\documentclass[a4paper]{scrartcl}
\usepackage[reqno]{amsmath}
\usepackage{amssymb, amsthm}
\usepackage{mathtools, MnSymbol,bbm}
\usepackage{IEEEtrantools}
\usepackage{mleftright}
\mleftright
\usepackage{pgfplots,tikz}
\usetikzlibrary{shapes, shapes.misc, snakes, calc, patterns, positioning}
\usepackage{fancybox}
\usepackage{csquotes}
\usepackage{enumitem}
\usepackage{hyperref}
\usepackage{pdfpages}

\setlist[enumerate, 1]{nosep, label=(\arabic*)}
\setlist[enumerate, 2]{nosep, label=(\roman*)}
\setlist[itemize, 1]{nosep, label={\bfseries $\bullet$}}
\setlist[itemize, 2]{nosep, label=$\circ$} 
\hypersetup{
    pdftitle={Free Probability Theory},
    pdfauthor={Roland Speicher},
    pdfsubject={Free Probability Theory},
    pdfkeywords={free probability theory, free cumulants, free convolution, random matrices, free group factors}
}

\makeatletter
\newcommand{\leqnomode}{\tagsleft@true}
\newcommand{\reqnomode}{\tagsleft@false}
\makeatother

\theoremstyle{plain} 
\newtheorem{theorem}{Theorem}[section]
\newtheorem{proposition}[theorem]{Proposition}
\newtheorem{lemma}[theorem]{Lemma}
\newtheorem{corollary}[theorem]{Corollary}

\theoremstyle{definition}
\newtheorem{definition}[theorem]{Definition}
\newtheorem{remark}[theorem]{Remark}
\newtheorem{example}[theorem]{Example}
\newtheorem{notation}[theorem]{Notation}
\newtheorem*{conclusion}{Conclusion}

\newtheorem{facts}[theorem]{Facts}

\newtheorem{auf}{Exercise}

\DeclareMathOperator{\trace}{Tr}
\DeclareMathOperator{\normalizedtrace}{tr}
\DeclareMathOperator{\realpart}{Re}
\DeclareMathOperator{\imaginarypart}{Im}
\DeclareMathOperator{\algebrageneratedby}{alg}
\DeclareMathOperator{\vonneumannalgebrageneratedby}{vN}
\DeclareMathOperator{\support}{supp}
\DeclareMathOperator{\range}{ran}

\newcommand{\complexnumbers}{\mathbb{C}}
\newcommand{\naturalnumbers}{\mathbb{N}}
\newcommand{\integers}{\mathbb{Z}}
\newcommand{\rationalnumbers}{\mathbb{Q}}
\newcommand{\realnumbers}{\mathbb{R}}
\newcommand{\equalperdefinition}{\coloneq}
\newcommand{\setofpartitionsof}{\mathcal{P}}
\newcommand{\setofnoncrossingpartitionsof}{{NC}}
\newcommand{\convergesindistributionto}{\overset{\mathrm{distr}}{\longrightarrow}}
\newcommand{\convergesindistributiontoo}{\longrightarrow}
\newcommand{\weingarten}{\mathrm{Wg}}
\newcommand{\expectation}{\mathbb{E}}

\newcommand{\unitcircle}{\mathbb{S}^1}
\newcommand{\GUE}{\hbox{\textsc{gue}}}
\newcommand{\GUEN}{\hbox{\textsc{gue(n)}}}
\newcommand{\Fundgr}{\mathcal F}

\newcommand{\ncpone}{{
  \begin{tikzpicture}[baseline=-2pt-0.375em]
       \node[inner sep=1pt] (n1) at (0em,0) {};
       \draw (n1) --++(0,-0.75em);      
     \useasboundingbox (-0.25em,-1em) rectangle (0.25em, 0.125em);
   \end{tikzpicture}
}}

\newcommand{\ncptwotwo}{{
\begin{tikzpicture}[baseline=-2pt-0.375em]
      \node[inner sep=1pt] (n1) at (0em,0) {};
      \node[inner sep=1pt] (n2) at (1em,0) {};
      \draw (n1) --++(0,-0.75em) -| (n2);
    \useasboundingbox (-0.25em,-0.875em) rectangle (1.25em, 0.125em);
  \end{tikzpicture}
}}
\newcommand{\ncptwoone}{{
  \begin{tikzpicture}[baseline=-2pt-0.375em]
      \node[inner sep=1pt] (n1) at (0em,0) {};
      \node[inner sep=1pt] (n2) at (1em,0) {};
      \draw (n1) --++(0,-0.75em);
      \draw (n2) --++(0,-0.75em);      
    \useasboundingbox (-0.25em,-0.875em) rectangle (1.25em, 0.125em);
  \end{tikzpicture}
}}
\newcommand{\ncpthreefive}{{
   \begin{tikzpicture}[baseline=-2pt-0.375em,xscale=0.7]
      \node[inner sep=1pt] (n1) at (0em,0) {};
      \node[inner sep=1pt] (n2) at (1em,0) {};
      \node[inner sep=1pt] (n3) at (2em,0) {};
      \draw (n1) --++(0,-0.75em) -| (n2);
      \draw ($(n2)+(0,-0.75em)$) -| (n3);      
    \useasboundingbox (-0.25em,-1em) rectangle (2.25em, 0.125em);
  \end{tikzpicture}
}}
\newcommand{\ncpthreefour}{{
  \begin{tikzpicture}[baseline=-2pt-0.375em,xscale=0.7]
      \node[inner sep=1pt] (n1) at (0em,0) {};
      \node[inner sep=1pt] (n2) at (1em,0) {};
      \node[inner sep=1pt] (n3) at (2em,0) {};
      \draw (n2) --++(0,-0.75em) -| (n3);
      \draw (n1) --++(0,-0.75em);      
    \useasboundingbox (-0.25em,-1em) rectangle (2.25em, 0.125em);
  \end{tikzpicture}
}}
\newcommand{\ncpthreethree}{{
  \begin{tikzpicture}[baseline=-2pt-0.5em,xscale=0.7]
      \node[inner sep=1pt] (n1) at (0em,0) {};
      \node[inner sep=1pt] (n2) at (1em,0) {};
      \node[inner sep=1pt] (n3) at (2em,0) {};
      \draw (n1) --++(0,-1em) -| (n3);
      \draw (n2) --++(0,-0.5em);      
    \useasboundingbox (-0.25em,-1.125em) rectangle (2.25em, 0.125em);
  \end{tikzpicture}
}}
\newcommand{\ncpthreetwo}{{
  \begin{tikzpicture}[baseline=-2pt-0.375em,xscale=0.7]
      \node[inner sep=1pt] (n1) at (0em,0) {};
      \node[inner sep=1pt] (n2) at (1em,0) {};
      \node[inner sep=1pt] (n3) at (2em,0) {};
      \draw (n1) --++(0,-0.75em) -| (n2);
      \draw (n3) --++(0,-0.75em);      
    \useasboundingbox (-0.25em,-1em) rectangle (2.25em, 0.125em);
  \end{tikzpicture}
}}
\newcommand{\ncpthreeone}{{
  \begin{tikzpicture}[baseline=-2pt-0.375em,xscale=0.7]
      \node[inner sep=1pt] (n1) at (0em,0) {};
      \node[inner sep=1pt] (n2) at (1em,0) {};
      \node[inner sep=1pt] (n3) at (2em,0) {};
      \draw (n1) --++(0,-0.75em);      
      \draw (n2) --++(0,-0.75em);
      \draw (n3) --++(0,-0.75em);      
    \useasboundingbox (-0.25em,-1em) rectangle (2.25em, 0.125em);
  \end{tikzpicture}
}}
\title{Free Probability Theory}
\author{Roland Speicher}
\date{}
\begin{document}
\includepdf{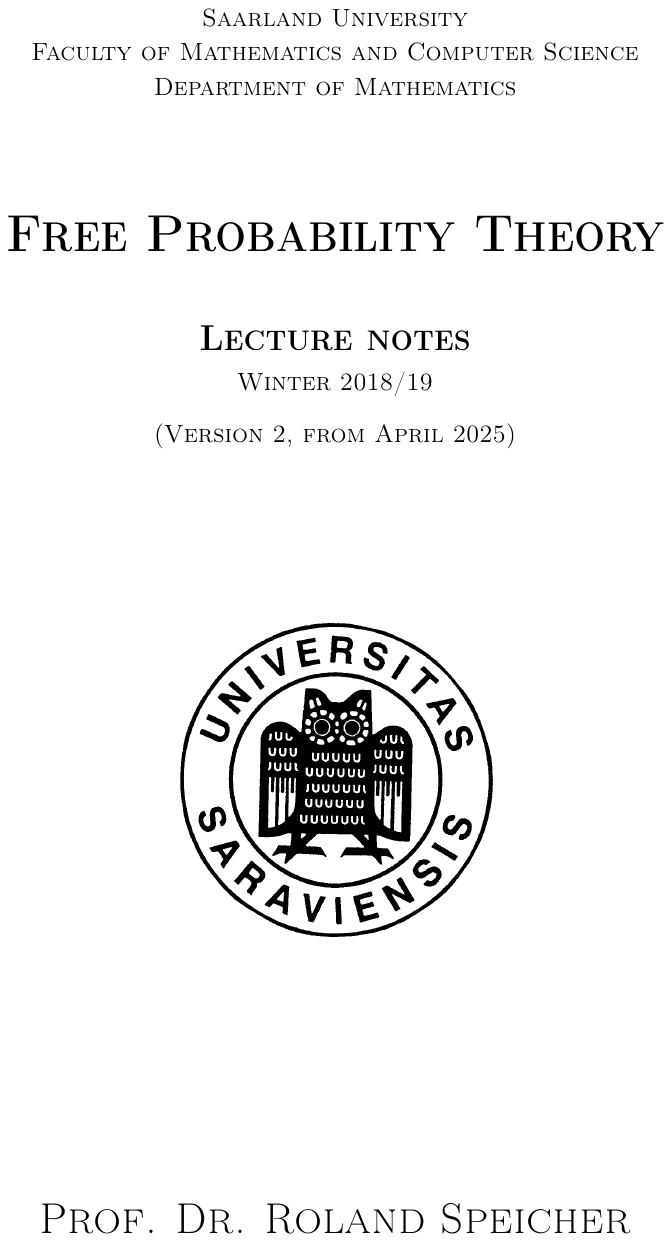}
\maketitle

\begin{abstract}
This in an introduction to free probability theory, covering the basic combinatorial and analytic theory, as well as the relations to random matrices and operator algebras. The material is mainly based on the two books of the lecturer, one joint with \href{http://rolandspeicher.com/literature/nica-speicher/}{Nica} \cite{NSp} and one joint with \href{http://rolandspeicher.com/literature/mingo-speicher/}{Mingo} \cite{MSp}. Free probability is here restricted to the scalar-valued setting, the operator-valued version is treated in the subsequent lecture series on \href{http://rolandspeicher.com/lectures/lectures-ncd19/}{``Non-Commutative Distributions''}.
The material here was presented in the winter term 2018/19 at Saarland University in 26 lectures of 90 minutes each. The lectures were recorded and can be found online at\\ \url{https://youtube.com/playlist?list=PLY11JnnnTUCYZni2Q7QNVa9hPGu77GK4M}
\end{abstract}
\quad
\vskip-2cm
\hskip-3.5cm
\begin{center} 
\centering
    \begin{tikzpicture}[scale=1]
      \useasboundingbox (-6,-4) rectangle (12,4);
      \node[draw=black, align=center] (a) at (1.5,-6) {free probability\\
\includegraphics[width=5cm]{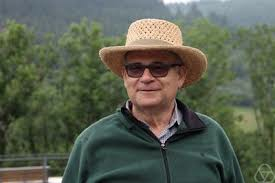}\\Dan Voiculescu};
      \node[draw=black, align=center] (b1) at (-4.5,0) {operator algebras\\
\includegraphics[width=2.5cm]{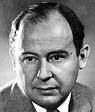}\\ von Neumann};
      \node[draw=black, align=center] (b2) at (7.5,0) {random matrices\\
\includegraphics[width=2.2cm]{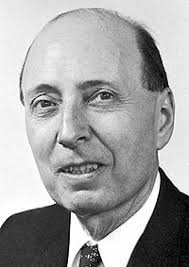}\\ Wigner};
      \draw (a) to (b1);
      \draw (a) to (b2);
      \draw[<->] (b1) to (b2);
    \end{tikzpicture}
 \end{center}
\thispagestyle{empty}
 
\newpage
\tableofcontents
\newpage
\setcounter{section}{-1}

%%%%%%%%%%%%%%%%%%%%%%%%%%%%%%%

\section{(Very Short) Introduction into Subject and History}
%[Introduction]
\begin{itemize}%[wide, itemsep=0.5em]
\item The field of \emph{Free Probability} was created by {Dan Voiculescu} in the 1980s.
\item {Voiculescu} isolated its central concept of \emph{{freeness}} or, synonymously, \emph{{free independence}} in the context of operator algebras.
\item The philosophy of free probability is to  investigate this notion in analogy to the concept of \enquote{independence} from (classical) probability theory. In this sense there are correspondences between 
  \begin{itemize}
  \item independence $\hateq$ free independence,
  \item central limit theorem $\hateq$ free central limit theorem,
  \item convolution $\hateq$ free convolution.
  \end{itemize}
\item Starting about 1990, a combinatorial theory of freeeness was developed by {Speicher} and by {Nica and Speicher} \href{http://rolandspeicher.com/literature/nica-speicher/}{\cite{NSp}}, featuring prominently
  \begin{itemize}
  \item the lattice of non-crossing partitions and
  \item free cumulants.
    \end{itemize}
    \item About 1991, {Voiculescu} discovered freeness also asymptotically for many kinds of {random matrices}.
      \begin{center}
    \centering
    \begin{tikzpicture}[scale=0.666]
      \useasboundingbox (-6,-4) rectangle (12,4);
      \node[ellipse, draw=black, align=center] (a) at (0,0) {abstract notion of\\ freeness};
      \node[draw=black, align=center] (b1) at (-3,3) {operator\\ algebras};
      \node[draw=black, align=center] (b2) at (3,3) {random\\ matrices};
      \node[rounded corners, draw=black, align=center] (c1) at (-3,-3) {combinatorial theory\\of freeness};
      \node[rounded corners, draw=black, align=center] (c2) at (3,-3) {analytical theory\\of freeness};
      \node[align=left, text width=3.3cm] (d1) at (9,3) {application of\\ random matrices,\\ e.g., in wireless networks};
      \node[align=left, text width=3.3cm] (d2) at (9,0) {asymptotics of\\ representation theory of \enquote{large} groups};
      \draw (a) to (b1);
      \draw (a) to (b2);
      \draw (a) to (c1);
      \draw (a) to (c2);
      \draw[->] (b2.east) to (d1.west);
      \draw[->] (b2.east) to (d2.west);
      \draw[<->] (b1) to (b2);
    \end{tikzpicture}
  \end{center}
  In the wake of {Voiculescu}'s discovery, the study of operator algebras was influenced by  random matrix theory. The option of modeling operator algebras asymptotically by random matrices lead to new results on von Neumann algebras, in particular on the so-called \enquote{free group factors}.
  \par
  Conversely, free probability brought to random matrix theory a conceptual approach and new tools for describing the asymptotic eigenvalue distribution of random matrices, in particular, for {functions} of several random matrices.
    \end{itemize}

%%%%%%%%%%%%%%%%%%%%%%%%%

\newpage
\section{The Notion of Freeness: Definition, Example, and Basic Properties}
\label{section-1}
%[The Notion of Freeness]

\begin{definition}[Voiculescu 1985]
  \label{definition:freeness}
The following are our basic definitions.
  \begin{enumerate}
  \item A \emph{non-commutative probability space} $(\mathcal{A},\varphi)$ consists of
    \begin{itemize}
    \item a unital (associative) algebra $\mathcal{A}$ (over $\complexnumbers$) and
    \item a unital linear functional $\varphi: \mathcal{A}\to \complexnumbers$ (meaning especially $\varphi(1)=1$).
    \end{itemize}
  \item Let $(\mathcal{A},\varphi)$ be a non-commutative probability space. Unital subalgebras $(\mathcal{A}_i)_{i\in I}$ of $\mathcal{A}$ are called \emph{free} (or \emph{freely independent}) in $(\mathcal A,\varphi)$ if
    %\begin{align*}
      $\varphi(a_1\ldots a_k)=0$
    %\end{align*}
    whenever 
    \begin{itemize}
    \item $k\in \naturalnumbers$,
    \item $i(j)\in I$ for all $j=1,\ldots,k$,
    \item $a_j\in \mathcal{A}_{i(j)}$ for all $j=1,\ldots,k$,      
    \item neighboring elements in $a_1\ldots a_k$ are from \enquote{different subalgebras}, which is to say
      \begin{align*}
        i(1)\neq i(2) \neq i(3) \neq \ldots \neq i(k-1) \neq i(k),
      \end{align*}
      (however, e.g., $i(1)=i(3)$, or, in particular, $i(1)=i(k)$ are allowed),
    \item $\varphi(a_j)=0$ for all $j=1,\ldots,k$.
    \end{itemize}
      Note that we do not require $\mathcal A_{i}\neq \mathcal A_{i'}$ for $i\neq i'$; cf.\ however Proposition~\ref{proposition:classical-and-free-independence} below.
  \end{enumerate}
\end{definition}

Voiculescu gave this definition in the context of von Neumann algebras of free products of groups. We will now present the algebraic version of this.

\begin{example}
  \label{example:group-algebra}
  Let $G$ be a group.
  \begin{enumerate}
  \item\label{example:group-algebra-1} Its \emph{group algebra} $\complexnumbers G$ is
 a complex vector space having a basis indexed by the elements of $G$, i.e.
      \begin{align*}
        \complexnumbers G \equalperdefinition \bigg\{ \sum_{g\in G}\alpha_g g \ \big\mid\ \alpha_g\in \complexnumbers \text{ for all }g\in G,
        %&\hspace{6.25em}
\alpha_g\neq 0 \text{ for only finitely many }g\in G\bigg\},
      \end{align*}
equipped with the multiplication
      \begin{align*}
        \Bigl(\sum_{g\in G} \alpha_g g\Bigr)\cdot \Bigl(\sum_{h\in G} \beta_h h\Bigr)\equalperdefinition \sum_{g,h\in G}\alpha_g\beta_h (gh)=\sum_{k\in G}\Bigl(\sum_{\substack{g,h\in G\\gh=k}}\alpha_g\beta_h\Bigr)k
      \end{align*}
      for all $(\alpha_g)_{g\in G}, (\beta_g)_{g\in G}$ such that $\alpha_g,\beta_g\in \complexnumbers$ for all $g\in G$ and  $\alpha_g\neq 0$ or $\beta_g\neq 0$  for only finitely many $g\in G$.
      Then, $\complexnumbers G$ is a unital algebra with unit
        $1=e=1\cdot e$,
      where $e$ is the neutral element of $G$.
    \item\label{example:group-algebra-2} On $\complexnumbers G$  we define the unital functional $\tau_G: \complexnumbers  G\to \complexnumbers$ by
      \begin{align*}
        \tau_G\Bigl(\sum_{g\in G}\alpha_g g\Bigr)\equalperdefinition \alpha_e
      \end{align*}
            for all $(\alpha_g)_{g\in G}$ such that $\alpha_g\in \complexnumbers$ for all $g\in G$ and  $\alpha_g\neq 0$ for only finitely many $g\in G$.\par
      The pair $(\complexnumbers G,\tau_G)$ is then a non-commutative probability space.

(We can identify elements $\sum_{g\in G}\alpha_gg$ of $\complexnumbers G$ with functions  $\alpha: G\to \complexnumbers$ of finite support by defining
        $\alpha(g)\equalperdefinition\alpha_g$.
      Multiplication in $\complexnumbers G$ then corresponds to convolution with respect to the counting measure of $G$.)
    \item\label{example:group-algebra-3} If $(G_i)_{i\in I}$ is a family of subgroups of $G$, then $\complexnumbers G_i$ is a unital subalgebra of $\complexnumbers G$ for every $i\in I$.
    \item\label{example:group-algebra-4} Subgroups $(G_i)_{i\in I}$ of $G$ are called \emph{free} in $G$ (in an algebraic sense) if there are \enquote{no non-trivial relations between different elements of the family}. This can be formulated in terms of a universal property. But one can also define it concretely as follows: For all $k\in \naturalnumbers$, all $i(1),\ldots,i(k)\in I$ and $g_1,\ldots,g_k\in G$ such that $g_j\in G_{i(j)}$ ($j=1,\ldots,k$) we have:
        $g_1\ldots g_k\neq e$,
      whenever $g_1,\ldots,g_k\neq e$ and $i(1)\neq i(2)\neq \ldots\neq i(k)$.
  \end{enumerate}
\end{example}

\begin{proposition}
  \label{proposition:freeness-group-algebras}
  Let $(G_i)_{i\in I}$ be subgroups of a group $G$. Then the following statements are equivalent:
  \begin{enumerate}
  \item\label{proposition:freeness-group-algebras-1} The subgroups $(G_i)_{i\in I}$ are free in $G$.
  \item\label{proposition:freeness-group-algebras-2} The subalgebras $(\complexnumbers G_i)_{i\in I}$ are freely independent in the non-commutative probability space $(\complexnumbers G,\tau_G)$.
  \end{enumerate}
\end{proposition}
\begin{proof}
  \ref{proposition:freeness-group-algebras-2} $\Rightarrow$ \ref{proposition:freeness-group-algebras-1}: Let $k\in \naturalnumbers$ be arbitrary and let $i(1),\ldots, i(k)\in I$ with $i(1)\neq i(2)\neq \ldots \neq i(k)$ and $g_1,\ldots, g_k\in G$ be such that $g_i\in G_{i(j)}$ for every $j=1,\ldots k$ and such that $g_j\neq e$ for every $j=1,\ldots, k$. For every $j\in I$, since we can embed $G_j\hookrightarrow \complexnumbers G_{j}$, the condition $g_j\neq e$ requires $\tau_G(g_j)=0$ by definition of $\tau_G$. Assuming statement \ref{proposition:freeness-group-algebras-2} therefore implies $\tau_G(g_1\ldots g_k)=0$, which particularly necessitates $g_1\ldots g_k\neq e$.
  \par
  \ref{proposition:freeness-group-algebras-1} $\Rightarrow$ \ref{proposition:freeness-group-algebras-2}:  Conversely, consider $k\in \naturalnumbers$, $i(1),\ldots, i(k)\in I$ with $i(1)\neq i(2)\neq \ldots\neq i(k)$ and for every $j=1,\ldots,k$ an element
  \begin{align}
    a_j\equalperdefinition \sum_{g\in G_{i(j)}}\alpha_g^{(j)}g\in \complexnumbers G_{i(j)} 
\qquad
\text{with}\qquad
    \label{eq:proposition-freeness-group-algebras-1}
    \alpha_e^{(j)}=\tau_G(a_j)=0.
  \end{align}
  Then, by definition of the multiplication in $\complexnumbers G$,
  \begin{align*}
    \tau_G(a_1\ldots a_k)&=\tau_G\Bigl(\bigl(\sum_{g_1\in G_{i(1)}}\alpha_g^{(1)}g_1\bigr)\ldots \bigl(\sum_{g_k\in G_{i(k)}}\alpha_g^{(k)}g_k\bigr)\Bigr)\\
\quad\\
                             &=\sum_{g_1\in G_{i(1)},\ldots,g_k\in G_{i(k)}}\underset{(\ast)}{\underbrace{\alpha_{g_1}^{(1)}\ldots \alpha_{g_k}^{(k)}}}\tau_G(g_1\ldots g_k)\\
\quad\\&=0,
  \end{align*}
  where the last equality is justified by the following argument: For any $g_1,\ldots, g_k\in G$ with $g_j\in G_{i(j)}$ for all $j=1,\ldots,k$, assuming that the product $\alpha_{g_1}^{(1)}\ldots \alpha_{g_k}^{(k)}$ in $(\ast)$ is non-zero requires $\alpha_{g_j}^{(j)}\neq 0$ for all $j=1,\ldots, k$. Due to the assumption that the subgroups $G_i$ are free, the latter is only possible if $g_j\neq e$ for all $j=1,\ldots,k$. Hence, if so, then supposing statement \ref{proposition:freeness-group-algebras-1} implies $g_1\ldots g_k\neq e$ and thus $\tau_G(g_1\ldots g_k)=0$ by definition of $\tau_G$.
\end{proof}

\begin{remark}
  On the level of the group algebras, Proposition~\ref{proposition:freeness-group-algebras} is just a rewriting of the algebraic condition of \enquote{absence of non-trivial relations} in terms of the linear functional $\tau_G$. If one goes over to the corresponding $C^\ast$- or von Neumann algebras, which consist of infinite sums, then the algebraic condition does not make sense anymore, while the condition in terms of $\tau_G$  survives for those operator algebras.
\end{remark}

\begin{proposition}
  \label{proposition:freeness-state-factorization}
  Let $(\mathcal{A},\varphi)$ be a non-commutative probability space and let $(\mathcal{A}_i)_{i\in I}$  be a family of free unital subalgebras of $\mathcal A$. Let $\mathcal{B}\equalperdefinition \algebrageneratedby (\bigcup_{i\in I}\mathcal{A}_i)$ be the subalgebra generated by all $(\mathcal{A}_i)_{i\in I}$. Then $\left.\varphi\right|_{\mathcal{B}}$ is uniquely determined by $(\varphi|_{\mathcal{A}_i})_{i\in I}$ and by the free independence condition.\\ (That means, if $\psi$ is such that $(\mathcal A,\psi)$ is a non-commutative probability space, such that $(\mathcal A_i)_{i\in I}$ is freely independent in $(\mathcal A,\psi)$ and such that $\left.\psi\right|_{\mathcal{A}_i}=\left.\varphi\right|_{\mathcal{A}_i}$ for all $i\in I$, then $\varphi\vert_\mathcal{B}=\psi\vert_\mathcal{B}$.)
\end{proposition}
\begin{proof}
  Elements in $\mathcal{B}$ are linear combinations of products $a_1\ldots a_k$ where $k\in \naturalnumbers$, $i(1),\ldots, \allowbreak i(k)\in I$ and $a_j\in \mathcal{A}_{i(j)}$ for every $j=1,\ldots, k$. We can also assume that $i(1)\neq i(2)\neq \ldots \neq i(k)$ by combining neighboring factors. Consider now such a product $a_1\ldots a_k\in \mathcal{B}$. Then, we  have to show that $\varphi(a_1\ldots a_k)$ is uniquely determined by $(\varphi |_{\mathcal{A}_i})_{i\in I}$. This we prove by induction over $k$.
  \par
  The base case $k=1$ is clear, since $a_1\in \mathcal{A}_{i(1)}$ by assumption. For general $k\in \naturalnumbers$, define
  \begin{align*}
    a^\circ_j\equalperdefinition a_j-\varphi(a_j)\cdot 1\in \mathcal{A}_{i(j)} \quad\text{for all } j=1,\ldots, k,
  \end{align*}
  where we have relied on the assumption that the subalgebras $(\mathcal A_i)_{i\in I}$ are unital. It then follows by linearity of $\varphi$ that
  \begin{align}
    \label{eq:proposition-freeness-state-factorization-1}
    \varphi(a_1\ldots a_k)=\varphi [(a_1^\circ+\varphi(a_1)\cdot 1)\ldots (a_k^\circ + \varphi(a_k)\cdot 1)]
    =\varphi(a_1^\circ a_2^\circ \ldots a_k^\circ)+\sum \ldots ,
  \end{align}
  where the remaining summands in \eqref{eq:proposition-freeness-state-factorization-1}  are all of the form
  \begin{align*}
    \varphi(a_{s(1)}^\circ)\ldots \varphi(a_{s(l)}^\circ)\cdot \varphi(a_{t(1)}^\circ \ldots a_{t(m)}^\circ)
  \end{align*}
  for some $l,m\in \naturalnumbers$,  $s(1),\ldots,s(l),t(1),\ldots,t(m)\in \{1,\ldots, k\}$ with, crucially, $m<k$. The latter namely ensures that the value of $\varphi$ at the product $a_{t(1)}^\circ \ldots a_{t(m)}^\circ$ is  determined by $(\varphi|_{\mathcal{A}_i})_{i\in I}$ by the induction hypothesis. Hence, \eqref{eq:proposition-freeness-state-factorization-1}  implies that the same is true for $\varphi(a_1\ldots a_k)$, since the term $\varphi(a_1^\circ a_2^\circ \ldots a_k^\circ)$ is zero by the definition of freeness.
\end{proof}

\begin{example}
  \label{example:moment-formulas-for-small-order}
  Let $\mathcal{A}_1$, $\mathcal{A}_2$ be free unital subalgebras  in $(\mathcal{A},\varphi)$.
  \begin{enumerate}
  \item \label{example:moment-formulas-for-small-order-1} Consider $a\in \mathcal{A}_1$ and $b\in \mathcal{A}_2$. Then, 
    \begin{align*}
      0&=\varphi[(a-\varphi(a)\cdot 1)(b-\varphi(b)\cdot 1)]\\
       &=\varphi(ab)-\varphi(a\cdot 1)\varphi(b)-\varphi(a)\varphi(1\cdot b)+\varphi(a)\varphi(b)\varphi(1)\\
       &=\varphi(ab)-\varphi(a)\varphi(b)
    \end{align*}
    implies
    \begin{align*}
      \varphi(ab)=\varphi(a)\varphi(b).
    \end{align*}
  \item \label{example:moment-formulas-for-small-order-2} Similarly, for $a_1,a_2\in \mathcal{A}_1$ and $b,b_1,b_2\in \mathcal{A}_2$, from
    \begin{align*}
      0=\varphi[(a_1-\varphi(a_1)\cdot 1)(b-\varphi(b)\cdot 1)(a_2-\varphi(a_2)\cdot 1)]
    \end{align*}
    one can derive
    \begin{align*}
      \varphi(a_1 ba_2)=\varphi(a_1a_2)\varphi(b).
    \end{align*}
    And, likewise, 
    \begin{align*}
      0=\varphi[(a_1-\varphi(a_1)\cdot 1)(b-\varphi(b_1)\cdot 1)(a_2-\varphi(a_2)\cdot 1)(b_2-\varphi(b_2)\cdot 1)]
    \end{align*}
    allows one to conclude (see Assignment~\hyperref[assignment-1]{1}, Exercise~2)
    \begin{align*}
      \varphi(a_1b_1a_2b_2)&=\varphi(a_1a_2)\varphi(b_1)\varphi(b_2)\\
                           &\phantom{{}={}}+\varphi(a_1)\varphi(a_2)\varphi(b_1b_2)\\
      &\phantom{{}={}}-\varphi(a_1)\varphi(a_2)\varphi(b_1)\varphi(b_2).
    \end{align*}
    \item \label{example:moment-formulas-for-small-order-3} For longer alternating products there is in the same way a formula, but the calculation via Proposition~\ref{proposition:freeness-group-algebras} is getting too complex; and there is no apparent structure of the final result.
  \end{enumerate}
\end{example}

\begin{remark}
  We can consider the formulas in Example~\ref{example:moment-formulas-for-small-order} as \enquote{non-commutative} analogues of formulas from  probability theory for the calculation of joint moments of independent random variables. Consider a classical probability space
    $(\Omega,\mathcal F,\mathbb P)$
  (meaning that $\Omega$ is a set of \enquote{outcomes}, $\mathcal F$ a sigma-algebra of \enquote{events} over $\Omega$ and $\mathbb P$ a probability measure on $(\Omega,\mathcal F)$, the \enquote{likelihood} of events.)
  Then, we choose as a unital algebra 
  \begin{align*}
    \mathcal A\equalperdefinition L^\infty(\Omega,\mathbb P)
  \end{align*}
  the algebra of bounded measurable functions (\enquote{random variables}) $X:\Omega\to \complexnumbers$
and as a unital linear functional   $\varphi$ on $\mathcal A$
the \enquote{expectation} $\mathbb{E}$ of random variables $X:\Omega\to \complexnumbers$ with respect to $\mathbb P$:
\begin{align*}
  \varphi(X)=\mathbb{E}[X]=\int_\Omega X(\omega) \, d\mathbb{P}(\omega).
\end{align*}
(Then, $\varphi(1)=1$ corresponds to $\mathbb{P}(\Omega)=1$.)
\par
Random variables $X,Y$ are independent if their joint moments factorize into the moments of the individual variables:
\begin{align*}
  \mathbb{E}[X^nY^m]=  \mathbb{E}[X^n]\mathbb{E}[Y^m]
\end{align*}
for all $m,n\in \naturalnumbers$.
Note that classical random variables commute.
\par
Furthermore, we have some more \enquote{positivity} structure in such a context. Namely our algebra carries a $\ast$-structure and the expectation is positive and faithful. Often  such additional structure is available in a non-commutative context as well. The notion of freeness is compatible with this extra structure.
\end{remark}
\begin{definition}
  Let $(\mathcal A,\varphi)$ be a non-commutative probability space.
  \begin{enumerate}
  \item If $\varphi$ is a \emph{trace}, i.e.\ if
    \begin{align*}
      \varphi(ab)=\varphi(ba) \quad \text{for all } a,b\in \mathcal{A},
    \end{align*}
    then we call $(\mathcal{A},\varphi)$ a \emph{tracial} non-commutative probability space.
  \item If $\mathcal{A}$ is a $\ast$-algebra and $\varphi$ is \emph{positive}, i.e.\ if
    \begin{align*}
      \varphi(a^\ast a)\geq 0 \quad \text{for all } a\in \mathcal{A},
    \end{align*}
    then we call $\varphi$ a \emph{state} and $(\mathcal{A},\varphi)$ a $\ast$-probability space.
    \par
    A state $\varphi$ is \emph{faithful} if for all $a\in \mathcal A$
    \begin{align*}
        \varphi(a^\ast a)=0 \implies a=0.
    \end{align*}
  \item Elements in $\mathcal{A}$ are called (non-commutative) \emph{random variables}. The \emph{moments} of a random variable $a\in \mathcal{A}$ are the numbers $(\varphi(a^n))_{n\in \naturalnumbers}$. The joint moments of a family $(a_1,\ldots, a_s)$ of random variables $a_1,\ldots, a_s \in \mathcal{A}$, $s\in \naturalnumbers$, is the collection of all numbers
    \begin{align*}
      \varphi(a_{r(1)}\ldots a_{r(n)}), \quad\text{where } n\in \naturalnumbers, r(1),\ldots, r(n)\in \{1,\ldots, s\}.
    \end{align*}
    \par
    If $(\mathcal A,\varphi)$ is a $\ast$-probability space, then the $\ast$-moments of a random variable $a\in \mathcal A$ are the joint moments of $(a,a^\ast)$ and the $\ast$-moments of $(a_1,\ldots, a_s)$ for $a_1,\ldots,a_s\in \mathcal A$, $s\in \naturalnumbers$, are the joint moments of $(a_1,a_1^\ast,\ldots, a_s,a_s^\ast)$.
    \par
    The ($\ast$-)distribution of $a$ or of $(a_1,\ldots, a_s)$ is the collection of all corresponding ($\ast$-)moments.
    \item Random variables $(a_i)_{i\in I}$ in $\mathcal A$ are called \emph{free} if the generated unital subalgebras $(\algebrageneratedby(1,a_i))_{i\in I}$ are free. In case $(\mathcal A,\varphi)$ is a $\ast$-probability space, then $(a_i)_{i\in I}$ are  \emph{$\ast$-free} if the generated unital $\ast$-subalgebras $(\algebrageneratedby(1,a_i,a_i^\ast))_{i\in I}$ are free.
  \end{enumerate}
\end{definition}

\begin{remark}
  \label{remark:moment-formulas-classical-and-free}
  So, we can now say: Freeness is a rule for calculating joint moments of free variables from the moments of the individual variables. For example, if  $a$ and $b$ are free from each other, then, as seen in Example~\ref{example:moment-formulas-for-small-order}, 
  \begin{align*}
    \varphi(ab)=\varphi(a)\varphi(b)
  \end{align*}
  and
  \begin{align*}
    \varphi(abab)=\varphi(a^2)\varphi(b)^2+\varphi(a)^2\varphi(b^2)-\varphi(a)^2\varphi(b)^2.
  \end{align*}
  Note that the first factorization is the same as for (classically) independent random variables. The second, however, is not compatible with commutativity.
\end{remark}

\begin{proposition}
  \label{proposition:classical-and-free-independence}
  Let $(\mathcal{A},\varphi)$ be a $\ast$-probability space and $\varphi$ faithful. Assume that the self-adjoint random variables $x,y\in \mathcal{A}$
  \begin{itemize}
  \item are free from each other and
  \item commute with each other.
  \end{itemize}
  Then, at least one of them must be a constant, i.e.\
  \begin{align*}
    x=\varphi(x)\cdot 1\quad\text{or} \quad y=\varphi(y)\cdot 1.
  \end{align*}
\end{proposition}
\begin{proof}
  Since $x$ and $y$ commute, Remark~\ref{remark:moment-formulas-classical-and-free} provides us with two distinct rules for calculating $\varphi(xyxy)$, justifying respectively the first and the last identity in 
  \begin{align*}
    \varphi(x^2)\varphi(y^2)&=\varphi(x^2y^2)\\
                            &=\varphi(xyxy)\\ 
    &=\varphi(x^2)\varphi(y)^2+\varphi(x)^2\varphi(y^2)-\varphi(x)^2\varphi(y)^2.
  \end{align*}
  It follows
  \begin{align*} 0&=\varphi(x^2)\varphi(y^2)+\varphi(x)^2\varphi(y)^2-\varphi(x^2)\varphi(y)^2-\varphi(x)^2\varphi(y^2)\\
    &=[\varphi(x^2)-\varphi(x)^2][\varphi(y^2)-\varphi(y)^2].
  \end{align*}
  Thus, at least one of the factors must vanish, say
  \begin{align*}
    0&=\varphi(x^2)-\varphi(x)^2\\
     &=\varphi[\hspace{-0.6em}\underset{\displaystyle =(x-\varphi(x)\cdot 1)^\ast}{\underbrace{(x-\varphi(x)\cdot 1)}}\hspace{-0.5em}(x-\varphi(x)\cdot 1)]\\
    &=\varphi[(x-\varphi(x)\cdot 1)^\ast(x-\varphi(x)\cdot 1)].
  \end{align*}
  Because $\varphi$ is faithful, we conclude
        $x-\varphi(x)\cdot 1=0$,
  proving  $x=\varphi(x)\cdot 1$ and thus the claim.
\end{proof}

\begin{proposition}
  \label{proposition:constants-are-free-from-anything}
  Let $(\mathcal{A},\varphi)$ be a non-commutative probability space. Then, constants are \enquote{free from anything}: For any unital subalgebra $\mathcal{B}$ of $\mathcal{A}$ we have that $\mathcal{B}$ and $\complexnumbers \cdot 1$  are free in $(\mathcal{A},\varphi)$.
\end{proposition}
\begin{proof}
Let $k\in \naturalnumbers$ and $a_1,\ldots, a_k$ be as in Definition~\ref{definition:freeness} of free independence. The case $k=1$ is trivial. Hence, let $k\geq 2$. But then, for at least for one $j=1,\ldots,k$ it must hold that $a_j\in \complexnumbers \cdot 1$. Thus, the assumption $\varphi(a_j)=0$ implies $a_j=0$ for this $j$. It follows $a_1\ldots a_k=0$ and thus $\varphi(a_1\ldots a_k)=0$.
\end{proof}

\begin{conclusion}
  What we are doing here has a kind of stochastic flavor, but our random variables typically do not commute. In this sense, free probability is a non-commutative probability theory.
\end{conclusion}

\newpage

\section{Emergence of the Combinatorics of Free Probability Theory: Free (Central) Limit Theorem}%[Free Central Limit Theorem]
\begin{remark}
\label{paragraph:non-crossing-moment-factorization}
 Let the unital subalgebras $(\mathcal{A}_i)_{i\in I}$ be free in a non-commutative probability space $(\mathcal{A},\varphi)$  and let $\mathcal{A}=\algebrageneratedby(\bigcup_{i\in I}\mathcal A_i)$. Then, by Proposition~\ref{proposition:freeness-state-factorization}, the functional $\varphi$ is completely determined by $(\varphi|_{\mathcal{A}_i})_{i\in I}$. We have to understand better the structure of those formulas relating the values of $\varphi$ to those of $(\varphi|_{\mathcal{A}_i})_{i\in I}$ as in Example~\ref{example:moment-formulas-for-small-order}. Since elements of $\mathcal{A}$ are linear combinations of products $a_1\ldots a_k$ for $k\in \naturalnumbers$, $i(1),\ldots, i(k)\in I$ with $i(1)\neq i(2)\neq \ldots\neq i(k)$ and $a_j\in \mathcal{A}_{i(j)}$ for all $j=1,\ldots,k$, and since $\varphi$ is linear, it suffices to understand the formulas for such products $\varphi(a_1\ldots a_k)$.
\par
Freeness tells us the following:
\begin{itemize}
\item If, in addition, $\varphi(a_1)=\varphi(a_2)=\ldots=\varphi(a_k)=0$, then  $\varphi(a_1\ldots a_k)=0$.
  \item The general case can be reduced to this, but might give complicated formulas.
  \end{itemize}

  We have seen both
  \begin{itemize}
  \item easy factorization formulas
    \begin{align}
      \label{equation:alternating-moment-fomula-three}
      \varphi(a_1ba_2)=\varphi(a_1a_2)\varphi(b) \quad \text{if $\{a_1,a_2\}$ and $b$ are free}
    \end{align}
  \item and complicated formulas with no apparent factorization structure, and many additive terms:
    \begin{align}
      \label{equation:alternating-moment-fomula-four}
      \varphi(a_1b_1a_2b_2)&=\varphi(a_1a_2)\varphi(b_1)\varphi(b_2)\notag\\
                           &\phantom{{}={}}+\varphi(a_1)\varphi(a_2)\varphi(b_1b_2)      \\
      &\phantom{{}={}}-\varphi(a_1)\varphi(a_2)\varphi(b_1)\varphi(b_2)\notag
    \end{align}
    if $\{a_1,a_2\}$ and $\{b_1,b_2\}$ are free.
  \end{itemize}
  \par
  Note that formula \eqref{equation:alternating-moment-fomula-three} has a \enquote{nested} structure
  \begin{center}\begin{tikzpicture}
    \begin{scope}[xshift=-0.5pt]
    \draw (0em,0) -- ++ (0,-1.5em) -| (2em,0);
    \draw (1em,0) -- ++ (0,-1em);
    \end{scope}
    \node[fill=white, inner sep=1pt] at (1em,0) {$a_1ba_2$};    
  \end{tikzpicture}
\end{center}
  with corresponding factorization
  \begin{center}
    \begin{tikzpicture}
      \begin{scope}[xshift=-4em]
        \draw (0em,0) -- ++ (0,-1.5em) -| (2em,0);
        \draw (1em,0) -- ++ (0,-1em);
      \end{scope}
      \begin{scope}[xshift=1.375em]
        \draw (0em,0) -- ++ (0,-1em) -| (1em,0);
      \end{scope}            
      \begin{scope}[xshift=3.625em]
        \draw (1em,0) -- ++ (0,-1em);
      \end{scope}
      \node[fill=white, inner sep=1pt] at (0,0) {$\varphi(a_1ba_2)=\varphi(a_1a_2)\varphi(b)$};  
    \end{tikzpicture}
\end{center}
and that this can be iterated to more complicated \enquote{nested} situations: For example, if $\{a_1,a_2,a_3\},\{b_1,b_2\},\{c\},\{d\}$ are free, then
\begin{align*}
  \begin{tikzpicture}[baseline=-2pt]
    \begin{scope}[xshift=-2.2em, xscale=0.8]
      \draw (0,0em) -- ++ (0,-2em) -- ++(4em,0) -- ++(0,2em);
      \draw (4em,-2em) -| ++(2em,2em);
      \draw (5em,0) -- ++ (0,-1em);
      \draw (1em,0) -- ++ (0,-1.5em) -| (3em,0);
      \draw (2em,0) -- ++ (0,-1em);      
    \end{scope}
     \node[fill=white, inner sep=1pt] at (0,0) {$\varphi(a_1b_1cb_2a_2da_3)$};  
   \end{tikzpicture}
  &= \varphi(\hspace{-0.5em}\underset{\displaystyle =\colon a_1'}{ \underbrace{a_1}}\hspace{-0.5em}\cdot \underset{\displaystyle=\colon b'}{\underbrace{(b_1cb_2)}}\cdot \underset{\displaystyle=\colon a_2'}{\underbrace{(a_2da_3)}}), \, \begin{tikzpicture}[baseline=1em+2pt]
    \node[align=left] at (0em,0.75em) {where $\{a_1',a_2'\},\{b'\}$ are free,\\ as $\{a_1,a_2,a_3,d\}$, $\{b_1,b_2,c\}$ are free,};    \end{tikzpicture}\\
    &\hspace{-0.15em}\overset{\text{\eqref{equation:alternating-moment-fomula-three}}}{=} \varphi(a_1'a_2')\varphi(b')\\
    &=
    \begin{tikzpicture}[baseline=-2pt]
    \begin{scope}[xshift=-0.7em, xscale=0.75]
      \draw (0,0em) -- ++ (0,-1.5em) -- ++(1em,0) -- ++(0,1.5em);
      \draw (1em,-1.5em) -| ++(2em,1.5em);
      \draw (2em,0) -- ++ (0,-1em);
    \end{scope}
    \node[fill=white, inner sep=1pt] at (0,0) {$\varphi(a_1a_2da_3)$};
    %\draw [brown] (-2.5em,-1.625em) rectangle (2.5em, 0.75em);
    \useasboundingbox (-2.5em,-1.625em) rectangle (2.5em, 0.75em);
   \end{tikzpicture}
      \begin{tikzpicture}[baseline=-2pt]
          \begin{scope}[xshift=0.5em,xscale=0.75]
    \draw (0em,0) -- ++ (0,-1.5em) -| (2em,0);
    \draw (1em,0) -- ++ (0,-1em);
    \end{scope}
    \node[fill=white, inner sep=1pt] at (1em,0) {$\varphi(b_1cb_2)$};
    %\draw [brown] (-0.875em,-1.625em) rectangle (3em, 0.75em);
    \useasboundingbox (-0.875em,-1.625em) rectangle (3em, 0.75em); 
  \end{tikzpicture} \\
    &=
          \begin{tikzpicture}[baseline=-2pt]
    \begin{scope}[xshift=-0.7em, xscale=1]
      \draw (0,0em) -- ++ (0,-1em) -- ++(1em,0) -- ++(0,1em);
      \draw (1em,-1em) -| ++(1em,1em);
    \end{scope}
    \node[fill=white, inner sep=1pt] at (0,0) {$\varphi(a_1a_2a_3)$};
    %\draw [brown] (-2.25em,-1.125em) rectangle (2.25em, 0.75em);
    \useasboundingbox (-2.25em,-1.125em) rectangle (2.25em, 0.75em);
  \end{tikzpicture}
            \begin{tikzpicture}[baseline=-2pt]
          \begin{scope}[xshift=0.375em,xscale=1]
    \draw (1em,0) -- ++ (0,-1em);
    \end{scope}
    \node[fill=white, inner sep=1pt] at (1em,0) {$\varphi(d)$};
    %\draw [brown] (-0.05em,-1.125em) rectangle (2.075em, 0.75em);
    \useasboundingbox (-0.05em,-1.125em) rectangle (2.075em, 0.75em); 
  \end{tikzpicture}
                \begin{tikzpicture}[baseline=-2pt]
    \begin{scope}[xshift=-0.125em, xscale=1]
      \draw (0em,0em) -- ++ (0,-1em) -| ++(1em,1em);
    \end{scope}
    \node[fill=white, inner sep=1pt] at (0,0) {$\varphi(b_1b_2)$};
    %\draw [brown] (-1.65em,-1.125em) rectangle (1.65em, 0.75em);
    \useasboundingbox (-1.65em,-1.125em) rectangle (1.65em, 0.75em);
  \end{tikzpicture}
                  \begin{tikzpicture}[baseline=-2pt]
          \begin{scope}[xshift=0.375em,xscale=1]
    \draw (1em,0) -- ++ (0,-1em);
    \end{scope}
    \node[fill=white, inner sep=1pt] at (1em,0) {$\varphi(c)$};
    %\draw [brown] (-0.05em,-1.125em) rectangle (2.075em, 0.75em);
    \useasboundingbox (-0.05em,-1.125em) rectangle (2.075em, 0.75em); 
  \end{tikzpicture}.
\end{align*}
Formula \eqref{equation:alternating-moment-fomula-four} on the other hand has no \enquote{nested} structure. Instead, it is \enquote{crossing}:
\begin{center}
    \begin{tikzpicture}[baseline=-2pt]
    \begin{scope}[xshift=-1.5em, xscale=0.95]
      \draw (0,0em) -- ++ (0,-1em) -- ++(2em,0) -- ++(0,1em);
      \draw (1em,0em) -- ++ (0,-1.5em) -- ++(2em,0) -- ++(0,1.5em);
    \end{scope}
    \node[fill=white, inner sep=1pt] at (0,0) {$a_1b_1a_2b_2$};
    %\draw [brown] (-2em,-1.625em) rectangle (2em, 0.75em);
    \useasboundingbox (-2em,-1.625em) rectangle (2em, 0.75em);
   \end{tikzpicture}  
 \end{center}
 This \enquote{crossing moment}
 \begin{center}
    \begin{tikzpicture}[baseline=-2pt]
    \begin{scope}[xshift=-1.25em, xscale=0.95]
      \draw (0,0em) -- ++ (0,-1em) -- ++(2em,0) -- ++(0,1em);
      \draw (1em,0em) -- ++ (0,-1.5em) -- ++(2em,0) -- ++(0,1.5em);
    \end{scope}
    \node[fill=white, inner sep=1pt] at (0,0) {$\varphi(a_1b_1a_2b_2)$};
    %\draw [brown] (-2.625em,-1.625em) rectangle (2.625em, 0.75em);
    \useasboundingbox (-2.625em,-1.625em) rectangle (2.625em, 0.75em);
   \end{tikzpicture}  
 \end{center}
 does not factorize. Actually the corresponding product $\varphi(a_1a_2)\varphi(b_1b_2)$ does not show up on the right hand side of \eqref{equation:alternating-moment-fomula-four} at all. Rather, there only \enquote{nested} contributions appear:
 \begin{align*}
       \begin{tikzpicture}[baseline=-2pt]
    \begin{scope}[xshift=-1.25em, xscale=0.95]
      \draw (0,0em) -- ++ (0,-1em) -- ++(2em,0) -- ++(0,1em);
      \draw (1em,0em) -- ++ (0,-1.5em) -- ++(2em,0) -- ++(0,1.5em);
    \end{scope}
    \node[fill=white, inner sep=1pt] at (0,0) {$\varphi(a_1b_1a_2b_2)$};
    %\draw [brown] (-2.625em,-1.625em) rectangle (2.625em, 0.75em);
    \useasboundingbox (-2.625em,-1.625em) rectangle (2.625em, 0.75em);
  \end{tikzpicture}  &=\begin{tikzpicture}[baseline=-2pt]
    \begin{scope}[xshift=-0.125em, xscale=1]
      \draw (0em,0em) -- ++ (0,-1em) -| ++(1em,1em);
    \end{scope}
    \node[fill=white, inner sep=1pt] at (0,0) {$\varphi(a_1a_2)$};
    %\draw [brown] (-1.65em,-1.125em) rectangle (1.65em, 0.75em);
    \useasboundingbox (-1.65em,-1.125em) rectangle (1.65em, 0.75em);
  \end{tikzpicture}
                  \begin{tikzpicture}[baseline=-2pt]
          \begin{scope}[xshift=0.375em,xscale=1]
    \draw (1em,0) -- ++ (0,-1em);
    \end{scope}
    \node[fill=white, inner sep=1pt] at (1em,0) {$\varphi(b_1)$};
    %\draw [brown] (-0.05em,-1.125em) rectangle (2.075em, 0.75em);
    \useasboundingbox (-0.05em,-1.125em) rectangle (2.075em, 0.75em); 
  \end{tikzpicture}\begin{tikzpicture}[baseline=-2pt]
                                 \begin{scope}[xshift=0.375em,xscale=1]
    \draw (1em,0) -- ++ (0,-1em);
    \end{scope}
    \node[fill=white, inner sep=1pt] at (1em,0) {$\varphi(b_2)$};
    %\draw [brown] (-0.05em,-1.125em) rectangle (2.075em, 0.75em);
    \useasboundingbox (-0.05em,-1.125em) rectangle (2.075em, 0.75em); 
                       \end{tikzpicture}\\[-0.5em]
                       &\phantom{{}={}} +            \begin{tikzpicture}[baseline=-2pt]                      \begin{scope}[xshift=0.375em,xscale=1]
    \draw (1em,0) -- ++ (0,-1em);
    \end{scope}
    \node[fill=white, inner sep=1pt] at (1em,0) {$\varphi(a_1)$};
    %\draw [brown] (-0.05em,-1.125em) rectangle (2.075em, 0.75em);
    \useasboundingbox (-0.05em,-1.125em) rectangle (2.075em, 0.75em); 
  \end{tikzpicture}\begin{tikzpicture}[baseline=-2pt]
                                   \begin{scope}[xshift=0.375em,xscale=1]
    \draw (1em,0) -- ++ (0,-1em);
    \end{scope}
    \node[fill=white, inner sep=1pt] at (1em,0) {$\varphi(a_2)$};
    %\draw [brown] (-0.05em,-1.125em) rectangle (2.075em, 0.75em);
    \useasboundingbox (-0.05em,-1.125em) rectangle (2.075em, 0.75em); 
                       \end{tikzpicture} \begin{tikzpicture}[baseline=-2pt]
    \begin{scope}[xshift=-0.125em, xscale=1]
      \draw (0em,0em) -- ++ (0,-1em) -| ++(1em,1em);
    \end{scope}
    \node[fill=white, inner sep=1pt] at (0,0) {$\varphi(b_1b_2)$};
    %\draw [brown] (-1.65em,-1.125em) rectangle (1.65em, 0.75em);
    \useasboundingbox (-1.65em,-1.125em) rectangle (1.65em, 0.75em);
  \end{tikzpicture}\\
                         &\phantom{{}={}} -          \begin{tikzpicture}[baseline=-2pt]                        \begin{scope}[xshift=0.375em,xscale=1]
    \draw (1em,0) -- ++ (0,-1em);
    \end{scope}
    \node[fill=white, inner sep=1pt] at (1em,0) {$\varphi(a_1)$};
    %\draw [brown] (-0.05em,-1.125em) rectangle (2.075em, 0.75em);
    \useasboundingbox (-0.05em,-1.125em) rectangle (2.075em, 0.75em); 
  \end{tikzpicture}
                           \begin{tikzpicture}[baseline=-2pt]                        \begin{scope}[xshift=0.375em,xscale=1]
    \draw (1em,0) -- ++ (0,-1em);
    \end{scope}
    \node[fill=white, inner sep=1pt] at (1em,0) {$\varphi(a_2)$};
    %\draw [brown] (-0.05em,-1.125em) rectangle (2.075em, 0.75em);
    \useasboundingbox (-0.05em,-1.125em) rectangle (2.075em, 0.75em); 
  \end{tikzpicture}
                           \begin{tikzpicture}[baseline=-2pt]                        \begin{scope}[xshift=0.375em,xscale=1]
    \draw (1em,0) -- ++ (0,-1em);
    \end{scope}
    \node[fill=white, inner sep=1pt] at (1em,0) {$\varphi(b_1)$};
    %\draw [brown] (-0.05em,-1.125em) rectangle (2.075em, 0.75em);
    \useasboundingbox (-0.05em,-1.125em) rectangle (2.075em, 0.75em); 
  \end{tikzpicture}
                           \begin{tikzpicture}[baseline=-2pt]                        \begin{scope}[xshift=0.375em,xscale=1]
    \draw (1em,0) -- ++ (0,-1em);
    \end{scope}
    \node[fill=white, inner sep=1pt] at (1em,0) {$\varphi(b_2)$};
    %\draw [brown] (-0.05em,-1.125em) rectangle (2.075em, 0.75em);
    \useasboundingbox (-0.05em,-1.125em) rectangle (2.075em, 0.75em); 
  \end{tikzpicture}.
 \end{align*}
\end{remark}

\setcounter{theorem}{1}
 \begin{definition}
   Let $S$ be a finite set.
   \begin{enumerate}
   \item We call $\pi=\{V_1,\ldots,V_r\}$ a \emph{partition} of the set $S$ if
     \begin{itemize}
       \item $k\in \naturalnumbers$,
     \item $V_i\subseteq S$ and $V_i\neq \emptyset$ for all $i=1,\ldots,r$,
     \item $V_i\cap V_j=\emptyset$ for all $i,j=1,\ldots,r$ with $i\neq j$ and
     \item $V_1\cup \ldots \cup V_r=S$.
     \end{itemize}
     \par
     The elements $V_1,\ldots,V_r$ of $\pi$ are called its \emph{blocks}. The integer $\# \pi\equalperdefinition r$ denotes the number of blocks of $\pi$.
     For all $p,q\in S$ we write
     \begin{align*}
       p\sim_\pi q\quad \text{if and only if}\quad p,q \text{ are in the same block of }\pi.
     \end{align*}
   \item We write  $\setofpartitionsof(S)$ for the set of all partitions of $S$. And for all $n\in \naturalnumbers$, we abbreviate
     \begin{align*}
       [n]\equalperdefinition \{1,\ldots,n\}
\qquad
     \text{and}
\qquad
       \setofpartitionsof(n)\equalperdefinition \setofpartitionsof([n]).
     \end{align*}
   \item If $S$ is totally ordered, then a partition $\pi\in \setofpartitionsof(S)$  is called \emph{crossing} if there exist $p_1,q_1,p_2,q_2\in S$ such that
     \begin{gather*}
       p_1<q_1<p_2<q_2 \text{ and at the same time }
       p_1\sim_\pi p_2, \, q_1\sim_\pi q_2 \text{ and } p_1\not\sim_\pi q_1.\\
      \begin{tikzpicture}
        \node (n1) at (0,0) {};
        \node (n2) at (1.5,0) {$p_1$};
        \node (n3) at (3,0) {$q_1$};
        \node (n4) at (4.5,0) {$p_2$};
        \node (n5) at (6,0) {$q_2$};
        \node (n6) at (7.5,0) {};
        \path (n1) -- node[pos=0.5] {$\ldots$} (n2);
        \path (n2) -- node[pos=0.5] {$\ldots$} (n3);
        \path (n3) -- node[pos=0.5] {$\ldots$} (n4);
        \path (n4) -- node[pos=0.5] {$\ldots$} (n5);
        \path (n5) -- node[pos=0.5] {$\ldots$} (n6);
        \draw (n2) --++ (0,-0.75) -| (n4);
        \draw (n3) --++ (0,-1.5) -| (n5);
        \draw[dotted] ($(n2)+(0,-0.75)$) --++ (-1,0);
        \draw[dotted] ($(n4)+(0,-0.75)$) --++ (1,0);
        \draw[dotted] ($(n3)+(0,-1.5)$) --++ (-1,0);
        \draw[dotted] ($(n5)+(0,-1.5)$) --++ (1,0);                
     \end{tikzpicture}
     \end{gather*}
     If $\pi$ is not crossing, then it is called \emph{non-crossing}.
   \item Lastly, if $S$ is totally ordered, we write $\setofnoncrossingpartitionsof(S)$ for the subset of all partitions in $\setofpartitionsof(S)$ which are non-crossing. And, for every $n\in \naturalnumbers$, we abbreviate
     \begin{align*}
       \setofnoncrossingpartitionsof(n)\equalperdefinition \setofnoncrossingpartitionsof([n]).
     \end{align*}
   \end{enumerate}
 \end{definition}

 \begin{remark}
   \label{remark:interval-stripping}
   \begin{enumerate}
   \item\label{remark:interval-stripping-1}  The partition
     \begin{align*}
           \begin{tikzpicture}[baseline=-2pt]
      \node[inner sep=1pt] (n1) at (0em,0) {$1$};
      \node[inner sep=1pt] (n2) at (1em,0) {$2$};
      \node[inner sep=1pt] (n3) at (2em,0) {$3$};
      \node[inner sep=1pt] (n4) at (3em,0) {$4$};
      \draw (n1) --++(0,-1em) -| (n3);
      \draw (n2) --++(0,-1.5em) -| (n4);      
    %\draw [brown] (-0.25em,-1.625em) rectangle (3.25em, 0.75em);
    \useasboundingbox (-0.25em,-1.625em) rectangle (3.25em, 0.75em);
  \end{tikzpicture}  \in \setofpartitionsof(4)
     \end{align*}
     is crossing, while the partition
     \begin{align*}
    \begin{tikzpicture}[baseline=-2pt]
      \node[inner sep=1pt] (n1) at (0em,0) {$1$};
      \node[inner sep=1pt] (n2) at (1em,0) {$2$};
      \node[inner sep=1pt] (n3) at (2em,0) {$3$};
      \node[inner sep=1pt] (n4) at (3em,0) {$4$};
      \node[inner sep=1pt] (n5) at (4em,0) {$5$};
      \node[inner sep=1pt] (n6) at (5em,0) {$6$};
      \node[inner sep=1pt] (n7) at (6em,0) {$7$};      
      \draw (n1) --++(0,-2em) -| (n5);
      \draw ($(n5)+(0,-2em)$) -| (n7);      
      \draw (n2) --++(0,-1.5em) -| (n4);
      \draw (n3) --++(0,-1em);
      \draw (n6) --++(0,-1em);      
    %\draw [brown] (-0.25em,-2.125em) rectangle (6.25em, 0.75em);
    \useasboundingbox (-0.25em,-2.125em) rectangle (6.25em, 0.75em);
  \end{tikzpicture}  \in \setofnoncrossingpartitionsof(7)       
     \end{align*}
 is non-crossing.
\item\label{remark:interval-stripping-2} It is easy to see that \emph{non-crossing} is the same as \enquote{nested} in the sense of Remark~\ref{paragraph:non-crossing-moment-factorization}: For every $n\in \naturalnumbers$, a partition $\pi\in\setofpartitionsof(n)$ is non-crossing if and only if there exists a block $V\in\pi$ such that $V$ is an interval and such that $\pi\backslash \{V\}$ is non-crossing (meaning that $V=\{k,k+1,k+2,\ldots,k+p\}$ for some $k,p\in \naturalnumbers$ with $1\leq k\leq n$, $0\leq p$, $k+p\leq n$ and $\pi\backslash \{V\} \in \setofnoncrossingpartitionsof(\{1,\ldots, k-1,k+p+1,\ldots,n\})\cong \setofnoncrossingpartitionsof(n-(p+1))$.) 
  This means that we can reduce non-crossing partitions by successive \enquote{interval-stripping}.  Example:
  \begin{align*}
    \begin{tikzpicture}[baseline=-2pt]
      \node[inner sep=1pt] (n1) at (0em,0) {$1$};
      \node[inner sep=1pt] (n2) at (1em,0) {$2$};
      \node[inner sep=1pt] (n3) at (2em,0) {$3$};
      \node[inner sep=1pt] (n4) at (3em,0) {$4$};
      \node[inner sep=1pt] (n5) at (4em,0) {$5$};
      \node[inner sep=1pt] (n6) at (5em,0) {$6$};
      \node[inner sep=1pt] (n7) at (6em,0) {$7$};      
      \draw (n1) --++(0,-2em) -| (n5);
      \draw ($(n5)+(0,-2em)$) -| (n7);      
      \draw (n2) --++(0,-1.5em) -| (n4);
      \draw (n3) --++(0,-1em);
      \draw (n6) --++(0,-1em);      
    %\draw [brown] (-0.25em,-2.125em) rectangle (6.25em, 0.75em);
    \useasboundingbox (-0.25em,-2.125em) rectangle (6.25em, 0.75em);
  \end{tikzpicture} \quad
    \underset{\displaystyle\displaystyle(3) \text{ and } (6)}{\displaystyle\overset{\displaystyle\text{remove intervals}}{\longrightarrow}} \quad&\quad
    \begin{tikzpicture}[baseline=-2pt]
      \node[inner sep=1pt] (n1) at (0em,0) {$1$};
      \node[inner sep=1pt] (n2) at (1em,0) {$2$};
      \node[inner sep=1pt] (n4) at (3em,0) {$4$};
      \node[inner sep=1pt] (n5) at (4em,0) {$5$};
      \node[inner sep=1pt] (n7) at (6em,0) {$7$};      
      \draw (n1) --++(0,-2em) -| (n5);
      \draw ($(n5)+(0,-2em)$) -| (n7);      
      \draw (n2) --++(0,-1.5em) -| (n4);
    %\draw [brown] (-0.25em,-2.125em) rectangle (6.25em, 0.75em);
    \useasboundingbox (-0.25em,-2.125em) rectangle (6.25em, 0.75em);
  \end{tikzpicture}\\
       \underset{\displaystyle\text{interval } (2,4)}{\overset{\displaystyle\text{remove}}{\longrightarrow}}\quad&\quad
    \begin{tikzpicture}[baseline=-2pt]
      \node[inner sep=1pt] (n1) at (0em,0) {$1$};
      \node[inner sep=1pt] (n5) at (4em,0) {$5$};
      \node[inner sep=1pt] (n7) at (6em,0) {$7$};      
      \draw (n1) --++(0,-2em) -| (n5);
      \draw ($(n5)+(0,-2em)$) -| (n7);      
    %\draw [brown] (-0.25em,-2.125em) rectangle (6.25em, 0.75em);
    \useasboundingbox (-0.25em,-2.125em) rectangle (6.25em, 0.75em);
  \end{tikzpicture}    \\
       \underset{\displaystyle\text{interval } (1,5,7)}{\overset{\displaystyle\text{remove}}{\longrightarrow}}\quad&\quad \hspace{3em}\emptyset
  \end{align*}
   \end{enumerate}
 \end{remark}

 \begin{definition}
   Let $I$ be a set. A multi-index $i=(i(1),\ldots, i(k))$ with $k\in \naturalnumbers$ and $i(1),\ldots, i(k)\in I$ will also be considered as a function $i:[k]\to I$, where $[k]\equalperdefinition \{1,\ldots,k\}$. For such an $i:[k]\to I$, we define its \emph{kernel} $\ker (i)\in \setofpartitionsof(k)$ by declaring for all $p,q\in [k]$:
   \begin{align*}
     p\sim_{\ker (i)}q\quad \text{if and only if}\quad i(p)=i(q).
   \end{align*}
 \end{definition}

 \begin{example}
   \begin{enumerate}
   \item The kernel of the multi-index $i=(1,2,3,2,1,4,1)$
     \begin{align*}
       i=    \begin{tikzpicture}[baseline=-2pt]
      \node[inner sep=1pt] (n1) at (0em,0) {$1$};
      \node[inner sep=1pt] (n2) at (1em,0) {$2$};
      \node[inner sep=1pt] (n3) at (2em,0) {$3$};
      \node[inner sep=1pt] (n4) at (3em,0) {$2$};
      \node[inner sep=1pt] (n5) at (4em,0) {$1$};
      \node[inner sep=1pt] (n6) at (5em,0) {$4$};
      \node[inner sep=1pt] (n7) at (6em,0) {$1$};      
      \draw (n1) --++(0,-2em) -| (n5);
      \draw ($(n5)+(0,-2em)$) -| (n7);      
      \draw (n2) --++(0,-1.5em) -| (n4);
      \draw (n3) --++(0,-1em);
      \draw (n6) --++(0,-1em);      
    %\draw [brown] (-0.25em,-2.125em) rectangle (6.25em, 0.75em);
    \useasboundingbox (-0.25em,-2.125em) rectangle (6.25em, 0.75em);
  \end{tikzpicture}
     \end{align*}
     is given by
  \begin{align*}
    \ker(i) = \begin{tikzpicture}[baseline=-2pt-0.75em]
      \node[inner sep=1pt] (n1) at (0em,0) {};
      \node[inner sep=1pt] (n2) at (1em,0) {};
      \node[inner sep=1pt] (n3) at (2em,0) {};
      \node[inner sep=1pt] (n4) at (3em,0) {};
      \node[inner sep=1pt] (n5) at (4em,0) {};
      \node[inner sep=1pt] (n6) at (5em,0) {};
      \node[inner sep=1pt] (n7) at (6em,0) {};      
      \draw (n1) --++(0,-1.5em) -| (n5);
      \draw ($(n5)+(0,-1.5em)$) -| (n7);      
      \draw (n2) --++(0,-1em) -| (n4);
      \draw (n3) --++(0,-0.5em);
      \draw (n6) --++(0,-0.5em);      
    %\draw [brown] (-0.25em,-1.625em) rectangle (6.25em, 0.125em);
    \useasboundingbox (-0.25em,-1.625em) rectangle (6.25em, 0.75em);
  \end{tikzpicture}
    =\{\{1,5,7\},\{2,4\}, \{3\}, \{6\}\}\in \setofnoncrossingpartitionsof(7).
  \end{align*}
  Note that
  \begin{align*}\ker(
    \begin{tikzpicture}[baseline=-3pt]
      \node[inner sep=1pt] (n1) at (0em,0) {$8$};
      \node[inner sep=1pt] (n2) at (1em,0) {$3$};
      \node[inner sep=1pt] (n3) at (2em,0) {$4$};
      \node[inner sep=1pt] (n4) at (3em,0) {$3$};
      \node[inner sep=1pt] (n5) at (4em,0) {$8$};
      \node[inner sep=1pt] (n6) at (5em,0) {$1$};
      \node[inner sep=1pt] (n7) at (6em,0) {$8$};      
      \draw (n1) --++(0,-2em) -| (n5);
      \draw ($(n5)+(0,-2em)$) -| (n7);      
      \draw (n2) --++(0,-1.5em) -| (n4);
      \draw (n3) --++(0,-1em);
      \draw (n6) --++(0,-1em);      
    %\draw [brown] (-0.25em,-2.125em) rectangle (6.25em, 0.75em);
      \useasboundingbox (-0.25em,-2.125em) rectangle (6.25em, 0.75em);
  \end{tikzpicture})=\ker(i).
  \end{align*}
\item The multi-index $i=(1,2,1,2)$ has the crossing kernel
  \begin{align*}
    \ker(i)=    \begin{tikzpicture}[baseline=-2pt-0.75em]
      \node[inner sep=1pt] (n1) at (0em,0) {};
      \node[inner sep=1pt] (n2) at (1em,0) {};
      \node[inner sep=1pt] (n3) at (2em,0) {};
      \node[inner sep=1pt] (n4) at (3em,0) {};
      \draw (n1) --++(0,-1em) -| (n3);
      \draw (n2) --++(0,-1.5em) -| (n4);      
    %\draw [brown] (-0.25em,-1.625em) rectangle (3.25em, 0.125em);
    \useasboundingbox (-0.25em,-1.625em) rectangle (3.25em, 0.125em);
  \end{tikzpicture}\in \setofpartitionsof (4)\backslash \setofnoncrossingpartitionsof(4).
  \end{align*}
   \end{enumerate}
 \end{example}

 \begin{remark}
   \label{remark:non-crossing-moments}
   \begin{enumerate}
   \item   \label{remark:non-crossing-moments-1} Let  $(\mathcal{A},\varphi)$ be a non-commutative probability space and therein $(\mathcal{A}_i)_{i\in I}$ a family of free unital subalgebras. Consider $k\in \naturalnumbers$, $i(1),\ldots, i(k)\in I$ and random variables $a_1,\ldots,a_k$ with $a_j\in \mathcal{A}_{i(j)}$ for every $j\in [k]$. If for the multi-index
     \begin{align*}
       i\equalperdefinition(i(1),\ldots,i(k))\qquad
\text{it holds that}
     \qquad
       \ker(i)\in \setofnoncrossingpartitionsof(k),
     \end{align*}
     then we have seen in Remark~\ref{paragraph:non-crossing-moment-factorization} that, writing
       $\ker(i)=\colon\{V_1,\ldots, V_r\}$,
     the moment $\varphi(a_1\ldots a_k)$ factorizes as 
     \begin{align*}
       \varphi(a_1\ldots a_k)&=\varphi\Bigl(\prod_{\substack{l\in [k]\\i(l)\in V_1}}^{\longrightarrow}a_l\Bigr)\cdot \ldots \cdot \varphi\Bigl(\prod_{\substack{l\in [k]\\i(l)\in V_r}}^{\longrightarrow}a_l\Bigr)=\prod_{V\in\, \ker(i)}\varphi\Bigl(\prod_{j\in V}^{\longrightarrow}a_j\Bigr),
     \end{align*}
     where $\overset{\rightarrow}{\prod}$ denotes the product of factors in the same order as they appear in the product $a_1\ldots a_k$, i.e., in this case, in  ascending order. (We don't even need to assume $i(1)\neq i(2)\neq \ldots\neq i(k)$ in order for this identity to hold.)
     \par
     As an example, let $\{1,2,3,7\}\subseteq I$ and
     \begin{align*}
       a_1,a_2,a_3 \in \mathcal{A}_1,\quad  b_1,b_2\in \mathcal A_3, \quad c\in \mathcal A_7 \quad\text{and}\quad d\in \mathcal A_2,
     \end{align*}
     and consider the moment
       $\varphi(a_1b_1cb_2a_2da_3)$,
     corresponding to the multi-index
       $i=(1,3,7,3,1,2,1)$,
     whose kernel is given by 
       \begin{align*}
                  \ker(i)&= \begin{tikzpicture}[baseline=-2pt-0.75em]
      \node[inner sep=1pt] (n1) at (0em,0) {};
      \node[inner sep=1pt] (n2) at (1em,0) {};
      \node[inner sep=1pt] (n3) at (2em,0) {};
      \node[inner sep=1pt] (n4) at (3em,0) {};
      \node[inner sep=1pt] (n5) at (4em,0) {};
      \node[inner sep=1pt] (n6) at (5em,0) {};
      \node[inner sep=1pt] (n7) at (6em,0) {};      
      \draw (n1) --++(0,-1.5em) -| (n5);
      \draw ($(n5)+(0,-1.5em)$) -| (n7);      
      \draw (n2) --++(0,-1em) -| (n4);
      \draw (n3) --++(0,-0.5em);
      \draw (n6) --++(0,-0.5em);      
    %\draw [brown] (-0.25em,-1.625em) rectangle (6.25em, 0.125em);
    \useasboundingbox (-0.25em,-1.625em) rectangle (6.25em, 0.75em);
  \end{tikzpicture}
  =\{\{1,5,7\},\{2,4\},\{3\},\{6\}\}
  \in \setofnoncrossingpartitionsof(7).
       \end{align*}
       As already seen in Remark~\ref{paragraph:non-crossing-moment-factorization}, the moment is given by
       \begin{align*}
  \begin{tikzpicture}[baseline=-2pt]
    \begin{scope}[xshift=-2.2em, xscale=0.8]
      \draw (0,0em) -- ++ (0,-2em) -- ++(4em,0) -- ++(0,2em);
      \draw (4em,-2em) -| ++(2em,2em);
      \draw (5em,0) -- ++ (0,-1em);
      \draw (1em,0) -- ++ (0,-1.5em) -| (3em,0);
      \draw (2em,0) -- ++ (0,-1em);      
    \end{scope}
     \node[fill=white, inner sep=1pt] at (0,0) {$\varphi(a_1b_1cb_2a_2da_3)$};  
   \end{tikzpicture}
    &=
          \begin{tikzpicture}[baseline=-2pt]
    \begin{scope}[xshift=-0.7em, xscale=1]
      \draw (0,0em) -- ++ (0,-1em) -- ++(1em,0) -- ++(0,1em);
      \draw (1em,-1em) -| ++(1em,1em);
    \end{scope}
    \node[fill=white, inner sep=1pt] at (0,0) {$\varphi(a_1a_2a_3)$};
    %\draw [brown] (-2.25em,-1.125em) rectangle (2.25em, 0.75em);
    \useasboundingbox (-2.25em,-1.125em) rectangle (2.25em, 0.75em);
  \end{tikzpicture}
                \begin{tikzpicture}[baseline=-2pt]
    \begin{scope}[xshift=-0.125em, xscale=1]
      \draw (0em,0em) -- ++ (0,-1em) -| ++(1em,1em);
    \end{scope}
    \node[fill=white, inner sep=1pt] at (0,0) {$\varphi(b_1b_2)$};
    %\draw [brown] (-1.65em,-1.125em) rectangle (1.65em, 0.75em);
    \useasboundingbox (-1.65em,-1.125em) rectangle (1.65em, 0.75em);
  \end{tikzpicture}
                  \begin{tikzpicture}[baseline=-2pt]
          \begin{scope}[xshift=0.375em,xscale=1]
    \draw (1em,0) -- ++ (0,-1em);
    \end{scope}
    \node[fill=white, inner sep=1pt] at (1em,0) {$\varphi(c)$};
    %\draw [brown] (-0.05em,-1.125em) rectangle (2.075em, 0.75em);
    \useasboundingbox (-0.05em,-1.125em) rectangle (2.075em, 0.75em); 
  \end{tikzpicture}
                  \begin{tikzpicture}[baseline=-2pt]
          \begin{scope}[xshift=0.375em,xscale=1]
    \draw (1em,0) -- ++ (0,-1em);
    \end{scope}
    \node[fill=white, inner sep=1pt] at (1em,0) {$\varphi(d)$};
    %\draw [brown] (-0.05em,-1.125em) rectangle (2.075em, 0.75em);
    \useasboundingbox (-0.05em,-1.125em) rectangle (2.075em, 0.75em); 
  \end{tikzpicture},
       \end{align*}
       which precisely fits the above factorization formula.
     \item If in Part~\ref{remark:non-crossing-moments-1}, the partition $\ker(i)$  is crossing, then the structure of the formula for $\varphi(a_1\ldots a_k)$ is not clear. In order to get some insight also in those cases, we will treat now the free analogue of the central limit theorem.  Recall first the classical version, in our language.
   \end{enumerate}
 \end{remark}

 \begin{theorem}[Classical Central Limit Theorem]
   \label{theorem:classical-central-limit-theorem}   
   Let $(\mathcal{A},\varphi)$ be a non-com\-mu\-ta\-ti\-ve probability space  and $(a_i)_{i\in \naturalnumbers}$ a family of classically independent random variables in  $(\mathcal{A},\varphi)$, i.e., suppose:
   \begin{itemize}
   \item For all $i,j\in \naturalnumbers$ with $i\neq j$, the variables $a_i$ and $a_j$ commute.
   \item For all $i:[k]\to\naturalnumbers$ with $i(1)<i(2)<\ldots <i(k)$ and all $r_1,\ldots, r_k\in \naturalnumbers$,
     \begin{align*}
      \varphi(a^{r_1}_{i(1)}a^{r_2}_{i(2)}\ldots a^{r_k}_{i(k)})=\varphi(a^{r_1}_{i(1)})\varphi(a^{r_2}_{i(2)})\ldots \varphi(a^{r_k}_{i(k)}). 
     \end{align*}
     (Note that all joint moments of $(a_i)_{i\in \naturalnumbers}$ can, by commutativity, be brought in this form.)
   \end{itemize}
   Assume furthermore:
   \begin{itemize}
   \item The random variables $(a_i)_{i\in \naturalnumbers}$ are identically distributed: $\varphi(a_i^r)=\varphi(a_j^r)$ for all $i,j\in \naturalnumbers$ and all $r\in \naturalnumbers$.
   \item The random variables are centered: $\varphi(a_i)=0$ for all $i\in \naturalnumbers$.
   \item Their variances are normalized: $\varphi(a_i^2)=1$ for all $i\in \naturalnumbers$.
   \end{itemize}
   And define for every $n\in \naturalnumbers$
   \begin{align*}
    S_n\equalperdefinition \frac{a_1+\ldots+a_n}{\sqrt{n}} .
   \end{align*}
   \par
   Then, $(S_n)_{n\in \naturalnumbers}$ converges in distribution to a normal ($\hateq$ Gaussian) random variable, which in our algebraic setting means that, for all $k\in \naturalnumbers$,
      \begin{align*}
     \lim_{n\to\infty}\varphi(S_n^k)&=\frac{1}{\sqrt{2\pi}}\int_\realnumbers t^ke^{-{t^2}/{2}}\, dt=
       \begin{cases}
         0,& \text{if $k$ is odd},\\
         (k-1)!!, &\text{otherwise},
       \end{cases}
      \end{align*}
      where $(k-1)!!\equalperdefinition (k-1)\cdot (k-3)\cdot\ldots\cdot 5\cdot 3\cdot 1$.
 \end{theorem}
 One of the first results in free probability theory was a free analogue of this by {Voiculescu}.
 \begin{theorem}[Free Central Limit Theorem, {Voiculescu} 1985]
   \label{theorem:free-central-limit-theorem}
   Let $(\mathcal{A},\varphi)$ be a non-com\-mu\-ta\-ti\-ve probability space and let $(a_i)_{i\in I}$ be a family of freely independent random variables in $(\mathcal{A},\varphi)$. Assume furthermore:
   \begin{itemize}
   \item The random variables $(a_i)_{i\in I}$ are identically distributed: $\varphi(a_i^r)=\varphi(a_j^r)$ for all $i,j\in \naturalnumbers$ and all $r\in \naturalnumbers$.
   \item The random variables are centered: $\varphi(a_i)=0$ for all $i\in \naturalnumbers$.
     \item Their variances are normalized: $\varphi(a_i^2)=1$ for all  $i\in \naturalnumbers$.
     \end{itemize}
     And define for every $n\in \naturalnumbers$
     \begin{align*}
       S_n\equalperdefinition \frac{a_1+\ldots+a_n}{\sqrt{n}}.
     \end{align*}
     \par
     Then, $(S_n)_{n\in \naturalnumbers}$ converges in distribution to a semicircular variable, which means that, for every $k\in \naturalnumbers$,
     \begin{align*}
       \lim_{n\to\infty}\varphi(S_n^k)&=\frac{1}{2\pi}\int_{-2}^2t^k\sqrt{4-t^2}\, dt=
         \begin{cases}
           0, &\text{if }k\text{ is odd},\\
           \frac{1}{m+1}\binom{2m}{m},&\text{if }k=2m \text{ for some }m\in \naturalnumbers.
         \end{cases}
     \end{align*}
   \end{theorem}

   \begin{proof}[Proof (of both Theorems~\ref{theorem:classical-central-limit-theorem} and~\ref{theorem:free-central-limit-theorem})]
     Assume that $(a_i)_{i\in \naturalnumbers}$ are either classically independent or freely independent. Let $k,n\in \naturalnumbers$ be arbitrary. Then,
          \begin{align*}
       \varphi[(a_1+\ldots+a_n)^k]&=\sum_{i:[k]\to [n]}\varphi(a_{i(1)}\ldots a_{i(k)})=\sum_{\pi\in \setofpartitionsof(k)} \ \sum_{\substack{i:[k]\to[n]\\\ker(i)=\pi}}\varphi(a_{i(1)}\ldots a_{i(k)}).
          \end{align*}
          For every  $\pi \in \setofpartitionsof(k)$ and $i:[k]\to [n]$ with $\ker(i)=\pi$ the value of
          \begin{align*}
            g(\pi)\equalperdefinition\varphi(a_{i(1)}\ldots a_{i(k)})
          \end{align*}
          depends only on $\pi$, meaning that for every $j: [k]\to[n]$ with $\ker(i)=\ker(j)$ it holds that $\varphi(a_{i(1)}\ldots a_{i(k)})=\varphi(a_{j(1)}\ldots a_{j(k)})$ 
     since both classical and free independence are rules for calculating joint moments from individual moments: e.g.
     \begin{align*}
         \begin{tikzpicture}[baseline=-2pt]
    \begin{scope}[xshift=-2.2em, xscale=1]
      \draw (0,0em) -- ++ (0,-1.5em) -- ++(2em,0) -- ++(0,1.5em);
      \draw (2em,-1.5em) -| ++(1em,1.5em);
      \draw (5em,0) -- ++ (0,-1em);
      \draw (1em,0) -- ++ (0,-1em) -| (4em,0);
    \end{scope}
     \node[fill=white, inner sep=1pt] at (0,0) {$\varphi(a_1a_2a_1a_1a_2a_3)$};  
   \end{tikzpicture}
       =
       \begin{tikzpicture}[baseline=-2pt]
    \begin{scope}[xshift=-2.2em, xscale=1]
      \draw (0,0em) -- ++ (0,-1.5em) -- ++(2em,0) -- ++(0,1.5em);
      \draw (2em,-1.5em) -| ++(1em,1.5em);
      \draw (5em,0) -- ++ (0,-1em);
      \draw (1em,0) -- ++ (0,-1em) -| (4em,0);
    \end{scope}
     \node[fill=white, inner sep=1pt] at (0,0) {$\varphi(a_7a_5a_7a_7a_5a_2)$};  
   \end{tikzpicture}
       =\colon
       g(\begin{tikzpicture}[baseline=-2pt-0.5em]
      \draw (0,0em) -- ++ (0,-1em) -- ++(2em,0) -- ++(0,1em);
      \draw (2em,-1em) -| ++(1em,1em);
      \draw (5em,0) -- ++ (0,-0.5em);
      \draw (1em,0) -- ++ (0,-0.5em) -| (4em,0);
   \end{tikzpicture})
     \end{align*}
     Hence, we can deduce
     \begin{align*}
       \varphi[(a_1+\ldots+a_n)^k]&=\sum_{\pi\in \setofpartitionsof(k)}g(\pi) \cdot\hspace{-4.5em} \underset{\hspace{4.5em}\begin{aligned}&= n\cdot (n-1)\cdot(n-2)\cdot \ldots \cdot (n-(\#\pi-1))\\ &\sim n^{\#\pi} \ (n\to\infty)\end{aligned}}{\underbrace{\#\{ i:[k]\to[n]\mid \ker(i)=\pi \}.}}
     \end{align*}
     \par
     If a partition  $\pi\in \setofpartitionsof(k)$ has a \emph{singleton block}, i.e.\ if there exists a block $V\in \pi$ with $\# V=1$, say $V=\{s\}$ for some $s\in [k]$, then, with $\ker(i)=\pi$, it follows by Example~\ref{example:moment-formulas-for-small-order} and the same rule for the classical case:
     \begin{align*}
       g(\pi)&=\varphi(a_{i(1)}\ldots a_{i(s-1)}\hspace{-4em}\overset{\substack{\displaystyle\text{free/ind. from the rest,}\\\displaystyle\text{appears only once}\\\displaystyle\downarrow}}{a_{i(s)}}\hspace{-4em}a_{i(s+1)}\ldots a_{i(k)})\\ &\hspace{-0.25em}\overset{\text{\ref{example:moment-formulas-for-small-order}}}{=} \varphi(a_{i(1)}\ldots a_{i(s-1)}a_{i(s+1)}\ldots a_{i(k)})\cdot \underset{=0}{\underbrace{\varphi(a_{i(s)})}}\\[-1.5em]
       &=0.
     \end{align*}
     Hence, if there exists $V\in\pi$ with $\# V=1$, then $g(\pi)=0$. Equivalently, in order  for $g(\pi)\neq 0$ to hold, we need $\#V\geq 2$ for all $V\in \pi$.
     For such $\pi\in\setofpartitionsof(k)$ it must always be true that $\# \pi\leq \frac{k}{2}$.
     We conclude
     \begin{align*}
       \varphi(S^k_n)&=\frac{1}{n^{\frac{k}{2}}}\sum_{\substack{\pi\in\setofpartitionsof(k)\\\#\pi\leq \frac{k}{2}}} g(\pi)\cdot \underset{\displaystyle \sim n^{\#\pi} \ (n\to\infty)}{\underbrace{\#\{i:[k]\to [n]\mid \ker(i)=\pi\}}}\\
\quad\\
       &\hspace{-0.5em}\overset{n\to\infty}{\longrightarrow}\sum_{\substack{\pi\in\setofpartitionsof(k)\\\#\pi\leq \frac{k}{2}}} g(\pi)\cdot\hspace{-3em}  \underset{\hspace{3em}\displaystyle=
       \begin{cases}
         1, & \text{if }\#\pi=\frac{k}{2},\\
         0, & \text{if }\#\pi<\frac{k}{2}.
       \end{cases}
}{\underbrace{\lim_{n\to\infty}\frac{n^{\# \pi}}{n^{\frac{k}{2}}}}.}
     \end{align*}
     Any $\pi\in \setofpartitionsof(k)$ with $\#\pi=\frac{k}{2}$ and without singleton blocks must be a \emph{pairing} or, synonmously, \emph{pair partition}, i.e.\ satisfy $\# V=2$ for all $V\in \pi$. Thus, we have shown
     \begin{align*}
       \lim_{n\to\infty}\varphi(S^k_n)=\sum_{\pi\in\setofpartitionsof_2(k)}g(\pi),
     \end{align*}
     where $\setofpartitionsof_2(k)$ denotes the subset of $\setofpartitionsof(k)$ consisting of all pairings. Note that $\setofpartitionsof_2(k)=\emptyset$ if $k$ is odd. Hence, the parts of the two claims saying
     \begin{align*}
       \lim_{n\to\infty}\varphi(S^k_n)=0 \text{ if }k \text{ is odd}
     \end{align*}
     has already been established. For even $k$, we distinguish between classical and free independence of $(a_i)_{i\in \naturalnumbers}$.
     \begin{enumerate}[wide]
     \item Classical case: Since $g(\pi)=1$ for all $\pi\in \setofpartitionsof_2(k)$, we deduce for even $k$, if $(a_i)_{i\in \naturalnumbers}$ are classically independent,
       \begin{align*}
         \lim_{n\to\infty}\varphi(S^k_n)&=\#\setofpartitionsof_2(k)=(k-1)\cdot (k-3)\cdot \ldots \cdot 5\cdot 3\cdot 1=(k-1)!!,
       \end{align*}
       thus proving Theorem~\ref{theorem:classical-central-limit-theorem}.
     \item Free case: Let $m\in \naturalnumbers$ be such that $k=2m$. Then, for all  $\pi\in \setofpartitionsof_2(2m)$,
       \begin{align*}
         g(\pi)=
         \begin{cases}
           1,&\text{if }\pi\text{ is non-crossing},\\
           0,&\text{otherwise}.
         \end{cases}
       \end{align*}
       For non-crossing $\pi\in \setofpartitionsof_2(2m)$ this follows from Remark~\hyperref[remark:non-crossing-moments-1]{\ref*{remark:non-crossing-moments}~\ref*{remark:non-crossing-moments-1}}. For crossing $\pi\in \setofpartitionsof_2(2m)$ we strip intervals as in Remark~\hyperref[remark:interval-stripping-2]{\ref*{remark:interval-stripping}~\ref*{remark:interval-stripping-2}} until we arrive at a situation where neighbors are from different algebras; $\varphi$ on this is then zero by the definition of freeness.  Example:
       \begin{align*}
                  \begin{tikzpicture}[baseline=-2pt]
    \begin{scope}[xshift=-3.2em, xscale=1]
      \draw (0,0em) -- ++ (0,-2em) -- ++(4em,0) -- ++(0,2em);
      \draw (1em,0) -- ++ (0,-1.5em) -| (7em,0);
      \draw (2em,0) -- ++ (0,-1em) -| ++(1em,1em);
      \draw (5em,0) -- ++ (0,-1em) -| ++(1em,1em);      
    \end{scope}
     \node[fill=white, inner sep=1pt] at (0,0) {$\varphi(a_1a_2a_3a_3a_1a_4a_4a_2)$};  
   \end{tikzpicture}
         =\varphi(a_4a_4)\cdot\varphi(a_3a_3)\cdot
             \begin{tikzpicture}[baseline=-2pt]
    \begin{scope}[xshift=-1.25em, xscale=0.95]
      \draw (0,0em) -- ++ (0,-1em) -- ++(2em,0) -- ++(0,1em);
      \draw (1em,0em) -- ++ (0,-1.5em) -- ++(2em,0) -- ++(0,1.5em);
    \end{scope}
    \node[fill=white, inner sep=1pt] at (0,0) {$\varphi(a_1a_2a_1a_2)$};
    %\draw [brown] (-2.625em,-1.625em) rectangle (2.625em, 0.75em);
    \useasboundingbox (-2.625em,-1.625em) rectangle (2.625em, 0.75em);
   \end{tikzpicture}  
          =0.
       \end{align*}
       Hence, we infer: If $(a_{i})_{i\in \naturalnumbers}$ are freely independent, then
       \begin{align*}
         \lim_{n\to\infty}\varphi(S^k_n)=\#\setofnoncrossingpartitionsof_2(k),
       \end{align*}
       where $\setofnoncrossingpartitionsof_2(k)$ denotes the subset of $\setofpartitionsof_2(k)$ encompassing all \emph{non-crossing pairings}.
       That those numbers are the moments of the semicircle follows from the next lemma. Thus we have proved Theorem~\ref{theorem:free-central-limit-theorem}.\qedhere
     \end{enumerate}
   \end{proof}

   \begin{lemma}
     \label{lemma:catalan-numbers}
The number of non-crossing pairings can be determined as follows.
     \begin{enumerate}
     \item\label{lemma:catalan-numbers-1} If we define $C_m\equalperdefinition \# \setofnoncrossingpartitionsof_2(2m)$ for every $m\in \naturalnumbers$, then the sequence $(C_m)_{m\in\naturalnumbers}$ satisfies the recursion equation
       \begin{align*}
         C_m=\sum_{k=1}^mC_{k-1}C_{m-k} \quad\text{for all }m\in \naturalnumbers,
       \end{align*}
       where $C_0\equalperdefinition 1$.
     \item The unique solution to the recursion problem in Part~\ref{lemma:catalan-numbers-1} is given by the sequence $(C_m)_{m\in \naturalnumbers}$ with
       \begin{align*}
         C_m=\frac{1}{m+1}\binom{2m}{m}\quad \text{for all }m\in \naturalnumbers.
       \end{align*}
     \end{enumerate}
   \end{lemma}

   \begin{definition}
     The numbers $(C_m)_{m\geq 0}$ from Lemma~\ref{lemma:catalan-numbers} are called the \emph{Catalan numbers}.
   \end{definition}

   \begin{example}
     The first four Catalan numbers are
     \begin{align*}
       C_0&=1,\\
       C_1&=C_0\cdot C_0=1,\\
       C_2&=C_1C_0+C_0C_1=1+1=2,\\
       C_3&=C_2C_0+C_1C_1+C_0C_2=2+1+2=5,\\
       C_4&=C_3C_0+C_2C_1+C_1C_2+C_0C_3=5+2+2+5=14.
     \end{align*}
     The sequence continues with $C_5=42$, $C_6=132$ and $C_7=429$.
   \end{example}

   \begin{proof}[Proof of Lemma~\ref{lemma:catalan-numbers}]
     \begin{enumerate}[wide]
     \item Let $m\in \naturalnumbers$ be arbitrary. Every non-crossing pair partition $\pi\in \setofnoncrossingpartitionsof_2(2m)$ is of the form
         $\pi=\{\{1,l\}\}\cup \tilde \pi$,
       for some $l\in [2m]$, where $\{1,l\}$ is the block of $\pi$ containing $1$, and where $$\tilde {\pi}\in \setofnoncrossingpartitionsof_2(\{2,\ldots,l-1,l+1,\ldots,2m\}).$$
\begin{center}
  \begin{tikzpicture}
    \node[circle, fill=white, draw=black, scale=0.4, label=above:{$1$}] (n1) at (0,0) {};
    \node[circle, fill=white, draw=black, scale=0.4] (n2) at (2,0) {};
    \node[circle, fill=white, draw=black, scale=0.4, label=above:{$l$}] (n3) at (2.5,0) {};
    \node[circle, fill=white, draw=black, scale=0.4] (n4) at (3,0) {};
    \node[circle, fill=white, draw=black, scale=0.4, label=above:{$2m$}] (n5) at (5,0) {};
    \path (n1) -- node [pos=0.5] {$\ldots$} (n2);
    \path (n4) -- node [pos=0.5] {$\ldots$} (n5);
    \draw[gray, dashed] ($(n1)+(0.2,-0.15)$)  rectangle ($(n2)+(0.2,0.15)$);
    \draw[gray, dashed] ($(n4)+(-0.2,-0.15)$)  rectangle ($(n5)+(0.2,0.15)$);
    \draw (n1) -- ++ (0,-1) -| (n3);
    \node[align=left] at (4,-2) {Those areas cannot be connected by $\tilde \pi$, otherwise\\there would be a crossing with block $\{1,l\}$.};
    \draw[densely dotted, ->] (1,-1.5) -- (1.25,-0.2);   
    \draw[densely dotted, ->] (1,-1.5) -- (3.75,-0.2);
    \node at ($(n1)+(-1.5em,0em)$) {$\pi$};
  \end{tikzpicture}
\end{center}
Hence, we can decompose $\tilde\pi=\pi_1\cup \pi_2$, where 
$$\pi_1\in \setofnoncrossingpartitionsof_2(\{2,\ldots,l-1\})\qquad
\text{and}\qquad \pi_2\in \setofnoncrossingpartitionsof_2(\{l+1,\ldots,2m\}).$$ 
Taking all possible locations of the partner $l$ of $1$ in $\pi$ into account, it follows
\begin{align*}
  C_m&=\sum_{l=2}^{2m}\underset{\displaystyle =
  \begin{cases}
    C_{k-1},&\text{if $l=2k$ is even},\\
    0,&\text{if $l$ is odd}.
  \end{cases}
\hspace{0.2em}}{\underbrace{\#\setofnoncrossingpartitionsof_2(\{2,\ldots,l-1\})}}\hspace{-0.2em}\cdot\hspace{-0.2em}\underset{\hspace{0.2em}\displaystyle =
  \begin{cases}
    C_{m-k},&\text{if $l=2k$ is even},\\
    0,&\text{if $l$ is odd}.
  \end{cases}}{\underbrace{\setofnoncrossingpartitionsof_2(\{l+1,\ldots,2m\})}}
=\sum_{k=1}^mC_{k-1}C_{m-k}.
\end{align*}
\item The second claim can be seen, for example, by a formal power series argument, see Assignment~\hyperref[assignment-3]{3}, Exercise~1.
\qedhere
     \end{enumerate}
   \end{proof}

   \begin{definition}
     \label{definition:semicircular-variable}
     Let $(\mathcal{A},\varphi)$ be a $\ast$-probability space and let $\sigma\in \realnumbers$, $\sigma>0$. A self-adjoint random variable $s\in \mathcal{A}$ is called \emph{semicircular element} of variance $\sigma^2$ if its moments are of the form
     \begin{align*}
       \varphi(s^k)=
       \begin{cases}
         0,&\text{if }k\text{ is odd},\\
         \sigma^{2m}C_m&\text{if }k=2m\text{ for some }m\in \naturalnumbers,
       \end{cases}
     \end{align*}
     for all $k\in \naturalnumbers$. (Note $\varphi(s^2)=\sigma^2$.)\par
     In the case $\sigma=1$ we call it a \emph{standard} semicircular element.
   \end{definition}

   \begin{definition}
     Let $(\mathcal{A}_n,\varphi_n)$ for every $n\in \naturalnumbers$ as well as $(\mathcal{A},\varphi)$ be non-commutative probability spaces. Let $(b_n)_{n\in \naturalnumbers}$ be a sequence of random variables with $b_n\in \mathcal{A}_n$ for every $n\in \naturalnumbers$ and let $b\in\mathcal{A}$. We say that $(b_n)_{n\in \naturalnumbers}$ \emph{converges in distribution} to $b$, denoted by
     \begin{align*}
       b_n \convergesindistributionto b,
\qquad
     \text{if it holds}\qquad
       \lim_{n\to\infty} \varphi_n(b^k_n)=\varphi(b^k)\quad\text{for all } k\in \naturalnumbers.
     \end{align*}
Often, we will also just write $b_n\convergesindistributiontoo b$ instead of $b_n \convergesindistributionto b$
   \end{definition}

   \begin{remark}
     \label{remark:limit-distribution}
     \begin{enumerate}
     \item\label{remark:limit-distribution-1} The moments of a semicircular element $s$ of variance $\sigma^2\in \realnumbers^+$ are 
       \begin{align*}
         \varphi(s^k)=\frac{1}{2\pi\sigma^2}\int_{-2\sigma}^{2\sigma}t^k\sqrt{4\sigma^2-t^2}\, dt
       \end{align*}
       for every $k\in \naturalnumbers$, see Assignment~\hyperref[assignment-3]{3}, Exercise~2.
       \begin{center}
         \begin{tikzpicture}
           \draw[->] (-3,0) -- (3,0);
           \draw[->] (0,-0.5) -- (0,2.5);
           \draw (2,0) arc (0:180:2);
           \node[below] at (-2,0) {$-2\sigma$};
           \node[below] at (2,0) {$+2\sigma$};           
         \end{tikzpicture}
       \end{center}
     \item\label{remark:limit-distribution-2} \enquote{Convergence in distribution} is a kind of probabilistic concept. If our random variables are operators from operator algebras, our notion of convergence is quite different from the ones usually considered there (uniform, strong, weak convergence). Note that for convergence in distribution the sequence and the limit do not even have to live in the same space! (And neither do the individual elements of the sequence have to share the same non-commutative probability space.)
     \item\label{remark:limit-distribution-3} In this language, the conclusion of our Free Central Limit Theorem~\ref{theorem:free-central-limit-theorem} can be expressed as
       \begin{align*}
      \frac{a_1+\ldots+a_n}{\sqrt{n}}\convergesindistributionto s
       \end{align*}
       for a standard semicircular element $s$ (in some non-commutative probability space) and with $(a_i)_{i\in \naturalnumbers}$ as in Theorem~\ref{theorem:free-central-limit-theorem}.
     \item\label{remark:limit-distribution-4} There are other important limit theorems in classical probability theory about the sum of independent and identically-distributed (\emph{i.i.d.}) random variables, in particular, the \enquote{Law of Rare Events}: Consider a family $(a_{n,i})_{n\in \naturalnumbers, i\in [n]}$ of $\{0,1\}$-valued random variables 
       \begin{center}
       \begin{tikzpicture}
         \node (n1) at (0,0em) {$a_{1,1}$};
         \node (n21) at (-0.5,-1.5em) {$a_{2,1}$};
         \node (n22) at (0.5,-1.5em) {$a_{2,2}$};
         \node (n31) at (-1,-3em) {$a_{3,1}$};
         \node (n32) at (0,-3em) {$a_{3,2}$};
         \node (n33) at (1,-3em) {$a_{3,3}$};
         \node at (0,-4.5em) {$\vdots$};
         \node at (2.5,-1.5em) {$\leftarrow$ i.i.d.};
         \node at (2.5,-3em) {$\leftarrow$ i.i.d.};
         \node at (-1.5,-4.5em) {\rotatebox{90}{$\ddots$}};
         \node at (1.5,-4.5em) {$\ddots$};
         \draw [densely dotted, darkgray] ($(n21)+(-1em,-0.5em)$) rectangle ($(n22)+(1em,0.65em)$);
         \draw [densely dotted, darkgray] ($(n31)+(-1em,-0.5em)$) rectangle ($(n33)+(1em,0.65em)$);
         %\draw [brown] (-10em,-5.5em) rectangle (10em,1em);
         \useasboundingbox (-10em,-5.5em) rectangle (10em,1em); 
       \end{tikzpicture}
     \end{center}
     such that, for each $n\in \naturalnumbers$, the random variables
       $a_{n,1},\ldots, a_{n,n}$ are i.i.d.,
     and such that, for some $\lambda\in \realnumbers^+$, the distribution of $a_{n,i}$ is given, for $n\geq \lambda$ and $i\in [n]$, by the Bernoulli distribution
     \begin{align}
       \label{eq:remark-limit-distribution-1}
       \frac{\lambda}{n}\delta_1+\left(1-\frac{\lambda}{n}\right)\delta_0.
     \end{align}
     \par
     Then, $(a_{n,1}+\ldots+ a_{n,n})_{n\in \naturalnumbers}$ converges in distribution to a Poisson variable with parameter $\lambda$, which is to say
       $a_{n,1}+\ldots+a_{n,n} \convergesindistributiontoo a$,
     for some random variable $a$ (in some non-commutative probability space)  whose distribution is given  by
     \begin{align*}
       \sum_{p=0}^\infty e^{-\lambda}\frac{\lambda^p}{p!}\delta_p.
     \end{align*} 
     An interpretation could be: The random variable $a_{n,i}$ indicates the decay of the radioactive atom $i$ within one time unit; $a_{n,1}+\ldots+a_{n,n}$ gives then the number of decayed atoms within one time unit.
     \par
     If we replace classical independence by freeness, the sequence $(a_{n,1}+\ldots+ a_{n,n})_{n\in \naturalnumbers}$ should converge to a \enquote{free Poisson variable}.
   \item\label{remark:limit-distribution-5}  We have to reformulate the condition in Equation~\eqref{eq:remark-limit-distribution-1} on the single variable distribution of  $a_{n,i}$ in terms of moments. It translates as
     \begin{align}
       \label{eq:remark-limit-distribution-2}
       \varphi(a_{n,i}^r)=\frac{\lambda}{n}\cdot 1^r+\left(1-\frac{\lambda}{n}\right)\cdot 0^r=\frac{\lambda}{n}
     \end{align}
     for $n\in \naturalnumbers$, $i\in[n]$ and for all $r\in \naturalnumbers$.
   \item\label{remark:limit-distribution-6} We will actually treat an even more general situation than in Equation~\eqref{eq:remark-limit-distribution-2}, namely we will allow the parameter $\lambda$ to depend on the order $r$ of the moment: For a given sequence  $(\kappa_r)_{r\in \naturalnumbers}$ in $\complexnumbers$ we assume for all $n\in \naturalnumbers$, $i\in [n]$ and $r\in \naturalnumbers$ that
     \begin{align}
       \label{eq:remark-limit-distribution-3}       
       \varphi(a^r_{n,i})=\frac{\kappa_r}{n}.
     \end{align}
     If we want positivity, i.e.\ if we want $\varphi$ to be a state, then we must make additional assumptions about the $\kappa_r$. Without positivity, however, we can choose  $(\kappa_r)_{r\in \naturalnumbers}$ arbitrarily.
     Let us now check what we get in the situation given by Equation~\eqref{eq:remark-limit-distribution-3} in the limit, again looking at the classical and free situation in parallel.
   \end{enumerate}   
 \end{remark}

 \begin{remark}[Calculation of the limit distribution]
   \label{paragraph:poisson-limit-theorem}
   Let $(\mathcal{A},\varphi)$ be a non-commutative probability space, let $(a_{n,i})_{n\in \naturalnumbers,i\in [n]}$ be a family of random variables in $\mathcal A$, let $(\kappa_r)_{r\in \naturalnumbers}$ be a sequence in $\complexnumbers$, and suppose that
   \begin{itemize}
   \item either for every $n\in \naturalnumbers$ the random variables $a_{n,1},\ldots, a_{n,n}$ are classically independent, or for every $n\in \naturalnumbers$ the random variables $a_{n,1},\ldots, a_{n,n}$ are freely independent, and 
   \item for all $n\in \naturalnumbers$, $i\in [n]$ and $r\in \naturalnumbers$ it holds
     \begin{align*}
       \varphi(a^r_{n,i}) =\frac{\kappa_r}{n}.
     \end{align*}
   \end{itemize}
   We want to see whether there exists a random variable $a$ in some non-com\-mu\-ta\-ti\-ve probability space such that
     $a_{n,1}+\ldots+a_{n,n}\convergesindistributiontoo a$
   and, if so, how we can describe the distribution of $a$.
   \par
   Let $n\in \naturalnumbers$ be arbitrary. Then, for every  $k\in \naturalnumbers$,
   \begin{align*}
     \varphi[(a_{n,1}+\ldots+a_{n,n})^k]&=\sum_{i:[k]\to[n]}\varphi(a_{n,i(1)}\ldots a_{n,i(k)})\\
\quad\\
                                        &=\sum_{\pi\in \setofpartitionsof(k)}\sum_{\substack{i:[k]\to[n]\\\ker(i)=\pi}}\hspace{-1em}\underset{\substack{\displaystyle =\colon g_n(\pi)\\\displaystyle\text{(as in proof of Theorems~\ref{theorem:classical-central-limit-theorem} and~\ref{theorem:free-central-limit-theorem},}\\\displaystyle\text{but now dependent on }n)}}
{\underbrace{\varphi(a_{n,i(1)}\ldots a_{n,i(k)})}}\\
\quad\\
    &= \sum_{\pi\in \setofpartitionsof(k)}g_n(\pi) \cdot \underset{\displaystyle\sim n^{\#\pi}\ (n\to\infty)}{\underbrace{n\cdot (n-1)\cdot\ldots\cdot (n-\#\pi+1)}}.
   \end{align*}
   Hence, in the limit, we obtain for every $k\in \naturalnumbers$
   \begin{align*}
     \varphi(a^k)=\lim_{n\to\infty}\varphi[(a_{n,1}+\ldots+a_{n,n})^k]=\sum_{\pi \in \setofpartitionsof(k)}\underset{\displaystyle=\colon g(\pi)}{\underbrace{\lim_{n\to\infty}g_n(\pi)\cdot n^{\#\pi}}}.
   \end{align*}
   As in the proof of Theorems~\ref{theorem:classical-central-limit-theorem} and~\ref{theorem:free-central-limit-theorem}, we now distinguish between classical and free independence. 
   \begin{enumerate}[wide]
   \item Classical case: Suppose that $(a_{n,i})_{i\in [n]}$ is classically independent for every $n\in \naturalnumbers$. Then, for all $n,k\in \naturalnumbers$ and $\pi\in \setofpartitionsof(k)$, the moment $g_{n}(\pi)$ factorizes into a product of terms $\kappa_{\# V}/n$ for each block $V\in \pi$:
     \begin{align*}
       \varphi(a^k)&=\lim_{n\to\infty} \sum_{\pi\in \setofpartitionsof (k)}\left[\prod_{V\in \pi}\frac{\kappa_{\# V}}{n}\right]\cdot n^{\# \pi}=\sum_{\pi \in \setofpartitionsof(k)}\prod_{V\in \pi}\kappa_{\# V}.
     \end{align*}
   \item Free case: Let now $(a_{n,i})_{i\in [n]}$ be freely independent for every $n\in \naturalnumbers$. Given $k\in \naturalnumbers$, for non-crossing $\pi\in \setofnoncrossingpartitionsof(k)$ the moment $g(\pi)$ factorizes as for the classical case: 
     \begin{align*}
       g(\pi)=\prod_{V\in \pi}{\kappa_{\# V}} \quad \text{if }\pi \in \setofnoncrossingpartitionsof (k).
     \end{align*}
     For every crossing $\pi\in \setofpartitionsof(k)$ and every $n\in \naturalnumbers$ we know, by Assignment~\hyperref[assignment-2]{2}, Exercise~3, that $g_n(\pi)$ is given by a polynomial in the moments where each summand contains at least $\#\pi+1$ many moments. Since each moment gives a factor $\frac{1}{n}$, that means
     \begin{align*}
       g_n(\pi)\sim \frac{1}{n^{\#\pi+1}} \ (n\to\infty) \quad\text{ if }\pi\text{ is crossing}.
     \end{align*}
     Hence, we can conclude that
       $g(\pi)=0$ for all $\pi\in \setofpartitionsof(k)\backslash \setofnoncrossingpartitionsof(k)$.
     Altogether, that means
     \begin{align*}
       \varphi(a^k)=\sum_{\pi\in \setofnoncrossingpartitionsof(k)}\prod_{V\in \pi}\kappa_{\# V}.
     \end{align*}
   \end{enumerate}
 \end{remark}

 \begin{theorem}[Poisson Limit Theorem]
   \label{theorem:poisson-limit-theorem}
      Let $(\mathcal{A},\varphi)$ be a non-commutative probability space, let $(a_{n,i})_{n\in \naturalnumbers,i\in [n]}$ be a family of random variables in $\mathcal A$, let $(\kappa_r)_{r\in \naturalnumbers}$ be a sequence in $\complexnumbers$, and suppose that
      \begin{itemize}
        \item for every $n\in \naturalnumbers$ the random variables $a_{n,1},\ldots, a_{n,n}$ are identically distributed,
   \item either for every $n\in \naturalnumbers$ the random variables $a_{n,1},\ldots, a_{n,n}$ are classically independent, or for every $n\in \naturalnumbers$ the random variables $a_{n,1},\ldots, a_{n,n}$ are freely independent, and 
   \item for all $n\in \naturalnumbers$ and $r\in \naturalnumbers$ it holds that
     \begin{align*}
       \lim_{n\to\infty}n\cdot\varphi(a^r_{n,i}) =\kappa_r.
     \end{align*}
   \end{itemize}
     Then, 
       $a_{n,1}+\ldots+a_{n,n}\convergesindistributiontoo a$ (for convenience, we can assume that $a\in\mathcal{A}$)
     and where the distribution of $a$ is given by
     \begin{align*}
       \varphi(a^k)&=
                      \begin{cases}
                        \displaystyle \hspace{0.3em}\sum_{\pi\in \setofpartitionsof(k)}\hspace{0.3em}\prod_{V\in \pi}\kappa_{\# V} &\text{ in the classical case and}\\
\quad\\
                     \displaystyle \sum_{\pi\in \setofnoncrossingpartitionsof(k)}\prod_{V\in \pi}\kappa_{\# V}&\text{ in the free case}.   
                      \end{cases}
     \end{align*}
     for all $k\in \naturalnumbers$.
 \end{theorem}

 \begin{example}
   \begin{enumerate}
   \item For $k=1,2,3$ there is no difference between the classical and free formulas in Theorem~\ref{theorem:poisson-limit-theorem} since $\setofpartitionsof(k)=\setofnoncrossingpartitionsof(k)$ for $k=1,2,3$.
     \begin{align*}
       \varphi(a^1)&=\kappa_1 &&\begin{tikzpicture}[baseline=-2pt-0.375em]
      \node[inner sep=1pt] (n1) at (0em,0) {};
      \draw (n1) --++(0,-0.75em);      
    %\draw [brown] (-0.25em,-1em) rectangle (0.25em, 0.125em);
    \useasboundingbox (-0.25em,-1em) rectangle (0.25em, 0.125em);
  \end{tikzpicture}\\
       \varphi(a^2)&=\kappa_2+\kappa_1^2 &&
      \begin{tikzpicture}[baseline=-2pt-0.375em]
      \node[inner sep=1pt] (n1) at (0em,0) {};
      \node[inner sep=1pt] (n2) at (1em,0) {};
      \draw (n1) --++(0,-0.75em) -| (n2);
    %\draw [brown] (-0.25em,-1em) rectangle (1.25em, 0.125em);
    \useasboundingbox (-0.25em,-1em) rectangle (1.25em, 0.125em);
  \end{tikzpicture}, \
                     \begin{tikzpicture}[baseline=-2pt-0.375em]
      \node[inner sep=1pt] (n1) at (0em,0) {};
      \node[inner sep=1pt] (n2) at (1em,0) {};
      \draw (n1) --++(0,-0.75em);
      \draw (n2) --++(0,-0.75em);      
    %\draw [brown] (-0.25em,-1em) rectangle (1.25em, 0.125em);
    \useasboundingbox (-0.25em,-1em) rectangle (1.25em, 0.125em);
  \end{tikzpicture}
       \\
       \varphi(a^3)&=\kappa_3+3\kappa_2\kappa_1+\kappa_1^3. &&
\begin{tikzpicture}[baseline=-2pt-0.375em]
      \node[inner sep=1pt] (n1) at (0em,0) {};
      \node[inner sep=1pt] (n2) at (1em,0) {};
      \node[inner sep=1pt] (n3) at (2em,0) {};
      \draw (n1) --++(0,-0.75em) -| (n2);
      \draw ($(n2)+(0,-0.75em)$) -| (n3);      
    %\draw [brown] (-0.25em,-1em) rectangle (2.25em, 0.125em);
    \useasboundingbox (-0.25em,-1em) rectangle (2.25em, 0.125em);
  \end{tikzpicture}, \
                     \begin{tikzpicture}[baseline=-2pt-0.375em]
      \node[inner sep=1pt] (n1) at (0em,0) {};
      \node[inner sep=1pt] (n2) at (1em,0) {};
      \node[inner sep=1pt] (n3) at (2em,0) {};
      \draw (n2) --++(0,-0.75em) -| (n3);
      \draw (n1) --++(0,-0.75em);      
    %\draw [brown] (-0.25em,-1em) rectangle (2.25em, 0.125em);
    \useasboundingbox (-0.25em,-1em) rectangle (2.25em, 0.125em);
  \end{tikzpicture},  \
                     \begin{tikzpicture}[baseline=-2pt-0.5em]
      \node[inner sep=1pt] (n1) at (0em,0) {};
      \node[inner sep=1pt] (n2) at (1em,0) {};
      \node[inner sep=1pt] (n3) at (2em,0) {};
      \draw (n1) --++(0,-1em) -| (n3);
      \draw (n2) --++(0,-0.5em);      
    %\draw [brown] (-0.25em,-1.125em) rectangle (2.25em, 0.125em);
    \useasboundingbox (-0.25em,-1.125em) rectangle (2.25em, 0.125em);
  \end{tikzpicture}, \ 
                     \begin{tikzpicture}[baseline=-2pt-0.375em]
      \node[inner sep=1pt] (n1) at (0em,0) {};
      \node[inner sep=1pt] (n2) at (1em,0) {};
      \node[inner sep=1pt] (n3) at (2em,0) {};
      \draw (n1) --++(0,-0.75em) -| (n2);
      \draw (n3) --++(0,-0.75em);      
    %\draw [brown] (-0.25em,-1em) rectangle (2.25em, 0.125em);
    \useasboundingbox (-0.25em,-1em) rectangle (2.25em, 0.125em);
  \end{tikzpicture}, \
                     \begin{tikzpicture}[baseline=-2pt-0.375em]
      \node[inner sep=1pt] (n1) at (0em,0) {};
      \node[inner sep=1pt] (n2) at (1em,0) {};
      \node[inner sep=1pt] (n3) at (2em,0) {};
      \draw (n1) --++(0,-0.75em);      
      \draw (n2) --++(0,-0.75em);
      \draw (n3) --++(0,-0.75em);      
    %\draw [brown] (-0.25em,-1em) rectangle (2.25em, 0.125em);
    \useasboundingbox (-0.25em,-1em) rectangle (2.25em, 0.125em);
  \end{tikzpicture}
     \end{align*}
     For $k\geq 4$, however, the formulas will differ:
     \begin{align*}
       \varphi(a^4_{\text{free}})&= \kappa_4 && \begin{tikzpicture}[baseline=-2pt-0.375em]
      \node[inner sep=1pt] (n1) at (0em,0) {};
      \node[inner sep=1pt] (n2) at (1em,0) {};
      \node[inner sep=1pt] (n3) at (2em,0) {};
      \node[inner sep=1pt] (n4) at (3em,0) {};
      \draw (n1) --++(0,-0.75em) -| (n2);
      \draw ($(n2)+(0,-0.75em)$) -| (n3);
      \draw ($(n3)+(0,-0.75em)$) -| (n4);
    %\draw [brown] (-0.25em,-1em) rectangle (3.25em, 0.125em);
    \useasboundingbox (-0.25em,-1em) rectangle (3.25em, 0.125em);
  \end{tikzpicture}, \ \\
                                 &\phantom{{}={}} +4\kappa_3\kappa_1 &&
 \begin{tikzpicture}[baseline=-2pt-0.375em]
      \node[inner sep=1pt] (n1) at (0em,0) {};
      \node[inner sep=1pt] (n2) at (1em,0) {};
      \node[inner sep=1pt] (n3) at (2em,0) {};
      \node[inner sep=1pt] (n4) at (3em,0) {};
      \draw (n2) --++(0,-0.75em) -| (n3);
      \draw ($(n3)+(0,-0.75em)$) -| (n4);
      \draw (n1) --++(0,-0.75em);            
     %\draw [brown] (-0.25em,-1em) rectangle (3.25em, 0.125em);
    \useasboundingbox (-0.25em,-1em) rectangle (3.25em, 0.125em);
  \end{tikzpicture}, \
 \begin{tikzpicture}[baseline=-2pt-0.5em]
      \node[inner sep=1pt] (n1) at (0em,0) {};
      \node[inner sep=1pt] (n2) at (1em,0) {};
      \node[inner sep=1pt] (n3) at (2em,0) {};
      \node[inner sep=1pt] (n4) at (3em,0) {};
      \draw (n1) --++(0,-1em) -| (n3);
      \draw ($(n3)+(0,-1em)$) -| (n4);
      \draw (n2) --++(0,-0.5em);            
     %\draw [brown] (-0.25em,-1.125em) rectangle (3.25em, 0.125em);
    \useasboundingbox (-0.25em,-1.125em) rectangle (3.25em, 0.125em);
  \end{tikzpicture}, \
 \begin{tikzpicture}[baseline=-2pt-0.5em]
      \node[inner sep=1pt] (n1) at (0em,0) {};
      \node[inner sep=1pt] (n2) at (1em,0) {};
      \node[inner sep=1pt] (n3) at (2em,0) {};
      \node[inner sep=1pt] (n4) at (3em,0) {};
      \draw (n1) --++(0,-1em) -| (n2);
      \draw ($(n2)+(0,-1em)$) -| (n4);
      \draw (n3) --++(0,-0.5em);            
     %\draw [brown] (-0.25em,-1.125em) rectangle (3.25em, 0.125em);
    \useasboundingbox (-0.25em,-1.125em) rectangle (3.25em, 0.125em);
  \end{tikzpicture}, \
 \begin{tikzpicture}[baseline=-2pt-0.375em]
      \node[inner sep=1pt] (n1) at (0em,0) {};
      \node[inner sep=1pt] (n2) at (1em,0) {};
      \node[inner sep=1pt] (n3) at (2em,0) {};
      \node[inner sep=1pt] (n4) at (3em,0) {};
      \draw (n1) --++(0,-0.75em) -| (n2);
      \draw ($(n2)+(0,-0.75em)$) -| (n3);
      \draw (n4) --++(0,-0.75em);            
     %\draw [brown] (-0.25em,-1em) rectangle (3.25em, 0.125em);
    \useasboundingbox (-0.25em,-1em) rectangle (3.25em, 0.125em);
  \end{tikzpicture}, \
       \\
                                 &\phantom{{}={}} +6\kappa_2\kappa_1^2 &&
 \begin{tikzpicture}[baseline=-2pt-0.375em]
      \node[inner sep=1pt] (n1) at (0em,0) {};
      \node[inner sep=1pt] (n2) at (1em,0) {};
      \node[inner sep=1pt] (n3) at (2em,0) {};
      \node[inner sep=1pt] (n4) at (3em,0) {};
      \draw (n1) --++(0,-0.75em) -| (n2);
      \draw (n3) --++(0,-0.75em);
      \draw (n4) --++(0,-0.75em);                  
     %\draw [brown] (-0.25em,-1em) rectangle (3.25em, 0.125em);
    \useasboundingbox (-0.25em,-1em) rectangle (3.25em, 0.125em);
  \end{tikzpicture}, \
 \begin{tikzpicture}[baseline=-2pt-0.5em]
      \node[inner sep=1pt] (n1) at (0em,0) {};
      \node[inner sep=1pt] (n2) at (1em,0) {};
      \node[inner sep=1pt] (n3) at (2em,0) {};
      \node[inner sep=1pt] (n4) at (3em,0) {};
      \draw (n1) --++(0,-1em) -| (n3);
      \draw (n2) --++(0,-0.5em);
      \draw (n4) --++(0,-0.5em);                  
     %\draw [brown] (-0.25em,-1.125em) rectangle (3.25em, 0.125em);
    \useasboundingbox (-0.25em,-1.125em) rectangle (3.25em, 0.125em);
  \end{tikzpicture}, \
                                                                           \begin{tikzpicture}[baseline=-2pt-0.5em]
      \node[inner sep=1pt] (n1) at (0em,0) {};
      \node[inner sep=1pt] (n2) at (1em,0) {};
      \node[inner sep=1pt] (n3) at (2em,0) {};
      \node[inner sep=1pt] (n4) at (3em,0) {};
      \draw (n1) --++(0,-1em) -| (n4);
      \draw (n2) --++(0,-0.5em);
      \draw (n3) --++(0,-0.5em);                  
     %\draw [brown] (-0.25em,-1.125em) rectangle (3.25em, 0.125em);
    \useasboundingbox (-0.25em,-1.125em) rectangle (3.25em, 0.125em);
  \end{tikzpicture}, \
 \begin{tikzpicture}[baseline=-2pt-0.5em]
      \node[inner sep=1pt] (n1) at (0em,0) {};
      \node[inner sep=1pt] (n2) at (1em,0) {};
      \node[inner sep=1pt] (n3) at (2em,0) {};
      \node[inner sep=1pt] (n4) at (3em,0) {};
      \draw (n2) --++(0,-1em) -| (n4);
      \draw (n1) --++(0,-0.5em);
      \draw (n3) --++(0,-0.5em);                  
     %\draw [brown] (-0.25em,-1.125em) rectangle (3.25em, 0.125em);
    \useasboundingbox (-0.25em,-1.125em) rectangle (3.25em, 0.125em);
  \end{tikzpicture},\\
                                        &  &&
 \begin{tikzpicture}[baseline=-2pt-0.375em]
      \node[inner sep=1pt] (n1) at (0em,0) {};
      \node[inner sep=1pt] (n2) at (1em,0) {};
      \node[inner sep=1pt] (n3) at (2em,0) {};
      \node[inner sep=1pt] (n4) at (3em,0) {};
      \draw (n2) --++(0,-0.75em) -| (n3);
      \draw (n1) --++(0,-0.75em);
      \draw (n4) --++(0,-0.75em);                  
     %\draw [brown] (-0.25em,-1em) rectangle (3.25em, 0.125em);
    \useasboundingbox (-0.25em,-1em) rectangle (3.25em, 0.125em);
  \end{tikzpicture}, \
\begin{tikzpicture}[baseline=-2pt-0.375em]
      \node[inner sep=1pt] (n1) at (0em,0) {};
      \node[inner sep=1pt] (n2) at (1em,0) {};
      \node[inner sep=1pt] (n3) at (2em,0) {};
      \node[inner sep=1pt] (n4) at (3em,0) {};
      \draw (n3) --++(0,-0.75em) -| (n4);
      \draw (n1) --++(0,-0.75em);
      \draw (n2) --++(0,-0.75em);                  
     %\draw [brown] (-0.25em,-1em) rectangle (3.25em, 0.125em);
    \useasboundingbox (-0.25em,-1em) rectangle (3.25em, 0.125em);
  \end{tikzpicture}, \                                                                          
       \\
                                 &\phantom{{}={}} +2\kappa_2^2 &&
\begin{tikzpicture}[baseline=-2pt-0.375em]
      \node[inner sep=1pt] (n1) at (0em,0) {};
      \node[inner sep=1pt] (n2) at (1em,0) {};
      \node[inner sep=1pt] (n3) at (2em,0) {};
      \node[inner sep=1pt] (n4) at (3em,0) {};
      \draw (n1) --++(0,-0.75em) -| (n2);
      \draw (n3) --++(0,-0.75em) -| (n4);      
     %\draw [brown] (-0.25em,-1em) rectangle (3.25em, 0.125em);
    \useasboundingbox (-0.25em,-1em) rectangle (3.25em, 0.125em);
  \end{tikzpicture}, \
 \begin{tikzpicture}[baseline=-2pt-0.5em]
      \node[inner sep=1pt] (n1) at (0em,0) {};
      \node[inner sep=1pt] (n2) at (1em,0) {};
      \node[inner sep=1pt] (n3) at (2em,0) {};
      \node[inner sep=1pt] (n4) at (3em,0) {};
      \draw (n2) --++(0,-0.5em) -| (n3);
      \draw (n1) --++(0,-1em) -| (n4);      
     %\draw [brown] (-0.25em,-1.125em) rectangle (3.25em, 0.125em);
    \useasboundingbox (-0.25em,-1.125em) rectangle (3.25em, 0.125em);
  \end{tikzpicture}, \
       \\
                                 &\phantom{{}={}} +\kappa_1^4 &&
\begin{tikzpicture}[baseline=-2pt-0.375em]
      \node[inner sep=1pt] (n1) at (0em,0) {};
      \node[inner sep=1pt] (n2) at (1em,0) {};
      \node[inner sep=1pt] (n3) at (2em,0) {};
      \node[inner sep=1pt] (n4) at (3em,0) {};
      \draw (n1) --++(0,-0.75em);
      \draw (n2) --++(0,-0.75em);
      \draw (n3) --++(0,-0.75em);
      \draw (n4) --++(0,-0.75em);      
     %\draw [brown] (-0.25em,-1em) rectangle (3.25em, 0.125em);
    \useasboundingbox (-0.25em,-1em) rectangle (3.25em, 0.125em);
  \end{tikzpicture}                                                                 
       \\
\quad\\
       \varphi(a^4_{\text{classical}})&=\varphi(a^4_{\text{free}})+\kappa_2^2 && \text{all the above plus} \
 \begin{tikzpicture}[baseline=-2pt-0.5em]
      \node[inner sep=1pt] (n1) at (0em,0) {};
      \node[inner sep=1pt] (n2) at (1em,0) {};
      \node[inner sep=1pt] (n3) at (2em,0) {};
      \node[inner sep=1pt] (n4) at (3em,0) {};
      \draw (n1) --++(0,-0.5em) -| (n3);
      \draw (n2) --++(0,-1em) -| (n4);      
     %\draw [brown] (-0.25em,-1.125em) rectangle (3.25em, 0.125em);
    \useasboundingbox (-0.25em,-1.125em) rectangle (3.25em, 0.125em);
  \end{tikzpicture}.
     \end{align*}
   \item Let $(a_i)_{i\in \naturalnumbers}$ be as in the Central Limit Theorems~\ref{theorem:classical-central-limit-theorem} or~\ref{theorem:free-central-limit-theorem} and define
     %\begin{align*}
       $a_{n,i}\equalperdefinition {a_i}/{\sqrt{n}}$
     %\end{align*}
     for every $n\in \naturalnumbers$ and $i\in [n]$. We thus obtain a family $(a_{n,i})_{n\in \naturalnumbers, i\in [n]}$ as in Theorem~\ref{theorem:poisson-limit-theorem} if we choose the sequence $(\kappa_r)_{r\in \naturalnumbers}$ in $\complexnumbers$ according to
     \begin{align*}
       \kappa_1&\overset{!}{=}\lim_{n\to\infty} n\cdot \frac{\varphi(a_i)}{\sqrt{n}}=0 && \text{since } \varphi(a_i)=0,\\
       \kappa_2&\overset{!}{=}\lim_{n\to\infty} n\cdot \frac{\varphi(a_i^2)}{n}=1 && \text{since } \varphi(a_i^2)=1 \text{ and}\\
       \kappa_r&\overset{!}{=}\lim_{n\to\infty} n\cdot \frac{\varphi(a_i^r)}{n^{{r}/{2}}}=0 && \text{for all }r\in \naturalnumbers \text{ with } r\geq 3.       
     \end{align*}
     Hence, the formulas in Theorem~\ref{theorem:poisson-limit-theorem} also prove Theorems~\ref{theorem:classical-central-limit-theorem} and~\ref{theorem:free-central-limit-theorem}.
   \item If we choose in  Theorem~\ref{theorem:poisson-limit-theorem} the family $(a_{i,n})_{n\in \naturalnumbers,i\in [n]}$ and the sequence $(\kappa_r)_{r\in \naturalnumbers}$ as given in Remark~\hyperref[remark:limit-distribution-5]{\ref*{remark:limit-distribution}~\ref*{remark:limit-distribution-5}} for the Law of Rare Events, and, in particular, for the parameter $1$, then we see with
     \begin{align*}
       \kappa_r\equalperdefinition \lim_{n\to\infty}n\cdot\frac{1}{n}=1
     \end{align*}
     that Theorem~\ref{theorem:poisson-limit-theorem} gives in this case
     \begin{align*}
       \varphi(a^k)&=
                     \begin{cases}
                       \#\setofpartitionsof(k) &\text{in the classical case and }  \\
                       \#\setofnoncrossingpartitionsof(k) &\text{in the free case.}
                     \end{cases}
     \end{align*}
     Indeed, the moments of the classical Poisson distribution of parameter $1$,
     \begin{align*}
       \sum_{p=0}^\infty\frac{1}{e}\frac{1}{p!}\delta_p,
     \end{align*}
     are given by the sequence $(\#\setofpartitionsof(k))_{k\in \naturalnumbers}$: For all $k\in \naturalnumbers_0$ it holds that 
     \begin{align*}
       \sum_{p=0}^\infty p^k\cdot\frac{1}{e}\frac{1}{p!}=\# \setofpartitionsof(k)
     \end{align*}
     (which is not obvious, see Assignment~\hyperref[assignment-4]{4}, Exercise~3).
     The numbers $\#\setofpartitionsof(k)$ are called the \emph{Bell numbers}. They satisfy nice recursion equations, but there is no easy explicit formula for them:
     \begin{align*}
              \begin{array}{c| c c c c c c c c}
         k & 1 & 2 & 3 & 4 & 5 & 6 & 7 &\ldots \\[0.2em] \hline
         \#\setofpartitionsof(k)&  1 & 2 & 5 & 15 & 52 & 203 & 877 & \ldots
       \end{array}
     \end{align*}
   \end{enumerate}   
 \end{example}

 \begin{definition}
   \label{definition:free-poisson-variable}
   Let $(\mathcal{A},\varphi)$ be a $\ast$-probability space and $\lambda\in \realnumbers$, $\lambda >0$. A self-adjoint random variable $x\in \mathcal{A}$ is called \emph{free Poisson element} with parameter $\lambda$ if its moments are of the form
   \begin{align*}
     \varphi(x^k)=\sum_{\pi\in \setofnoncrossingpartitionsof(k)} \lambda^{\#\pi}
   \end{align*}
   for every $k\in \naturalnumbers$. In particular, $\lambda=1$ means that
     $\varphi(x^k)=\#\setofnoncrossingpartitionsof(k)$.
 \end{definition}

 \begin{remark}
   So what are the \enquote{free Bell numbers}? It turns out that, for all $k\in \naturalnumbers$,
   \begin{align*}
     \#\setofnoncrossingpartitionsof(k)=C_k=\#\setofnoncrossingpartitionsof_2(2k).
   \end{align*}
   Hence the \enquote{free Bell numbers} are just the Catalan numbers again.
   \begin{align*}
     \begin{matrix}
       & \setofnoncrossingpartitionsof(k) & \# & \setofnoncrossingpartitionsof_2(2k) \\
       k=1 & \begin{tikzpicture}[baseline=-2pt-0.375em]
      \node[inner sep=1pt] (n1) at (0em,0) {};
      \draw (n1) --++(0,-0.75em);      
    %\draw [brown] (-0.25em,-1em) rectangle (0.25em, 0.125em);
    \useasboundingbox (-0.25em,-1em) rectangle (0.25em, 0.125em);
  \end{tikzpicture} & 1 & \begin{tikzpicture}[baseline=-2pt-0.375em]
      \node[inner sep=1pt] (n1) at (0em,0) {};
      \node[inner sep=1pt] (n2) at (1em,0) {};
      \draw (n1) --++(0,-0.75em) -| (n2);
    %\draw [brown] (-0.25em,-1em) rectangle (1.25em, 0.125em);
    \useasboundingbox (-0.25em,-1em) rectangle (1.25em, 0.125em);
  \end{tikzpicture} \\[1em]
       k=2 &       \begin{tikzpicture}[baseline=-2pt-0.375em]
      \node[inner sep=1pt] (n1) at (0em,0) {};
      \node[inner sep=1pt] (n2) at (1em,0) {};
      \draw (n1) --++(0,-0.75em) -| (n2);
    %\draw [brown] (-0.25em,-1em) rectangle (1.25em, 0.125em);
    \useasboundingbox (-0.25em,-1em) rectangle (1.25em, 0.125em);
  \end{tikzpicture}, \
                     \begin{tikzpicture}[baseline=-2pt-0.375em]
      \node[inner sep=1pt] (n1) at (0em,0) {};
      \node[inner sep=1pt] (n2) at (1em,0) {};
      \draw (n1) --++(0,-0.75em);
      \draw (n2) --++(0,-0.75em);      
    %\draw [brown] (-0.25em,-1em) rectangle (1.25em, 0.125em);
    \useasboundingbox (-0.25em,-1em) rectangle (1.25em, 0.125em);
  \end{tikzpicture}& 2 & \begin{tikzpicture}[baseline=-2pt-0.375em]
      \node[inner sep=1pt] (n1) at (0em,0) {};
      \node[inner sep=1pt] (n2) at (1em,0) {};
      \node[inner sep=1pt] (n3) at (2em,0) {};
      \node[inner sep=1pt] (n4) at (3em,0) {};
      \draw (n1) --++(0,-0.75em) -| (n2);
      \draw (n3) --++(0,-0.75em) -| (n4);      
     %\draw [brown] (-0.25em,-1em) rectangle (3.25em, 0.125em);
    \useasboundingbox (-0.25em,-1em) rectangle (3.25em, 0.125em);
  \end{tikzpicture}, \
 \begin{tikzpicture}[baseline=-2pt-0.5em]
      \node[inner sep=1pt] (n1) at (0em,0) {};
      \node[inner sep=1pt] (n2) at (1em,0) {};
      \node[inner sep=1pt] (n3) at (2em,0) {};
      \node[inner sep=1pt] (n4) at (3em,0) {};
      \draw (n2) --++(0,-0.5em) -| (n3);
      \draw (n1) --++(0,-1em) -| (n4);      
     %\draw [brown] (-0.25em,-1.125em) rectangle (3.25em, 0.125em);
    \useasboundingbox (-0.25em,-1.125em) rectangle (3.25em, 0.125em);
  \end{tikzpicture}\\[1em]
       k=3 & \begin{gathered}
\begin{tikzpicture}[baseline=-2pt-0.375em]
      \node[inner sep=1pt] (n1) at (0em,0) {};
      \node[inner sep=1pt] (n2) at (1em,0) {};
      \node[inner sep=1pt] (n3) at (2em,0) {};
      \draw (n1) --++(0,-0.75em) -| (n2);
      \draw ($(n2)+(0,-0.75em)$) -| (n3);      
    %\draw [brown] (-0.25em,-1em) rectangle (2.25em, 0.125em);
    \useasboundingbox (-0.25em,-1em) rectangle (2.25em, 0.125em);
  \end{tikzpicture}, \
                     \begin{tikzpicture}[baseline=-2pt-0.375em]
      \node[inner sep=1pt] (n1) at (0em,0) {};
      \node[inner sep=1pt] (n2) at (1em,0) {};
      \node[inner sep=1pt] (n3) at (2em,0) {};
      \draw (n2) --++(0,-0.75em) -| (n3);
      \draw (n1) --++(0,-0.75em);      
    %\draw [brown] (-0.25em,-1em) rectangle (2.25em, 0.125em);
    \useasboundingbox (-0.25em,-1em) rectangle (2.25em, 0.125em);
  \end{tikzpicture},  \
                     \begin{tikzpicture}[baseline=-2pt-0.5em]
      \node[inner sep=1pt] (n1) at (0em,0) {};
      \node[inner sep=1pt] (n2) at (1em,0) {};
      \node[inner sep=1pt] (n3) at (2em,0) {};
      \draw (n1) --++(0,-1em) -| (n3);
      \draw (n2) --++(0,-0.5em);      
    %\draw [brown] (-0.25em,-1.125em) rectangle (2.25em, 0.125em);
    \useasboundingbox (-0.25em,-1.125em) rectangle (2.25em, 0.125em);
  \end{tikzpicture}, \\
                     \begin{tikzpicture}[baseline=-2pt-0.375em]
      \node[inner sep=1pt] (n1) at (0em,0) {};
      \node[inner sep=1pt] (n2) at (1em,0) {};
      \node[inner sep=1pt] (n3) at (2em,0) {};
      \draw (n1) --++(0,-0.75em) -| (n2);
      \draw (n3) --++(0,-0.75em);      
    %\draw [brown] (-0.25em,-1em) rectangle (2.25em, 0.125em);
    \useasboundingbox (-0.25em,-1em) rectangle (2.25em, 0.125em);
  \end{tikzpicture}, \
                     \begin{tikzpicture}[baseline=-2pt-0.375em]
      \node[inner sep=1pt] (n1) at (0em,0) {};
      \node[inner sep=1pt] (n2) at (1em,0) {};
      \node[inner sep=1pt] (n3) at (2em,0) {};
      \draw (n1) --++(0,-0.75em);      
      \draw (n2) --++(0,-0.75em);
      \draw (n3) --++(0,-0.75em);      
    %\draw [brown] (-0.25em,-1em) rectangle (2.25em, 0.125em);
    \useasboundingbox (-0.25em,-1em) rectangle (2.25em, 0.125em);
  \end{tikzpicture}
\end{gathered} & 5 & \begin{gathered}\begin{tikzpicture}[baseline=-2pt-0.375em]
      \node[inner sep=1pt] (n1) at (0em,0) {};
      \node[inner sep=1pt] (n2) at (1em,0) {};
      \node[inner sep=1pt] (n3) at (2em,0) {};
      \node[inner sep=1pt] (n4) at (3em,0) {};
      \node[inner sep=1pt] (n5) at (4em,0) {};
      \node[inner sep=1pt] (n6) at (5em,0) {};      
      \draw (n1) --++(0,-0.75em) -| (n2);
      \draw (n3) --++(0,-0.75em) -| (n4);
      \draw (n5) --++(0,-0.75em) -| (n6);      
    %\draw [brown] (-0.25em,-1em) rectangle (5.25em, 0.125em);
    \useasboundingbox (-0.25em,-1em) rectangle (5.25em, 0.125em);
  \end{tikzpicture}, \
  \begin{tikzpicture}[baseline=-2pt-0.5em]
      \node[inner sep=1pt] (n1) at (0em,0) {};
      \node[inner sep=1pt] (n2) at (1em,0) {};
      \node[inner sep=1pt] (n3) at (2em,0) {};
      \node[inner sep=1pt] (n4) at (3em,0) {};
      \node[inner sep=1pt] (n5) at (4em,0) {};
      \node[inner sep=1pt] (n6) at (5em,0) {};      
      \draw (n1) --++(0,-1em) -| (n4);
      \draw (n2) --++(0,-0.5em) -| (n3);
      \draw (n5) --++(0,-0.75em) -| (n6);      
    %\draw [brown] (-0.25em,-1.125em) rectangle (5.25em, 0.125em);
    \useasboundingbox (-0.25em,-1.125em) rectangle (5.25em, 0.125em);
  \end{tikzpicture}, \
    \begin{tikzpicture}[baseline=-2pt-0.5em]
      \node[inner sep=1pt] (n1) at (0em,0) {};
      \node[inner sep=1pt] (n2) at (1em,0) {};
      \node[inner sep=1pt] (n3) at (2em,0) {};
      \node[inner sep=1pt] (n4) at (3em,0) {};
      \node[inner sep=1pt] (n5) at (4em,0) {};
      \node[inner sep=1pt] (n6) at (5em,0) {};      
      \draw (n3) --++(0,-1em) -| (n6);
      \draw (n4) --++(0,-0.5em) -| (n5);
      \draw (n1) --++(0,-0.75em) -| (n2);      
    %\draw [brown] (-0.25em,-1.125em) rectangle (5.25em, 0.125em);
    \useasboundingbox (-0.25em,-1.125em) rectangle (5.25em, 0.125em);
  \end{tikzpicture},\\
  \begin{tikzpicture}[baseline=-2pt-0.5em]
      \node[inner sep=1pt] (n1) at (0em,0) {};
      \node[inner sep=1pt] (n2) at (1em,0) {};
      \node[inner sep=1pt] (n3) at (2em,0) {};
      \node[inner sep=1pt] (n4) at (3em,0) {};
      \node[inner sep=1pt] (n5) at (4em,0) {};
      \node[inner sep=1pt] (n6) at (5em,0) {};      
      \draw (n1) --++(0,-1em) -| (n6);
      \draw (n2) --++(0,-0.5em) -| (n3);
      \draw (n4) --++(0,-0.5em) -| (n5);      
    %\draw [brown] (-0.25em,-1.125em) rectangle (5.25em, 0.125em);
    \useasboundingbox (-0.25em,-1.125em) rectangle (5.25em, 0.125em);
  \end{tikzpicture}, \
      \begin{tikzpicture}[baseline=-2pt-0.75em]
      \node[inner sep=1pt] (n1) at (0em,0) {};
      \node[inner sep=1pt] (n2) at (1em,0) {};
      \node[inner sep=1pt] (n3) at (2em,0) {};
      \node[inner sep=1pt] (n4) at (3em,0) {};
      \node[inner sep=1pt] (n5) at (4em,0) {};
      \node[inner sep=1pt] (n6) at (5em,0) {};      
      \draw (n2) --++(0,-1em) -| (n5);
      \draw (n3) --++(0,-0.5em) -| (n4);
      \draw (n1) --++(0,-1.5em) -| (n6);      
    %\draw [brown] (-0.25em,-1.625em) rectangle (5.25em, 0.125em);
    \useasboundingbox (-0.25em,-1.625em) rectangle (5.25em, 0.125em);
  \end{tikzpicture}
\end{gathered}\\[1em]
       k=4 & \ldots & 15 & \ldots \\       
     \end{matrix}
   \end{align*}
   This is very different from the classical world!
      \begin{align*}
     \begin{matrix}
       & \setofpartitionsof(k) & \# &\#& \setofpartitionsof_2(2k) \\
       k=1 & \begin{tikzpicture}[baseline=-2pt-0.375em]
      \node[inner sep=1pt] (n1) at (0em,0) {};
      \draw (n1) --++(0,-0.75em);      
    %\draw [brown] (-0.25em,-1em) rectangle (0.25em, 0.125em);
    \useasboundingbox (-0.25em,-1em) rectangle (0.25em, 0.125em);
  \end{tikzpicture}&1 &1 & \begin{tikzpicture}[baseline=-2pt-0.375em]
      \node[inner sep=1pt] (n1) at (0em,0) {};
      \node[inner sep=1pt] (n2) at (1em,0) {};
      \draw (n1) --++(0,-0.75em) -| (n2);
    %\draw [brown] (-0.25em,-1em) rectangle (1.25em, 0.125em);
    \useasboundingbox (-0.25em,-1em) rectangle (1.25em, 0.125em);
  \end{tikzpicture}\\[1em]
       k=2 & \begin{tikzpicture}[baseline=-2pt-0.375em]
      \node[inner sep=1pt] (n1) at (0em,0) {};
      \node[inner sep=1pt] (n2) at (1em,0) {};
      \draw (n1) --++(0,-0.75em) -| (n2);
    %\draw [brown] (-0.25em,-1em) rectangle (1.25em, 0.125em);
    \useasboundingbox (-0.25em,-1em) rectangle (1.25em, 0.125em);
  \end{tikzpicture}, \
                     \begin{tikzpicture}[baseline=-2pt-0.375em]
      \node[inner sep=1pt] (n1) at (0em,0) {};
      \node[inner sep=1pt] (n2) at (1em,0) {};
      \draw (n1) --++(0,-0.75em);
      \draw (n2) --++(0,-0.75em);      
    %\draw [brown] (-0.25em,-1em) rectangle (1.25em, 0.125em);
    \useasboundingbox (-0.25em,-1em) rectangle (1.25em, 0.125em);
  \end{tikzpicture}&2 &3 & \begin{tikzpicture}[baseline=-2pt-0.375em]
      \node[inner sep=1pt] (n1) at (0em,0) {};
      \node[inner sep=1pt] (n2) at (1em,0) {};
      \node[inner sep=1pt] (n3) at (2em,0) {};
      \node[inner sep=1pt] (n4) at (3em,0) {};
      \draw (n1) --++(0,-0.75em) -| (n2);
      \draw (n3) --++(0,-0.75em) -| (n4);      
     %\draw [brown] (-0.25em,-1em) rectangle (3.25em, 0.125em);
    \useasboundingbox (-0.25em,-1em) rectangle (3.25em, 0.125em);
  \end{tikzpicture}, \
 \begin{tikzpicture}[baseline=-2pt-0.5em]
      \node[inner sep=1pt] (n1) at (0em,0) {};
      \node[inner sep=1pt] (n2) at (1em,0) {};
      \node[inner sep=1pt] (n3) at (2em,0) {};
      \node[inner sep=1pt] (n4) at (3em,0) {};
      \draw (n2) --++(0,-0.5em) -| (n3);
      \draw (n1) --++(0,-1em) -| (n4);      
     %\draw [brown] (-0.25em,-1.125em) rectangle (3.25em, 0.125em);
    \useasboundingbox (-0.25em,-1.125em) rectangle (3.25em, 0.125em);
  \end{tikzpicture}, \  \begin{tikzpicture}[baseline=-2pt-0.5em]
      \node[inner sep=1pt] (n1) at (0em,0) {};
      \node[inner sep=1pt] (n2) at (1em,0) {};
      \node[inner sep=1pt] (n3) at (2em,0) {};
      \node[inner sep=1pt] (n4) at (3em,0) {};
      \draw (n1) --++(0,-0.5em) -| (n3);
      \draw (n2) --++(0,-1em) -| (n4);      
     %\draw [brown] (-0.25em,-1.125em) rectangle (3.25em, 0.125em);
    \useasboundingbox (-0.25em,-1.125em) rectangle (3.25em, 0.125em);
  \end{tikzpicture}\\[1em]
       k=3 & \ldots &5 &15 & \ldots\\[1em]
       k=4 & \ldots &15 &105 & \ldots \\
       & & \uparrow& \uparrow & &\\
       &&\text{Bell}& (2k-1)!! &&\\
       &&\text{numbers}&  &&       
     \end{matrix}
      \end{align*}
      The \enquote{coincidence} that $\# \setofnoncrossingpartitionsof(k)=\# \setofnoncrossingpartitionsof_2(2k)$ for all $k\in \naturalnumbers$ means in particular that the square $s^2$ of a standard semicircle element $s$ is a free Poisson variable with parameter $1$. For every $k\in \naturalnumbers$,
      \begin{align*} \varphi((s^2)^k)=\varphi(s^{2k})=\#\setofnoncrossingpartitionsof_2(2k)=\#\setofnoncrossingpartitionsof(k).
      \end{align*}
      We can determine the density of the distribution of $s^2$. For every $k\in \naturalnumbers$,
      \begin{align*}
        \varphi((s^2)^k)&=\frac{1}{2\pi}\int_{-2}^2t^{2k}\sqrt{4-t^2}\, dt\\
                        &=\frac{1}{\pi}\int_{0}^2t^{2k}\sqrt{4-t^2}\, dt\quad\quad    \begin{aligned}x&=t^2\\dx&=2t\,dt\end{aligned}\\
                        &=\frac{1}{\pi}\int_{0}^4x^{k}\sqrt{4-x}\, \frac{dx}{2\sqrt{x}}   \\
        &=\frac{1}{2\pi}\int_{0}^4x^{k}\sqrt{\frac{4}{x}-1}\, dx,
      \end{align*}
      yielding as density
      \begin{align*}
      x\mapsto
        \begin{cases}
          \frac 1{2\pi}\sqrt{\frac{4}{x}-1}, &\text{if }x \in (0,4],\\
          0,&\text{otherwise.}
        \end{cases}
      \end{align*}
      In classical probability theory there is no relation between the square of a normal variable and a Poisson variable!
   
$$
        \begin{tikzpicture}[baseline=3cm]
          \begin{axis}[xmin=-0.20, ymin=-0.05, xmax=4*1.1, domain=0:4, ymax=1, samples=500, title={Free Poisson of parameter $1$}, axis lines=middle, width=18em, height=16em]
  \addplot[thick] {(0.5*pi^-1)*((4*(x^-1)-1)^0.5)};
\end{axis}
\end{tikzpicture}\quad
        \begin{tikzpicture}[baseline=3cm]
          \begin{axis}[xmin=-0.25, ymin=-0.05, xmax=7, ymax=1, samples at={0,1,2,3,4,5,6}, title={Classical Poisson of parameter $1$}, axis lines=middle, height=16em, width=20em]
  \addplot[thick, only marks] {(e^-1)*((factorial(x))^-1)};
\end{axis}
\end{tikzpicture}
$$
 \end{remark}

 \begin{remark}[Joint moments of free variables]
   Let $(\mathcal A,\varphi)$ be a non-com\-mu\-ta\-ti\-ve probability space, let $a_1,a_2\in \mathcal A$, let
     $(a^{(1)}_{n,i})_{n\in \naturalnumbers, i\in [n]}$ and $(a^{(2)}_{n,i})_{n\in \naturalnumbers, i\in [n]}$
   be two families of random variables
as in Remark~\ref{paragraph:poisson-limit-theorem}; in particular,  let $(\kappa_r^{(1)})_{r\in \naturalnumbers}$ and $(\kappa_r^{(2)})_{r\in \naturalnumbers}$ be two sequences in $\complexnumbers$ and assume that, for all $n\in \naturalnumbers$, $i\in [n]$ and $r\in \naturalnumbers$,
   \begin{align*}
     \varphi((a^{(1)}_{n,i})^r)=\frac{\kappa^{(1)}_r}{n} \qquad\text{and}\qquad \varphi((a^{(2)}_{n,i})^r)=\frac{\kappa^{(2)}_r}{n}
   \end{align*}
    and that
   \begin{align*}
     a^{(1)}_{n,1}+\ldots+a^{(1)}_{n,n}\convergesindistributionto a_1 \qquad\text{and}\qquad
     a^{(2)}_{n,1}+\ldots+a^{(2)}_{n,n}\convergesindistributionto a_2.
   \end{align*}
Let us look in the following at the case of free variables; the case of independent variables works in the same way.
  Assume in addition that, for every $n\in \naturalnumbers$, the sets 
     $\{a^{(1)}_{n,1},\ldots,a^{(1)}_{n,n}\}$ and $\{a^{(2)}_{n,1},\ldots,a^{(2)}_{n,n}\}$ are free.
   This implies then (see Assignment~\hyperref[assignment-3]{3}, Exercise~4) that also the limits $a_1$ and $a_2$ are free. Hence we can calculate joint moments in free $a_1, a_2$ via this representation and try to express it in terms of $(\kappa_r^{(1)})_{r\in \naturalnumbers}$ and $(\kappa^{(2)}_r)_{r\in \naturalnumbers}$. Fix $k\in \naturalnumbers$ and $p_1,\ldots,p_k\in \{1,2\}$. Then, 
   \begin{align*}
     \varphi(a_{p_1}\ldots a_{p_k})&=\lim_{n\to\infty}\varphi[(a^{(p_1)}_{n,1}+\ldots+a^{(p_1)}_{n,n})\ldots(a^{(p_k)}_{n,1}+\ldots+a^{(p_k)}_{n,n})]\\
                                   &=\lim_{n\to\infty}\hspace{-3.8em}\underset{\hspace{4em}\displaystyle \sim \sum_{\pi\in \setofpartitionsof(k)}g^{(p_1,\ldots,p_k)}_n(\pi)\cdot n^{\#\pi} \ (n\to\infty)}{\underbrace{\sum_{i:[k]\to[n]}\varphi(a^{(p_1)}_{n,i(1)}\ldots a^{(p_k)}_{n,i(k)})}}\\
     &=\sum_{\pi\in \setofpartitionsof(k)}\lim_{n\to\infty} \underset{\displaystyle (\ast)}{\underbrace{n^{\# \pi} \cdot g^{(p_1,\ldots,p_k)}_n(\pi)}}.
   \end{align*}
For which $\pi\in \setofpartitionsof(k)$ does the term $(\ast)$ survive in the limit $n\to\infty$? For $\pi\in \setofnoncrossingpartitionsof(k)$ the moment $g^{(p_1,\ldots,p_k)}_n(\pi)$ factorizes into products of moments according to the blocks of $\pi$. If in a block $V\in \pi$ of such a $\pi$ all $p_1,\ldots,p_k$ belonging to $V$ are $1$ or all $2$, then this $\pi$ gives the contribution $\kappa^{(1)}_{\# V}$ or $\kappa^{(2)}_{\# V}$ respectively. If, however, both $1$ and $2$ appear among the $p_1,\ldots,p_k$ in a given block $V\in \pi$, then this moment factorizes further into (joint) moments of $(a^{(1)}_{n,i})_{n\in \naturalnumbers,i\in [n]}$ and $(a^{(2)}_{n,i})_{n\in \naturalnumbers,i\in [n]}$. In that case $g^{(p_1,\ldots,p_k)}_n(\pi)$ provides more than $\# \pi$ many factors $\frac{1}{n}$, which we do not have enough $n$-factors to compensate for. Hence, such terms go to zero. Thus, we conclude:
   \begin{align*}
     \varphi(a_{p_1}\ldots a_{p_k})=\hspace{-2em}\sum_{\substack{\pi\in \setofnoncrossingpartitionsof(k)\\\pi\leq \ker(p)\\\displaystyle\uparrow\\\displaystyle p \text{ must be constant}\\\displaystyle\text{on each block of }\pi}}\hspace{-2em}\prod_{V\in \pi} \kappa^{(\left.p\right|_V)}_{\# V}.
   \end{align*}
   for every $k\in \naturalnumbers$ and every $p:[k]\to [2]$.
 \end{remark}

 \begin{theorem}
   \label{theorem:joint-moments}
   Let $(\mathcal{A},\varphi)$ be a non-commutative probability space, $a_1,a_2\in \mathcal{A}$ and $a_1$ and $a_2$ either classically or freely independent. If there are sequences $(\kappa_r^{(1)})_{r\in \naturalnumbers}$ and  $(\kappa_r^{(2)})_{r\in \naturalnumbers}$ in $\complexnumbers$ such that, for every $l\in [2]$, the moments of $a_l$ can be written in the form
        \begin{align*}
       \varphi(a^k_l)&=
                      \begin{cases}
                        \displaystyle \hspace{0.3em}\sum_{\pi\in \setofpartitionsof(k)}\hspace{0.3em}\prod_{V\in \pi}\kappa_{\# V}^{(l)} &\text{ in the classical case and}\\
\quad\\
                     \displaystyle \sum_{\pi\in \setofnoncrossingpartitionsof(k)}\prod_{V\in \pi}\kappa_{\# V}^{(l)}&\text{ in the free case},   
                      \end{cases}
        \end{align*}
        for all $k\in \naturalnumbers$,
        then the joint moments of $(a_1,a_2)$ are of the form
                \begin{align*}
       \varphi(a_{p_1}\ldots a_{p_k})&=
                      \begin{cases}
                        \displaystyle \hspace{0em}\sum_{\substack{\pi\in \setofpartitionsof(k)\\\ker(p)\leq \pi}}\hspace{0em}\prod_{V\in \pi}\kappa_{\# V}^{(\left.p\right|_V)} &\text{ in the classical case and}\\
\quad\\
                     \displaystyle \sum_{\substack{\pi\in \setofnoncrossingpartitionsof(k)\\\ker(p)\leq \pi}}\prod_{V\in \pi}\kappa_{\# V}^{(\left.p\right|_V)}&\text{ in the free case}   
                      \end{cases}
        \end{align*}
      for all $k\in \naturalnumbers$ and $p:[k]\to [2]$.
 \end{theorem}

 \begin{example}
   In the free case we have
   \begin{align*}
     \varphi(a_1a_2a_1a_2)&=\\
                           \begin{tikzpicture}[baseline=-2pt-0.5em]
      \node[inner sep=1pt] (n1) at (0em,0) {};
      \node[inner sep=1pt] (n2) at (1em,0) {};
      \node[inner sep=1pt] (n3) at (2em,0) {};
      \node[inner sep=1pt] (n4) at (3em,0) {};
      \draw (n1) --++(0,-1em) -| (n3);
      \draw (n2) --++(0,-0.5em);
      \draw (n4) --++(0,-0.5em);                  
     %\draw [brown] (-0.25em,-1.125em) rectangle (3.25em, 0.125em);
    \useasboundingbox (-0.25em,-1.125em) rectangle (3.25em, 0.125em);
  \end{tikzpicture}
\hspace{0.8em}&\phantom{{}={}}\phantom{{}+{}}\kappa^{(1)}_2\kappa_1^{(2)}\kappa_1^{(2)}\\
 \begin{tikzpicture}[baseline=-2pt-0.5em]
      \node[inner sep=1pt] (n1) at (0em,0) {};
      \node[inner sep=1pt] (n2) at (1em,0) {};
      \node[inner sep=1pt] (n3) at (2em,0) {};
      \node[inner sep=1pt] (n4) at (3em,0) {};
      \draw (n2) --++(0,-1em) -| (n4);
      \draw (n1) --++(0,-0.5em);
      \draw (n3) --++(0,-0.5em);                  
     %\draw [brown] (-0.25em,-1.125em) rectangle (3.25em, 0.125em);
    \useasboundingbox (-0.25em,-1.125em) rectangle (3.25em, 0.125em);
  \end{tikzpicture}                          \hspace{0.8em}&\phantom{{}={}}+\kappa^{(1)}_1\kappa_1^{(1)}\kappa_2^{(2)}\\
\begin{tikzpicture}[baseline=-2pt-0.375em]
      \node[inner sep=1pt] (n1) at (0em,0) {};
      \node[inner sep=1pt] (n2) at (1em,0) {};
      \node[inner sep=1pt] (n3) at (2em,0) {};
      \node[inner sep=1pt] (n4) at (3em,0) {};
      \draw (n1) --++(0,-0.75em);
      \draw (n2) --++(0,-0.75em);
      \draw (n3) --++(0,-0.75em);
      \draw (n4) --++(0,-0.75em);      
     %\draw [brown] (-0.25em,-1em) rectangle (3.25em, 0.125em);
    \useasboundingbox (-0.25em,-1em) rectangle (3.25em, 0.125em);
  \end{tikzpicture}                                                                      \hspace{0.8em}&\phantom{{}={}}+\kappa^{(1)}_1\kappa_1^{(2)}\kappa_1^{(1)}\kappa_1^{(2)},
   \end{align*}
   whereas in the classical case we get the additional term 
   \begin{align*}\hspace{0.2em}
      \begin{tikzpicture}[baseline=-2pt-0.5em]
      \node[inner sep=1pt] (n1) at (0em,0) {};
      \node[inner sep=1pt] (n2) at (1em,0) {};
      \node[inner sep=1pt] (n3) at (2em,0) {};
      \node[inner sep=1pt] (n4) at (3em,0) {};
      \draw (n1) --++(0,-0.5em) -| (n3);
      \draw (n2) --++(0,-1em) -| (n4);      
     %\draw [brown] (-0.25em,-1.125em) rectangle (3.25em, 0.125em);
    \useasboundingbox (-0.25em,-1.125em) rectangle (3.25em, 0.125em);
  \end{tikzpicture}\hspace{1.7em} +\kappa_2^{(1)}\kappa_2^{(2)}.\hspace{2.5em}
   \end{align*}
 \end{example}

 \begin{conclusion}
   If we write the moments of random variables $a_1,a_2$ as in Theorem~\ref{theorem:joint-moments} as sums over non-crossing partitions multiplicatively in terms of sequences $(\kappa_r^{(1)})_{r\in \naturalnumbers}$ and $(\kappa_r^{(2)})_{r\in \naturalnumbers}$,  then freeness between $a_1$ and $a_2$ corresponds to having in joint moments no \enquote{mixed $\kappa$'s} (no blocks connecting different variables)! The sequences  $(\kappa_r^{(1)})_{r\in \naturalnumbers}$ and $(\kappa_r^{(2)})_{r\in \naturalnumbers}$ will be called ``free cumulants''. We will treat them more systematically in the next section.
 \end{conclusion}

\newpage

 \section{The Combinatorics of Free Probability Theory: Free Cumulants}
 \begin{definition}Let $n\in \naturalnumbers$ be arbitrary.
   \begin{enumerate}
   \item Given $\pi,\sigma\in \setofnoncrossingpartitionsof(n)$, we write $\pi\leq \sigma$ if each block of $\pi$ is completely contained in one of the blocks of $\sigma$. With this partial order, $\setofnoncrossingpartitionsof(n)$ becomes a partially ordered set (\emph{poset}). We also write $\pi<\sigma$ for: $\pi\leq \sigma$ and $\pi\neq \sigma$.
   \item The unique maximal element of $\setofnoncrossingpartitionsof(n)$ is denoted by
     \begin{align*}
       1_n\equalperdefinition\{\{1,2,\ldots,n\}\}=  \begin{tikzpicture}[baseline=-2pt-0.375em]
      \node[inner sep=1pt] (n1) at (0em,0) {};
      \node[inner sep=1pt] (n2) at (1em,0) {};
      \node[inner sep=1pt] (n3) at (2em,0) {};
      \node[inner sep=1pt] (n4) at (5em,0) {};
      \node[inner sep=1pt] (n5) at (6em,0) {};      
      \draw (n1) --++(0,-0.75em) -| (n2);
      \draw ($(n1)+(0,-0.75em)$) -| (n3);
      \draw (n4) --++(0,-0.75em) -| (n5);
      \draw ($(n3)+(0,-0.75em)$) -- ++ (0.5em,0);
      \draw ($(n4)+(0,-0.75em)$) -- ++ (-0.5em,0);
      \path ($(n3)+(0,-0.75em)$) -- node[pos=0.5] {$\ldots$}  ($(n4)+(0,-0.75em)$);
     %\draw [brown] (-0.25em,-1em) rectangle (6.25em, 0.125em);
    \useasboundingbox (-0.25em,-1em) rectangle (6.25em, 0.125em);
  \end{tikzpicture}
     \end{align*}
     and the unique minimal element of $\setofnoncrossingpartitionsof(n)$ by
     \begin{align*}
       0_n\equalperdefinition \{\{1\},\{2\},\ldots,\{n\}\}=  \begin{tikzpicture}[baseline=-2pt-0.375em]
      \node[inner sep=1pt] (n1) at (0em,0) {};
      \node[inner sep=1pt] (n2) at (1em,0) {};
      \node[inner sep=1pt] (n3) at (2em,0) {};
      \node[inner sep=1pt] (n4) at (5em,0) {};
      \node[inner sep=1pt] (n5) at (6em,0) {};      
      \draw (n1) --++(0,-0.75em);
      \draw (n2) --++(0,-0.75em);
      \draw (n3) --++(0,-0.75em);
      \draw (n4) --++(0,-0.75em);
      \draw (n5) --++(0,-0.75em);      
      \path ($(n3)+(0,-0.375em)$) -- node[pos=0.5] {$\ldots$}  ($(n4)+(0,-0.375em)$);
     %\draw [brown] (-0.25em,-1em) rectangle (6.25em, 0.125em);
    \useasboundingbox (-0.25em,-1em) rectangle (6.25em, 0.125em);
  \end{tikzpicture}.
     \end{align*}     
   \end{enumerate}
 \end{definition}
 \begin{example}
   In $\setofnoncrossingpartitionsof(8)$,
   \begin{align*}
     \begin{matrix}
       1\ \, 2\ \,  3 \ \,  4 \  \, 5\ \, 6\ \, 7\ \, 8 \ \, & \\
\begin{tikzpicture}[baseline=-2pt-0.5em]
      \node[inner sep=1pt] (n1) at (0em,0) {};
      \node[inner sep=1pt] (n2) at (1em,0) {};
      \node[inner sep=1pt] (n3) at (2em,0) {};
      \node[inner sep=1pt] (n4) at (3em,0) {};
      \node[inner sep=1pt] (n5) at (4em,0) {};
      \node[inner sep=1pt] (n6) at (5em,0) {};
      \node[inner sep=1pt] (n7) at (6em,0) {};
      \node[inner sep=1pt] (n8) at (7em,0) {};            
      \draw (n2) --++(0,-0.5em);
      \draw (n4) --++(0,-0.5em)-| (n5);
      \draw (n1) --++(0,-1em)-| (n3);
      \draw ($(n3)+(0,-1em)$) -| (n6);
      \draw ($(n6)+(0,-1em)$) -| (n7);
      \draw ($(n7)+(0,-1em)$) -| (n8);                  
     %\draw [brown] (-0.25em,-1.125em) rectangle (7.25em, 0.125em);
    \useasboundingbox (-0.25em,-1.125em) rectangle (7.25em, 0.125em);
  \end{tikzpicture}  \     &\pi\\
       &\rotatebox{90}{$\leq$}\\
\begin{tikzpicture}[baseline=-2pt-0.5em]
      \node[inner sep=1pt] (n1) at (0em,0) {};
      \node[inner sep=1pt] (n2) at (1em,0) {};
      \node[inner sep=1pt] (n3) at (2em,0) {};
      \node[inner sep=1pt] (n4) at (3em,0) {};
      \node[inner sep=1pt] (n5) at (4em,0) {};
      \node[inner sep=1pt] (n6) at (5em,0) {};
      \node[inner sep=1pt] (n7) at (6em,0) {};
      \node[inner sep=1pt] (n8) at (7em,0) {};            
      \draw (n2) --++(0,-0.5em);
      \draw (n4) --++(0,-0.75em)-| (n5);
      \draw (n1) --++(0,-1em)-| (n3);
      \draw (n7) --++(0,-0.5em);
      \draw (n6) --++(0,-1em)-| (n8);       
     %\draw [brown] (-0.25em,-1.125em) rectangle (7.25em, 0.125em);
    \useasboundingbox (-0.25em,-1.125em) rectangle (7.25em, 0.125em);
  \end{tikzpicture}  \        &\sigma
     \end{matrix}
   \end{align*}
   is true. The binary relation
   \enquote{$\leq$} is indeed only a partial order. E.g., in $\setofnoncrossingpartitionsof(3)$ 
   \begin{center}
     \begin{tikzpicture}
\begin{scope}[shift={(0cm,1.5cm)}]
      \node[inner sep=1pt] (n1) at (0em,0) {};
      \node[inner sep=1pt] (n2) at (1em,0) {};
      \node[inner sep=1pt] (n3) at (2em,0) {};
      \draw (n1) --++(0,-0.75em) -| (n2);
      \draw ($(n2)+(0,-0.75em)$) -| (n3);      
    %\draw [brown] (-0.25em,-1em) rectangle (2.25em, 0.125em);
    \useasboundingbox (-0.25em,-1em) rectangle (2.25em, 0.125em);
  \end{scope}
                     \begin{scope}[shift={(-2cm,0cm)}]
      \node[inner sep=1pt] (n1) at (0em,0) {};
      \node[inner sep=1pt] (n2) at (1em,0) {};
      \node[inner sep=1pt] (n3) at (2em,0) {};
      \draw (n2) --++(0,-0.75em) -| (n3);
      \draw (n1) --++(0,-0.75em);      
    %\draw [brown] (-0.25em,-1em) rectangle (2.25em, 0.125em);
    \useasboundingbox (-0.25em,-1em) rectangle (2.25em, 0.125em);
  \end{scope}
                     \begin{scope}[shift={(0cm,0cm)}]
      \node[inner sep=1pt] (n1) at (0em,0) {};
      \node[inner sep=1pt] (n2) at (1em,0) {};
      \node[inner sep=1pt] (n3) at (2em,0) {};
      \draw (n1) --++(0,-1em) -| (n3);
      \draw (n2) --++(0,-0.5em);      
    %\draw [brown] (-0.25em,-1.125em) rectangle (2.25em, 0.125em);
    \useasboundingbox (-0.25em,-1.125em) rectangle (2.25em, 0.125em);
  \end{scope}
                     \begin{scope}[shift={(2cm,0cm)}]
      \node[inner sep=1pt] (n1) at (0em,0) {};
      \node[inner sep=1pt] (n2) at (1em,0) {};
      \node[inner sep=1pt] (n3) at (2em,0) {};
      \draw (n1) --++(0,-0.75em) -| (n2);
      \draw (n3) --++(0,-0.75em);      
    %\draw [brown] (-0.25em,-1em) rectangle (2.25em, 0.125em);
    \useasboundingbox (-0.25em,-1em) rectangle (2.25em, 0.125em);
  \end{scope}
                     \begin{scope}[shift={(0cm,-1.5cm)}]
      \node[inner sep=1pt] (n1) at (0em,0) {};
      \node[inner sep=1pt] (n2) at (1em,0) {};
      \node[inner sep=1pt] (n3) at (2em,0) {};
      \draw (n1) --++(0,-0.75em);      
      \draw (n2) --++(0,-0.75em);
      \draw (n3) --++(0,-0.75em);      
    %\draw [brown] (-0.25em,-1em) rectangle (2.25em, 0.125em);
    \useasboundingbox (-0.25em,-1em) rectangle (2.25em, 0.125em);
  \end{scope}
  \begin{scope}[shift={(1em,-0.375em)}]
  \draw[->,gray,shorten >= 11pt, shorten <= 7.5pt] (0cm,-1.5cm) -- (0cm,0cm);
  \draw[->,gray,shorten >= 12.5pt, shorten <= 10pt] (0cm,-1.5cm) -- (-2cm,0cm);
  \draw[->,gray,shorten >= 12.5pt, shorten <= 10pt] (0cm,-1.5cm) -- (2cm,0cm);    
  \draw[->,gray,shorten >= 7.5pt, shorten <= 7.5pt] (0cm,0cm) -- (0cm,1.5cm);
  \draw[->,gray,shorten >= 12.5pt, shorten <= 12.5pt] (-2cm,0cm) -- (0cm,1.5cm);
  \draw[->,gray,shorten >= 12.5pt, shorten <= 12.5pt] (2cm,0cm) -- (0cm,1.5cm);
\end{scope}
\node[align=left] (la) at (6,-0.375em) {those three elements\\cannot be compared.};
\draw[densely dotted, ->] (la.west) --++ (-2em,0);
     \end{tikzpicture}
   \end{center}
 \end{example}
 
 \begin{definition}[Rota 1964]
   \label{definition:incidence-algebra}
     Let $P$ be a finite partially ordered set.
     \begin{enumerate}
     \item   \label{definition:incidence-algebra-1} We put
       \begin{align*}
         P^{(2)}\equalperdefinition \{(\pi,\sigma)\mid \pi,\sigma\in P, \pi\leq \sigma\}
       \end{align*}
       and define for every two functions $F,G:\, P^{(2)}\to \complexnumbers$
       their convolution
       \begin{align*}
         F\ast G: \, P^{(2)}\to \complexnumbers
       \end{align*}
       by demanding for all $(\pi,\sigma)\in P^{(2)}$ that
       \begin{align*}
         (F\ast G)(\pi,\sigma)\equalperdefinition \sum_{\substack{\tau\in P\\\pi\leq \tau\leq \sigma}} F(\pi,\tau)G(\tau,\sigma).
       \end{align*}
       We also introduce a one-variable version of this: For functions $f:\,P\to\complexnumbers$ and $G:\,P^{(2)}\to \complexnumbers$ we define $f\ast G:\, P\to \complexnumbers$ by requiring, for all $\sigma\in P$,
       \begin{align*}
         (f\ast G)(\sigma)\equalperdefinition \sum_{\substack{\tau\in P\\ \tau\leq \sigma}} f(\tau)G(\tau,\sigma).
       \end{align*}
     \item   \label{definition:incidence-algebra-2} The special functions $\delta,\zeta:\, P^{(2)}\to \complexnumbers$, the latter  named \emph{zeta function}, are defined by the condition that, for all  $(\pi,\sigma)\in P^{(2)}$,
       \begin{align*}
         \delta(\pi,\sigma)\equalperdefinition
         \begin{cases}
           1,&\text{if }\pi=\sigma,\\
           0,&\text{if }\pi<\sigma,
         \end{cases}
               \qquad\text{and}\qquad \zeta(\pi,\sigma)\equalperdefinition 1.
       \end{align*}
     \end{enumerate}
   \end{definition}

   \begin{remark}
    \label{remark:incidence-algebra}
     Let $P$ be a finite partially ordered set.
     \begin{enumerate}
     \item\label{remark:incidence-algebra-1} Suppose that $P$ has a unique minimal element $0$. Given a one-variable function $f:\, P\to\complexnumbers$, one should think of $f$ as the restriction of some function $F:\,P^{(2)}\to \complexnumbers$ with  $f(\sigma)=F(0,\sigma)$ for all $\sigma\in P$.
     \item\label{remark:incidence-algebra-2} Given two functions $F,G:\, P^{(2)}\to\complexnumbers$, we can think of $F$ and $G$ as functions operating not on pairs $(\pi,\sigma)\in P^{(2)}$ but on intervals
       \begin{align*}
         [\pi,\sigma]\equalperdefinition \{\tau\in P\mid \pi\leq \tau\leq \sigma\}.
       \end{align*}       
       In the Definition~\hyperref[definition:incidence-algebra-1]{\ref*{definition:incidence-algebra}~\ref*{definition:incidence-algebra-1}} of $(F\ast G)(\pi,\sigma)$ for $(\pi,\sigma)\in P^{(2)}$ we sum over all decompositions of the interval $[\pi,\sigma]$ into two subintervals $[\pi,\tau]$ and $[\tau,\sigma]$, where $\tau\in P$ with $\pi\leq \tau\leq \sigma$.
     \item\label{remark:incidence-algebra-3} The function $\delta$ from Definition~\hyperref[definition:incidence-algebra-2]{\ref*{definition:incidence-algebra}~\ref*{definition:incidence-algebra-2}} is clearly the unit of the convolution operation: For all $F:\,P^{(2)}\to\complexnumbers$, 
       \begin{align*}
         F\ast \delta=\delta\ast F=F.
       \end{align*}
     \item\label{remark:incidence-algebra-4}  Note that the convolution $\ast$ is associative: For all $F,G, H:\, P^{(2)}\to \complexnumbers$ it holds that
       \begin{align*}
         (F\ast G)\ast H=F\ast (G\ast H).
       \end{align*}
       But $\ast$ is, in general, not commutative.
     \item\label{remark:incidence-algebra-5} The set of all functions $F:P^{(2)}\to \complexnumbers$ equipped with pointwise defined addition and with the convolution $\ast$ as multiplication is a unital (associative) algebra over $\complexnumbers$, usually called the \emph{incidence algebra} of $P$.
     \item\label{remark:incidence-algebra-6} We are mainly interested in the special case of convolutions $f=g\ast\zeta$ with the zeta function of Definition~\hyperref[definition:incidence-algebra-2]{\ref*{definition:incidence-algebra}~\ref*{definition:incidence-algebra-2}} for functions $f,g:\,P\to\complexnumbers$, i.e.\ in equations of the form
       \begin{align*}
         f(\sigma)&=\sum_{\substack{\tau\in P\\\tau\leq \sigma}}g(\tau)\underset{\displaystyle=1}{\underbrace{\zeta(\tau,\sigma)}}
         =\sum_{\substack{\tau\in P\\\tau\leq \sigma}}g(\tau),
       \end{align*}
       holding for all $\sigma\in P$.
          For us, $P$ will correspond to the set of non-crossing partitions,  $f$ to moments and $g$ to cumulants. In order to define the cumulants in terms of the moments we should solve the equation $f=g\ast \zeta$ for $g$.
     \end{enumerate}
   \end{remark}

   \begin{proposition}
     \label{proposition:moebius-function}
     Let $P$ be a finite partially ordered set. 
     Its zeta function is invertible: There exists $\mu: P^{(2)}\to \complexnumbers$, called \emph{M\"obius function}, such that
     \begin{align*}
       \zeta\ast\mu=\delta=\mu\ast \zeta.
     \end{align*}
   \end{proposition}
   \begin{proof}
     The second desired relation for $\mu$ asks that for all $(\pi,\sigma)\in P^{(2)}$ we have
     \begin{align*}
       \left.
       \begin{aligned}
         &1,&&\text{if }\pi=\sigma,\\
         &0,&&\text{if }\pi<\sigma
       \end{aligned}\right\}=\delta(\pi,\sigma)\overset{!}{=}(\mu\ast \zeta)(\pi,\sigma)=\sum_{\substack{\tau\in P\\\pi\leq \tau\leq \sigma}}\mu(\pi,\tau)\underset{\displaystyle=1}{\underbrace{\zeta(\tau,\sigma)}}=\sum_{\substack{\tau\in P\\\pi\leq \tau\leq \sigma}}\mu(\pi,\tau);
     \end{align*}     
    this can be solved recursively by defining, for all $\pi\in P$,
     \begin{align*}
       \mu(\pi,\pi)\equalperdefinition 1
    \qquad \text{and}\qquad
       \mu(\pi,\sigma)&\equalperdefinition-\sum_{\substack{\tau\in P\\\pi\leq \tau <\sigma}}\mu(\pi,\tau),
     \end{align*}
for all $(\pi,\sigma)\in P^{(2)}$.

     Note that, given functions $F,G:\, P^{(2)}\to\complexnumbers$, we can also view the Definition~\hyperref[definition:incidence-algebra-1]{\ref*{definition:incidence-algebra}~\ref*{definition:incidence-algebra-1}} of $F\ast G$ as matrix multiplication. Let $r\in \naturalnumbers$ and $(\pi_i)_{i=1}^r$ be such that $P=\{\pi_1,\ldots,\pi_r\}$ and put, for all $F: \, P^{(2)}\to \complexnumbers$,
     %\begin{align*}
       $\hat F\equalperdefinition(F(\pi_i,\pi_j))_{i,j=1}^r$,
     %\end{align*}
     where, for all $i,j\in [r]$,
    % \begin{align*}
     $ F(\pi_i,\pi_j) \equalperdefinition 0$ if $\pi_i\not\leq\pi_j$.
     %\end{align*}
     Then, given functions $F,G,H:\, P^{(2)}\to \complexnumbers$, the identity $H=F\ast G$ is the same as
     \begin{align*}
       \hat H=\hat F\hspace{-4.35em}\underset{\substack{\displaystyle\uparrow\\\displaystyle\text{matrix multiplication}}}{\cdot}\hspace{-4.35em} \hat G.
     \end{align*}
     Since $\hat\delta$ is the identity matrix, the relation $\hat\mu\cdot\hat\zeta=\hat\delta$ implies automatically that also $\hat\zeta\cdot\hat\mu=\hat\delta$. 
     (Since we are in finite dimensions, any left inverse with respect to matrix multiplication is also a right inverse.)
     Hence, the function $\mu$ defined above satisfies also the first desired relation $\zeta\ast\mu=\delta$.
   \end{proof}

   \begin{corollary}[M\"obius inversion]
     \label{corollary:moebius-inversion}
     Let $P$ be a finite partially ordered set. There exists a uniquely determined M\"obius function $\mu: \, P^{(2)}\to\complexnumbers$ such that for any $f,g:\,P\to\complexnumbers$ the following statements are equivalent:
\begin{enumerate}
     \item $f=g\ast \zeta$, meaning  
$$f(\sigma)=\sum_{\substack{\pi\in P\\\pi\leq \sigma}}g(\pi) \quad\text{ for all }\pi\in P.$$
     \item $g=f\ast \mu$, meaning
$$g(\sigma)=\sum_{\substack{\pi\in P\\\pi\leq \sigma}}f(\pi)\mu(\pi,\sigma) \quad\text{ for all }\pi\in P.$$
     \end{enumerate}
   \end{corollary}
\reqnomode
   \begin{remark}
     Let $\setofnoncrossingpartitionsof\equalperdefinition \bigcup_{n\geq 1}\setofnoncrossingpartitionsof(n)$. Our upcoming cumulant functions $g:\,\setofnoncrossingpartitionsof\to \complexnumbers$ will have a quite special \enquote{multiplicative} structure. Namely, for every $\pi\in \setofnoncrossingpartitionsof$, the cumulant $g(\pi)$ will factorize into a product of cumulants according to the blocks of $\pi$. The same will then be the case for the moment functions $f:\, \setofnoncrossingpartitionsof \to \complexnumbers$. Actually, instead of looking at moments and cumulants of single random variables, we will consider multivariate joint versions.
   \end{remark}

   \begin{definition}
     \label{definition:multiplicative-family}
     Let $\mathcal{A}$ be a unital (associative) algebra over $\complexnumbers$. For a given sequence $(\rho_n)_{n\in \naturalnumbers}$ of multilinear functionals on $\mathcal{A}$, where, for every $n\in \naturalnumbers$,
     \begin{align*}
       \rho_n:\, \mathcal{A}^n\to \complexnumbers,\,       (a_1,\ldots,a_n)\mapsto \rho_n(a_1,\ldots,a_n) \quad \text{is }n\text{-linear},
     \end{align*}
     we extend $(\rho_n)_{n\in \naturalnumbers}$ to a family
       $(\rho_\pi)_{\pi\in \setofnoncrossingpartitionsof}$
     of multilinear functionals on
     \begin{align*}
       \setofnoncrossingpartitionsof\equalperdefinition \bigcup_{n\in \naturalnumbers}\setofnoncrossingpartitionsof(n)
     \end{align*}
     such that, again, for every $n\in \naturalnumbers$ and every $\pi\in \setofnoncrossingpartitionsof(n)$,
     \begin{align*}
              \rho_\pi:\,\mathcal{A}^n\to \complexnumbers,\,  (a_1,\ldots,a_n)\mapsto \rho_\pi(a_1,\ldots,a_n) \quad \text{is }n\text{-linear},
     \end{align*}
     by defining, for every $n\in \naturalnumbers$, every $\pi\in \setofnoncrossingpartitionsof(n)$ and all $a_1,\ldots,a_n\in \mathcal A$,
     \begin{align*}
       \rho_\pi(a_1,\ldots,a_n)\equalperdefinition \prod_{V\in \pi}\rho_{\# V}(a_1,\ldots,a_n\mid V),
     \end{align*}
     where for all  $s\in \naturalnumbers$, $i_1,\ldots, i_s\in [n]$ with $i_1<i_2<\ldots<i_s$ and  $V=\{i_1,\ldots,i_s\}$,
     \begin{align*}
       \rho_s(a_1,\ldots,a_n\mid V)\equalperdefinition \rho_s(a_{i_1},\ldots,a_{i_s}).
     \end{align*}
     Then, $(\rho_\pi)_{\pi\in \setofnoncrossingpartitionsof}$ is called the \emph{multiplicative family of functionals on $\setofnoncrossingpartitionsof$ determined by $(\rho_n)_{n\geq 1}$}.     
   \end{definition}

   \begin{example}
     Let $\mathcal A$ and $(\rho_n)_{n\in \naturalnumbers}$ be as in Definition~\ref{definition:multiplicative-family}.
     \begin{enumerate}
     \item For all $n\in \naturalnumbers$, note that $\rho_n=\rho_{1_n}$: 
       \begin{align*}
         \rho_{\begin{tikzpicture}[baseline=-2pt-0.375em, xscale=0.7]
      \node[inner sep=1pt] (n1) at (0em,0) {};
      \node[inner sep=1pt] (n2) at (1em,0) {};
      \node[inner sep=1pt] (n3) at (2em,0) {};
      \node[inner sep=1pt] (n4) at (5em,0) {};
      \node[inner sep=1pt] (n5) at (6em,0) {};      
      \draw (n1) --++(0,-0.75em) -| (n2);
      \draw ($(n1)+(0,-0.75em)$) -| (n3);
      \draw (n4) --++(0,-0.75em) -| (n5);
      \draw ($(n3)+(0,-0.75em)$) -- ++ (0.5em,0);
      \draw ($(n4)+(0,-0.75em)$) -- ++ (-0.5em,0);
      \path ($(n3)+(0,-0.75em)$) -- node[pos=0.5] {$\ldots$}  ($(n4)+(0,-0.75em)$);
     %\draw [brown] (-0.25em,-1em) rectangle (6.25em, 0.125em);
    \useasboundingbox (-0.25em,-1em) rectangle (6.25em, 0.125em);
  \end{tikzpicture}}(a_1,\ldots,a_n)&=\rho_n(a_1,\ldots,a_n)
       \end{align*}
       for all $a_1,\ldots,a_n\in \mathcal A$. That justifies thinking of $(\rho_\pi)_{\pi\in \setofnoncrossingpartitionsof}$ as \enquote{extending} the original family $(\rho_n)_{n\in \naturalnumbers}$.
     \item For the partition $\pi=\{\{1,10\},\{2,5,9\}, \{3,4\}, \{6\}, \{7,8\} \}\in \setofnoncrossingpartitionsof(10)$ and $a_1,\ldots,a_{10}\in \mathcal A$,
Definition~\ref{definition:multiplicative-family} means
\begin{align*} \rho_\pi(a_1,\ldots,a_{10})&=\rho_2(a_1,a_{10})\rho_3(a_2,a_5,a_9)\rho_2(a_3,a_4)\rho_1(a_6)\rho_2(a_7,a_8),
\end{align*}
according to
       \begin{center}
         \begin{tikzpicture}[baseline=-2pt-1em]
      \node[inner sep=1pt] (n1) at (0em,0) {$a_1$};
      \node[inner sep=1pt] (n2) at (1.5em,0) {$a_2$};
      \node[inner sep=1pt] (n3) at (3em,0) {$a_3$};
      \node[inner sep=1pt] (n4) at (4.5em,0) {$a_4$};
      \node[inner sep=1pt] (n5) at (6em,0) {$a_5$};
      \node[inner sep=1pt] (n6) at (7.5em,0) {$a_6$};
      \node[inner sep=1pt] (n7) at (9em,0) {$a_7$};
      \node[inner sep=1pt] (n8) at (10.5em,0) {$a_8$};
      \node[inner sep=1pt] (n9) at (12em,0) {$a_9$};
      \node[inner sep=1pt] (n10) at (13.5em,0) {$a_{10}$};            
      \draw (n1) --++(0,-2em) -| (n10);
      \draw (n2) --++(0,-1.5em) -| (n5);      
      \draw ($(n5)+(0,-1.5em)$) -| (n9);
      \draw (n3) --++(0,-1em) -| (n4);
      \draw (n7) --++(0,-1em) -| (n8);      
      \draw (n6) --++(0,-1em);
     %\draw [brown] (-0.25em,-2.125em) rectangle (13.75em, 0.75em);
    \useasboundingbox (-0.25em,-2.125em) rectangle (13.75em, 0.75em);
  \end{tikzpicture}
\end{center}

     \end{enumerate}
   \end{example}

   \begin{definition}[Speicher 1994]
     \label{definition:free-cumulants}
     Let $(\mathcal{A},\varphi)$ be a non-commutative probability space. Then, we define, for every $n\in \naturalnumbers$, the $n$-linear functional $\varphi_n$ on $\mathcal A$ by
     \begin{align*}
       \varphi_n(a_1,\ldots,a_n)\equalperdefinition \varphi(a_1\ldots a_n)
     \end{align*}
     for all $a_1,\ldots,a_n\in \mathcal A$, 
     and extend $(\varphi_n)_{n\in \naturalnumbers}$ to the corresponding multiplicative family of moment functionals $\varphi\equalperdefinition (\varphi_\pi)_{\pi\in \setofnoncrossingpartitionsof}$ on $\setofnoncrossingpartitionsof$ by defining
     \begin{align*}
       \varphi_\pi(a_1,\ldots,a_n)=\prod_{V\in \pi}\varphi_{\# V}(a_1,\ldots,a_n\mid V)
     \end{align*}
     for all $n\in \naturalnumbers$, $\pi\in \setofnoncrossingpartitionsof(n)$ and $a_1,\ldots,a_n\in \mathcal A$.
     \par
     The corresponding \emph{free cumulants} $\kappa\equalperdefinition (\kappa_\pi)_{\pi\in  \setofnoncrossingpartitionsof}$ are defined by
       $\kappa\equalperdefinition\varphi\ast \mu$,
     which means, by
     \begin{align*}
       \kappa_\sigma(a_1,\ldots,a_n)\equalperdefinition \sum_{\substack{\pi\in\setofnoncrossingpartitionsof(n)\\\pi\leq \sigma}}\varphi_\pi(a_1,\ldots,a_n)\mu(\pi,\sigma)
     \end{align*}
     for all $n\in \naturalnumbers$, $\sigma\in\setofnoncrossingpartitionsof(n)$ and $a_1,\ldots,a_n\in \mathcal{A}$.
   \end{definition}

   \begin{proposition}
     Let $(\mathcal A,\varphi)$ be a non-commutative probability space with free cumulants $(\kappa_\pi)_{\pi\in \setofnoncrossingpartitionsof}$.
     \begin{enumerate}
     \item For each $n\in \naturalnumbers$ and $\pi\in \setofnoncrossingpartitionsof(n)$, the free cumulant functional
         $\kappa_\pi: \,\mathcal{A}^n\to \complexnumbers$
       is linear in each of its $n$ arguments.
     \item The family $(\kappa_\pi)_{\pi\in \setofnoncrossingpartitionsof}$ is multiplicative, determined by the family $(\kappa_n)_{n\in \naturalnumbers}$, where, for every $n\in \naturalnumbers$,
        $\kappa_n\equalperdefinition \kappa_{1_n}$. 
     \end{enumerate}
   \end{proposition}
   \begin{proof}
     \begin{enumerate}[wide]
     \item Clear by Definition~\ref{definition:free-cumulants} since all $\varphi_\pi$ are multilinear.
     \item Let $n\in \naturalnumbers$, $\sigma=\{V_1,\ldots,V_r\} \in \setofnoncrossingpartitionsof (n)$ and $a_1,\ldots,a_n\in \mathcal A$ and  consider $\kappa_\sigma(a_1,\ldots,a_n)$. Any $\pi\in \setofnoncrossingpartitionsof(n)$ with $\pi\leq \sigma$ decomposes then into
         $\pi=\pi_1\cup\ldots\cup \pi_r$,
       where $\pi_i\in \setofnoncrossingpartitionsof(V_i)$ for every $i\in [r]$. And, the interval $[\pi,\sigma]$ decomposes accordingly
       \begin{align*}
         [\pi,\sigma]\cong [\pi_1,1_{\# V_1}]\times \ldots\times [\pi_r,1_{\# V_r}]\subseteq \setofnoncrossingpartitionsof(V_1)\times \ldots\times \setofnoncrossingpartitionsof(V_r).
       \end{align*}
       \par
       Since the value $\mu(\pi,\sigma)$ of the M\"obius function at $(\pi,\sigma)$ depends only on the interval $[\pi,\sigma]$ (by the recursive definition in the proof of Proposition~\ref{proposition:moebius-function}) and since the M\"obius function of a product of partially ordered sets is the product of the M\"obius functions, we find
       \begin{align*}
         \mu(\pi,\sigma)=\mu(\pi_1,1_{\# V_1})\cdot \ldots \cdot \mu(\pi_r,1_{\# V_r})
       \end{align*}
       and thus
       \begin{align*}
         \kappa_\sigma(a_1,\ldots, a_n)&=\sum_{\substack{\pi\in \setofnoncrossingpartitionsof(n)\\\pi\leq \sigma}}\varphi_\pi(a_1,\ldots,a_n)\mu(\pi,\sigma)\\
                                       &=\sum_{\substack{\pi_1\in\setofnoncrossingpartitionsof(V_1),\ldots,\\
         \pi_r\in \setofnoncrossingpartitionsof(V_r)}}\prod_{i=1}^r\varphi_{\pi_i}(a_1,\ldots,a_n\mid V_i)\cdot \mu(\pi_i,1_{\# V_i})\\
                                       &=\prod_{i=1}^r\underset{\displaystyle=\kappa_{\# V_i}(a_1,\ldots,a_n\mid V_i)}{ \underbrace{\sum_{\pi_i\in \setofnoncrossingpartitionsof(V_i)}\varphi_{\pi_i}(a_1,\ldots,a_n\mid V_i)\cdot \mu(\pi_i,1_{\# V_i})}}\\
         &=\prod_{V\in \sigma}\kappa_{\# V}(a_1,\ldots,a_n\mid V).
       \end{align*}
       That is just what we needed to show.\qedhere
     \end{enumerate}     
   \end{proof}

   \begin{remark}
     \label{remark:moment-cumulant-formula}
     Let $(\mathcal A,\varphi)$ be a non-commutative probability space with free cumulants $\kappa$.
     By M\"obius inversion, Corollary~\ref{corollary:moebius-inversion}, the definition $\kappa\equalperdefinition\varphi\ast\mu$  of $\kappa$ in Definition~\ref{definition:free-cumulants} is equivalent to requiring
       $\varphi=\kappa\ast\zeta$.
     More precisely, the free cumulants are determined by the facts that $(\kappa_\pi)_{\pi\in \setofnoncrossingpartitionsof}$ is a multiplicative family of functionals and that, for all $n\in \naturalnumbers$ and $a_1,\ldots,a_n\in \mathcal{A}$,
     \begin{align*}
       \varphi(a_1\ldots a_n)=\varphi_n(a_1,\ldots,a_n)=\sum_{\pi\in \setofnoncrossingpartitionsof (n)}\kappa_\pi(a_1,\ldots,a_n).
     \end{align*}
   \end{remark}

   \begin{notation}
     Let $(\mathcal A,\varphi)$ be a non-commutative probability space with free cumulants $\kappa$.
     The formulas $\kappa=\varphi\ast\mu$ and $\varphi=\kappa\ast\zeta$ from Definition~\ref{definition:free-cumulants} and Remark~\ref{remark:moment-cumulant-formula} are called \emph{moment-cumulant formulas}.
   \end{notation}

   \begin{example}
     Let $(\mathcal A,\varphi)$ be a non-commutative probability space.
     \label{example:free-cumulants}
     \begin{enumerate}
     \item\label{example:free-cumulants-1} Let us calculate cumulants $\kappa_n$ for small $n\in \naturalnumbers$ by explicitly inverting the equation $\varphi=\kappa\ast\zeta$. In this way, we also learn values of the M\"obius function of $\setofnoncrossingpartitionsof$. 
       \begin{enumerate}
       \item\label{example:free-cumulants-1-1} Case~$n=1$: Since $\setofnoncrossingpartitionsof(1)=\{\ncpone\}=\{1_1\}$, the equation $\varphi=\kappa\ast\zeta$ yields, for every $a_1\in \mathcal A$,
   \begin{align*}
     \underset{\displaystyle =\varphi(a_1)}{\underbrace{\varphi_1(a_1)}}=\kappa_{\ncpone}(a_1)=\kappa_1(a_1),
   \end{align*}
   which allows us to conclude that
     $\kappa_1(a_1)=\varphi(a_1)$.
   Since this formula corresponds by definition to
   \begin{align*}
          \kappa_1(a_1)=\varphi_1(a_1)\underset{\displaystyle=1}{\underbrace{\mu(\ncpone,\ncpone)}}.
   \end{align*}
   we have verified $\mu(\ncpone,\ncpone)=1$.
          \item\label{example:free-cumulants-1-2} Case~$n=2$: It holds $\setofnoncrossingpartitionsof(2)=\{\ncptwoone,\ncptwotwo\}=\{0_2,1_2\}$. From $\varphi=\kappa\ast\zeta$ follows here, for all $a_1,a_2\in \mathcal A$,
   \begin{align*}
     \underset{\displaystyle = \varphi(a_1a_2)}{\underbrace{\varphi_2(a_1,a_2)}}=\underset{\displaystyle =\kappa_2(a_1,a_2)}{\underbrace{\kappa_\ncptwotwo(a_1,a_2)}}+\underset{\displaystyle =\kappa_1(a_1)\kappa_1(a_2)}{\underbrace{\kappa_\ncptwoone(a_1,a_2)}}.
   \end{align*}
   We solve for $\kappa_2(a_1,a_2)$ and use the results $\kappa_1(a_1)=\varphi(a_1)$ and $\kappa_1(a_2)=\varphi(a_2)$ from Part~\ref{example:free-cumulants-1-1}, yielding
   \begin{align*}
     \kappa_2(a_1,a_2)
     &=\varphi(a_1a_2)-\varphi(a_1)\varphi(a_2).
   \end{align*}
   We can learn further values of the M\"obius function by comparing with the definition:
        \begin{align*}
     \kappa_2(a_1,a_2)
       &=\varphi_2(a_1,a_2)\underset{\displaystyle = 1}{\underbrace{\mu(\ncptwoone,\ncptwotwo)}}+\varphi_1(a_1)\varphi_2(a_2)\underset{\displaystyle = -1}{\underbrace{\mu(\ncptwotwo,\ncptwotwo)}}.
   \end{align*}
 \item\label{example:free-cumulants-1-3} Case~$n=3$: Here, expanding the definitions in $\varphi=\kappa\ast\zeta$ leads to, for all $a_1,a_2,a_3\in \mathcal A$,
      \begin{align*}
     \underset{\displaystyle =\varphi(a_1a_2a_3)}{\underbrace{\varphi_3(a_1,a_2,a_3)}}&=\kappa_{\ncpthreefive}(a_1,a_2,a_3)+\kappa_{\ncpthreefour}(a_1,a_2,a_3)\\[-1.5em]
                                                                         &\phantom{{}=}+\kappa_{\ncpthreethree}(a_1,a_2,a_3)+\kappa_{\ncpthreetwo}(a_1,a_2,a_3)\\
     &\phantom{{}=}+\kappa_{\ncpthreeone}(a_1,a_2,a_3)\\
     &=\kappa_3(a_1,a_2,a_3)+\kappa_1(a_1)\kappa_2(a_2,a_3)\\
                                                                                      &\phantom{{}=}+\kappa_2(a_1,a_3)\kappa_1(a_2)+\kappa_2(a_1,a_2)\kappa_1(a_3)\\
     &\phantom{{}=}+\kappa_1(a_1)\kappa_1(a_2)\kappa_1(a_3).
   \end{align*}
      Combining the results of Parts~\ref{example:free-cumulants-1-1} and~\ref{example:free-cumulants-1-2}, we can thus conclude
   \begin{align*}
     \kappa_3(a_1,a_2,a_3)&=\varphi(a_1a_2a_3)-\varphi(a_1)\varphi(a_2a_3)\\
                          &\phantom{{}=}-\varphi(a_2)\varphi(a_1a_3)-\varphi(a_3)\varphi(a_1a_2)\\
     &\phantom{{}=}+2\varphi(a_1)\varphi(a_2)\varphi(a_3).
   \end{align*}
   And thus we can read off the following values of the M\"obius function:
   \begin{align*}
     \begin{array}{c||c|c|c|c|c}
       \sigma & \ncpthreefive & \ncpthreefour & \ncpthreethree & \ncpthreetwo & \ncpthreeone\\ \hline 
       \mu(\sigma,\ncpthreefive) & 1 & -1 & -1 & -1 & 2
     \end{array}
   \end{align*}
\end{enumerate}
\item\label{example:free-cumulants-2} Let $s$ be a standard semicircular element (see Definition~\ref{definition:semicircular-variable}). Then, for all $m\in \naturalnumbers$,
  \begin{align*}
    \varphi(s^{2m})=\#\setofnoncrossingpartitionsof_2(2m)=\sum_{\pi\in \setofnoncrossingpartitionsof_2(2m)}\prod_{V\in \pi}1.
  \end{align*}
  Thus the free cumulants of $s$ are, for every $n\in \naturalnumbers$, of the form
  \begin{align*}
    \kappa_n(s,s,\ldots,s)=
    \begin{cases}
      1,&\text{if }n=2,\\0,&\text{otherwise}.
    \end{cases}
  \end{align*}
  This is because, for every $n\in \naturalnumbers$, the corresponding multiplicative extension is given by
  \begin{align*}
    \kappa_\pi(s,s,\ldots,s)&=\prod_{V\in \pi}\underset{\displaystyle =\delta_{\# V,2}}{\underbrace{\kappa_{\# V}(s,\ldots,s)}}
    =
      \begin{cases}
        1,&\text{if }\pi\in\setofnoncrossingpartitionsof_2(n),\\
        0,&\text{otherwise}
      \end{cases}
  \end{align*}
  for all  $\pi\in \setofnoncrossingpartitionsof (n)$,
   which gives the right values for the moments, and because (by M\"obius inversion, Corollary~\ref{corollary:moebius-inversion}) the cumulants $(\kappa_n)_{n\in \naturalnumbers}$ are uniquely determined by the moments $(\varphi_n)_{n\in \naturalnumbers}$.
\item\label{example:free-cumulants-3} In the same way, for every $\lambda\in \realnumbers$, $\lambda >0$, the description of the moments of a free Poisson element $x$ of parameter $\lambda$ (see Definition~\ref{definition:free-poisson-variable}) as 
  \begin{align*}
    \varphi(x^k) =\sum_{\pi\in \setofnoncrossingpartitionsof(k)}\lambda^{\# \pi}=\sum_{\pi\in \setofnoncrossingpartitionsof(k)}\prod_{V\in \pi}\lambda,
  \end{align*}
  for every $k\in \naturalnumbers$,
  tells us that its cumulants are, for every $n\in \naturalnumbers$, of the form
  \begin{align*}
    \kappa_n(x,x,\ldots,x)=\lambda.
  \end{align*}
     \end{enumerate}
   \end{example}

   \begin{remark}
     Let $(\mathcal A,\varphi)$ be a non-commutative probability space.
     With respect to the additive structure of $\mathcal{A}$, the free cumulant $\kappa_\pi$ is a multilinear functional for every $\pi\in \setofnoncrossingpartitionsof$. We also need to understand their behavior with respect to the multiplicative structure of $\mathcal{A}$. For the moment functionals this is easy, they are \enquote{associative}: For example, for all $a_1,a_2,a_3\in \mathcal A$,
     \begin{align*}
       \varphi_2(a_1a_2,a_3)&=\varphi((a_1a_2)a_3)=\varphi(a_1(a_2a_3))=\varphi_2(a_1,a_2a_3).
     \end{align*}
     But, what about $\kappa_2$? The functional $\kappa_2$ is \emph{not} associative in this sense, i.e., for $a_1,a_2,a_3\in \mathcal A$,
       $\kappa_2(a_1a_2,a_3)\neq \kappa_2(a_1,a_2a_3)$ in general.
     However, there is a nice replacement for this.
   \end{remark}
   \begin{notation}
     \label{notation:cumulant-product-rule}
     Let $(\mathcal A,\varphi)$ be a non-commutative probability space, fix $m,n\in \naturalnumbers$ with $m<n$ and $i:[m]\to [n]$ with $1\leq i(1)<i(2)<\ldots<i(m)=n$. Consider $a_1,\ldots,a_n\in \mathcal{A}$ and put
     \begin{align*}
       A_1&\equalperdefinition a_1\ldots a_{i(1)}\\
       A_2&\equalperdefinition a_{i(1)+1}\ldots a_{i(2)}\\
          & \vdots\\
       A_m&\equalperdefinition a_{i(m-1)+1}\ldots a_{i(m)}.
     \end{align*}
     We want to relate the cumulants of $(a_1,\ldots,a_n)$ and $(A_1,\ldots, A_m)$. On the level of moments this is simple: For each $\pi\in \setofnoncrossingpartitionsof(m)$ there is a $\hat\pi\in \setofnoncrossingpartitionsof(n)$ such that
     \begin{align*}
       \varphi_\pi(A_1,\ldots,A_m)=\varphi_{\hat\pi}(a_1,a_2,\ldots,a_n).
     \end{align*}
     Namely, for $\pi\in\setofnoncrossingpartitionsof(m)$ this partition $\hat\pi$ is determined by writing $i_0\equalperdefinition 0$ and requiring, for all $j,k\in [n]$, 
     \begin{align*}
       j\sim_{\hat{ \pi}}k\quad\text{if and only if}\quad &\text{there exist }p,q\in [m] \text{ such that}\\
       &a_j\text{ is a factor in }A_p: \, j\in \{i(p-1)+1,\ldots, i(p)\},\\
        &a_k\text{ is a factor in }A_q: \, k\in \{i(q-1)+1,\ldots, i(q)\},\text{ and}\\       
       & p\sim_\pi q.
     \end{align*}
      The mapping  $\hat{\cdot}: \setofnoncrossingpartitionsof(m)\to \setofnoncrossingpartitionsof(n)$ is an embedding of partially ordered sets.
   \end{notation}

   \begin{example}
     %Let $m,\in \naturalnumbers$ with $m<n$ be arbitrary. 
     \begin{enumerate}
     \item For $m=3$, $n=6$, $a_1,\ldots,a_6\in \mathcal A$, and
       \begin{align*}
        A_1\equalperdefinition a_1, \quad A_2\equalperdefinition a_2a_3a_4, \quad A_3\equalperdefinition a_5a_6,
       \end{align*}
the embedding $\hat{\cdot}: \setofnoncrossingpartitionsof(3)\to \setofnoncrossingpartitionsof(6)$ from Notation~\ref{notation:cumulant-product-rule} looks as follows:
       \begin{align*}
         \begin{matrix}
         A_1   A_2  A_3 &\overset{\hat{\cdot}}{\longrightarrow}& a_1  a_2   a_3   a_4   a_5   a_6\\
                     \begin{tikzpicture}[baseline=-2pt-0.375em]
      \node[inner sep=1pt] (n1) at (0em,0) {};
      \node[inner sep=1pt] (n2) at (1em,0) {};
      \node[inner sep=1pt] (n3) at (2em,0) {};
      \draw (n1) --++(0,-0.75em);      
      \draw (n2) --++(0,-0.75em);
      \draw (n3) --++(0,-0.75em);      
    %\draw [brown] (-0.25em,-1em) rectangle (2.25em, 0.125em);
    \useasboundingbox (-0.25em,-1em) rectangle (2.25em, 0.125em);
  \end{tikzpicture}         && \begin{tikzpicture}[baseline=-2pt-0.375em]
      \node[inner sep=1pt] (n1) at (0em,0) {};
      \node[inner sep=1pt] (n2) at (1em,0) {};
      \node[inner sep=1pt] (n3) at (2em,0) {};
      \node[inner sep=1pt] (n4) at (3em,0) {};
      \node[inner sep=1pt] (n5) at (4em,0) {};
      \node[inner sep=1pt] (n6) at (5em,0) {};      
      \draw (n1) --++(0,-0.75em);      
      \draw (n2) --++(0,-0.75em) -| (n3);
      \draw ($(n3)+(0,-0.75em)$) -| (n4);
      \draw (n5) --++(0,-0.75em) -| (n6);      
    %\draw [brown] (-0.25em,-0.875em) rectangle (5.25em, 0.125em);
    \useasboundingbox (-0.25em,-0.875em) rectangle (5.25em, 0.125em);
  \end{tikzpicture}\\[0.5em]
\begin{tikzpicture}[baseline=-2pt-0.375em]
      \node[inner sep=1pt] (n1) at (0em,0) {};
      \node[inner sep=1pt] (n2) at (1em,0) {};
      \node[inner sep=1pt] (n3) at (2em,0) {};
      \draw (n1) --++(0,-0.75em) -| (n2);
      \draw (n3) --++(0,-0.75em);      
    %\draw [brown] (-0.25em,-1em) rectangle (2.25em, 0.125em);
    \useasboundingbox (-0.25em,-1em) rectangle (2.25em, 0.125em);
  \end{tikzpicture}                              && \begin{tikzpicture}[baseline=-2pt-0.375em]
      \node[inner sep=1pt] (n1) at (0em,0) {};
      \node[inner sep=1pt] (n2) at (1em,0) {};
      \node[inner sep=1pt] (n3) at (2em,0) {};
      \node[inner sep=1pt] (n4) at (3em,0) {};
      \node[inner sep=1pt] (n5) at (4em,0) {};
      \node[inner sep=1pt] (n6) at (5em,0) {};      
      \draw (n1) --++(0,-0.75em) -| (n2);
      \draw ($(n2)+(0,-0.75em)$) -| (n3);
      \draw ($(n3)+(0,-0.75em)$) -| (n4);      
      \draw (n5) --++(0,-0.75em) -| (n6);      
    %\draw [brown] (-0.25em,-0.875em) rectangle (5.25em, 0.125em);
    \useasboundingbox (-0.25em,-0.875em) rectangle (5.25em, 0.125em);
  \end{tikzpicture}                     \\[0.5em]
\begin{tikzpicture}[baseline=-2pt-0.5em]
      \node[inner sep=1pt] (n1) at (0em,0) {};
      \node[inner sep=1pt] (n2) at (1em,0) {};
      \node[inner sep=1pt] (n3) at (2em,0) {};
      \draw (n1) --++(0,-1em) -| (n3);
      \draw (n2) --++(0,-0.5em);      
    %\draw [brown] (-0.25em,-1.125em) rectangle (2.25em, 0.125em);
    \useasboundingbox (-0.25em,-1.125em) rectangle (2.25em, 0.125em);
  \end{tikzpicture}                              && \begin{tikzpicture}[baseline=-2pt-0.5em]
      \node[inner sep=1pt] (n1) at (0em,0) {};
      \node[inner sep=1pt] (n2) at (1em,0) {};
      \node[inner sep=1pt] (n3) at (2em,0) {};
      \node[inner sep=1pt] (n4) at (3em,0) {};
      \node[inner sep=1pt] (n5) at (4em,0) {};
      \node[inner sep=1pt] (n6) at (5em,0) {};      
      \draw (n1) --++(0,-1em) -| (n5);
      \draw ($(n5)+(0,-1em)$) -| (n6);
      \draw (n2) --++(0,-0.5em) -| (n3);
      \draw ($(n3)+(0,-0.5em)$) -| (n4);      
    %\draw [brown] (-0.25em,-1.125em) rectangle (5.25em, 0.125em);
    \useasboundingbox (-0.25em,-1.125em) rectangle (5.25em, 0.125em);
  \end{tikzpicture}\\[0.5em]
       \begin{tikzpicture}[baseline=-2pt-0.375em]
      \node[inner sep=1pt] (n1) at (0em,0) {};
      \node[inner sep=1pt] (n2) at (1em,0) {};
      \node[inner sep=1pt] (n3) at (2em,0) {};
      \draw (n2) --++(0,-0.75em) -| (n3);
      \draw (n1) --++(0,-0.75em);      
    %\draw [brown] (-0.25em,-1em) rectangle (2.25em, 0.125em);
    \useasboundingbox (-0.25em,-1em) rectangle (2.25em, 0.125em);
  \end{tikzpicture}                       && \begin{tikzpicture}[baseline=-2pt-0.375em]
      \node[inner sep=1pt] (n1) at (0em,0) {};
      \node[inner sep=1pt] (n2) at (1em,0) {};
      \node[inner sep=1pt] (n3) at (2em,0) {};
      \node[inner sep=1pt] (n4) at (3em,0) {};
      \node[inner sep=1pt] (n5) at (4em,0) {};
      \node[inner sep=1pt] (n6) at (5em,0) {};      
      \draw (n1) --++(0,-0.75em);      
      \draw (n2) --++(0,-0.75em) -| (n3);
      \draw ($(n3)+(0,-0.75em)$) -| (n4);
      \draw ($(n4)+(0,-0.75em)$) -| (n5);
      \draw ($(n5)+(0,-0.75em)$) -| (n6);      
    %\draw [brown] (-0.25em,-0.875em) rectangle (5.25em, 0.125em);
    \useasboundingbox (-0.25em,-0.875em) rectangle (5.25em, 0.125em);
  \end{tikzpicture}\\[0.5em]
\begin{tikzpicture}[baseline=-2pt-0.375em]
      \node[inner sep=1pt] (n1) at (0em,0) {};
      \node[inner sep=1pt] (n2) at (1em,0) {};
      \node[inner sep=1pt] (n3) at (2em,0) {};
      \draw (n1) --++(0,-0.75em) -| (n2);
      \draw ($(n2)+(0,-0.75em)$) -| (n3);      
    %\draw [brown] (-0.25em,-1em) rectangle (2.25em, 0.125em);
    \useasboundingbox (-0.25em,-1em) rectangle (2.25em, 0.125em);
  \end{tikzpicture}                              &&   \begin{tikzpicture}[baseline=-2pt-0.375em]
      \node[inner sep=1pt] (n1) at (0em,0) {};
      \node[inner sep=1pt] (n2) at (1em,0) {};
      \node[inner sep=1pt] (n3) at (2em,0) {};
      \node[inner sep=1pt] (n4) at (3em,0) {};
      \node[inner sep=1pt] (n5) at (4em,0) {};
      \node[inner sep=1pt] (n6) at (5em,0) {};      
      \draw (n1) --++(0,-0.75em) -| (n2);
      \draw ($(n2)+(0,-0.75em)$) -| (n3);      
      \draw ($(n3)+(0,-0.75em)$) -| (n4);
      \draw ($(n4)+(0,-0.75em)$) -| (n5);
      \draw ($(n5)+(0,-0.75em)$) -| (n6);      
    %\draw [brown] (-0.25em,-0.875em) rectangle (5.25em, 0.125em);
    \useasboundingbox (-0.25em,-0.875em) rectangle (5.25em, 0.125em);
  \end{tikzpicture}   
                              \end{matrix}
       \end{align*}
            
       We can relate the moments of $(a_1,a_2,\ldots,a_6)$ and $(A_1,A_2,A_3)$  by
       \begin{align*}
         \varphi_{\begin{tikzpicture}[baseline=-2pt-0.5em,xscale=0.7]
      \node[inner sep=1pt] (n1) at (0em,0) {};
      \node[inner sep=1pt] (n2) at (1em,0) {};
      \node[inner sep=1pt] (n3) at (2em,0) {};
      \draw (n1) --++(0,-1em) -| (n3);
      \draw (n2) --++(0,-0.5em);      
    %\draw [brown] (-0.25em,-1.125em) rectangle (2.25em, 0.125em);
    \useasboundingbox (-0.25em,-1.125em) rectangle (2.25em, 0.125em);
  \end{tikzpicture}}(A_1,A_2,A_3)&=\varphi(A_1A_3)\varphi(A_2)\\
                                &=\varphi(a_1a_5a_6)\varphi(a_2a_3a_4)%\\
         =\varphi_{\begin{tikzpicture}[baseline=-2pt-0.5em,xscale=0.7]
      \node[inner sep=1pt] (n1) at (0em,0) {};
      \node[inner sep=1pt] (n2) at (1em,0) {};
      \node[inner sep=1pt] (n3) at (2em,0) {};
      \node[inner sep=1pt] (n4) at (3em,0) {};
      \node[inner sep=1pt] (n5) at (4em,0) {};
      \node[inner sep=1pt] (n6) at (5em,0) {};      
      \draw (n1) --++(0,-1em) -| (n5);
      \draw ($(n5)+(0,-1em)$) -| (n6);
      \draw (n2) --++(0,-0.5em) -| (n3);
      \draw ($(n3)+(0,-0.5em)$) -| (n4);      
    %\draw [brown] (-0.25em,-1.125em) rectangle (5.25em, 0.125em);
    \useasboundingbox (-0.25em,-1.125em) rectangle (5.25em, 0.125em);
  \end{tikzpicture}}(a_1,a_2,\ldots,a_6).
       \end{align*}
     \item Note that $\hat{1}_m=1_n$, but $\hat{0}_m\neq \hat{0}_n$ (unless $m=n$) and that
       \begin{align*}
         \{ \hat{\pi}\mid \pi\in \setofnoncrossingpartitionsof(m) \}&=[\hat{0}_m,\hat{1}_n]=\{ \sigma\in \setofnoncrossingpartitionsof(n)\mid \hat{0}_m\leq \sigma \}.
       \end{align*}
     \item The mapping $\setofnoncrossingpartitionsof(m)\to \setofnoncrossingpartitionsof(n),\, \tau\mapsto\hat{\tau}$ preserves the partial order. Hence, in particular,
         $\mu(\pi,\tau)=\mu(\hat{\pi},\hat{\tau})$
       for all $\pi,\tau\in \setofnoncrossingpartitionsof(m)$ with $\pi\leq \tau$.
     \end{enumerate}
   \end{example}

   \begin{proposition}
\label{proposition:non-crossing-lattice}
     For each $n\in \naturalnumbers$, the partially ordered set $\setofnoncrossingpartitionsof(n)$ is a \emph{lattice}:
     \begin{enumerate}
     \item\label{proposition:non-crossing-lattice-1} For all $\pi,\sigma\in \setofnoncrossingpartitionsof(n)$ there exists a unique smallest $\tau\in \setofnoncrossingpartitionsof(n)$ with the properties $\pi \leq \tau$ and $\sigma\leq \tau$.  (It is denoted by $\pi\vee \sigma$ and called the \emph{join} (or \emph{maximum}) of $\pi$ and $\sigma$.)
     \item\label{proposition:non-crossing-lattice-2} For all $\pi,\sigma\in \setofnoncrossingpartitionsof(n)$ there exists a unique largest $\tau\in \setofnoncrossingpartitionsof(n)$ with the properties $\tau\leq \pi$ and $\tau\leq \sigma$. (It is denoted by $\pi\wedge \sigma$ and called the \emph{meet} (or \emph{minimum}) of $\pi$ and $\sigma$.)
     \end{enumerate}
   \end{proposition}
   \begin{proof}
     Let $n\in \naturalnumbers$ and $\pi,\sigma\in \setofnoncrossingpartitionsof(n)$ be arbitrary. We prove Claim~\ref{proposition:non-crossing-lattice-2} first.
     \begin{enumerate}[wide]
      \setcounter{enumi}{1}
     \item If  $\pi=\{V_1,\ldots,V_r\}$ and $\sigma=\{W_1,\ldots,W_s\}$, then,
       \begin{align*}
         \pi\wedge \sigma=\{ V_i\cap W_j\mid i\in [r],\, j\in [s],\, V_i\cap W_j\neq \emptyset\}.
       \end{align*}
       In other words, for all $p,q\in [n]$,
       \begin{align*}
         p\sim_{\pi\wedge \sigma} q\quad \text{if and only if}\quad p\sim_\pi q\text{ and }p\sim_\sigma q.
       \end{align*}
       For example, 
       \begin{align*}
         \begin{tikzpicture}[baseline=-2pt-0.375em]
      \node[inner sep=1pt] (n1) at (0em,0) {};
      \node[inner sep=1pt] (n2) at (1em,0) {};
      \node[inner sep=1pt] (n3) at (2em,0) {};
      \node[inner sep=1pt] (n4) at (3em,0) {};
      \draw (n2) --++(0,-0.75em) -| (n3);
      \draw ($(n3)+(0,-0.75em)$) -| (n4);
      \draw (n1) --++(0,-0.75em);            
     %\draw [brown] (-0.25em,-1em) rectangle (3.25em, 0.125em);
    \useasboundingbox (-0.25em,-1em) rectangle (3.25em, 0.125em);
  \end{tikzpicture}\quad\wedge\quad
        \begin{tikzpicture}[baseline=-2pt-0.5em]
      \node[inner sep=1pt] (n1) at (0em,0) {};
      \node[inner sep=1pt] (n2) at (1em,0) {};
      \node[inner sep=1pt] (n3) at (2em,0) {};
      \node[inner sep=1pt] (n4) at (3em,0) {};
      \draw (n2) --++(0,-0.5em) -| (n3);
      \draw (n1) --++(0,-1em) -| (n4);      
     %\draw [brown] (-0.25em,-1.125em) rectangle (3.25em, 0.125em);
    \useasboundingbox (-0.25em,-1.125em) rectangle (3.25em, 0.125em);
  \end{tikzpicture}
      \qquad =\qquad
       \begin{tikzpicture}[baseline=-2pt-0.375em]
      \node[inner sep=1pt] (n1) at (0em,0) {};
      \node[inner sep=1pt] (n2) at (1em,0) {};
      \node[inner sep=1pt] (n3) at (2em,0) {};
      \node[inner sep=1pt] (n4) at (3em,0) {};
      \draw (n2) --++(0,-0.75em) -| (n3);
      \draw (n1) --++(0,-0.75em);
      \draw (n4) --++(0,-0.75em);                  
     %\draw [brown] (-0.25em,-1em) rectangle (3.25em, 0.125em);
    \useasboundingbox (-0.25em,-1em) rectangle (3.25em, 0.125em);
  \end{tikzpicture}.
  \end{align*}

      \setcounter{enumi}{0}       
     \item The first guess, trying to define $\pi\vee \sigma$ by requiring  $p\sim_\pi q$ or $p\sim_\sigma q$ for $p,q\in [n]$, does not work. But we can reduce joins to meets as follows:
       \begin{itemize}
       \item By induction, any finite number of non-crossing partitions has a meet:
         \begin{align*}
           \pi_1\wedge \pi_2 \wedge \ldots \wedge \pi_k \in \setofnoncrossingpartitionsof(k)
         \end{align*}
         for all $k\in \naturalnumbers$ and $\pi_1,\ldots, \pi_k\in \setofnoncrossingpartitionsof(n)$.
       \item There is a maximal element $1_n$, which is the join of all of $\setofnoncrossingpartitionsof(n)$.
       \item It then follows
         \begin{align*}
           \pi\vee \sigma&= \bigwedge \underset{\displaystyle \neq \emptyset\text{ since it contains }1_n}{\underbrace{\{ \tau\in \setofnoncrossingpartitionsof(n)\mid  \pi\leq \tau, \,\sigma\leq \tau \}}}.
         \end{align*}
       \end{itemize}
%       \vspace{-1.5em}
       That concludes the proof.\qedhere
     \end{enumerate}     
   \end{proof}

   \begin{remark}
     \begin{enumerate}
     \item For $n\in \naturalnumbers$ and not necessarily non-crossing $\pi,\sigma\in \setofpartitionsof(n)$, we can define $\pi\leq \sigma$ in the same way as for $\pi,\sigma\in\setofnoncrossingpartitionsof(n)$. Then, $\setofpartitionsof(n)$ is also a lattice. The meet of $\setofpartitionsof(n)$ restricts to the meet on $\setofnoncrossingpartitionsof(n)$. But for the join, both operations are different. E.g., the non-crossing partitions
       \begin{align*}
          \begin{tikzpicture}[baseline=-2pt-0.5em]
      \node[inner sep=1pt] (n1) at (0em,0) {};
      \node[inner sep=1pt] (n2) at (1em,0) {};
      \node[inner sep=1pt] (n3) at (2em,0) {};
      \node[inner sep=1pt] (n4) at (3em,0) {};
      \draw (n2) --++(0,-1em) -| (n4);
      \draw (n1) --++(0,-0.5em);
      \draw (n3) --++(0,-0.5em);                  
     %\draw [brown] (-0.25em,-1.125em) rectangle (3.25em, 0.125em);
    \useasboundingbox (-0.25em,-1.125em) rectangle (3.25em, 0.125em);
  \end{tikzpicture}                                  \quad\text{and}\quad
\begin{tikzpicture}[baseline=-2pt-0.5em]
      \node[inner sep=1pt] (n1) at (0em,0) {};
      \node[inner sep=1pt] (n2) at (1em,0) {};
      \node[inner sep=1pt] (n3) at (2em,0) {};
      \node[inner sep=1pt] (n4) at (3em,0) {};
      \draw (n1) --++(0,-1em) -| (n3);
      \draw (n2) --++(0,-0.5em);
      \draw (n4) --++(0,-0.5em);                  
     %\draw [brown] (-0.25em,-1.125em) rectangle (3.25em, 0.125em);
    \useasboundingbox (-0.25em,-1.125em) rectangle (3.25em, 0.125em);
  \end{tikzpicture}
       \end{align*}
       have the joins
       \begin{align*}
         \begin{tikzpicture}[baseline=-2pt-0.5em]
      \node[inner sep=1pt] (n1) at (0em,0) {};
      \node[inner sep=1pt] (n2) at (1em,0) {};
      \node[inner sep=1pt] (n3) at (2em,0) {};
      \node[inner sep=1pt] (n4) at (3em,0) {};
      \draw (n1) --++(0,-0.5em) -| (n3);
      \draw (n2) --++(0,-1em) -| (n4);      
     %\draw [brown] (-0.25em,-1.125em) rectangle (3.25em, 0.125em);
    \useasboundingbox (-0.25em,-1.125em) rectangle (3.25em, 0.125em);
  \end{tikzpicture}\text{ in }\setofpartitionsof(4) \quad\text{and}\quad
         \begin{tikzpicture}[baseline=-2pt-0.375em]
      \node[inner sep=1pt] (n1) at (0em,0) {};
      \node[inner sep=1pt] (n2) at (1em,0) {};
      \node[inner sep=1pt] (n3) at (2em,0) {};
      \node[inner sep=1pt] (n4) at (3em,0) {};
      \draw (n1) --++(0,-0.75em) -| (n2);
      \draw ($(n2)+(0,-0.75em)$) -| (n3);
      \draw ($(n3)+(0,-0.75em)$) -| (n4);
    %\draw [brown] (-0.25em,-1em) rectangle (3.25em, 0.125em);
    \useasboundingbox (-0.25em,-1em) rectangle (3.25em, 0.125em);
  \end{tikzpicture}\text{ in }\setofnoncrossingpartitionsof(4).
       \end{align*}
     \item In general, the join in $\setofnoncrossingpartitionsof$ is given by taking the join in $\setofpartitionsof$ and then merging blocks together which have a crossing.
     \item Let $n\in \naturalnumbers$. Even in $\setofpartitionsof(n)$ the join $\pi\vee \sigma$ of partitions $\pi,\sigma\in \setofpartitionsof(n)$ is \emph{not} just found by declaring for $p,q\in [n]$ that $p\sim_{\pi\vee \sigma} q$ if and only if $p\sim_\pi q$ or $p\sim_\sigma q$. Instead, one has to allow longer alternating connections via $\pi$ and $\sigma$. For example, the join of
       \begin{align*}
         \begin{tikzpicture}[baseline=-2pt-0.375em]
      \node[inner sep=1pt] (n1) at (0em,0) {};
      \node[inner sep=1pt] (n2) at (1em,0) {};
      \node[inner sep=1pt] (n3) at (2em,0) {};
      \node[inner sep=1pt] (n4) at (3em,0) {};
      \node[inner sep=1pt] (n5) at (4em,0) {};
      \node[inner sep=1pt] (n6) at (5em,0) {};
      \node[inner sep=1pt] (n7) at (6em,0) {};
      \node[inner sep=1pt] (n8) at (7em,0) {};
      \node[inner sep=1pt] (n9) at (8em,0) {};
      \node[inner sep=1pt] (n10) at (9em,0) {};
      \draw (n1) --++(0,-0.75em);
      \draw (n10) --++(0,-0.75em);
      \draw (n2) --++(0,-0.75em) -| (n3);
      \draw (n4) --++(0,-0.75em) -| (n5);
      \draw (n6) --++(0,-0.75em) -| (n7);
      \draw (n8) --++(0,-0.75em) -| (n9);      
    %\draw [brown] (-0.25em,-0.875em) rectangle (9.25em, 0.125em);
    \useasboundingbox (-0.25em,-0.875em) rectangle (9.25em, 0.125em);
  \end{tikzpicture}\quad\text{and}\quad
         \begin{tikzpicture}[baseline=-2pt-0.375em]
      \node[inner sep=1pt] (n1) at (0em,0) {};
      \node[inner sep=1pt] (n2) at (1em,0) {};
      \node[inner sep=1pt] (n3) at (2em,0) {};
      \node[inner sep=1pt] (n4) at (3em,0) {};
      \node[inner sep=1pt] (n5) at (4em,0) {};
      \node[inner sep=1pt] (n6) at (5em,0) {};
      \node[inner sep=1pt] (n7) at (6em,0) {};
      \node[inner sep=1pt] (n8) at (7em,0) {};
      \node[inner sep=1pt] (n9) at (8em,0) {};
      \node[inner sep=1pt] (n10) at (9em,0) {};
      \draw (n1) --++(0,-0.75em) -| (n2);      
      \draw (n3) --++(0,-0.75em) -| (n4);
      \draw (n5) --++(0,-0.75em) -| (n6);
      \draw (n7) --++(0,-0.75em) -| (n8);
      \draw (n9) --++(0,-0.75em) -| (n10);      
    %\draw [brown] (-0.25em,-0.875em) rectangle (9.25em, 0.125em);
    \useasboundingbox (-0.25em,-0.875em) rectangle (9.25em, 0.125em);
  \end{tikzpicture}
       \end{align*}
       in $\setofpartitionsof(10)$ (and $\setofnoncrossingpartitionsof(10)$) is given by
       \begin{align*}
         \begin{tikzpicture}[baseline=-2pt-0.375em]
      \node[inner sep=1pt] (n1) at (0em,0) {};
      \node[inner sep=1pt] (n2) at (1em,0) {};
      \node[inner sep=1pt] (n3) at (2em,0) {};
      \node[inner sep=1pt] (n4) at (3em,0) {};
      \node[inner sep=1pt] (n5) at (4em,0) {};
      \node[inner sep=1pt] (n6) at (5em,0) {};
      \node[inner sep=1pt] (n7) at (6em,0) {};
      \node[inner sep=1pt] (n8) at (7em,0) {};
      \node[inner sep=1pt] (n9) at (8em,0) {};
      \node[inner sep=1pt] (n10) at (9em,0) {};
      \draw (n1) --++(0,-0.75em) -| (n2);
      \draw ($(n2)+(0,-0.75em)$) -| (n3);
      \draw ($(n3)+(0,-0.75em)$) -| (n4);
      \draw ($(n4)+(0,-0.75em)$) -| (n5);
      \draw ($(n5)+(0,-0.75em)$) -| (n6);
      \draw ($(n6)+(0,-0.75em)$) -| (n7);
      \draw ($(n7)+(0,-0.75em)$) -| (n8);
      \draw ($(n8)+(0,-0.75em)$) -| (n9);
      \draw ($(n9)+(0,-0.75em)$) -| (n10);            
    %\draw [brown] (-0.25em,-0.875em) rectangle (9.25em, 0.125em);
    \useasboundingbox (-0.25em,-0.875em) rectangle (9.25em, 0.125em);
  \end{tikzpicture}.
       \end{align*}
     \end{enumerate}
   \end{remark}
 
   \begin{theorem}
     \label{theorem:cumulant-product-rule}
     Let $(\mathcal{A},\varphi)$ be a non-commutative probability space with free cumulants $(\kappa_\pi)_{\pi\in \setofnoncrossingpartitionsof}$, let $m,n\in \naturalnumbers$ satisfy $m<n$, let $i:[m]\to[n]$ with $1\leq i(1)<i(2)<\ldots <i(m)=n$ and $a_1,\ldots,a_n\in \mathcal A$ be arbitrary and define
\begin{align*}
       A_1\equalperdefinition a_1\ldots a_{i(1)},\qquad
       A_2\equalperdefinition a_{i(1)+1}\ldots a_{i(2)},\qquad
          & \ldots\qquad
       A_m\equalperdefinition a_{i(m-1)+1}\ldots a_{i(m)}.
      \end{align*}
      Then,  for all $\tau\in \setofnoncrossingpartitionsof(m)$,
      \begin{align*}
        \kappa_\tau(A_1,\ldots,A_m)=\sum_{\substack{\pi\in \setofnoncrossingpartitionsof(n)\\\pi\vee \hat{0}_m=\hat{\tau}}}\kappa_\pi(a_1,\ldots, a_n).
      \end{align*}
      In particular, 
      \begin{align*}
        \kappa_m(A_1, \ldots,A_m)=\sum_{\substack{\pi\in \setofnoncrossingpartitionsof(n)\\\pi\vee \hat{0}_m={1}_n}}\kappa_\pi(a_1,\ldots, a_n).
      \end{align*}
   \end{theorem}
   \begin{proof}
     For every $\tau\in \setofnoncrossingpartitionsof(m)$,
     \begin{IEEEeqnarray*}{rCl}
       \kappa_\tau(A_1,\ldots,A_m)&=&\sum_{\substack{\sigma\in \setofnoncrossingpartitionsof(m)\\\sigma\leq \tau}}\varphi_\sigma(A_1,\ldots,A_m)\mu(\sigma,\tau)\\
                                  &=&\sum_{\substack{\sigma\in \setofnoncrossingpartitionsof(m)\\ 0_m\leq \sigma\leq \tau}}\varphi_{\hat\sigma}(a_1,\ldots,a_n)\mu(\hat\sigma,\hat\tau)\\
                                  &\overset{\omega=\hat\sigma}{=}& \sum_{\substack{\omega\in \setofnoncrossingpartitionsof(n)\\\hat{0}_m\leq \omega\leq \hat{\tau}}} \varphi_\omega(a_1,\ldots,a_n)\mu(\omega,\hat\tau)\\
                                  &=&\underset{\displaystyle\sum_{\substack{\pi\in \setofnoncrossingpartitionsof(n)\\\pi\leq \hat{\tau}}}\sum_{\substack{\omega\in \setofnoncrossingpartitionsof(n)\\\hat{0}_m\vee \pi\leq \omega\leq \hat{\tau}}}}{\underbrace{\sum_{\substack{\omega\in \setofnoncrossingpartitionsof(n)\\\hat{0}_m\leq \omega\leq \hat{\tau}}} \sum_{\substack{\pi\in \setofnoncrossingpartitionsof(n)\\\pi\leq \omega}} }}\kappa_\pi(a_1,\ldots,a_n)\mu(\omega,\hat\tau)\\
       &=&\sum_{\substack{\pi\in \setofnoncrossingpartitionsof(n)\\\pi\leq \hat{\tau}}}  \kappa_\pi(a_1,\ldots,a_n)\hspace{-1em}\underset{\displaystyle \overset{\zeta\ast\mu=\delta}{=}\begin{cases}1,&\text{if }\hat{0}_m\vee\pi=\hat{\tau},\\0,&\text{if }\hat{0}_m\vee \pi<\hat\tau. \end{cases}}{\underbrace{\sum_{\substack{\omega\in \setofnoncrossingpartitionsof(n)\\\hat{0}_m\vee\pi\leq \omega\leq \hat\tau}} \mu(\omega,\hat{\tau})}}\\
&=&\sum_{\substack{\pi\in\setofnoncrossingpartitionsof(n)\\\pi \leq\hat{\tau}\\ \hat{0}_m\vee \pi=\hat{\tau}}}\kappa_\pi(a_1,\ldots,a_n).
     \end{IEEEeqnarray*}
     For every $\pi \in \setofnoncrossingpartitionsof(n)$ we have: $\pi\leq \hat{\tau}$ and $\hat{0}_m\vee \pi=\hat{\tau}$ holds if and only if $\hat{0}\vee \pi=\hat{\tau}$; this proves then the claim.
   \end{proof}

   \begin{example}
     Let $s$ be a standard semicircular in a $\ast$-probability space $(\mathcal A,\varphi)$ with free cumulants $(\kappa_\pi)_{\pi\in \setofnoncrossingpartitionsof}$. Then, by Example~\hyperref[example:free-cumulants-2]{\ref*{example:free-cumulants}~\ref*{example:free-cumulants-2}}, for every $n\in \naturalnumbers$,
     \begin{align*}
       \kappa_n(s,s,\ldots,s)=
       \begin{cases}
         1,&\text{if } n=2,\\
         0,&\text{otherwise}.
       \end{cases}
     \end{align*}
     Let us check that $x\equalperdefinition s^2$ is a free Poisson of parameter $1$, i.e.\ that, for all $m\in \naturalnumbers$, we have
       $\kappa_m(x,x,\ldots,x)=1$.
     By Theorem~\ref{theorem:cumulant-product-rule}, for all $m\in \naturalnumbers$,
     \begin{align*}
       \kappa_m(ss,ss,\ldots,ss)\quad=\sum_{\substack{\pi\in \setofnoncrossingpartitionsof(2m)\\\pi\vee \hat{0}_m=1_{2m}}}\qquad\hspace{-0em}\underset{\hspace{-2em}\displaystyle=
       \begin{cases}
         1,&\text{if }\pi\in\setofnoncrossingpartitionsof_2,\\
         0,&\text{otherwise}.
       \end{cases}}{\underbrace{\kappa_\pi(s,s,\ldots,s)}}%\\
%\quad\\
      \quad =\quad\# \{\pi\in \setofnoncrossingpartitionsof_2 (2m)\mid \pi\vee\hat{0}_m=1_{2m}\}.
     \end{align*}
     Note that, here, $\hat{0}_m=\{ \{1,2\}, \{3,4\},\ldots, \{2m-1,2m\}  \}$.
     \begin{align*}
       \begin{tikzpicture}[baseline=-2pt-0.375em]
      \node[circle, scale=0.4, fill=white, draw=black] (n1) at (0em,0) {};
      \node[circle, scale=0.4, fill=white, draw=black] (n2) at (1.5em,0) {};
      \node[circle, scale=0.4, fill=white, draw=black] (n3) at (3em,0) {};
      \node[circle, scale=0.4, fill=white, draw=black] (n4) at (4.5em,0) {};
      \node[circle, scale=0.4, fill=white, draw=black] (n5) at (6em,0) {};
      \node[circle, scale=0.4, fill=white, draw=black] (n6) at (7.5em,0) {};
      \node[circle, scale=0.4, fill=white, draw=black] (n7) at (9em,0) {};
      \node[circle, scale=0.4, fill=white, draw=black] (n8) at (10.5em,0) {};
      \draw (n1) --++(0,1.125em) -| (n2);
      \draw (n3) --++(0,1.125em) -| (n4);
      \draw (n5) --++(0,1.125em) -| (n6);
      \draw (n7) --++(0,1.125em) -| (n8);
      \draw ($(n2)+(0,-1.5em)$) --++(0,-1.125em) -| ($(n6)+(0,-1.5em)$);
      \draw ($(n2)+(0,-4.5em)$) --++(0,-1.125em) -| ($(n7)+(0,-4.5em)$);
      \draw ($(n1)+(0,-7.5em)$) --++(0,-1.5em) -| ($(n8)+(6em,-7.5em)$);
      \draw ($(n2)+(0,-7.5em)$) --++(0,-0.75em) -| ($(n3)+(0,-7.5em)$);
      \draw ($(n4)+(0,-7.5em)$) --++(0,-0.75em) -| ($(n5)+(0,-7.5em)$);
      \draw ($(n6)+(0,-7.5em)$) --++(0,-0.75em) -| ($(n7)+(0,-7.5em)$);
      \draw ($(n8)+(0,-7.5em)$) --++(0,-0.75em) -| ($(n8)+(0,-7.5em)+(1.5em,0)$);
      \path (n8) -- node[pos=0.5] {$\ldots$}  ++(6em,0);
      \path ($(n8)+(1em,-7.5em)$) -- node[pos=0.5] {$\ldots$}  ++(4.5em,0);
      \draw [gray, pattern=north east lines, pattern color=lightgray] ($(n2)+(0.25em,-2.25em-0.125)$) rectangle ($(n6)+(-0.25em,-1em)$);
      \node [ inner sep=1pt] at ($(n4)+(0,-1.5em-0.1em)$) {odd};
      \draw [gray, pattern=north east lines, pattern color= lightgray] ($(n2)+(0.25em,-5.25em-0.125)$) rectangle ($(n7)+(-0.25em,-4em)$);
      \node [ inner sep=1pt] at ($(n4)+(0.75em,-4.5em-0.1em)$) {not connected};
      \node[align=left,right] at ($(n8)+(8em,-1.5em)$) {not possible};
      \node[align=left,right] at ($(n8)+(8em,-4.5em)$) {not possible};
      \node[align=left,right] at ($(n8)+(8em,-7.5em)$) {only possibility};      
%     \draw [brown] (-0.25em,-9.125em) rectangle (28.75em, 1.25em);
    \useasboundingbox (-0.25em,-9.125em) rectangle (28.75em, 1.25em);
       \end{tikzpicture}
       \end{align*}
       For every $m\in \naturalnumbers$, there is exactly one $\pi\in\setofnoncrossingpartitionsof_2(2m)$ which satisfies $\pi\vee \hat{0}_m=1_{2m}$, namely
       \begin{align*}
         \pi=\{\{1,2m\}, \{2,3\},\{4,5\}, \ldots \{2m-2,2m-1\} \}.
       \end{align*}
       It follows $\kappa_m(x,x,\ldots,x)=\# \{\pi\in\setofnoncrossingpartitionsof_2(2m)\mid \pi\vee \hat{0}_m=1_{2m}\}=1$ for every  $m\in\naturalnumbers$, which is what we wanted to see.
   \end{example}

   \begin{proposition}
     \label{proposition:cumulants-with-one}
     Let $(\mathcal A,\varphi)$ be a non-commutative probability space with free cumulants $(\kappa_n)_{n\in \naturalnumbers}$. Consider $n\in \naturalnumbers$ with $n\geq 2$ and $a_1,\ldots,a_n\in \mathcal A$. If there exists at least one $i\in [n]$ such that $a_i=1$, then 
       $\kappa_n(a_1,\ldots,a_n)=0$.
   \end{proposition}
Note that for $n=1$ we have $\kappa_1(1)=\varphi(1)=1$.
   \begin{proof}
     For simplicity, we only consider the case $a_n=1$. Then, we have to show, for all $n\in\naturalnumbers$ with $n\geq 2$, that
       $\kappa_n(a_1,\ldots,a_{n-1},1)\overset{!}{=}0$.
     We do this by induction over $n$.\par
     In the base case $n=2$, the claim is true by Example~\hyperref[example:free-cumulants-1-2]{\ref*{example:free-cumulants}~\ref*{example:free-cumulants-1}~\ref*{example:free-cumulants-1-2}}:
     \begin{align*}
       \kappa_2(a_1,1)=\varphi(a_1\cdot 1)-\varphi(a_1)\varphi(1)=0.
     \end{align*}
     Assume the statement is true for all $k\in \naturalnumbers$ with $k<n$. We prove it for $n$.
     By Theorem~\ref{theorem:cumulant-product-rule},
     \begin{align*}
       \kappa_{n-1}(a_1,\ldots,a_{n-1}\cdot 1)=\sum_{\substack{\pi\in \setofnoncrossingpartitionsof(n)\\\pi\vee \hat{0}_{n-1}=1_n}}\kappa_\pi(a_1,a_2,\ldots, a_{n-1},1),
     \end{align*}
     where $\hat{0}_{n-1}=\{ \{1\},\ldots, \{n-2\}, \{n-1,n\}  \}$.
     Possible $\pi\in \setofnoncrossingpartitionsof(n)$ with $\pi\vee \hat{0}_{n-1}=1_n$ can only be of two kinds: Clearly, $\pi=1_n$ satisfies this equation
          \begin{align*}
              \begin{tikzpicture}[baseline=-2pt-0.375em]
      \node[circle, scale=0.4, fill=white, draw=black] (n1) at (0em,0) {};
      \node[circle, scale=0.4, fill=white, draw=black] (n2) at (1.5em,0) {};
      \node[circle, scale=0.4, fill=white, draw=black] (n6) at (9em,0) {};
      \node[circle, scale=0.4, fill=white, draw=black] (n7) at (10.5em,0) {};
      \node[circle, scale=0.4, fill=white, draw=black] (n8) at (12em,0) {};
      %--------------------------------------------------      
      \draw (n1) -- ++ (0,1.125em);
      \draw (n2) -- ++ (0,1.125em);
      \draw (n6) -- ++ (0,1.125em);
      \draw (n7) -- ++ (0,1.125em) -| (n8);
      \path (n2) -- node [pos=0.5] {\ldots} (n6);
      \draw ($(n1)+(0,-1.5em)$) --++(0,-0.75em) -| ($(n2)+(0,-1.5em)$);
      \draw ($(n2)+(0,-2.25em)$)  -- ++ (0.5em,0);
      \path ($(n2)+(0,-2.25em)$) -- node [pos=0.5] {\ldots} ($(n6)+(0,-2.25em)$);      
      \draw ($(n6)+(0,-2.25em)$)  -- ++ (-0.5em,0);      
      \draw ($(n6)+(0,-1.5em)$) --++(0,-0.75em) -| ($(n7)+(0,-1.5em)$);
      \draw ($(n7)+(0,-2.25em)$)  -| ($(n8)+(0,-1.5em)$);
      % --------------------------------------------------
      \node at (-2em, 0.5em) {$\hat{0}_n$};
      \node at (-2em, -2em) {$1_n$};      
      % --------------------------------------------------
%     \draw [brown] (-0.25em,-3.125em) rectangle (12.25em, 1.25em);
    \useasboundingbox (-0.25em,-3.125em) rectangle (12.25em, 1.25em);
       \end{tikzpicture}
          \end{align*}
          and contributes
          \begin{align*}
            \kappa_{1_n}(a_1,\ldots,a_{n-1},1)=\kappa_n(a_1,\ldots, a_{n-1},1).
          \end{align*}
          The other possibility is that there exists $r\in \naturalnumbers_0$ with $r<n-1$ such that
          \begin{align*}
            \pi=\{\{1,2,\ldots,r,n\}, \{r+1,r+2,\ldots,n-1\}\},
          \end{align*}
     \begin{align*}
              \begin{tikzpicture}[baseline=-2pt-0.375em]
      \node[circle, scale=0.4, fill=white, draw=black] (n1) at (0em,0) {};
      \node[circle, scale=0.4, fill=white, draw=black] (n2) at (1.5em,0) {};
%      \node[circle, scale=0.4, fill=white, draw=black] (n3) at (4.5em,0) {};
%      \node[circle, scale=0.4, fill=white, draw=black] (n4) at (6em,0) {};
%      \node[circle, scale=0.4, fill=white, draw=black] (n5) at (7.5em,0) {};
      \node[circle, scale=0.4, fill=white, draw=black] (n6) at (9em,0) {};
      \node[circle, scale=0.4, fill=white, draw=black] (n7) at (10.5em,0) {};
      \node[circle, scale=0.4, fill=white, draw=black] (n8) at (12em,0) {};
      % --------------------------------------------------
      \draw (n1) -- ++ (0,1.125em);
      \draw (n2) -- ++ (0,1.125em);
      \draw (n6) -- ++ (0,1.125em);
      \draw (n7) -- ++ (0,1.125em) -| (n8);
      \path (n2) -- node [pos=0.5] {\ldots} (n6);
      % --------------------------------------------------      
      \node[inner sep=1pt] (l1) at (0em,-1.5em) {$1$};
      \node[inner sep=1pt] (l2) at (1.5em,-1.5em) {};
      \node[inner sep=1pt] (l3) at (4.5em,-1.5em) {$r$};
      \node[inner sep=1pt] (l4) at (6em,-1.5em) {$r\!+\!1$};
      \node[inner sep=1pt] (l5) at (7.5em,-1.5em) {};
      \node[inner sep=1pt] (l6) at (9em,-1.5em) {$\phantom{n\!\!-\!\!2}$};
      \node[inner sep=1pt] (l7) at (10.5em,-1.5em) {$n\!\!-\!\!1$};
      \node[inner sep=1pt] (l8) at (12em,-1.5em) {$n$};
      % --------------------------------------------------
      \draw (l1) --++(0,-2.25em) -- ++ (0.5em,0);
      \draw (l3) --++(0,-2.25em) -- ++ (-0.5em,0);
      \path ($(l1)+(0,-2.25em)$) -- node [pos=0.5] {\ldots} ($(l3)+(0,-2.25em)$);
      \draw ($(l3)+(0,-2.25em)$) -| (l8);
      \draw (l4) --++(0,-1.375em) --++(0.5em,0);
      \draw (l6) --++(0,-1.375em) --++(-0.5em,0);
      \draw ($(l6)+(0,-1.375em)$) -| (l7);
      \path ($(l4)+(0,-1.375em)$) -- node [pos=0.5] {\ldots} ($(l6)+(0,-1.375em)$);
      % --------------------------------------------------
      \node at (-2em, 0.5em) {$\hat{0}_n$};
      \node at (-2em, -2.5em) {$\pi$};            
      %---------------------------------------------------
%     \draw [brown] (-0.25em,-4.125em) rectangle (12.25em, 1.25em);
    \useasboundingbox (-0.25em,-4.125em) rectangle (12.25em, 1.25em); 
       \end{tikzpicture}
     \end{align*}
     contributing
     \begin{align*}
       \kappa_\pi(a_1,\ldots,a_{n-1},1)=\kappa_{r+1}(a_1,\ldots, a_r,1)\kappa_{n-r-1}(a_{r+1},\ldots,a_{n-1}).
     \end{align*}
     If $r\neq 0$, then the induction hypothesis implies $\kappa_{r+1}(a_1,\ldots,a_r,1)=0$ and thus, for such $\pi$, $\kappa_\pi(a_1,\ldots,a_{n-1},1)=0$.
     Hence, the only potentially non-zero contribution comes for such $\pi$ from the one with $r=0$ and amounts to $\kappa_1(1)\kappa_{n-1}(a_1,\ldots,a_{n-1})$.
     \par
     Consequently, since $\kappa_1(1)=1$,
     \begin{align*}
       \kappa_{n-1}(a_1,\ldots,\underset{\displaystyle= a_{n-1}}{\underbrace{a_{n-1}\cdot 1}})=\kappa_n(a_1,\ldots,a_{n-1},1)+\kappa_{n-1}(a_1,\ldots,a_{n-1}),
     \end{align*}
     which proves $\kappa_{n}(a_1,\ldots,a_{n-1},1)=0$, as claimed.
        \end{proof}

        \begin{theorem}[Freeness $\hateq$ Vanishing of Mixed Cumulants, {Speicher} 1994]
          \label{theorem:freeness-vanishing-mixed-cumulants}
          Let $(\mathcal A,\varphi)$ be a non-commutative probability space with free cumulants $(\kappa_n)_{n\in \naturalnumbers}$ and let $(\mathcal A_i)_{i\in I}$ be a family of unital subalgebras of  $\mathcal A$. Then, the following statements are equivalent.
          \begin{enumerate}
          \item\label{theorem:freeness-vanishing-mixed-cumulants-1} The subalgebras  $(\mathcal A_i)_{i\in I}$ are freely independent in $(\mathcal A,\varphi)$.
          \item\label{theorem:freeness-vanishing-mixed-cumulants-2} Mixed cumulants in the subalgebras $(\mathcal A_i)_{i\in I}$ vanish: For all $n\in \naturalnumbers$ with $n\geq 2$, all $i:[n]\to I$ and all $a_1,\ldots,a_n\in \mathcal A$ with $a_j\in \mathcal A_{i(j)}$ for every $j\in [n]$ we have that
              $\kappa_n(a_1,\ldots,a_n)=0$
            whenever there exist $l,k\in [n]$ such that $i(l)\neq i(k)$.
          \end{enumerate}
        \end{theorem}
        \begin{proof}
          \ref{theorem:freeness-vanishing-mixed-cumulants-2} $\Rightarrow$ \ref{theorem:freeness-vanishing-mixed-cumulants-1}: Consider a situation as in the definition of freeness, i.e.\ let $n\in \naturalnumbers$, $i:[n]\to I$, $i(1)\neq i(2)\neq \ldots\neq i(n)$,  $a_1,\ldots,a_n\in \mathcal A$ with $a_j\in \mathcal{A}_{i(j)}$ and $\varphi(a_j)=0$ for all $j\in [n]$. Then, we have to show that $\varphi(a_1\ldots a_n)=0$. This follows from
          \begin{align*}
            \varphi(a_1\ldots a_n)=\sum_{\pi\in \setofnoncrossingpartitionsof(n)}\hspace{-4em}\underset{\hspace{4em}\displaystyle =\prod_{V\in \pi} \kappa_{\# V}(a_1,\ldots,a_n\mid V)}{\underbrace{\kappa_\pi(a_1,\ldots,a_n).}}
          \end{align*}
          by the following reasoning:
          For every $\pi\in \setofnoncrossingpartitionsof(n)$ and every block $V\in \pi$ which is an \emph{interval}  it holds that
            $\kappa_{\# V}(a_1,\ldots,a_n\mid V)=0$;
          in the case $\# V=1$ because $a_1,\ldots,a_n$ are centered, and in the case $\# V>1$ because of the Assumption~\ref{theorem:freeness-vanishing-mixed-cumulants-2} of vanishing of mixed cumulants. 
          Since each $\pi\in \setofnoncrossingpartitionsof(n)$ contains at least one interval block, it follows $\varphi(a_1\ldots a_n)=0$, as claimed.
          \par
          \ref{theorem:freeness-vanishing-mixed-cumulants-1} $\Rightarrow$ \ref{theorem:freeness-vanishing-mixed-cumulants-2}: Let $n\in \naturalnumbers$, $i:[n]\to I$, $a_1,\ldots,a_n\in \mathcal A$ and $a_j\in \mathcal A_{i(j)}$ for every $j\in [n]$. Assume first that $a_1,\ldots,a_n$ are centered and {alternating}, i.e.\ that
          \begin{gather*}
            \varphi(a_1)=\ldots=\varphi(a_n)=0 \quad\text{and}\quad i(1)\neq i(2)\neq \ldots\neq i(n).
          \end{gather*}
          Then, by definition,
          \begin{align*}
            \kappa_n(a_1,\ldots,a_n)&=\sum_{\pi\in \setofnoncrossingpartitionsof(n)}\hspace{-4em}\underset{\hspace{4em}\displaystyle = \prod_{V\in \pi}\varphi_{\# V}(a_1,\ldots,a_n\mid V)}{\underbrace{\varphi_\pi(a_1,\ldots,a_n)}}\hspace{-4em} \mu(\pi,1_n).
          \end{align*}
          Again, for every $\pi\in \setofnoncrossingpartitionsof(n)$ which contains an interval block $V\in \pi$ it holds that
           $\varphi_{\# V}(a_1,\ldots,a_n\mid V)=0$
          due to the freeness Assumption~\ref{theorem:freeness-vanishing-mixed-cumulants-1}. Thus, since every non-crossing partition has at least one interval block, it follows
            $\kappa_n(a_1,\ldots,a_n)=0$,
          and thus the claim for centered and {alternating} variables.
          \par
          Because, by Proposition~\ref{proposition:cumulants-with-one}, since $n\geq 2$,
          \begin{align*}
            \kappa_n(a_1,\ldots,a_n)=\kappa_n(a_1-\varphi(a_1)\cdot 1,\ldots,a_n-\varphi(a_n)\cdot 1),
          \end{align*}
          we can get rid of the assumption $\varphi(a_1)=\ldots=\varphi(a_n)=0$.
         Finally,  we also want to see the vanishing of the cumulant if arguments are only mixed, not necessarily alternating, i.e.\ if there exist $l,k\in [n]$ such that $i(l)\neq i(k)$, but \emph{not} necessarily $i(1)\neq \ldots\neq i(n)$. Given such {mixed} arguments $a_1,\ldots,a_n$ for $\kappa_n$, we multiply neighbors together to make them {alternating}: We choose $m\in \naturalnumbers$ and $i':[m]\to I$ such that
            $a_1\ldots a_n=A_1\ldots A_m$,
where $A_1,\ldots,A_m\in \mathcal A$, such that  $A_j\in \mathcal A_{i'(j)}$ for every $j\in [m]$ and such that $i'(1)\neq \ldots\neq i'(m)$.
          \par
          Note that $m\geq 2$ because $a_1,\ldots,a_n$ are {mixed}. Hence, for $A_1,\ldots,A_m$ we already know that $\kappa_m(A_1,\ldots,A_m)=0$ by what was shown above.
          On the other hand, by Theorem~\ref{theorem:cumulant-product-rule},
           \begin{align*}
            0=\kappa_m(A_1,\ldots,A_m)&=\sum_{\substack{\pi\in \setofnoncrossingpartitionsof(n)\\\pi\vee \hat{0}_m=1_n}}\kappa_\pi(a_1,\ldots,a_n)\\
                                    &=\underset{\displaystyle = \kappa_{n}(a_1,\ldots,a_n)}{\underbrace{\kappa_{1_n}(a_1,\ldots,a_n)}} +\sum_{\substack{\pi\in \setofnoncrossingpartitionsof(n)\\\pi\neq 1_n\\\pi\vee \hat{0}_m=1_n}}\kappa_\pi(a_1,\ldots,a_n).
           \end{align*}
      
    Note $\ker(i)\geq\hat {0}_m$.
  By induction, we can infer that any $\pi\in\setofnoncrossingpartitionsof(n)$ with $\pi\neq 1_m$ (which must have all blocks of size less than $n$) can only yield a potentially non-zero contribution $\kappa_\pi(a_1,\ldots,a_n)$ if each block of $\pi$ connects exclusively elements from the same subalgebra, i.e.\ if $\pi\leq \ker(i)$. For such $\pi$ the condition $\pi\vee \hat{0}_m=1_n$ would then give 
$\ker(i)=1_n$, saying that all $a_i$ are from the same subalgebra. But that would contradict $m\geq 2$. Hence, there are \emph{no} $\pi$ besides $1_n$ which could yield non-zero contributions. Thus,
          \begin{align*}
            \kappa_n(a_1,\ldots,a_n)=\kappa_m(A_1,\ldots,A_m)=0.
          \end{align*}
          To get the above induction started, consider the base case $n=2$. Use Example~\hyperref[example:free-cumulants-1-2]{\ref*{example:free-cumulants}~\ref*{example:free-cumulants-1}~\ref*{example:free-cumulants-1-2}} to find
          \begin{align*}
            \kappa_2(a_1,a_2)=\varphi(a_1a_2)-\varphi(a_1)\varphi(a_2).
          \end{align*}
 Assuming that $a_1,a_2$ are {mixed}, means that $a_1$ and $a_2$ are free, from which  then $\varphi(a_1a_2)=\varphi(a_1)\varphi(a_2)$  follows by Example \hyperref[example:moment-formulas-for-small-order-2]{\ref*{example:moment-formulas-for-small-order}~\ref*{example:moment-formulas-for-small-order-2}}. Hence, $\kappa_2(a_1,a_2)=0$, which completes the proof.          
        \end{proof}
        We can refine this in similar way (see Assignment~\hyperref[assignment-6]{6}, Exercise~2) to a characterization of freeness for random variables.

        \begin{theorem}
          \label{theorem:freeness-vanishing-mixed-cumulants-random-variables}
          Let $(\mathcal A,\varphi)$ be a non-commutative probability space with free cumulants  $(\kappa_n)_{n\in \naturalnumbers}$ and $(a_i)_{i\in I}$ a family of random variables in $\mathcal A$. Then, the following statements are equivalent:
          \begin{enumerate}
          \item The random variables $(a_i)_{i\in I}$ are freely independent in $(\mathcal A,\varphi)$.
          \item Mixed cumulants in the random variables $(a_i)_{i\in I}$ vanish: For all $n\in \naturalnumbers$ with $n\geq 2$ and all $i:[n]\to I$ we have that
              $\kappa_n(a_{i(1)},\ldots, a_{i(n)})=0$
            whenever there exist $l,k\in[n]$ such that $i(l)\neq i(k)$.
          \end{enumerate}
        \end{theorem}
        \begin{remark}
          \label{remark:moments-cumulants-convolution}
          Consider now a fixed single random variable $a$ in some non-commutative probability space $(\mathcal A,\varphi)$ with free cumulants $(\kappa_n)_{n\in\naturalnumbers}$. Then, its moments $(m_n)_{n\in \naturalnumbers}$ and its cumulants $(\kappa_n)_{n\in \naturalnumbers}$, where, for all $n\in \naturalnumbers$,
$$
            m_n\equalperdefinition \varphi_n(a,a,\ldots,a)=\varphi(a^n)\qquad
\text{and}\qquad
            \kappa_n\equalperdefinition\kappa_n(a,a,\ldots,a)
$$
          are just sequences of numbers, which we extend to \enquote{multiplicative} functions $m:\, \setofnoncrossingpartitionsof\to\complexnumbers$ and $\kappa:\, \setofnoncrossingpartitionsof\to\complexnumbers$ via
          \begin{align}
            \label{eq:remark-moments-cumulants-convolution}
            m(\pi)\equalperdefinition m_\pi\equalperdefinition \prod_{V\in \pi}m_{\# V}\quad\text{and}\quad \kappa(\pi)\equalperdefinition \kappa_\pi\equalperdefinition \prod_{V\in \pi}\kappa_{\# V} 
          \end{align}
          for all $n\in \naturalnumbers$ and $\pi\in \setofnoncrossingpartitionsof(n)$.
          Then, $m$ and $\kappa$ satisfy the relations
            $\kappa=m\ast\mu$ and $m=\kappa\ast\zeta$.

          Those combinatorial relations are conceptually nice but usually not so useful for concrete calculations. We need a more analytic reformulation of this.
        \end{remark}

        \begin{theorem}
          \label{theorem:moment-cumulant-series}
          Let $(m_n)_{n\in \naturalnumbers}$ and $(\kappa_n)_{n\in \naturalnumbers}$ be two sequences in $\complexnumbers$ and let two corresponding multiplicative functions $m,\kappa:\setofnoncrossingpartitionsof\to \complexnumbers$ be determined via Equation~\eqref{eq:remark-moments-cumulants-convolution}.
          Consider the corresponding formal power series in $\complexnumbers \lsem z\rsem$:
          \begin{align*}
            M(z)=1+\sum_{n=1}^\infty m_nz^n\quad
            \text{and}\quad C(z)&=1+\sum_{n=1}^\infty \kappa_nz^n.
          \end{align*}
          Then, the following statements are equivalent:
          \begin{enumerate}
          \item\label{theorem:moment-cumulant-series-1} $m=k\ast\zeta$, i.e.
we have for all $n\in \naturalnumbers$
\begin{equation*}
m_n=\sum_{\pi\in \setofnoncrossingpartitionsof(n)}\kappa_\pi.
\end{equation*}
            \item\label{theorem:moment-cumulant-series-2} 
We have for all $n\in \naturalnumbers$
$$ m_n=\sum_{s=1}^n \sum_{\substack{i_1,\ldots,i_s\in \{0,1,\ldots,n-s\}\\i_1+\ldots+i_s+s=n}}\kappa_sm_{i_1}\ldots m_{i_s}.$$
            \item\label{theorem:moment-cumulant-series-3} 
We have as functional relation in $\complexnumbers\lsem z\rsem$
$$C(z\cdot M(z))=M(z).$$
            \item\label{theorem:moment-cumulant-series-4} 
We have as functional relation in $\complexnumbers\lsem z\rsem$
$$M\Bigl(\frac{z}{C(z)}\Bigr)=C(z).$$
          \end{enumerate}
        \end{theorem}
        \begin{proof}
          \ref{theorem:moment-cumulant-series-1} $\Rightarrow$ \ref{theorem:moment-cumulant-series-2}:  Let $n\in \naturalnumbers$ and $\pi\in \setofnoncrossingpartitionsof(n)$ be arbitrary. If we let $V$ be the block of $\pi$ containing $1$, then we can write
            $\pi=\{V\}\cup \pi_1\cup \ldots\cup \pi_s$
and
          \begin{align*}
            V=\{1,\, i_1+2,\, i_1+i_2+3,\, \ldots,\, i_1+\ldots+i_{s-1}+s\}
          \end{align*}
          for certain $s=\# V\in [n]$, $i_1,\ldots,i_s\in [n]$ and $\pi_1\in \setofnoncrossingpartitionsof(i_1),\ldots,\pi_s\in \setofnoncrossingpartitionsof(i_s)$.                     
          \begin{align*}
            \begin{tikzpicture}
              \node [circle,scale=0.4, draw=black, fill=white, label=left:{$1$}] (n1) at (0em,0) {};
              \node [circle,scale=0.4, draw=black, fill=white] (n2) at (1.5em,0) {};
              %-----------------------
              \node [circle,scale=0.4, draw=black, fill=white] (n3) at (4.5em,0) {};
              \node [circle,scale=0.4, draw=black, fill=white] (n4) at (6em,0) {};
              \node [circle,scale=0.4, draw=black, fill=white] (n5) at (7.5em,0) {};
              %-----------------------              
              \node [circle,scale=0.4, draw=black, fill=white] (n6) at (10.5em,0) {};
              \node [circle,scale=0.4, draw=black, fill=white] (n7) at (12em,0) {};
              \node [circle,scale=0.4, draw=black, fill=white] (n8) at (13.5em,0) {};
              %-----------------------              
              \node [circle,scale=0.4, draw=black, fill=white] (n9) at (16.5em,0) {};
              \node [circle,scale=0.4, draw=black, fill=white] (n10) at (18em,0) {};
              \node [circle,scale=0.4, draw=black, fill=white] (n11) at (19.5em,0) {};
              %-----------------------              
              \node [circle,scale=0.4, draw=black, fill=white] (n12) at (22.5em,0) {};
              \node [circle,scale=0.4, draw=black, fill=white] (n13) at (24em,0) {};
              \node [circle,scale=0.4, draw=black, fill=white] (n13a) at (25.5em,0) {};
              %-----------------------              
              \node [circle,scale=0.4, draw=black, fill=white, label=right:{$n$}] (n14) at (30em,0) {};
              % -----------------------
              \draw (n1) --  ++ (0,-2em) -| node[pos=0.1,below] {$V$} (n4);
              \draw ($(n4)+(0,-2em)$)  -| (n7);
              \draw (n10) -- ++ (0,-2em) -| (n13);
              \draw ($(n7)+(0,-2em)$)  --++ (0.5em,0);
              \draw ($(n10)+(0,-2em)$)  --++ (-0.5em,0);
              \path (n2) -- node[pos=0.5] {$\ldots$} (n3);
              \path (n5) -- node[pos=0.5] {$\ldots$} (n6);
              \path (n8) -- node[pos=0.5] {$\ldots$} (n9);
              \path (n11) -- node[pos=0.5] {$\ldots$} (n12);
              \path ($(n7)+(0,-2em)$) -- node[pos=0.5] {$\ldots$} ($(n10)+(0,-2em)$);              
              \path (n13a) -- node[pos=0.5] {$\ldots$} (n14);              
              \draw[pattern=north west lines, pattern color=lightgray] ($(n1)+(0.4em,-1.6em)$) rectangle ($(n4)+(-0.4em,-0.5em)$) node [pos=0.5] {$\pi_1$};
              \draw[pattern=north west lines, pattern color=lightgray] ($(n4)+(0.4em,-1.6em)$) rectangle ($(n7)+(-0.4em,-0.5em)$) node [pos=0.5] {$\pi_2$};
              \draw[pattern=north west lines, pattern color=lightgray] ($(n10)+(0.4em,-1.6em)$) rectangle ($(n13)+(-0.4em,-0.5em)$) node [pos=0.5] {$\pi_{s-1}$};
              \draw[pattern=north west lines, pattern color=lightgray] ($(n13)+(0.4em,-1.6em)$) rectangle ($(n14)+(0.25em,-0.5em)$)node [pos=0.5] {$\pi_s$};
              \draw [decorate,decoration={brace,amplitude=3pt,raise=4pt},xshift=0pt]
              ($(n2)+(-0.25em,0)$) -- ($(n3)+(0.25em,0)$) node [black,midway,yshift=1.25em] {$i_1$};
              \draw [decorate,decoration={brace,amplitude=3pt,raise=4pt},xshift=0pt]
              ($(n5)+(-0.25em,0)$) -- ($(n6)+(0.25em,0)$) node [black,midway,yshift=1.25em] {$i_2$};
              \draw [decorate,decoration={brace,amplitude=3pt,raise=4pt},xshift=0pt]
              ($(n11)+(-0.25em,0)$) -- ($(n12)+(0.25em,0)$) node [black,midway,yshift=1.25em] {$i_{s-1}$};
              \draw [decorate,decoration={brace,amplitude=3pt,raise=4pt},xshift=0pt]
($(n13a)+(-0.25em,0)$) -- ($(n14)+(0.25em,0)$) node [black,midway,yshift=1.25em] {$i_{s}$};  
            \end{tikzpicture}
          \end{align*}
          Thus, by using these decompositions for all $\pi\in\setofnoncrossingpartitionsof(n)$,
          \begin{align*}
            m_n&=\sum_{\pi\in\setofnoncrossingpartitionsof(n)}\kappa_\pi.\\
               &=\sum_{s=1}^n\sum_{\substack{i_1,\ldots,i_s\in \{0,\ldots,n-s\}\\i_1+\ldots+i_s+s=n}}\sum_{\pi_1\in \setofnoncrossingpartitionsof(i_1)}\ldots \sum_{\pi_s\in \setofnoncrossingpartitionsof(i_s)} \kappa_s\kappa_{\pi_1}\ldots\kappa_{\pi_s}\\
                           &=\sum_{s=1}^n\sum_{\substack{i_1,\ldots,i_s\in \{0,\ldots,n-s\}\\i_1+\ldots+i_s+s=n}}\kappa_s\cdot \underset{\displaystyle =m_{i_1}}{\underbrace{\Bigl(\sum_{\pi_1\in \setofnoncrossingpartitionsof(i_1)}\kappa_{\pi_1}\Bigr)}}\quad\cdots\quad  \underset{\displaystyle =m_{i_s}}{\underbrace{\Bigl(\sum_{\pi_s\in \setofnoncrossingpartitionsof(i_s)} \kappa_{\pi_s}\Bigr)}}\\
            &=\sum_{s=1}^n\sum_{\substack{i_1,\ldots,i_s\in \{0,\ldots,n-s\}\\i_1+\ldots+i_s+s=n}} \kappa_sm_{i_1}\ldots m_{i_s},
          \end{align*}
          which is what we needed to see.
          \par
          \ref{theorem:moment-cumulant-series-2} $\Rightarrow$ \ref{theorem:moment-cumulant-series-3}: For every $n\in \naturalnumbers$ multiply the expression \ref{theorem:moment-cumulant-series-2} for $m_n$ with $z^n$ and then sum over all $n\in \naturalnumbers$:
          \begin{align*}
            M(z)&=1+\sum_{n=1}^\infty m_nz^n \\
&=1+\sum_{n=1}^\infty\sum_{s=1}^n\sum_{\substack{i_1,\ldots,i_s\in \{0,\ldots,n-s\}\\i_1+\ldots+i_s+s=n}} \kappa_sm_{i_1}\ldots m_{i_s} z^n\\
                &=1+\sum_{s=1}^\infty\sum_{i_1,\ldots,i_s}(\kappa_sz^s)(m_{i_1}z^{i_1})\ldots(m_{i_s}z^{i_s})\\
                &=1+\sum_{s=1}^\infty \kappa_s z^s\bigg(\underset{\displaystyle =M(z)}{\underbrace{\sum_{i=0}^\infty m_iz^i}}\bigg)^s\\
                &=1+\sum_{s=1}^\infty \kappa_s(zM(z))^s\\
            &=C(z\cdot M(z)).
          \end{align*}
          And that proves this implication.
          \par
          \ref{theorem:moment-cumulant-series-3} $\Rightarrow$ \ref{theorem:moment-cumulant-series-1}: Both \ref{theorem:moment-cumulant-series-1} and \ref{theorem:moment-cumulant-series-3} determine a unique relation between $(m_n)_{n\in \naturalnumbers}$  and $(\kappa_n)_{n\in \naturalnumbers}$. Hence, \ref*{theorem:moment-cumulant-series-1} $\Rightarrow$ \ref*{theorem:moment-cumulant-series-3} gives also \ref*{theorem:moment-cumulant-series-3} $\Rightarrow$ \ref*{theorem:moment-cumulant-series-1}.
          \par
          \ref{theorem:moment-cumulant-series-3} $\Rightarrow$ \ref{theorem:moment-cumulant-series-4}: Put $w=z\cdot M(z)$. Then
          \begin{align*}
            z=\frac{w}{M(z)}=\frac{w}{C(z\cdot M(z))}=\frac{w}{C(w)}
          \end{align*}
          and thus
          \begin{align*}
            C(w)=C(z\cdot M(z))=M(z)=M\Bigl(\frac{w}{C(w)}\Bigr).
          \end{align*}
          \par
          \ref{theorem:moment-cumulant-series-4} $\Rightarrow$ \ref{theorem:moment-cumulant-series-3}: The proof is similar to that of \ref*{theorem:moment-cumulant-series-3} $\Rightarrow$ \ref*{theorem:moment-cumulant-series-4}.
        \end{proof}

        \begin{example}
          \begin{enumerate}
          \item Consider the implications of Theorem~\ref{theorem:moment-cumulant-series} in the case of a standard semicircular element: Let $(m_n)_{n\in \naturalnumbers}$ and $(\kappa_n)_{n\in \naturalnumbers}$ be given by the moments respectively cumulants of a standard semicircular in some non-commutative probability space. By~\hyperref[example:free-cumulants-2]{\ref*{example:free-cumulants}~\ref*{example:free-cumulants-2}} we have $\kappa_n=\delta_{n,2}$ for all $n\in \naturalnumbers$. Hence, the formal power series $C(z)\in \complexnumbers\lsem z\rsem$ associated with the cumulants is 
             $C(z)=1+z^2$. 
            And, Relation~\ref{theorem:moment-cumulant-series-3} of Theorem~\ref{theorem:moment-cumulant-series} says about the formal power series $M(z)\in \complexnumbers\lsem z\rsem$ associated with the moments that
            \begin{align*}
              1+z^2 M(z)^2=M(z).
            \end{align*}
            Since odd moments vanish, we can write $M(z)=f(z^2)$ for a formal power series $f(z)\in \complexnumbers\lsem z\rsem$. 
The previous equation for $M(z)$ gives then
            \begin{align*}
              1+zf(z)^2=f(z).
            \end{align*}
            This is the equation for the generating series of the Catalan numbers, see Assignment~\hyperref[assignment-3]{3}, Exercise~1.
          \item Next, let  $(m_n)_{n\in \naturalnumbers}$ and $(\kappa_n)_{n\in \naturalnumbers}$ be  the moments and cumulants of a free Poisson element with parameter $1$. Then, in particular, $\kappa_n=1$ for all  $n\in\naturalnumbers$. Thus, the formal power series in $\complexnumbers\lsem z\rsem$ of the cumulants is given by
            \begin{align*}
              C(z)=1+\sum_{n=1}^\infty z^n=\frac{1}{1-z}.
            \end{align*}
            Here, Relation~\ref{theorem:moment-cumulant-series-3} of Theorem~\ref{theorem:moment-cumulant-series} translates as
            \begin{align*}\frac{1}{1-z M(z)}=M(z),
            \end{align*}
            which is equivalent to
            \begin{align*}
1+z M(z)^2&=M(z).
            \end{align*}
This is again the equation for the generating series of the Catalan numbers.
          \item Consider the single-variable restrictions of the convolution unit $\delta$ and the M\"obius function $\mu$ of $\setofnoncrossingpartitionsof$ defined by
            \begin{align*}
              \mu(\pi)\equalperdefinition\mu(0_n,\pi)\quad\text{and}\quad
              \delta(\pi)\equalperdefinition\delta(0_n,\pi)
            \end{align*}
            for all $n\in \naturalnumbers$ and $\pi\in\setofnoncrossingpartitionsof(n)$.
        In the sense of Theorem~\ref{theorem:moment-cumulant-series}, let $\delta$ induce the moment sequence $(m_n)_{n\in \naturalnumbers}$ and $\mu$ the cumulant sequence $(\kappa_n)_{n\in \naturalnumbers}$. Then, the corresponding formal power series $M(z),C(z)\in\complexnumbers\lsem z\rsem$ are
            \begin{align*}
              M(z)=1+\sum_{n=1}^\infty \underset{\displaystyle =\delta_{n,1}}{\underbrace{\delta(0_n,1_n)}}z^n=1+z \quad\text{and}\quad              C(z)=1+\sum_{n=1}^\infty \mu(0_n,1_n)z^n.
            \end{align*}
                 The identity $\mu\ast\zeta=\delta$, shown in  Proposition~\ref{proposition:moebius-function}, holds for the restricted versions of $\mu$ and $\delta$ as well. It is precisely Relation~\ref{theorem:moment-cumulant-series-1} of Theorem~\ref{theorem:moment-cumulant-series}.
            Hence, the theorem tells us Relation~\ref{theorem:moment-cumulant-series-4} must be true as well. In this case, the latter says
            \begin{align*}
              1+\frac{z}{C(z)}=C(z),
            \end{align*}
            from which we conclude
             $C(z)+z=C(z)^2$. 
            Again, define a formal power series $f(z)\in \complexnumbers\lsem z\rsem$ implicitly by
             $C(z)=1+z f(-z)$. 
            With this transformation, the above identity reads as 
            \begin{align*}
              1+z f(-z)+z=z^2f(-z)^2+2zf(-z)+1,
\qquad
            \text{or}\qquad
              -f(-z)+1=zf(-z)^2.
            \end{align*}
            Replacing now the indeterminate $z$ by $-z$, we obtain the equation
            \begin{align*}
                        1+zf(z)^2&=f(z).
            \end{align*}
            Hence, $f$ is again the generating function for the Catalan numbers. Thus,
            \begin{align*}
              C(z)&=1+zf(-z)\\
                  &=1+z\Bigl[1+\sum_{n=1}^\infty C_n(-z)^n\Bigr]\\
                  &=1+\sum_{n= 0}^\infty C_n(-z)^{n+1}\cdot(-1)\\
              &=1+\sum_{n=1}^\infty\underset{\displaystyle\overset{!}{=}\mu(0_n,1_n)}{\underbrace{(-1)^{n-1}C_{n-1}}} z^n.
            \end{align*}
            Comparing coefficients of $C(z)$ allows us to draw the following conclusion  relating $\mu$ and the Catalan numbers.
          \end{enumerate}
        \end{example}
        \begin{corollary}
          The M\"obius function of $\setofnoncrossingpartitionsof$ satisfies for all $n\in \naturalnumbers$:
          \begin{align*}
            \mu(0_n,1_n)=(-1)^{n-1}C_{n-1}.
          \end{align*}
        \end{corollary}

\newpage

        \section{Free Convolution of Compactly Supported Probability Measures and the $R$-Transform}
%[Free convolution: R-transform]
        \begin{remark}
          Classical convolution of probability measures corresponds to forming the distribution of the sum of independent variables from the distributions of the individual variables. We want to consider the free analogue. The classical distribution of a self-adjoint random variable is not just a collection of moments, but can be identified with a probability measure on $\realnumbers$. We will also consider our free analogue in such an analytical context. In order to avoid problems with unbounded operators and the moment problem we stick for now to probability measures with compact support. Let us collect some relevant facts about this situation. 
        \end{remark}

        \begin{facts}
          \label{facts:analytic-distribution}
          \begin{enumerate}
          \item\label{facts:analytic-distribution-1} Let $\mu$ be a probability measure on $\realnumbers$ with compact support, which means that there exists $M\in \realnumbers$, $M>0$ such that
              $\mu([-M,M])=1$.
            Then, the moments $(m_n)_{n\in \naturalnumbers}$ of $\mu$, defined by
            \begin{align*}
              m_n\equalperdefinition \int_\realnumbers t^n\,d\mu(t)
\qquad \text{(for every $n\in \naturalnumbers$)},
            \end{align*}
            \begin{enumerate}
            \item\label{facts:analytic-distribution-1-1} are all finite,
            \item\label{facts:analytic-distribution-1-2} are exponentially bounded with constant $M$, i.e.,              for all $n\in \naturalnumbers$,
              \begin{align*}
                |m_n|\leq M^n,
              \end{align*}
            \item\label{facts:analytic-distribution-1-3} determine the probability measure $\mu$ uniquely: For every probability measure $\nu$ on $\realnumbers$ with
              \begin{align*}
                \int_\realnumbers t^n\,d\nu(t)=m_n
              \end{align*}
              for all $n\in \naturalnumbers$ it follows that $\nu=\mu$. (One does not need to require $\nu$ to be compactly supported for this to be true.)
            \end{enumerate}
          \item\label{facts:analytic-distribution-2} If $(\mathcal A,\varphi)$ is a $\ast$-probability space and $x$ a self-adjoint random variable in $\mathcal A$ with exponentially bounded moments, i.e.\ such that there exists $M\in \realnumbers$, $M>0$ with $|\varphi(x^n)|\leq M^n$  for all $n\in \naturalnumbers$, then there exists a uniquely determined probability measure $\mu_x$ on $\realnumbers$ such that
            \begin{align*}
              \varphi(x^n)=\int_\realnumbers t^n\,d\mu_x(t)
            \end{align*}
            for all $n\in \naturalnumbers$. Actually, $\mu_x$ is compactly supported with $\mu_x([-M,M])=1$.
          \item\label{facts:analytic-distribution-3} For any compactly supported probability measure $\mu$ on $\realnumbers$ there exists a $\ast$-probability space $(\mathcal A,\varphi)$ and a self-adjoint random variable $x\in \mathcal A$ such that $\mu_x=\mu$.\par
            Indeed, we can choose  $\mathcal A\equalperdefinition\complexnumbers\langle x\rangle$, the polynomials in an indeterminate $x$, (which becomes a $\ast$-algebra via $x^\ast\equalperdefinition x$) and define, for all $p(x)\in \complexnumbers\langle x\rangle$,
            \begin{align*}
              \varphi(p(x))\equalperdefinition \int_\realnumbers p(t)\,d\mu(t).
            \end{align*}
            The positivity of $\mu$ then renders $\varphi$ positive.
          \end{enumerate}
        \end{facts}

        \begin{remark}
          \begin{enumerate}
          \item The proofs of Facts~\hyperref[facts:analytic-distribution-1-1]{\ref*{facts:analytic-distribution}~\ref*{facts:analytic-distribution-1}~\ref*{facts:analytic-distribution-1-1}} and~\hyperref[facts:analytic-distribution-1-2]{\ref*{facts:analytic-distribution-1-2}} are trivial. One proves Fact~\hyperref[facts:analytic-distribution-1-3]{\ref*{facts:analytic-distribution}~\ref*{facts:analytic-distribution-1}~\ref*{facts:analytic-distribution-1-3}} via the Stone-Weierstra\ss\ Theorem, where one should also note that the existence and exponential boundedness of the moments of $\nu$ imply that $\nu$ has compact support.
          \item Rough idea of the proof of Fact~\hyperref[facts:analytic-distribution-2]{\ref*{facts:analytic-distribution}~\ref*{facts:analytic-distribution-2}}: Define on the $\ast$-algebra $\complexnumbers \langle y\rangle$ (with $y^\ast\equalperdefinition y$) an inner product by linear extension of the map $\langle\,\cdot\, , \, \cdot \, \rangle$ determined by
            \begin{align*}
              \langle y^n,y^m\rangle=\varphi(x^{n+m}) 
            \end{align*}
            for all $m,n\in \naturalnumbers_0$.
            Dividing by the kernel of this inner product and subsequent completion yields a Hilbert space. The indeterminate $y$ acts on this Hilbert space as a multiplication operator. In general, the operator $y$ is unbounded. It is, however, always symmetric and has a self-adjoint extension $\tilde y$ with
            \begin{align*}
              \langle 1,\tilde{y}^n1\rangle=\langle 1,y^n1\rangle=\varphi(x^n)
            \end{align*}
            for all $n\in \naturalnumbers$. By the spectral theorem, this extends to the spectral measure $\tilde{\mu}$ of $\tilde{y}$, i.e.,
            \begin{align*}
              \langle 1,f(\tilde{y})1\rangle=\int_\realnumbers f(t)\,d\tilde{\mu}(t)
            \end{align*}
            for all bounded measurable functions $f:\,\realnumbers\to\complexnumbers$.
            Since the moments of $\tilde{y}$ agree with those of $x$, they are exponentially bounded, which is why  $\tilde{\mu}$ has compact support and $\tilde{y}$ is bounded, thus also unique. It follows $\tilde{\mu}=\mu_x$.
          \item In Fact~\hyperref[facts:analytic-distribution-3]{\ref*{facts:analytic-distribution}~\ref*{facts:analytic-distribution-3}}, we can also choose a $\ast$-algebra $\mathcal A$ with more analytic structure, like the continuous functions $C(\support(\mu))$ on the support of $\mu$, a $C^\ast$-algebra, or the $\mu$-essentially bounded functions $L^\infty(\mu)$ on the support of $\mu$, a von Neumann algebra.
          \end{enumerate}
        \end{remark}

        \begin{proposition}
          \label{proposition:moments-cumulants-exponentially-bounded}
          Let $(\mathcal A,\varphi)$ be a non-commutative probability space with free cumulants $(\kappa_n)_{n\in \naturalnumbers}$. For $a\in \mathcal A$ we denote the moments and  cumulants of $a$, respectively, by
          \begin{align*}
           m_n^a\equalperdefinition \varphi(a^n)\quad\text{and}\quad \kappa_n^a\equalperdefinition \kappa_n(a,a,\ldots,a) 
          \end{align*}
          for all $n\in \naturalnumbers$. Then, the following statements are equivalent:
          \begin{enumerate}
          \item\label{proposition:moments-cumulants-exponentially-bounded-1} The sequence $(m_n^a)_{n\in \naturalnumbers}$ is exponentially bounded.
          \item\label{proposition:moments-cumulants-exponentially-bounded-2} The sequence $(\kappa_n^a)_{n\in \naturalnumbers}$ is exponentially bounded.            
          \end{enumerate}
        \end{proposition}
        \begin{proof}
          \ref{proposition:moments-cumulants-exponentially-bounded-2} $\Rightarrow$ \ref{proposition:moments-cumulants-exponentially-bounded-1}: Let $M\in \realnumbers^+$ and assume $|\kappa_n^a|\leq M^n$ for all  $n\in \naturalnumbers$. Then, for every $n\in \naturalnumbers$,
          \begin{align*}
            |m_n^a|&\leq \sum_{\pi\in \setofnoncrossingpartitionsof(n)}\underset{\displaystyle \leq M^n}{\underbrace{\prod_{V\in \pi}\underset{\displaystyle\leq M^{\# V}}{\underbrace{|\kappa_{\# V}^a|}}}}
                          \leq M^n\cdot \underset{\displaystyle = C_n \leq  4^n}{\underbrace{\#\setofnoncrossingpartitionsof(n)}}
            \leq (4M)^n,
          \end{align*}
          showing that the moments are exponentially bounded.
          \par
          \ref{proposition:moments-cumulants-exponentially-bounded-1} $\Rightarrow$ \ref{proposition:moments-cumulants-exponentially-bounded-2}: Conversely, suppose $M\in \realnumbers^+$ and $|m_n^a|\leq M^n$ for all $n\in \naturalnumbers$. By Assignment~\hyperref[assignment-7]{7}, Exercise~2 it holds that
           $|\mu(\pi,1_n)|\leq 4^n$
          for all $n\in \naturalnumbers$ and $\pi\in\setofnoncrossingpartitionsof(n)$. We conclude, for every $n\in \naturalnumbers$, using $\#\setofnoncrossingpartitionsof(n)\leq 4^n$ once more,
          \begin{align*}
            |\kappa_n^a|&\leq \sum_{\pi\in \setofnoncrossingpartitionsof(n)}\underset{\displaystyle \leq M^n}{\underbrace{\prod_{V\in \pi}|m_{\# V}^a|}}\cdot \underset{\displaystyle \leq 4^n }{\underbrace{|\mu(\pi,1_n)|}}
                          \leq 4^n\cdot 4^n\cdot M^n
            \leq (16M)^n,
          \end{align*}
          which proves the cumulants to be exponentially bounded.
        \end{proof}

        \begin{theorem}
          \label{theorem:free-convolution}
          Let $\mu$ and $\nu$ be two compactly supported probability measures on $\realnumbers$. Then, there exists a $\ast$-probability space $(\mathcal A,\varphi)$ and self-adjoint variables $x,y\in \mathcal A$ such that
          \begin{itemize}
          \item the analytic distributions $\mu_x$ of $x$ and $\mu_y$ of $y$ with respect to $(\mathcal A,\varphi)$ satisfy
            \begin{align*}
            \mu_x =\mu\quad\text{and}\quad\mu_y=\nu,
            \end{align*}            
          \item the random variables $x$ and $y$ are free in $(\mathcal A,\varphi)$.
          \end{itemize}
          The moments $(\varphi((x+y)^n))_{n\in \naturalnumbers}$ of the self-adjoint random variable $x+y\in \mathcal A$ are then exponentially bounded and hence determine uniquely a compactly supported probability measure $\mu_{x+y}$ on $\realnumbers$. The moments of $x+y$ (and thus $\mu_{x+y}$) depend only on $\mu$ and $\nu$ and the fact that $x$ and $y$ are free, and not on the concrete realizations of $x$ and $y$.
        \end{theorem}

        \begin{definition}
          The probability measure $\mu_{x+y}$ from Theorem~\ref{theorem:free-convolution} is called the \emph{free  convolution} of $\mu$ and $\nu$ and is denoted  by $\mu\boxplus \nu$.
        \end{definition}

        \begin{proof}[Proof of Theorem~\ref{theorem:free-convolution}]
          We can realize $\mu$ as $\mu_x$ for the random variable $x$ in a $\ast$-probability space $(\complexnumbers\langle x\rangle,\varphi_x )$ and, likewise, $\nu$ as $\mu_y$ for $y$ in $(\complexnumbers\langle y\rangle,\varphi_y)$. Then we take the free product of these $\ast$-probability spaces (see Assignment~\hyperref[assignment-6]{6}, Exercise~4) and thus realize $x$ and $y$ in the $\ast$-probability space $(\complexnumbers\langle x,y\rangle,\varphi_x\ast\varphi_y)$ in such a way that $x$ and $y$ are free with respect to $\varphi_x\ast\varphi_y$. Let $(\kappa_n)_{n\in \naturalnumbers}$ be the free cumulants of  $(\complexnumbers\langle x,y\rangle,\varphi_x\ast\varphi_y)$ and define
          \begin{align*}
             \kappa^x_n\equalperdefinition \kappa_n(x,\ldots,x),\quad \kappa^y_n\equalperdefinition \kappa_n(y,\ldots,y)\quad\text{and}\quad \kappa_n^{x+y}\equalperdefinition \kappa_n(x+y,\ldots,x+y)
          \end{align*}
          for all $n\in \naturalnumbers$.
          \par
          Since, by Theorem~\ref{theorem:freeness-vanishing-mixed-cumulants-random-variables}, mixed cumulants of free variables vanish, the cumulants of $x+y$ satisfy, for all $n\in \naturalnumbers$,
          \begin{align*}
            \kappa_n^{x+y}= \kappa_n(x+y,\ldots,x+y)
                          =\kappa_n(x,\ldots,x)+\kappa_n(y,\ldots,y)=\kappa_n^x+\kappa_n^y,
          \end{align*}
          and thus are exponentially bounded. Hence, according to Proposition~\ref{proposition:moments-cumulants-exponentially-bounded}, the moments $((\varphi_x\ast\varphi_y)((x+y)^n))_{n\in \naturalnumbers}$ are exponentially bounded and are thus the moments of a compactly supported probability measure $\mu_{x+y}$ by Fact~\hyperref[facts:analytic-distribution-2]{\ref*{facts:analytic-distribution}~\ref*{facts:analytic-distribution-2}}.
          \par
          Because the moments of $x+y$ are determined by the moments of $x$, the moments of $y$ and the freeness condition, it is clear that $\mu\boxplus \nu=\mu_{x+y}$ depends only on $\mu_x=\mu$ and $\mu_\nu=\nu$, but not on the concrete realizations of $x$ and $y$.
        \end{proof}

        \begin{remark}
          \begin{enumerate}
          \item Let $\mu,\nu$ and $\mu_1,\mu_2,\mu_3$ be compactly supported probability measures on $\realnumbers$. It easy to see that the binary operation $\boxplus$ has the following properties:
            \begin{enumerate}
            \item Commutativity: $\mu\boxplus \nu=\nu\boxplus \mu$.
            \item Associativity: $\mu_1\boxplus (\mu_2\boxplus \mu_3)=(\mu_1\boxplus \mu_2)\boxplus \mu_3$.
            \item Neutral element: $\delta_0\boxplus \mu=\mu$.
            \item Translations: Actually, for any $s\in \realnumbers$, free convolution with $\delta_s$  has the effect of shifting the measure by the amount $s$: $\delta_s\boxplus \mu=\mu_{(s)}$,
              where $\mu_{(s)}(B)\equalperdefinition \mu(\{b-s\mid b\in B\})$ for all Borel sets $B\subseteq \realnumbers$. This is because $\delta_s$ is the distribution $\mu_{s\cdot 1}$ of the constant random variable $s\cdot 1$, and the latter is \enquote{free from anything} by Proposition~\ref{proposition:constants-are-free-from-anything}.
            \end{enumerate}
          \item In order to make concrete calculations of $\mu\boxplus \nu$ for  probability measures $\mu$ and $\nu$ on $\realnumbers$, we need an analytic description of the free convolution. In particular, we will encode the information about moments and cumulants in analytic functions which will then also allow us to extend the definition of $\boxplus$ to general, not necessarily compactly supported, probability measures.
            
          \end{enumerate}
        \end{remark}

        \begin{definition}
          For any probability measure $\mu$ on $\realnumbers$ we define its \emph{Cauchy transform} $G_\mu$ by
          \begin{align*}
            G_\mu(z)\equalperdefinition \int_\realnumbers \frac{1}{z-t}\, d\mu(t) \quad\quad\quad
            \begin{tikzpicture}[baseline=1em]
              \draw[->] (-1em,0) -- (5em,0);
              \draw[->] (0,-1em) -- (0,4em);
              \coordinate (t) at (2em,0) {};
              \coordinate (z) at (1em,2em) {};
              \draw (z) -- node[pos=0.5, xshift=1.5em] {$z-t$} (t);
              \draw (t) -- +(0,-0.25em) -- +(0,0.25em);
              \node [below =0.5em of t] {$t$};
              \node [above =0.5em of z] {$z$};
              \node[circle,scale=0.25, fill=black, draw=black] at (z) {};
%              \draw (z) -- +(0.25em,0.25em) (z) -- +(-0.25em,0.25em) (z) -- +(-0.25em,-0.25em) (z) -- +(0.25em,-0.25em);
            \end{tikzpicture}
          \end{align*}
          for all $z\in \complexnumbers^+\equalperdefinition\{z\in \complexnumbers\mid \imaginarypart(z)>0\}$.\par
          Often, (in particular in a random matrix context) a variant $S_\mu$ of $G_\mu$, called \emph{Stieltjes transform}, is studied, which is defined by
          \begin{align*}
            S_\mu(z)=\int_\realnumbers \frac{1}{t-z}\, d\mu(t)=-G_\mu(z)
          \end{align*}
          for all $z\in \complexnumbers^+$.
        \end{definition}

        \begin{theorem}
          \label{theorem:cauchy-transform}
          \begin{enumerate}
          \item\label{theorem:cauchy-transform-1} Let $G\equalperdefinition G_\mu$ be the Cauchy transform of a probability measure $\mu$ on $\realnumbers$. Then,
            \begin{enumerate}
            \item\label{theorem:cauchy-transform-1-1} $G:\complexnumbers^+\to \complexnumbers^-$,  where $\complexnumbers^-\equalperdefinition \{z\in \complexnumbers\mid \imaginarypart(z)<0\}$,
            \item\label{theorem:cauchy-transform-1-2} $G$ is analytic on $\complexnumbers^+$,
              \item\label{theorem:cauchy-transform-1-3} 
we have
$$\lim_{y\to\infty} iyG(iy)=1\qquad \text{and}\qquad \sup_{y>0,x\in \realnumbers}y|G(x+iy)|=1.$$
            \end{enumerate}
          \item\label{theorem:cauchy-transform-2} Any probability measure $\mu$ on $\realnumbers$ can be recovered from the Cauchy transform $G_\mu$ via the \emph{Stieltjes inversion formula}: For all $a,b\in \realnumbers$ with $a<b$,
            \begin{align*}
              -\lim_{\varepsilon \downarrow 0}\frac{1}{\pi}\int_a^b \imaginarypart( G_\mu(x+i\varepsilon))\, dx=\mu((a,b))+\frac{1}{2}\mu(\{a,b\}).
            \end{align*}
Equivalently, with $\lambda$ denoting the Lebesgue measure on $\realnumbers$,
            \begin{align*}
              \begin{tikzpicture}[baseline=1em]
                              \draw (-1em,0) -- (4em,0);
              \draw (0,-1em) -- (0,3em);
              \coordinate (x) at (2em,0) {};
              \coordinate (xe) at (2em,2em) {};
              \draw (x) -- +(0,-0.25em) -- +(0,0.25em);
              \draw [->,shorten >=0.5em] (xe) -- (x);              
              \node [below =0.5em of x] {$x$};
              \node [fill=white,draw,circle,scale=0.25, label=right:{$x+i\varepsilon$}] at (xe) {}; 
              \end{tikzpicture}\qquad
              \underset{\displaystyle\text{density of probability measure}}{\underbrace{-\frac{1}{\pi}\imaginarypart (G(x+i\varepsilon)) }\lambda}\,\hspace{-2.2em}\overset{\varepsilon\downarrow 0}{\longrightarrow} \mu\text{ weakly.}
            \end{align*}
            In particular, for any probability measure $\nu$ on $\realnumbers$ with Cauchy transform $G_\nu$ it follows that $\mu=\nu$ whenever $G_\mu=G_\nu$.
          \item\label{theorem:cauchy-transform-3} Let $G:\complexnumbers^+\to \complexnumbers^-$ be an analytic function which satisfies
            \begin{align*}
              \limsup_{y\to\infty} y|G(iy)|=1.
            \end{align*}
            Then, there exists a unique probability measure $\mu$ on $\realnumbers$ such that $G=G_\mu$.
          \end{enumerate}
        \end{theorem}
        \begin{proof}
          \begin{enumerate}[wide]
          \item
            \begin{enumerate}[wide]
            \item The first claim is immediately clear from the definition.
            \item For all  $w,z\in \complexnumbers^+$, by considering the difference
              \begin{align*}
                G(w)-G(z)&=\int\underset{\displaystyle=\frac{z-w}{(w-t)(z-t)}}{\underbrace{\left(\frac{1}{w-t}-\frac{1}{z-t}\right)}}\, d\mu(t)
              \end{align*}
             it follows for the difference quotient
              \begin{align*}
                 \frac{G(w)-G(z)}{w-z}&=-\int \frac{1}{(w-t)(z-t)}\,d\mu(t)
              \end{align*}
              and thus, by passing to the limit $w\to z$,
              \begin{align*}
                G'(z)&=-\int_\realnumbers\frac{1}{(z-t)^2}\,d\mu(t),
              \end{align*}
              which shows that $G$ is analytic.
            \item The third claim can be equivalently expressed as 
              \begin{align*}
                \begin{aligned}
                  \lim_{y\to\infty}y \imaginarypart (G(iy))= -1\quad \text{and}\quad \lim_{y\to\infty} y\realpart (G(iy))=0.
                  \end{aligned}
              \end{align*}
              We only prove the statement about the imaginary part of $G$. The proof for the real part is similar. For every $y\in \realnumbers$ with $y\neq 0$ it holds
              \begin{align*}
                y \imaginarypart (G(iy))&= \int_\realnumbers y\hspace{-4.85em} \underset{\hspace{5em}
                                        \begin{aligned}
                                          &=\frac{1}{2i}\left(\frac{1}{iy-t}-\frac{1}{-iy-t}\right)\\
&=\frac{2iy}{-2i(y^2+t^2)}         
                                        \end{aligned}
                                        }{\underbrace{\imaginarypart \left(\frac{1}{iy-t}\right)}}\hspace{-5em}\, d\mu(t)\\
                                      &=-\int_\realnumbers \frac{y^2}{y^2+t^2}\, d\mu(t)\\
                                      &=-\int_\realnumbers\underset{\displaystyle \leq 1}{\underbrace{\frac{1}{1+\left(\frac{t}{y}\right)^2}}} \, d\mu(t).
              \end{align*}
              Because, $\lim_{y\to\infty}1/(1+(t/y)^2)=1$ for all $t\in\realnumbers$,
             Lebesgue's dominated convergence theorem (with majorant $t\mapsto 1$) yields $\lim_{y\to \infty}y \imaginarypart (G(iy))=-1$ as claimed.
            \end{enumerate}
          \item For all $\varepsilon>0$ and $x\in \realnumbers$,
            \begin{align*}
                  \imaginarypart(G(x+i\varepsilon))=\int_\realnumbers \imaginarypart \left(\frac{1}{x-t+i\varepsilon}\right)\, d\mu(t)
    =-\int_\realnumbers \frac{y}{(x-t)^2+\varepsilon^2}\, d\mu(t)
            \end{align*}
  and thus, for all $a,b\in \realnumbers$ with $a<b$,
  \begin{align*}
    \int_a^b \imaginarypart (G(x+i\varepsilon))\, dx&=-\int_\realnumbers \int_a^b \frac{\varepsilon}{(x-t)^2+\varepsilon^2}\, dx\,d\mu(t)\\
    &=-\int_\realnumbers \hspace{-5em}\underset{\hspace{5em}
      \begin{aligned}[t] &=\tan^{-1}\left(\frac{b-t}{\varepsilon}\right)-\tan^{-1}\left(\frac{a-t}{\varepsilon}\right)\\
        &\overset{\varepsilon\downarrow 0}{\longrightarrow}
        \begin{cases}
          0,& t\notin [a,b]\\
          \frac{\pi}{2},&t\in \{a,b\}\\
          \pi, & t\in (a,b)
        \end{cases}
      \end{aligned}}{\underbrace{\int_{(a-t)/\varepsilon}^{(b-t)/\varepsilon} \frac{1}{1+x^2}\, dx}} \hspace{-5em}d\mu(t)\\
    &\overset{\varepsilon\downarrow 0}{\longrightarrow} -\pi\bigl[\mu((a,b))+\frac{1}{2}\mu(\{a,b\})\bigr],
  \end{align*}
  which proves one part of the claim.
  \begin{center}
    \begin{tikzpicture}[baseline=3cm]
          \begin{axis}[xmin=-5, ymin=1.15*(-3.14/2), xmax=5, ymax=1.15*(3.14/2), domain=-5:6, samples=500,  axis lines=middle, ytick={-1.57,-0.78,0,0.78,1.57}, yticklabels={$-\frac{\pi}{2}$,$-\frac{\pi}{4}$, $0$, $\frac{\pi}{4}$, $\frac{\pi}{2}$}, title={$\realnumbers \to (-\frac{\pi}{2},\frac{\pi}{2}),\,x\mapsto \tan^{-1}(x)$}]
            \addplot[thick] {rad(atan(x))};
          \end{axis}
          \end{tikzpicture}
        \end{center}
        \par
Now, let $\nu$ be an arbitrary probability measure on $\realnumbers$ with Cauchy transform $G_\nu$ and assume $G_\mu=G_\nu$. By what was just shown, $\mu((a,b))=\nu((a,b))$
for all $a,b\in \realnumbers$ with $a<b$ such that $a,b$ are atoms of neither $\mu$ nor $\nu$.
\par
Since $\mu$ and $\nu$ can each have only countably many atoms, we can write any open interval, $a,b\in \realnumbers$, $a<b$, as a union
\begin{align*}
  (a,b)=\bigcup_{n=1}^\infty (\underset{\text{not  atoms}}{\underbrace{a+\varepsilon_n},\underbrace{b-\varepsilon_n}})
\end{align*}
for some monotonic sequence $(\varepsilon_n)_{n\in \naturalnumbers}$ in $\realnumbers^+$ with $\lim_{n\to\infty} \varepsilon_n=0$ such that, for every $n\in \naturalnumbers$, both $a+\varepsilon_n$ and $b-\varepsilon_n$ are not atoms of $\mu$ nor $\nu$.
\par
By monotone convergence of measures, it follows for all $a,b\in\realnumbers$ with $a<b$:
\begin{align*}
  \mu((a,b))=\lim_{n\to\infty}\mu((a+\varepsilon_n,b-\varepsilon_n))
=\lim_{n\to\infty}\nu((a+\varepsilon_n,b-\varepsilon_n))
&=\nu((a,b)).
\end{align*}
That proves the other half of the claim.
\item The third claim follows from (non-trivial) results of \emph{Nevanlinna} about analytic functions $\varphi:\, \complexnumbers^+\to \complexnumbers^+$. For any such $\varphi$ there exist a unique finite Borel measure $\sigma$ on $\realnumbers$ and unique $\alpha,\beta\in \realnumbers$ with $\beta\geq 0$ such that
  \begin{align*}
    \varphi(z)=\alpha+\beta z+\int_\realnumbers \frac{1+tz}{t-z}\, d\sigma(z)
  \end{align*}
  for any $z\in \complexnumbers^+$.\qedhere
          \end{enumerate}
        \end{proof}

        \begin{proposition}
          \label{proposition:cauchy-transform-series-expansion}
          Let $\mu$ be a probability measure on $\realnumbers$ and $G_\mu$ its Cauchy transform.
          If $\mu$ is compactly supported, say $\mu([-r,r])=1$ for some $r\in \realnumbers$, $r>0$, then $G_\mu$ has a power series expansion (about $\infty$) as follows:
          \begin{align*}
            G_\mu(z)=\sum_{n=0}^\infty \frac{m_n}{z^{n+1}} \quad \text{for all }z\in \complexnumbers^+ \text{ with }|z|>r,
          \end{align*}
          where, for every $n\in\naturalnumbers$, the number $m_n=\int_\realnumbers t^n\, d\mu(t)$  is the $n$-th moment of $\mu$.
        \end{proposition}
        \begin{proof}
          For every $z\in \complexnumbers^+$ with $|z|>r$ we can expand for all $t\in [-r,r]$
          \begin{align*}
            \frac{1}{z-t}=\frac{1}{z\left(1-\frac{t}{z}\right)}=\frac{1}{z}\sum_{n=0}^\infty \Bigl(\frac{t}{z}\Bigr)^n.
          \end{align*}
          Since this series convergences uniformly in $t\in [-r,r]$,
          \begin{align*}
            G_\mu(z)&=\int_{-r}^r\frac{1}{z-t}\,d\mu(t)
            =\sum_{n=0}^\infty \int_{-r}^r \frac{t^n}{z^{n+1}}\,d\mu(t)
            =\sum_{n=0}^\infty \frac{m_n}{z^{n+1}}
          \end{align*}
          for all $z\in \complexnumbers^+$ with $|z|>r$.
        \end{proof}
        
        \begin{remark}
          \label{remark:r-transform-formal-power-series}
          Proposition~\ref{proposition:cauchy-transform-series-expansion} shows that the Cauchy transform $G(z)$ is a version of our moment series
            $M(z)=1+\sum_{n=1}^\infty m_n z^ n$
          from Theorem~\ref{theorem:moment-cumulant-series}, namely
          \begin{align*}
            G(z)={\textstyle\frac{1}{z}} M({\textstyle\frac{1}{z}}).
          \end{align*}
          The relation with the cumulant series
            $C(z)=1+\sum_{n=1}^\infty \kappa_n z^n$
          from Theorem~\ref{theorem:moment-cumulant-series}, which was
            $C(z M(z))=M(z)$,
          implies then
                    \begin{align*}
C(\underset{\displaystyle =G(z)}{\underbrace{{\textstyle\frac{1}{z}} M({\textstyle\frac{1}{z}})}})=M({\textstyle\frac{1}{z}})=z\cdot G(z)
          \end{align*}
          and thus
            ${C(G(z))}/{G(z)}=z$.
          Define the formal Laurent series
            $K(z)\equalperdefinition {C(z)}/{z}$.
          Then, $K(G(z))=z$, hence also $G(K(z))=z$. Since $K(z)$ has a pole $\frac{1}{z}$ we split this off and write
          \begin{align*}
            K(z)=\frac{1}{z}+R(z),
\qquad\text{where}\qquad
            R(z)\equalperdefinition \sum_{n=1}^\infty \kappa_nz^{n-1}.
          \end{align*}
          If $\mu$ has compact support, then the cumulants are exponentially bounded and the above power series $R$ converges for $|z|$ sufficiently small.
          \par
          However, it is at the moment not clear for us, whether
          \begin{itemize}
          \item in the compactly supported case $R$ can be extended to an analytic function on $\complexnumbers^+$ (not true!),
          \item in the case of general $\mu$ the formal power series $R$ makes any sense as an analytic function.
          \end{itemize}
        \end{remark}
        
        \begin{theorem}[Voiculescu 1986, Voiculescu and Bercovici 1992]
          \label{theorem:R-transform-compact-support}
For compactly supported probability measures one has the following analytic properties of the Cauchy and the $R$-transform.
          \begin{enumerate}
          \item\label{theorem:R-transform-compact-support-1} Let $\mu$ be a probability measure on $\realnumbers$  with compact support, contained in an interval $[-r,r]$ for some $r>0$. Consider its Cauchy transform $G\equalperdefinition G_\mu$ as an analytic function
            \begin{align*}
              G:\, \underset{\displaystyle =\colon U}{\underbrace{\{ z\in \complexnumbers\mid |z|>4r \}}} \to \{ z\in \complexnumbers\mid |z|<{\textstyle\frac{1}{3r}}\}.
            \end{align*}
            Then, $G$ is injective on $U$ and $G(U)$ contains $V\equalperdefinition \{ z\in \complexnumbers \mid |z|<\frac{1}{6r}\}$.
            \begin{center}
              \begin{tikzpicture}
                \draw[draw=none, pattern=north east lines, pattern color=lightgray] (-4em,-4em) rectangle (4em,4em);
                \draw[fill=white] (2em,0) arc(0:360:2em);
                \draw (0,0) -- node[pos=0.5,below] {$4r$} ++(25:2em);
                \node at (3em,-3em) {$U$};
                \draw[->] (5em,0) -- node[pos=0.5,above] {$G$} (8em,0);
%------------------------------
                \draw (19em,0) arc(360:0:5em);
                \draw[pattern=horizontal lines,pattern color=lightgray] (9.5em,-0.5em) to[out=270,in=130] (10em,-1.8em) to[out=310,in=170] (12.9em,-4.1em) to[out=350,in=210] (14.7em,-3.6em) to[out=30,in=255] (17.4em,-3em) to[out=75,in=290] (16.5em,0.2em) to[out=110,in=345em] (16.8em,2.6em) to[out=165,in=35] (14.8em, 4.3em) to[out=215,in=60] (13.4em,2.6em) to[out=240,in=20] (12.2em, 1.7em) to[out=200,in=90] (9.5em,-0.5em);
                \draw[ pattern=vertical lines, pattern color=lightgray] (15.5em,0) arc(360:0:1.5em);
                %----------------------------
                \draw (14em,0) -- node[pos=0.85,right] {$\frac{1}{6r}$} ++(30:1.5em);
                \draw (14em,0) -- node[pos=0.65, left] {$\frac{1}{3r}$} ++(305:5em);
                \node at (12.5em,-2.5em) {$G(U)$};
                \node at (13.5em,0.5em) {$V$};
              \end{tikzpicture}
            \end{center}
            Hence, $G$ has an inverse on $V$,
              $K \equalperdefinition G^{<-1>}:\, V\to U$.
            The function $K$ has a pole at $0$ and is of the form
            \begin{align*}
              K(z)={\textstyle\frac{1}{z}}+R(z),
            \end{align*}
            for some analytic function $R:\, V\to \complexnumbers$.
            \par
           We have 
            \begin{align*}
              G(R(z)+{\textstyle{\frac{1}{z}}})=z \quad \text{if } |z|<{\textstyle\frac{1}{6r}}
           \qquad
            \text{and} \qquad
              R(G(z))+{\textstyle\frac{1}{G(z)}}=z \quad \text{if } |z|>7r.
            \end{align*}
          \item\label{theorem:R-transform-compact-support-2} The function $R$ has on $V$ the power series expansion
            \begin{align*}
              R(z)=\sum_{n=1}^\infty \kappa_n z^{n-1} \qquad \text{if } |z|<{\textstyle\frac{1}{6r}},
            \end{align*}
            where $(\kappa_n)_{n\in \naturalnumbers}$ are the cumulants corresponding to the moments $(m_n)_{n\in \naturalnumbers}$ of $\mu$. If we want to indicate the dependence on $\mu$, we will also write $R_\mu$ for $R$.
          \item\label{theorem:R-transform-compact-support-3} Let $\mu,\nu$ be compactly supported probability measures on $\realnumbers$. Then, for all $z\in \complexnumbers$,
            \begin{align*}
              R_{\mu\boxplus \nu}(z)=R_\mu(z)+R_{\nu}(z) \qquad \text{if }|z| \text{ is sufficiently small}.
            \end{align*}
          \end{enumerate}
        \end{theorem}
        \begin{definition}
          For a compactly supported probability measure $\mu$ on $\realnumbers$, the analytic function $R_\mu$ from Theorem~\ref{theorem:R-transform-compact-support} is called the \emph{$R$-transform} of $\mu$.
        \end{definition}
        \begin{proof}
          \begin{enumerate}[wide]
          \item We begin by showing that $G$ is injective on $U$. For all $z\in U$, put $f(z)\equalperdefinition G({z}^{-1})$. Then, $f:\, U\to\complexnumbers$ has a power series expansion
            \begin{align*}
              f(z)=\sum_{n= 0}^\infty m_nz^{n+1} \qquad\text{for all }z\in \complexnumbers \text{ with } |z|<{\textstyle \frac{1}{r}}.
            \end{align*}
            Restricted to all $z\in \complexnumbers$ with $|z|<\frac{1}{4r}$, 
            \begin{align*}
              |f(z)|&\leq \sum_{n=0}^\infty\underset{\displaystyle <r^n}{\underbrace{|m_n|}}\hspace{-0.1em}\cdot\hspace{-1.1em}\overset{\displaystyle<\left(\frac{1}{4r}\right)^{n+1}}{\overbrace{|z|^{n+1}}}\hspace{-1em}<\frac{1}{4r}\cdot\hspace{-0.55em} \underset{\displaystyle=\frac{1}{1-\frac{1}{4}}=\frac{4}{3}}{\underbrace{\sum_{n=0}^\infty\left(\frac{1}{4}\right)^n}}\hspace{-0.55em}=\frac{1}{3r}.   
            \end{align*}
            Now, consider $z_1,z_2\in \complexnumbers$ with $z_1^{-1},z_2^{-1}\in V$, i.e.\ $|z_1|,|z_2|<\frac{1}{4r}$. If $z_1\neq z_2$, then, by the mean value theorem,
            \begin{align*}
              \left|\frac{f(z_1)-f(z_2)}{z_1-z_2}\right|&\geq \realpart\,\left(\frac{f(z_1)-f(z_2)}{z_1-z_2}\right)
                                                        =\int_0^1\realpart[f'(z_1+t(z_2-z_1))]\, dt.
            \end{align*}
            We want to obtain a lower bound for $\realpart(f')$. For all $z\in \complexnumbers$ with $|z|<\frac{1}{4r}$,
            \begin{align*}
              \realpart[f'(z)]=&\realpart\,\left[1+ \sum_{n= 1}^\infty (n+1)m_nz^{n}\right]\\
                              &\geq 1- \sum_{n= 1}^\infty (n+1)\hspace{-1.65em}\underset{\hspace{2em}\displaystyle <r^n\left(\frac{1}{4r}\right)^n}{\underbrace{m_n|z|^{n}}}%\\
                              \geq 2- \underset{\displaystyle =\frac{1}{(1-\frac{1}{4})^2}=\frac{16}{9}}{\underbrace{\sum_{n= 0}^\infty (n+1)\left(\frac{1}{4}\right)^n}}
%\\
              =\frac{2}{9}.
            \end{align*}
            By  combining this estimate with the inequality for the difference quotient, we find for all $z_1,z_2\in \complexnumbers$ with $|z_1|, |z_2|<\frac{1}{4r}$:
            \begin{align*}
              |f(z_1)-f(z_2)|\geq \frac{2}{9}|z_1-z_2|.
            \end{align*}
            It follows that, for such $z_1,z_2$, assuming $z_1\neq z_2$ entails $f(z_1)\neq f(z_2)$. That means $f$ is injective on $\{z\in \complexnumbers \mid |z|<\frac{1}{4r}\}$, and thus $G$ injective on $\{z\in \complexnumbers \mid |z|>4r\}$.  So, the first part of the claim has been established.
            \par
            Next, we verify the claims about the nature of $G^{<-1>}$.
            Let $w\in V$, i.e.\ $|w|<\frac{1}{6r}$. We want to show that there exists $z\in \complexnumbers$ with  $z^{-1}\in U$, i.e.\ $|z|<\frac{1}{4r}$, such that $f(z)=G(z^{-1})=w$. In other words, we want to prove that $z\mapsto f(z)-w$ has a zero in $U^{-1}=\{z\in \complexnumbers \mid |z|<\frac{1}{4r}\}$. We will compare this function via the argument principle to $z\mapsto z-w$ in $U^{-1}$. Of course, the latter has exactly one zero, $w$.
            \par
            Let $\Gamma\equalperdefinition \{ z\in \complexnumbers \mid |z|=\frac{1}{4r}\}=\partial U^{-1}$. Then, for every $z\in \Gamma$,\\
            \begin{minipage}{0.6\linewidth}
            \begin{align*}
              |f(z)-w-(z-w)|&=|f(z)-z|\\
                            &=\left|\sum_{n=1}^\infty m_n z^{n+1}\right|\\
                            &\leq \sum_{n=1}^\infty r^n\left(\frac{1}{4r}\right)^{n+1}\\
                            &=r\cdot \left(\frac{1}{4r}\right)^2\cdot \underset{\displaystyle=\frac{4}{3}}{\underbrace{\sum_{k=0}^\infty \left(\frac{r}{4r}\right)^k}}\\
                            &=\frac{1}{12r}\\
              &<|z-w|. 
            \end{align*}  
          \end{minipage}%
          \noindent
            \begin{minipage}{0.4\linewidth}\hspace{0em}
            \begin{tikzpicture}
              \draw (2em,0em) arc (0:360:2em);
              \draw (4em,0em) arc (0:360:4em);
              \draw (0em,0em) -- node[pos=1.15] {$\frac{1}{4r}$} ++(315:4em);
              \draw (0em,0em) -- node[pos=1.25] {$\frac{1}{6r}$} ++(225:2em);
              \draw[->, gray] (0em,0em) -- ++ (0em,5em);
              \draw[->, gray] (0em,0em) -- ++ (5em,0em);
              \draw[gray] (0em,0em) -- ++ (0em,-5em);
              \draw[gray] (0em,0em) -- ++ (-5em,0em);
              \node [circle,fill=black,draw=black,scale=0.15] (w) at (1em,1em) {};
              \node [circle,fill=black,draw=black,scale=0.15] (z) at ($(0em,0em)+(35:4em)$) {};
              \node[below =0 of w] {$w$};
              \node[right =0em of z] {$z$};              
              \node at ($(0em,0em)+(60:4.5em)$) {$\Gamma$};
              %\draw [brown] (-5em, -5em) rectangle (5em,5em);
              \useasboundingbox (-5em, -5em) rectangle (5em,5em);              
            \end{tikzpicture}  
          \end{minipage}
          \\
          Hence, 
          \begin{align*}
            |(f(z)-w)-(z-w)|<|z-w| \quad\text{for all } z\in \Gamma=\partial U^{-1}.
          \end{align*}
          Rouch\'e's Theorem therefore implies
          \begin{align*}
            \#\{\text{zeros of }z\mapsto f(z)-w \text{ in }U^{-1}  \}=\#\{\text{zeros of }z\mapsto z-w \text{ in }U^{-1}  \}
            =1.
          \end{align*}
          Thus, we have shown that there indeed exists $z\in U^{-1}$ with $f(z)=G(z^{-1})=w$.
          In consequence, $f$ has an analytic inverse, $f^{<-1>}:\, V\to U^{-1}$. And thus
            $K\equalperdefinition {1}/{f^{<-1>}}$
          gives the inverse for $G$.
          \par
          Since $f^{<-1>}$ has a simple zero at $0$, the function $K$ has simple pole at $0$. Since there are no other zeros of $f^{<-1>}$, there is an analytic function $R$ such that  $K(z)=\frac{1}{z}+R(z)$ for all $z\in V$. That is another part of the claim verified.
          \par
          It only remains to prove the assertions about the relationship between $G$ and $R$.
          By construction, for all $z\in \complexnumbers$ with  $|z|<\frac{1}{6r}$,
          \begin{align*}
            z=f(f^{<-1>}(z))=G\left({\textstyle\frac{1}{f^{<-1>}(z)}}\right)=G(K(z))=G({\textstyle \frac{1}{z}}+R(z)).
          \end{align*}
          For $z\in \complexnumbers$ with $|z|>7r$ we have to show that $|G(z)|<\frac {1}{6r}$. Then, by construction it follows that $K(G(z))=z$ for all $z\in \complexnumbers$ with $|z|>7r$. Equivalently, we have to show that, for all $z\in \complexnumbers$,
          \begin{align*}
            |f(z)|\overset{!}{<}\frac{1}{6r} \qquad\text{if } |z|<\frac{1}{7r}.
          \end{align*}
          So, let $z$ be such. Then,
          \begin{align*}
            |f(z)|&=\left|\sum_{n=0}^\infty m_n z^{n+1}\right|
                  <\sum_{n=0}^\infty r^n\left(\frac{1}{7r}\right)^{n+1}
                  =\frac{1}{7r}\cdot \sum_{n=0}^\infty \left(\frac{1}{7}\right)^n
                  =\frac{1}{7r}\cdot \frac{1}{1-\frac{1}{7}}
                  =\frac{1}{7r}\cdot\frac{7}{6}
   =\frac{1}{6r}.
          \end{align*}
          That completes the proof of Claim~\ref{theorem:R-transform-compact-support-1}.
        \item Since we know by Remark~\ref{remark:r-transform-formal-power-series}  that for $|z|$  sufficiently small, the function $\tilde R(z):=\sum_{n=1}^\infty \kappa_nz^{n-1}$ solves the equation $G(\tilde R(z)+\frac{1}{z})=z$, it must agree for small $\vert z\vert$ with $R$. Since $z\mapsto R(z)$ is analytic for $|z|<\frac{1}{6r}$, the power series expansion $\tilde{R}$ of $R$ must converge there.
        \item Let $\mu$ and $\nu$ be compactly supported probability measures on $\realnumbers$. Let $(\kappa_n^\mu)_{n\in \naturalnumbers}$, $(\kappa_n^\nu)_{n\in \naturalnumbers}$ and $(\kappa_n^{\mu\boxplus \nu})_{n\in \naturalnumbers}$ be the cumulants corresponding to the moments of $\mu$, $\nu$ and $\mu\boxplus \nu$, respectively. Then, by Part~\ref{theorem:R-transform-compact-support-2}, for all $z\in \complexnumbers$ with $|z|$ sufficiently small,
          \begin{align*}
            R_\mu(z)=\sum_{n=1}^\infty \kappa_n^\mu z^{n-1},\quad
            R_\nu(z)=\sum_{n=1}^\infty \kappa_n^\nu z^{n-1},\quad\text{and}\quad
            R_{\mu\boxplus\nu}(z)=\sum_{n=1}^\infty \kappa_n^{\mu\boxplus \nu} z^{n-1}.
          \end{align*}
          As in the proof of Theorem~\ref{theorem:free-convolution}, if we realize $\mu$ as $\mu_x$ and $\nu$ as $\mu_y$ for random variables $x,y$ in some $\ast$-probability space such that $x$ and $y$ are free there, then $\mu\boxplus \nu$ is given by $\mu_{x+y}$. So, by the vanishing of mixed cumulants in free variables, Theorem~\ref{theorem:freeness-vanishing-mixed-cumulants-random-variables}, we have for all  $n\in\naturalnumbers$:
          \begin{align*}
            \kappa_n^{\mu\boxplus \nu}&=\kappa_n(x+y,\ldots, x+y)
                                      =\kappa_n(x,\ldots,x)+\kappa_n(y,\ldots,y)
            =\kappa_n^\mu+\kappa_n^\nu.
          \end{align*}
          Hence, for all $z\in \complexnumbers$ with  $|z|$  sufficiently small, we have
            $R_{\mu\boxplus \nu}(z)=R_\mu(z)+R_\nu(z)$.
          That is what we needed to show.\qedhere
          \end{enumerate}
        \end{proof}

        \begin{example}
          \label{example:arcsine-distribution}
          We want to determine $\mu\boxplus \nu$ for the special case of 
          \begin{align*}
            \mu= \nu =\frac{1}{2}\left(\delta_{-1}+\delta_{1}\right).
          \end{align*}
          For all $z\in \complexnumbers^+$, the Cauchy transform $G_\mu$ is given by
          \begin{align*}
            G_\mu(z)&=\int\frac{1}{z-t}\, d\mu(t)
                    =\frac{1}{2} \left(\frac{1}{z+1}+ \frac{1}{z-1}\right)
            =\frac{z}{z^2-1}.
          \end{align*}
By Theorem~\hyperref[theorem:R-transform-compact-support-1]{\ref*{theorem:R-transform-compact-support}~\ref*{theorem:R-transform-compact-support-1}} the inverse $K_\mu$ of $G_\mu$ and the $R$-transform $R_\mu$ of $\mu$ exist on suitable domains and the same is true for the respective functions of $\mu\boxplus \nu$. In the following, let $z\in \complexnumbers$ always lie in the appropriate domain for the functions involved. Then, $K_\mu(z)$ is given as the solution of
          \begin{align*}
            z&=G_\mu(K_\mu(z))=\frac{K_\mu(z)}{K_\mu (z)^2-1}
\qquad\text{or, equivalently,}\qquad
            K_\mu(z)^2-\frac{K_\mu(z)}{z}=1.
          \end{align*}
          This has the two solutions
          \begin{align*}
            K_\mu(z)=\frac{1\pm\sqrt{1+4z^2}}{2z}.
          \end{align*}
          For the $R$-transform we hence find
          \begin{align*}
            R_\mu(z)&=K_\mu(z)-\frac{1}{z}=\frac{\pm \sqrt{1+4z^2}-1}{2z}.
          \end{align*}
          Since $R_\mu(0)=0$, we must have
          \begin{align*}
            R_\mu(z)=\frac{\sqrt{1+4z^2}-1}{2z}.
          \end{align*}
          Now, put $R\equalperdefinition R_{\mu\boxplus \nu}$ and $G\equalperdefinition G_{\mu\boxplus \nu}$. Then,           
          \begin{align*}
            R(z)&=R_\mu(z)+R_\mu(z)=2R_\mu(z)=\frac{\sqrt{1+4z^2}-1}{z}.
          \end{align*}
          For the inverse $K\equalperdefinition K_{\mu\boxplus \nu}$ of $G$, it follows
          \begin{align*}
            K(z)=R(z)+\frac{1}{z}=\frac{\sqrt{1+4z^2}}{z}.
          \end{align*}
          And thus $G$ is determined by
          \begin{align*}
            z=K(G(z))=\frac{\sqrt{1+4G(z)^2}}{G(z)}
\qquad\text{or}\qquad
            z^2G(z)^2=1+4G(z)^2.
          \end{align*}
          Solving for $G$ shows
          \begin{align*}
            G(z)=\frac{1}{\sqrt{z^2-4}}.
          \end{align*}
          By Theorem~\ref{theorem:R-transform-compact-support} we know this form of $z\mapsto G(z)$ only for $|z|$ large. But, since $G$ is an analytic function on $\complexnumbers^+$, this explicit form can be extended by analytic continuation to all of $\complexnumbers^+$. Hence, we can also use it for $z\downarrow \realnumbers$ and thus get the form of $\mu\boxplus \mu$ by the Stieltjes inversion formula, \hyperref[theorem:cauchy-transform-2]{\ref*{theorem:cauchy-transform}~\ref*{theorem:cauchy-transform-2}}: For all $t\in \realnumbers$,
          \begin{align*}
            \frac{d(\mu\boxplus \mu)(t)}{dt}&=-\frac{1}{\pi}\lim_{\varepsilon\downarrow 0}\imaginarypart\,\bigl( \frac{1}{\sqrt{(t+i\varepsilon)^2-4}}\bigr)
                                            =-\frac{1}{\pi}\imaginarypart \frac{1}{\sqrt{t^2-4}}
            =
              \begin{cases}
                \frac{1}{\pi}\frac{1}{\sqrt{4-t^2}},& |t|< 2\\
                0,&\text{otherwise.}
              \end{cases}
          \end{align*}
          That means  $\mu^{\boxplus 2}\equalperdefinition \mu\boxplus \mu$ is the \emph{arcsine distribution}.

          \begin{align*}
            \begin{tikzpicture}[baseline=4em]
                        \begin{axis}[xmin=-2.2, ymin=-0.025, xmax=2.2, ymax=0.6, domain=-1.1:1.1, axis lines=middle, xtick={-2,-1,0,1,2}, ytick={0.25,0.5}, yticklabels={$\frac{1}{4}$,$\frac{1}{2}$}, xticklabels={$-2$,$-1$,$0$,$1$,$2$},title={$\frac{1}{2}(\delta_{-1}+\delta_1)$}, width=14em]
                          \addplot[thick, samples=50, smooth] coordinates {(-1,0)(-1,0.5)};
                          \addplot[thick, samples=50, smooth] coordinates {(1,0)(1,0.5)};                          
          \end{axis}
            \end{tikzpicture}
            \boxplus
            \begin{tikzpicture}[baseline=4em]
                        \begin{axis}[xmin=-2.2, ymin=-0.025, xmax=2.2, ymax=0.6, domain=-1.1:1.1, axis lines=middle, xtick={-2,-1,0,1,2}, ytick={0.25,0.5}, yticklabels={$\frac{1}{4}$,$\frac{1}{2}$}, xticklabels={$-2$,$-1$,$0$,$1$,$2$},title={$\frac{1}{2}(\delta_{-1}+\delta_1)$}, width=14em]
                          \addplot[thick, samples=50, smooth] coordinates {(-1,0)(-1,0.5)};
                          \addplot[thick, samples=50, smooth] coordinates {(1,0)(1,0.5)};                          
          \end{axis}
            \end{tikzpicture}
        =
                \begin{tikzpicture}[baseline=4em]
          \begin{axis}[xmin=-2.2, ymin=-0.025, xmax=2.2, ymax=0.6, domain=-1.9999:1.9999, samples=500, ytick={0.25,0.5}, yticklabels={$\frac{1}{4}$,$\frac{1}{2}$}, axis lines=middle, title={$(-2,2) \to \realnumbers,\,x\mapsto \frac{1}{\pi\sqrt{4-t^2}}$},width=14em]  
            \addplot[thick] {(pi*((4-(x^2))^0.5))^-1};
            \addplot[dashed, samples=50, smooth] coordinates {(-2,0)(-2,0.6)};
                          \addplot[dashed, samples=50, smooth] coordinates {(2,0)(2,0.6)};           
          \end{axis}
          \end{tikzpicture}
          \end{align*}
        \end{example} 

        \begin{remark}
          \begin{enumerate}
          \item Example~\ref{example:arcsine-distribution} shows that free convolution behaves quite differently from classical convolution:
            \begin{enumerate}
            \item discrete $\boxplus$ discrete can be continuous -- in the classical case, the result is also discrete, e.g.,
              \begin{align*}
               \textstyle \frac{1}{2}(\delta_{-1}+\delta_1)\ast \frac{1}{2}(\delta_{-1}+\delta_1)=\frac{1}{4}\delta_{-2}+\frac{1}{2}\delta_{0}+\frac{1}{4}\delta_2.
              \end{align*}
            \item The classical case is easy to calculate since $\ast$ is distributive, e.g.,
              \begin{align*}
          \textstyle      {\textstyle\frac{1}{2}}(\delta_{-1}+\delta_{1})\ast {\textstyle\frac{1}{2}}(\delta_{-1}+\delta_{1})&=\textstyle\frac{1}{4}(\underset{\displaystyle =\delta_{-2}}{\underbrace{\delta_{-1}\ast\delta_{-1}}}+\underset{\displaystyle =\delta_0}{\underbrace{\delta_{-1}\ast\delta_1}}+\underset{\displaystyle =\delta_0}{\underbrace{\delta_1\ast\delta_{-1}}}+\underset{\displaystyle =\delta_2}{\underbrace{\delta_1\ast\delta_1}}).
              \end{align*}
              Since, for all $a,b\in \realnumbers$, it still holds that $\delta_a\boxplus \delta_b=\delta_{a+b}$, the above computation shows that $\boxplus$ is not distributive. It is a  non-linear operation. It is linear on the level of the $R$-transforms but the relation between $G$ and $R$ is non-linear.
            \end{enumerate}
          \item Let us generalize Example~\ref{example:arcsine-distribution} for $\mu\equalperdefinition \frac{1}{2}(\delta_{-1}+\delta_1)$ to the calculation of
            \begin{align*}
              \mu^{\boxplus n}=\underset{n\text{ times}}{\underbrace{\mu\boxplus\mu\boxplus\ldots\boxplus\mu}}
            \end{align*}
            for arbitrary $n\in \naturalnumbers$.
            Then, for $R\equalperdefinition R_{\mu^{\boxplus n}}$, $G\equalperdefinition G_{\mu^{\boxplus n}}$ and $z\in \complexnumbers$ with $|z|$ small enough:
            \begin{align*}
              R(z)=n\cdot R_\mu(z)=n\cdot\frac{\sqrt{1+4z^2}-1}{2z}.
            \end{align*}
            Writing $K\equalperdefinition K_{\mu^{\boxplus n}}$, it follows, for all $z\in \complexnumbers$ with $|z|$ sufficiently large, that we have
            \begin{align*}
              z=K(G(z))=\frac{\sqrt{1+4G(z)^2}}{2G(z)}\cdot n-\frac{n-2}{2G(z)},
            \end{align*}
            which has as solution
            \begin{align*}
              G(z)=\frac{n\cdot \sqrt{z^2-4(n-1)}-z(n-2)}{2(z^2-n^2)}.
            \end{align*}
            The density of $\mu^{\boxplus n}$, determined as
            \begin{align*}
              t\mapsto -\frac{1}{\pi}\imaginarypart\, \Bigl(\frac{n\sqrt{t^2-4(n-1)}}{2(t^2-n^2)}\Bigr)\, dt
            \end{align*}
            is, of course, positive whatever $n\in \naturalnumbers$ is. However, the same is true for all $n\in \realnumbers$ with $n\geq 1$, but not anymore for $n<1$. So, it seems that we can extend for our $\mu$ the convolution powers $\mu^{\boxplus n}$ also to real $n\geq 1$.
            This is not true in the classical case!
            In the free case, there is a probability measure $\nu\equalperdefinition \mu^{\boxplus \frac{3}{2}}$, i.e.\  such that
              $\nu^{\boxplus 2}=\mu^{\boxplus 3}$.
            However, classically, there is \emph{no} $\nu$ such that
            \begin{align*}
\nu^{\ast 2}=\mu^{\ast 3} ={\textstyle\frac{1}{8}}(\delta_{-3}+3\delta_{-1}+3\delta_1+\delta_3).
            \end{align*}
            Actually, the existence of such a convolution semigroup $(\mu^{\boxplus t})_{t\geq 1}$ is nothing special for our particular measure $\mu=\frac{1}{2}(\delta_{-1}+\delta_1)$, but is given for any compactly supported probability measure $\mu$ on $\realnumbers$.
          \end{enumerate}
        \end{remark}

        \begin{theorem}[{Nica} and {Speicher} 1996]
          \label{theorem:free-convolution-semigroup}
          Let $\mu$ be a compactly supported probability measure on $\realnumbers$. Then, there exists a semigroup $(\mu_t)_{t\geq 1}$ ($t\in\realnumbers$, $t\geq 1$) of compactly supported probability measures on $\realnumbers$ such that, for all $t\geq 1$ and $z\in \complexnumbers$,
          \begin{align*}
            R_{\mu_t}(z)=t\cdot R_\mu(z) \qquad\text{if } |z|\text{ sufficiently small}.
          \end{align*}
          In particular,
            $\mu_1=\mu$ and $\mu_{s+t}=\mu_s\boxplus \mu_t$ for all $s,t\geq 1$.
     Moreover, for all $n\in \naturalnumbers$, the maps  of the $n$-th moment $t\mapsto m_n^{\mu_t}$ and of the $n$-th cumulant $t\mapsto \kappa_n^{\mu_t}$ of $\mu$ are continuous in $t\geq 1$.
          We write $\mu_t=\colon \mu^{\boxplus t}$ for $t\geq 1$. 
        \end{theorem}
        \begin{proof}
          Let $t\geq 1$ be arbitrary. We will construct $\mu_t$ from a concrete realization. We have to find self-adjoint random variables $x_t$ in some $\ast$-probability space $(\mathcal A,\varphi)$ such that $\mu_{x_t}=\mu_t$, i.e., such that
          $ R_{\mu_{x_t}}=t\cdot R_{\mu_x}$.
          By Theorem~\hyperref[theorem:R-transform-compact-support-2]{\ref*{theorem:R-transform-compact-support}~\ref*{theorem:R-transform-compact-support-2}}, the latter is equivalent to
          \begin{align*}
            \kappa_n(x_t,x_t,\ldots,x_t)=t\cdot \kappa_n(x,x,\ldots,x)
          \end{align*}
          holding for all $n\in \naturalnumbers$; there $x=x_1$ has distribution $\mu_x=\mu$.
          \par
          We claim that we can get $x_t$ by compressing $x$ by a free projection $p$ with trace $\frac{1}{t}$. Start from a random variable $x=x^\ast$ in some $\ast$-probability space $(\mathcal A,\varphi)$ such that $\mu_x=\mu$, which we can find by Fact~\hyperref[facts:analytic-distribution-3]{\ref*{facts:analytic-distribution}~\ref*{facts:analytic-distribution-3}}. Then, by a free product construction, we can assume that there exists a projection $p\in \mathcal A$ (meaning $p^\ast =p=p^2$) such that $x$ and $p$ are free in $(\mathcal A,\varphi)$  and such that, for all $n\in \naturalnumbers$,
            $\varphi(p^n)=\varphi(p)={\textstyle\frac{1}{t}}$.
          Since the free product preserves traciality (see Assignment~\hyperref[assignment-6]{6}, Exercise~4), we can assume that $\varphi$ is a trace (as it is a trace restricted to the $*$-algebras $\algebrageneratedby(1,x)$ and $\algebrageneratedby(1,p))$.
          \par
          Now, we consider $(\mathcal A_t,\varphi_t)$, where
          \begin{align*}
            \mathcal A_t\equalperdefinition \{pap \mid a\in \mathcal A\}
\subseteq \mathcal A
\quad\text{and}\quad
            \varphi_t\equalperdefinition \frac{1}{\varphi(p)}\left.\varphi\right|_{\mathcal A_t},
          \end{align*}
          meaning especially, for all $a\in \mathcal A$,
          \begin{align*}
            \varphi_t(pap)=\frac{\varphi(pap)}{\varphi(p)}.
          \end{align*}
          Then, $(\mathcal A_t,\varphi_t)$ is a $\ast$-probability space with unit $1_{\mathcal A_t}=p$.
          \par
          Consider now the self-adjoint random variable
            $x_t\equalperdefinition p(tx)p$
          in $\mathcal A_t$. For all $n\in \naturalnumbers$ we have (by using $p^2=p$ and traciality in the last step)
          \begin{align*}
            \varphi_t(x^n_t)=t \varphi(x_t^n)
                            =t^{n+1} \varphi((pxp)^n)
                            =t^{n+1}\varphi(xpxp\ldots xp),
          \end{align*}
which we can then calculate 
by the moment-cumulant formula (Remark~\ref{remark:moment-cumulant-formula}) as
  \begin{align*}
t^{n+1}\varphi(xpxp\ldots xp)%\\
                            =t^{n+1}\sum_{
                              \underset{\textstyle\iff \sigma\leq K(\pi)}{\textstyle\underbrace{\substack{\pi\in \setofnoncrossingpartitionsof(1,3,\ldots,2n-1)\\
            \sigma\in \setofnoncrossingpartitionsof(2,4,\ldots,2n)\\
            \pi\cup\sigma\in \setofnoncrossingpartitionsof(2n)}}}
            } \kappa_\pi(x,x,\ldots,x)\kappa_\sigma(p,p,\ldots,p)
          \end{align*}

        For $\pi\in \setofnoncrossingpartitionsof(1,3,\ldots,2n-1)$ and $\sigma\in \setofnoncrossingpartitionsof(2,4,\ldots,2n)$ the condition $\pi\cup \sigma\in \setofnoncrossingpartitionsof(2n)$ is equivalent to $\sigma\leq K(\pi)$, where $K(\pi)$, defined as the join of all such $\sigma$, is the \emph{Kreweras complement} of $\pi$. By Assignment~\hyperref[assignment-7]{7}, Exercise~2(4), $\#\pi+\#K(\pi)=n+1$ in this case. Hence we can continue, by again using the moment-cumulant formula, 
          \begin{align*}
            \varphi_t(x^n_t)&=t^{n+1}\sum_{\pi\in \setofnoncrossingpartitionsof(n)}\kappa_\pi^x \cdot\hspace{-2.5em} \underset{
              \begin{aligned}
                &=\varphi_{K(\pi)}(p,p,\ldots,p)\\
                &=\left({\textstyle\frac{1}{t}}\right)^{\# K(\pi)}
              \end{aligned}
                  }{\underbrace{\sum_{\sigma\leq K(\pi)}\kappa_\sigma ^p}}%\\
                            &=\sum_{\pi\in \setofnoncrossingpartitionsof(n)}\kappa_\pi^x \cdot \underset{\displaystyle =t^{\# \pi}}{\underbrace{t^{n+1-\# K(\pi)}}}%\\
                              &=\sum_{\pi\in \setofnoncrossingpartitionsof(n)}(t\kappa^x)_\pi.
          \end{align*}
          Thus, if  we denote by $\kappa_n^{x_t}$ the $n$-th cumulant of $x_t$ with respect to $(\mathcal A_t,\varphi_t)$ (and \emph{not} with respect to $(\mathcal A,\varphi)$), then we have just proved $\kappa_n^{x_t}=t\kappa_n^x$. By what we observed initially, that proves the existence claim. The continuity is clear from the construction.
        \end{proof}

\newpage

\section{Free Convolution of Arbitrary Probability Measures and the Subordination Function}
%[Free Convolution: Subordination Function]
\begin{remark}
  \label{remark:subordination-a}
 We now want to extend our theory of free convolution to arbitrary probability measures.
  \begin{enumerate}
    \item%\label{remark:subordination-1}
In order to define $\mu\boxplus \nu$ for such  $\mu$ and $\nu$ there are two approaches:
  \begin{enumerate}
  \item\label{remark:subordination-1-1} Realizing $\mu$ and $\nu$ as the distributions of unbounded operators $x$ and $y$ on Hilbert spaces (\enquote{affiliated} to nice von Neumann algebras) and defining $\mu\boxplus \nu$ as the distribution of the unbounded operator $x+y$. 
    This was done by {Bercovici} and {Voiculescu} in 1993. However,
    one has to deal with technical issues of unbounded operators.
  \item\label{remark:subordination-1-2} Trying to extend the approach from the last chapter, via defining the $R$-transforms $R_\mu$ and $R_\nu$ of $\mu$ and $\nu$ implicitly by the equations
    \begin{align*}
      G_\mu[R_\mu(z)+{\textstyle\frac{1}{z}}]=z \qquad\text{and}\qquad G_\nu[R_\nu(z)+{\textstyle\frac{1}{z}}]=z
    \end{align*}
    from Theorem~\hyperref[theorem:R-transform-compact-support-2]{\ref*{theorem:R-transform-compact-support}~\ref*{theorem:R-transform-compact-support-2}}
    (note that $G_\mu$ and $G_\nu$ are well-defined even if $\mu$ and $\nu$ do not have compact support)  and then defining $\mu\boxplus \nu$ implicitly via $G_{\mu\boxplus \nu}$ (and Theorem~\hyperref[theorem:cauchy-transform-2]{\ref*{theorem:cauchy-transform}~\ref*{theorem:cauchy-transform-2}}), which in turn is defined implicitly via $R_{\mu \boxplus \nu}$, by the equations
    \begin{align*}
      R_{\mu\boxplus \nu}(z)=R_\mu(z)+R_\nu(z)
\qquad
\text{and}\qquad   
      G_{\mu\boxplus \nu}[R_{\mu\boxplus \nu}(z)+{\textstyle\frac{1}{z}}]=z.
    \end{align*}    
    This was done  about 2005 independently by {Chistyakov and G\"otze} and by {Belinschi and Bercovici}.
    However, instead of the $R$-transform one describes the theory in terms of subordination functions $\omega_1,\omega_2$, which have better analytic properties.
    \end{enumerate}
  \item\label{remark:subordination-2} Let us first rewrite the $R$-transform description into a subordination form on a formal level, treating the functions as formal power or Laurent series in the indeterminate $z$.
    We define
    \begin{align*}
      \omega_1(z)\equalperdefinition z-R_\nu(G_{\mu\boxplus \nu}(z)),\qquad
      \omega_2(z)\equalperdefinition z-R_\mu(G_{\mu\boxplus \nu}(z)).
    \end{align*}
    Then we have
    \begin{align*}
      G_\mu(\omega_1(z))&=G_\mu[\displaystyle z-\hspace{-8.5em}\underset{\hspace{8.5em}\displaystyle =  \underset{\hspace{0em}\displaystyle =z-\frac{1}{G_{\mu\boxplus \nu}(z)}}{\displaystyle \underbrace{R_{\mu\boxplus \nu}(G_{\mu\boxplus \nu}(z))}}\displaystyle-R_\mu(G_{\mu\boxplus \nu}(z))}{\displaystyle\underbrace{R_\nu(G_{\mu\boxplus \nu}(z))}}\hspace{-8.5em}]\\
                        &=G_\mu[R_\mu(G_{\mu\boxplus \nu}(z))+{\textstyle\frac{1}{G_{\mu\boxplus \nu}(z)}}]\\
      &=G_{\mu\boxplus \nu}(z).
    \end{align*}
    Hence, $G_{\mu\boxplus \nu}$ is subordinated to $G_\mu$ via the subordination function $\omega_1$:
      $G_{\mu\boxplus \nu}(z)=G_\mu(\omega_1(z))$,
    and, in the same way, to $G_\nu$ via $\omega_2$:
      $G_{\mu\boxplus \nu}(z)=G_\nu(\omega_2(z))$.

    The series $\omega_1$ and $\omega_2$ have been defined here via $R_\nu$ and $R_\mu$, respectively. So, this does not seem to give anything new. However, we can reformulate the defining equations for $\omega_1$ and $\omega_2$ in a form not invoking the $R$-transforms. Namely, 
    \begin{align*}
      \omega_1(z)+\omega_2(z)&=z-R_\nu(G_{\mu\boxplus \nu}(z))+z-R_\mu(G_{\mu\boxplus \nu}(z))%\\
      =z+\hspace{0em}\underset{\hspace{0em}\displaystyle =\frac{1}{G_{\mu\boxplus \nu}(z)}}{\underbrace{z-R_{\mu\boxplus \nu}(G_{\mu\boxplus \nu}(z))}}\hspace{0em}.
    \end{align*}
    Thus, by $G_{\mu\boxplus \nu}(z)=G_\mu(\omega_1(z))=G_\nu(\omega_2(z))$, we have
    \begin{align*}
      \omega_1(z)+\omega_2(z)-\frac{1}{G_\mu(\omega_1(z))}=z\qquad
      \text{and}\qquad\omega_1(z)+\omega_2(z)-\frac{1}{G_\nu(\omega_2(z))}=z.
    \end{align*}
    If we put
    \begin{align*}
      H_\mu(z)\equalperdefinition\frac{1}{G_\mu(z)}-z \qquad \text{and} \qquad H_\nu(z)\equalperdefinition\frac{1}{G_\nu(z)}-z,
    \end{align*}
    then we get
    \begin{align*}
      \omega_1(z)&=z+\frac{1}{G_\nu(\omega_2(z))}-\omega_2(z)=z+H_\nu(\omega_2(z))\\   \omega_2(z)&=z+\frac{1}{G_\mu(\omega_1(z))}-\omega_1(z)=z+H_\mu(\omega_1(z)).
    \end{align*}
    Thus, by inserting one into the other, we conclude
    \begin{align}
      \label{eq:subordination-formula}
      \omega_1(z)&=z+H_\nu[z+H_\mu(\omega_1(z))].
    \end{align}
    This is a fixed point equation for $\omega_1$. It only involves the variants $H_\nu$ or $H_\mu$ of $G_\nu$ and $G_\mu$, respectively, but \emph{not} the $R$-transforms $R_\mu$ nor $R_\nu$. 

It will turn out that Equation~\eqref{eq:subordination-formula} is much better behaved analytically than the equation $G[R(z)+\frac{1}{z}]$ for the $R$-transform.
    Let us check this for an example before we address the general theory.
  \end{enumerate}
\end{remark}

\begin{example}
  We reconsider Example~\ref{example:arcsine-distribution} from the subordination perspective, again in formal terms. The Cauchy transform of the measure
  \begin{align*}
    \mu=\nu=\frac{1}{2}(\delta_{-1}+\delta_1)
 \qquad \text{is given by} \qquad
    G_\mu(z)=G_\nu(z)=\frac{z}{z^2-1}.
  \end{align*}
  Hence
  \begin{align*}
    H_\mu(z)=\frac{1}{G_\mu(z)}-z=\frac{z^2-1}{z}-z=-\frac{1}{z}.
  \end{align*}
The above says that we can write $G\equalperdefinition G_{\mu\boxplus \mu}$ as
$ G(z)=G_\mu(\omega(z))$,
 where the subordination function $\omega_1=\omega_2=\omega$ is determined by Equation~\eqref{eq:subordination-formula}  as
  \begin{align*}
    \omega(z)=z-\frac{1}{z-\frac{1}{\omega(z)}},
\qquad\text{i.e.,}\qquad
    z-\omega(z)=\frac{1}{z-\frac{1}{\omega(z)}}.
  \end{align*}
  This identity implies
  \begin{align*}
    1= (z-\omega(z))(z-{\textstyle\frac{1}{\omega(z)}})=z^2-z(\omega(z)+{\textstyle\frac{1}{\omega(z)}})+1.
  \end{align*}
  Thus,
  \begin{align*}
    \omega(z)+\frac{1}{\omega(z)}=z,
  %\end{align*}
  \qquad\text{which has the solution}\qquad
  %\begin{align*}
    \omega(z)=\frac{z+\sqrt{z^2-4}}{2},
  \end{align*}
  where we have chosen the one which behaves like $\omega(z)\sim z$ for $|z|\uparrow \infty$.
  Thus,
  \begin{align*}
    G(z)&=G_\mu(\omega(z))
        =\frac{\omega(z)}{\omega(z)^2-1}
        =\frac{1}{\omega(z)-\frac{1}{\omega(z)}}
        =\frac{1}{2\omega(z)-z}
    =\frac{1}{\sqrt{z^2-4}},
  \end{align*}
  which agrees with our result from Examples~\ref{example:arcsine-distribution} (and gives the arcsine distribution for $\mu\boxplus \mu$).
  \par
  But the question remains: what is now the advantage of this subordination function $\omega(z)=\frac{1}{2}(z+\sqrt{z^2-4})$ over the $R$-transform  $R(z)=\frac{1}{z}(\sqrt{1+4z^2}-1)$ from \ref{example:arcsine-distribution}?
  \par
  Note that (in the compactly supported case) our formal calculations could be made rigorous on suitables domains in $\complexnumbers^+$, but $\omega$ and $R$ behave differently with respect to extension to all of $\complexnumbers^+$ or $\complexnumbers^-$:
  \begin{enumerate}
  \item The subordination function $\omega(z)=\frac{1}{2}(z+\sqrt{z^2-4})$, which is a priori only defined for large $z$, can be extended by this formula to $\complexnumbers^+$ by choosing branches for $\sqrt{z^2-4}=\sqrt{z-2}\cdot \sqrt{z+2}$ without cuts in $\complexnumbers^+$. Hence, $\omega:\complexnumbers^+\to \complexnumbers^+$ is a nice analytic object (like $G$).
  \item In contrast, the $R$-transform $R(z)=\frac{1}{z}(\sqrt{1+4z^2}-1)$, which is a priori only defined for small $z$, can by this formula not be extended to $\complexnumbers^+$ nor to $\complexnumbers^-$ since there exists no branch of $\sqrt{1+4z^2}$ in any neighborhood of $z=\frac{i}{2}$ or of $z=-\frac{i}{2}$.
  \end{enumerate}
\end{example}

\begin{remark}
  What we observe in this example is true in general:
  \begin{enumerate}
  \item We can still define $R$-transforms (not necessarily in balls anymore, but in wedge-like regions in $\complexnumbers^-$), but they cannot be analytically extended to all of $\complexnumbers^-$.
  \item The subordination functions, defined by fixed point equations, can always be extended to all of $\complexnumbers ^+$.
  \end{enumerate}
\end{remark}

\begin{definition}
  For $\alpha,\beta>0$ we define the \emph{truncated Stolz angle}
  \begin{align*}
    \Gamma_{\alpha,\beta}\equalperdefinition \{z=x+iy\in \complexnumbers ^+\mid \alpha y>|x|, y>\beta\}
  \end{align*}
  \begin{center}
    \begin{tikzpicture}
\path[pattern= north east lines, pattern color=lightgray] ($(0em,0em)+(130:9em)$) -- ($(0em,0em)+(130:4em)$) -- ($(0em,0em)+(50:4em)$) -- ($(0em,0em)+(50:9em)$) -- cycle;
     \draw[->] (-8em,0em) -- (8em,0em);
     \draw[->] (0em,-0.5em) -- (0em,8em);
     \draw[dashed] (0em,0em) -- ($(0em,0em)+ (50:4em)$);
     \draw[dashed] (0em,0em) -- ($(0em,0em)+ (130:4em)$);
     \draw[dashed] ($(0em,0em)+ (50:4em)$) -- ++(5em,0);     
     \draw[dashed] ($(0em,0em)+ (130:4em)$) -- ++(-5em,0);     
     \draw ($(0em,0em)+ (50:4em)$) -- node [pos=0.75,right] {$\alpha y=|x|$}  ++(50:5em);
     \draw ($(0em,0em)+ (130:4em)$) --++(130:5em);
     \draw ($(0em,0em)+ (50:4em)$) -- node [pos=1.4,left, below] {$y=\beta$} ($(0em,0em)+ (130:4em)$);
     \node at (2em,5em) {$\Gamma_{\alpha,\beta}$};
     \draw[dotted,->] (6em,0em) arc (0:50:6em);
     \node at (3.5em,1em) {$\tan^{-1}(\frac{1}{\alpha})$};
  \end{tikzpicture} 
  \end{center}
  and the \emph{wedge}
  \begin{align*}
    \Delta_{\alpha,\beta}\equalperdefinition \{z\in \complexnumbers^\times \mid z^{-1}\in \Gamma_{\alpha,\beta}\}. 
  \end{align*}
  \begin{center}
    \begin{tikzpicture}
    \draw[->] (-8em,0em) -- (8em,0em);
    \draw[->] (0em,0.5em) -- (0em,-8em);
    \draw[pattern=north east lines, pattern color =lightgray] (0em,0em) --++(230:7em) arc (230:310:7em) node[pos=0.2, below, yshift=-0.25em] {$|z|=\frac{1}{\beta}$} -- cycle;
    \node at (1.7em,-5.25em) {$\Delta_{\alpha,\beta}$};
     \draw[dotted,->] (-6em,0em) arc (180:230:6em);
     \node at (-3.5em,-1em) {$\tan^{-1}(\frac{1}{\alpha})$};    
  \end{tikzpicture} 
  \end{center}  
\end{definition}

\begin{theorem}
\label{theorem:r-transform-general-case}
  Let $\mu$ be a probability measure on $\realnumbers$ with Cauchy transform $G$ and put $F\equalperdefinition\frac{1}{G}$. For every $\alpha>0$ there exists $\beta>0$ such that
  \begin{align*}
    R(z)\equalperdefinition F^{<-1>}({\textstyle\frac{1}{z}})-{\textstyle\frac{1}{z}}
  \end{align*}
  is defined for $z\in \Delta_{\alpha,\beta}$ and such that
  \begin{enumerate}
  \item\label{theorem:r-transform-general-case-1} we have for all $z\in \Delta_{\alpha,\beta}$   $$G[R(z)+\frac{1}{z}]=z;$$
  \item\label{theorem:r-transform-general-case-2} 
we have for all $z\in \Gamma_{\alpha,\beta}$ 
$$R[G(z)]+\frac{1}{G(z)}=z.$$
  \end{enumerate}
\end{theorem}
\begin{proof}
  The proof is similar to the case of compact support 
(Theorem~\ref{theorem:R-transform-compact-support}), again using Rouch\'e's Theorem to show the existence of $F^{<-1>}$ in a suitable domain.
\end{proof}
\begin{remark}
  \label{remark:subordination-problem}
  \begin{enumerate}
  \item\label{remark:subordination-problem-1} We will now investigate for fixed but arbitrary $z\in \complexnumbers^+$ the subordination fixed point equation
    \begin{align}
      \label{eq:remark-subordination-problem-1}
      \omega_1(z)=z+H_\nu[z+H_\mu(\omega_1(z))]
    \end{align}
    and want to see that this always has a unique solution $\omega_1(z)\in \complexnumbers^+$.
    Note that for given $z\in \complexnumbers^+$, Equation~\eqref{eq:remark-subordination-problem-1} is the fixed point problem for the $z$-dependent mapping
    \begin{align}
      \label{eq:remark-subordination-problem-2}
      w\mapsto z+H_\nu[z+H_\mu(w)].
    \end{align}
    The naive approach for solving a problem like \eqref{eq:remark-subordination-problem-1} would be to iterate the maps from \eqref{eq:remark-subordination-problem-2} and hope that this converges to the solution. Surprisingly, this works! That relies on the fact that the mappings in  \eqref{eq:remark-subordination-problem-2} are analytic self-mappings of $\complexnumbers^+$ for whose iteration strong results are available. The two main ideas are:
    \begin{enumerate}
    \item\label{remark:subordination-problem-1-1} In the domains where the $R$-transforms are defined, say for $z\in \Omega$, we already know that a fixed point of \eqref{eq:remark-subordination-problem-1} exists. An application of the Schwarz lemma shows then that the iterations \eqref{eq:remark-subordination-problem-2} converge to this fixed point.
      \item\label{remark:subordination-problem-1-2} Once convergence for $z\in \Omega$ has been established, Vitali's Theorem shows that the iterations  \eqref{eq:remark-subordination-problem-2} actually converge for all $z\in \complexnumbers^+$. The limits $\omega_1(z)$ are then the unique fixed points of \eqref{eq:remark-subordination-problem-1}.
      \end{enumerate}
      \item\label{remark:subordination-problem-2} For the use of the Schwarz Lemma in Part~\hyperref[remark:subordination-problem-1-1]{\ref*{remark:subordination-problem-1}~\ref*{remark:subordination-problem-1-1}} we should exclude that the map \eqref{eq:remark-subordination-problem-2} is an automorphism of $\complexnumbers^+$. (Actually, we should first check that it is a mapping from $\complexnumbers^+$ to $\complexnumbers^+$.) For this we have to exclude the case that $\mu$ or $\nu$ is of the form $\delta_a$ for some $a\in \realnumbers$. But the latter case can be treated directly.
  \end{enumerate}
\end{remark}

\begin{notation}
  Let $\mu$ be a probability measure on $\realnumbers$ with Cauchy transform $G_\mu$. We put for all $z\in \complexnumbers^+$
  \begin{align*}
    F_\mu(z)\equalperdefinition \frac{1}{G_\mu(z)}
\qquad
  \text{and}\qquad
    H_\mu(z)\equalperdefinition F_\mu(z)-z=\frac{1}{G_\mu(z)}-z.
  \end{align*}
  Since $G_\mu:\complexnumbers^+\to \complexnumbers^-$, we have $F_\mu: \complexnumbers^+\to \complexnumbers^+$.
\end{notation}

\begin{lemma}
  \label{lemma:subordination-functions}
The functions $F_\mu$ and $H_\mu$ have the following properties.
  \begin{enumerate}
  \item\label{lemma:subordination-functions-1} For all $z\in \complexnumbers^+$,
    \begin{align*}
      \imaginarypart(z)\leq \imaginarypart (F_\mu(z)),
    \end{align*}
    with equality holding somewhere only if $\mu$ is a Dirac measure.
  \item\label{lemma:subordination-functions-2} In other words: If $\mu$ is not a Dirac measure, then
    \begin{align*}
      H_\mu: \, \complexnumbers^+\to \complexnumbers^+.
    \end{align*}    
  \item If $\mu=\delta_a$ for some $a\in\realnumbers$, then $H_\mu(z)=-a$ for all $z\in \complexnumbers^+$.
  \end{enumerate}
\end{lemma}
\begin{proof}
  \begin{enumerate}[wide]
    \item
 Abbreviating $G\equalperdefinition G_\mu$ and $F\equalperdefinition F_\mu$, we have
   \begin{align*}
     \imaginarypart(F(z))=\imaginarypart\, \left(\frac{1}{G(z)}\right)%\\
                         &=\frac{-\imaginarypart (G(z))}{|G(z)|^2}\\
                         &=-\frac{1}{|G(z)|^2}\int_\realnumbers \underset{\displaystyle =\frac{-\imaginarypart (z)}{|z-t|^2}}{\underbrace{\imaginarypart\, \left(\frac{1}{z-t}\right)}}\, d\mu(t)%\\
     =\frac{\imaginarypart(z)}{|G(z)|^2}\int_\realnumbers \frac{1}{|z-t|^2}\, d\mu(t).
   \end{align*}
   So, we have to show that
   \begin{align*}
     |G(z)|^2\overset{!}{\leq} \int_\realnumbers \frac{1}{|z-t|^2}\, d\mu(t).
   \end{align*}
   This follows by Cauchy-Schwartz:
   \begin{align*}
     |G(z)|^2=\left|\int_\realnumbers \frac{1}{z-t}\, d\mu(t)\right|^2&\leq \underset{\displaystyle= 1}{\underbrace{\int_\realnumbers 1^2\, d\mu(t)}}\cdot \int_\realnumbers \left|\frac{1}{z-t}\right|^2\, d\mu(t)
     =\int_\realnumbers \frac{1}{|z-t|^2}\,d\mu(t).
   \end{align*}
   Equality holds if and only if the maps $t\mapsto 1$ and $t\mapsto \frac{1}{z-t}$ are linearly dependent in $L^2(\mu)$, i.e.\ if $t\mapsto \frac{1}{z-t}$ is $\mu$-almost constant, which is only possible if $\mu$ is a Dirac measure.
 \item If $\mu$ is not a Dirac measure, then by Part~\ref{lemma:subordination-functions-1}
   \begin{align*}
     \imaginarypart(H(z))=\underset{\displaystyle >\imaginarypart(z)}{\underbrace{\imaginarypart(F(z))}}-\imaginarypart (z)>0
   \end{align*}
   for all $z\in \complexnumbers^+$ and hence $H:\,\complexnumbers^+\to \complexnumbers^+$.
   \item Lastly, if $\mu=\delta_a$ for some $a\in \realnumbers$ then it follows that $G_\mu(z)=\frac{1}{z-a}$ for all $z\in \complexnumbers^+$, thus $F_\mu(z)=z-a$ and $H_\mu(z)=-a$, for all $z\in \complexnumbers^+$.\qedhere
   \end{enumerate}
 \end{proof}

 \begin{notation}
   Let $\mu$ and $\nu$ be two probability measures on $\realnumbers$. Then, for every $z\in \complexnumbers^+$, we consider the analytic function
   \begin{align*}
     g_z:\, \complexnumbers^+&\to \complexnumbers,\\
w&\mapsto z+H_\nu[z+H_\mu(w)].
   \end{align*}
 \end{notation}

 \begin{lemma}
   \label{lemma:subordination-no-automorphism}
   For all $z\in \complexnumbers^+$ we have that
     $g_z: \, \complexnumbers^+\to \complexnumbers^+$
   and $g_z$ is not an automorphism of $\complexnumbers^+$. Indeed, for all $w\in \complexnumbers^+$:
     $\imaginarypart(g_z(w)) \geq \imaginarypart(z)$.
 \end{lemma}
 \begin{proof}
   For all $z,w\in \complexnumbers^+$ we have
   \begin{align*}
     \imaginarypart(g_z(w))=\imaginarypart(z)+\overset{\displaystyle \geq 0}{\overbrace{\imaginarypart [H_\nu(\underset{\displaystyle \in \complexnumbers^+}{\underbrace{z+H_\mu(w)}})]}}\geq \imaginarypart (z)
   \end{align*}
   by Lemma~\hyperref[lemma:subordination-functions-2]{\ref*{lemma:subordination-functions}~\ref*{lemma:subordination-functions-2}}. \qedhere
 \end{proof}

 \begin{theorem}[Belinschi and  Bercovici 2007]
   \label{theorem:subordination-iteration}
   Let $\mu$ and $\nu$ be two probability measures on $\realnumbers$. Then, the function
   \begin{align*}
     g_z:\,\complexnumbers^+\to \complexnumbers^+, \,     w\mapsto g_z(w)
   \end{align*}
   has, for each $z\in \complexnumbers ^+$, a unique fixed point $\omega(z)\in \complexnumbers^+$,  given by
   \begin{align*}
     \omega(z)=\lim_{n\to\infty} \hspace{-4.4em}\underset{\substack{\displaystyle\uparrow\\\displaystyle n\text{-fold composition of }g_z}}{g^{\circ n}_z}\hspace{-4.4em}(z_0),\qquad\qquad
\text{for any $z_0\in \complexnumbers^+$.}
   \end{align*}
    The function
     $\omega:\, \complexnumbers^+\to \complexnumbers ^+, \, z\mapsto \omega(z)$
   is analytic.
 \end{theorem}
 \begin{proof}
   We choose an arbitrary initial point $u_0\in \complexnumbers^+$ for our iteration and put, for every $n\in \naturalnumbers$ and all $z\in \complexnumbers^+$,
     $f_n(z)\equalperdefinition g^{\circ n}_z(u_0)$.
   Then, for every $n\in \naturalnumbers$, the mapping $f_n:\, \complexnumbers^+\to \complexnumbers^+$ is analytic. We want to see that the sequence $(f_n)_{n\in \naturalnumbers}$ converges to a limit function $\omega$.
   \begin{enumerate}[wide]
   \item By the $R$-transform description, Theorem~\ref{theorem:r-transform-general-case} and Lemma~\ref{lemma:subordination-no-automorphism}, we know that there is an open set $\Omega\subseteq \complexnumbers^+$ such that, for every $z\in \Omega$, the function $g_z:\, \complexnumbers^+\to \complexnumbers^+$ has a fixed point. Then, the Schwarz Lemma implies that, for those $z\in \Omega$, the sequence $(g^{\circ n}_z(u_0))_{n\in \naturalnumbers}$ converges to this (necessarily unique) fixed point. (Via conformal transformations one can map $\complexnumbers^+$ to the disc $\mathbb D\equalperdefinition\{z\in \complexnumbers \mid |z|<1\}$ and assume there also that the transformed version $\tilde g_z$ of $g_z$ satisfies $\tilde g_z(0)=0$. Then, note that $\tilde{g}_z$ is not an automorphism of $\mathbb D$ and apply Schwarz Lemma to it.)
   \item Hence, the sequence $(f_n)_{n\in \naturalnumbers}$ converges on the set $\Omega$. Then, Vitali's Theorem (or Montel's Theorem) implies that $(f_n)_{n\in \naturalnumbers}$ converges to an analytic function on all of $\complexnumbers^+$. Thus, there exists an analytic  function $\omega:\,\complexnumbers^+\to \complexnumbers$ with
        $\lim_{n\to \infty}f_n(z)=\omega(z)$
     for all $z\in \complexnumbers^+$. This limit is not constant (since it is not constant on $\Omega$). Hence, $\omega(z)\in \complexnumbers^+$ for every $z\in \complexnumbers^+$ and thus $\omega:\,\complexnumbers^+\to \complexnumbers^+$.
     \par
     Furthermore, for all $z\in \complexnumbers^+$, we have
     \begin{align*}
       g_z(\omega(z))&=g_z\left(\lim_{n\to\infty}g^{\circ n}_z(u_0)\right)=\lim_{n\to\infty}g^{\circ (n+1)}_z(u_0)=\omega(z),
     \end{align*}
     making $\omega(z)$ our wanted fixed point of $g_z$. It is necessarily unique.\qedhere
   \end{enumerate}
 \end{proof}

 \begin{remark}
   Given probability measures $\mu$ and $\nu$ on $\realnumbers$, we determine $\omega_1\equalperdefinition \omega$ according to Theorem~\ref{theorem:subordination-iteration} and define
     $G(z)\equalperdefinition G_\mu(\omega_1(z))$
   for every $z\in \complexnumbers^+$. One shows then that $\omega_1$ satisfies
   \begin{align*}
     \lim_{y\to\infty}\frac{\omega_1(iy)}{iy}=1,
   \end{align*}
   which implies by Theorem~\hyperref[theorem:cauchy-transform-3]{\ref*{theorem:cauchy-transform}~\ref*{theorem:cauchy-transform-3}} that $G$ is the Cauchy transform of some measure, which we then denote by $\mu\boxplus \nu$.
   \par
   By Remark~\ref{remark:subordination-a}, this measure then has the property that
     $R_{\mu\boxplus \nu}(z)=R_\mu(z)+R_\nu(z)$
   for all $z\in \complexnumbers^+$ such that  $R_{\mu\boxplus \nu}(z)$, $R_\mu(z)$ and $R_\nu(z)$ are defined. Hence, it reduces to our earlier definition in the compactly supported case.
 \end{remark}
 
 \begin{example}
   \begin{enumerate}
   \item Consider $\mu\boxplus \delta_a$ for some $a\in \realnumbers$ and an arbitrary probability measure $\mu$ on $\realnumbers$. Then, by Lemma~\ref{lemma:subordination-functions}, we have for every $z\in \complexnumbers^+$ that
       $H_{\delta_a}(z)=-a$.
     Thus, $\omega_1\equalperdefinition \omega$ is, by Theorem~\ref{theorem:subordination-iteration}, determined by
     \begin{align*}
       \omega_1(z)&=g_z(\omega_1(z))=z+(-a)=z-a
     \end{align*}
     for all $z\in \complexnumbers^+$. Thus, for all $z\in \complexnumbers^+$, we have
     \begin{align*}
       G_{\mu\boxplus \delta_a}(z)=G_\mu(\omega_1(z))=G_\mu(z-a).
     \end{align*}
     Hence, $\mu\boxplus \delta_a$ is also in general a shift of $\mu$ by the amount $a$.
   \item Consider the Cauchy distribution, i.e.\ the probability measure $\nu$ on $\realnumbers$ with density
     \begin{align*}
       d\nu(t)=\frac{1}{\pi}\frac{1}{1+t^2}\, dt
     \end{align*}
     for all $t\in \realnumbers$. (It has no  second or higher order moments.)
     \begin{center}
                       \begin{tikzpicture}[baseline=4em]
          \begin{axis}[xmin=-4.2, ymin=-0.025, xmax=4.2, ymax=0.5, domain=-3.9999:3.9999, samples=50,  axis lines=middle, title={$\realnumbers \to \realnumbers^+,\,t\mapsto \frac{1}{\pi}\frac{1}{1+t^2}$},width=24em, height=16em]  
            \addplot[thick] {(pi^-1)*((1+(x^2))^-1)};
          \end{axis}
          \end{tikzpicture}
     \end{center}
     The value of the Cauchy transform $G_\nu$ of the Cauchy distribution is, for every $z\in \complexnumbers^+$, given by
     \begin{align*}
       G_\nu(z)=\frac{1}{\pi}\int_\realnumbers \frac{1}{z-t}\cdot \frac{1}{1+t^2}\,dt=\frac{1}{z+i}.
     \end{align*}
     Thus, for all $z\in \complexnumbers^+$,
       $H_\nu(z)=z+i-z=i$,
     and hence, for all $w\in \complexnumbers^+$.
       $g_z(w)=z+i$.
This implies, by Theorem~\ref{theorem:subordination-iteration},  
       $\omega(z)=z+i$
     for all $z\in \complexnumbers^+$. Thus, for any probability measure $\mu$ on $\realnumbers$ and all  $z\in \complexnumbers^+$,
       $G_{\mu\boxplus \nu}(z)=G_\mu(z+i)$.
     It  turns out that this is the same result as classical convolution $\mu\ast\nu$ of $\mu$ with the Cauchy distribution $\nu$. This means in particular that the free analogue of the Cauchy distribution \emph{is} the usual Cauchy distribution.
   \end{enumerate}
 \end{example}

\newpage

\section{Gaussian Random Matrices and Asymptotic Freeness}
 \begin{definition}
%[Gaussian Random Matrices]
   For every $N\in \naturalnumbers$, a $\ast$-probability space of \emph{random $N\times N$-ma\-tri\-ces} is given by
     $(M_N(L^{\infty-}(\Omega,\mathbb P)), \normalizedtrace\otimes \mathbb E)$,
%\underset{\displaystyle =\colon \varphi}{\underbrace{\normalizedtrace\otimes \mathbb E}}),
where $(\Omega,P)$ is a classical probability space, where
   \begin{align*}
     L^{\infty-}(\Omega,\mathbb P)\equalperdefinition \bigcap_{1\leq p<\infty} L^p(\Omega,P), 
   \end{align*}
   and for any complex algebra $\mathcal A$, 
     $M_N(\mathcal A)\cong M_N(\complexnumbers)\otimes \mathcal {A}$
   denotes the  $N\times N$-matrices with entries from $\mathcal A$. Furthermore, $\mathbb E$ denotes the expectation with respect to $\mathbb P$ and  $\normalizedtrace$  the normalized trace on $M_N(\complexnumbers)$. That means our random variables are of the form
   \begin{align*}
     A=(a_{i,j})_{i,j=1}^N \quad\text{with } a_{i,j}\in L^{\infty-}(\Omega,\mathbb P)\text{ for all }i,j\in [N],
   \end{align*}
   and our relevant state is given by
   \begin{align*}
   \varphi(A)=  (\normalizedtrace  \otimes \expectation)(A)=\expectation[\normalizedtrace(A)]=\frac{1}{N}\sum_{i=1}^N\expectation[a_{i,i}].
   \end{align*}
 \end{definition}
 \begin{remark}
   \label{remark:GUE}
   \begin{enumerate}
   \item\label{remark:GUE-1} Consider a self-adjoint matrix $A\in M_N(\complexnumbers)$, let $\lambda_1,\ldots,\lambda_N$ be the eigenvalues of $A$, counted with multiplicity. Then, we can diagonalize $A$ with a unitary $U\in M_N(\complexnumbers)$ as $A=UDU^\ast$, where
     \begin{align*}
       D=
       \begin{bmatrix}
         \lambda_1& 0 & \hdots & 0\\
         0 &\lambda_2 &\ddots &\vdots\\
         \vdots&\ddots & \ddots& 0\\
         0&\hdots &0 &\lambda_N
       \end{bmatrix}.
     \end{align*}
     Hence, the moments of $A$ with respect to $\normalizedtrace$ are given by
     \begin{align*}
       \normalizedtrace(A^m)=\normalizedtrace((UDU^\ast)^m)%\\
                            =\normalizedtrace(UD^mU^\ast)%\\
                            =\normalizedtrace(D^m)%\\
                            =\frac{1}{N}(\lambda_1^m+\ldots+\lambda_N^m).
     \end{align*}
The latter can be written as
$$ \normalizedtrace(A^m) =\int_\realnumbers t^m\, d\mu_A(t),
     \qquad \text{where}\qquad
       \mu_A\equalperdefinition \frac{1}{N}(\delta_{\lambda_1}+\ldots+\delta_{\lambda_N})
   $$
     is the \emph{eigenvalue distribution} of $A$. In the same way, the distribution of self-adjoint random matrices with respect to $\normalizedtrace\otimes \mathbb E$ is the \emph{averaged} eigenvalue distribution.
   \item\label{remark:GUE-2} We will now consider the \enquote{nicest} kind of self-adjoint random matrix, one where the entries are independent Gaussian random variables. In fact, we will study an entire sequence of such random matrices, one of dimension $N\times N$ for every $N\in \naturalnumbers$. The intent behind that is to form the limit $N\to\infty$ of the averaged eigenvalue distributions of these random matrices. 
     \par
     Recall (from Theorem~\ref{theorem:classical-central-limit-theorem}) that the moments of a  real-valued Gaussian (classical) random variable $x$ of variance $1$ (on some probability space $(\Omega,\mathcal F,\mathbb P)$ with corresponding expectation $\mathbb E$) are given by
     \begin{align*}
       \mathbb E[x^m]=\#\setofpartitionsof_2(m)=\sum_{\pi\in\setofpartitionsof_2(m)}1
     \end{align*}
     for every $m\in \naturalnumbers$. Consider now (classically!) independent real-valued Gaussian random variables $x_1,\ldots,x_r$, $r\in \naturalnumbers$. Then, the joint moments of $(x_1,\ldots,x_r)$ are, for all $m\in \naturalnumbers$ and $i:[m]\to[r]$, given by (where  $n_k\equalperdefinition \#\{s\in [m]\mid i(s)=k\}$ for all $k\in [r]$)
     \begin{align*}
       \mathbb E[x_{i(1)}x_{i(2)}\ldots x_{i(m)}]&=\mathbb E[x_1^{n_1}x_2^{n_2}\ldots x_r^{n_r}], \\
                                                 &=\mathbb E[x_1^{n_1}]\mathbb E[x_2^{n_2}]\ldots \mathbb E[x_r^{n_r}]\\
\quad\\
                                                 &=\sum_{\substack{\pi\in \setofpartitionsof_2(m)\\\pi\leq \ker(i)}}1\\
\quad\\
                                                 &=\sum_{\pi\in \setofpartitionsof_2(m)}\prod_{\substack{\{r,s\}\in \pi\\r<s}}\hspace{-3em}\underset{\hspace{3em}\displaystyle=
                                                   \begin{cases}
                                                     1,&\text{if }i(r)=i(s),\\
                                                     0,&\text{otherwise}
                                                   \end{cases}
                                                   }{\underbrace{\mathbb E[x_{i(r)}x_{i(s)}].}}
     \end{align*}
     The above formula
     \begin{align}
       \label{eq:wick-formula}
       \mathbb E[x_{i(1)}\ldots x_{i(m)}]=\sum_{\pi\in \setofpartitionsof_2(m)}\prod_{\substack{\{r,s\}\in \pi\\r<s}}\mathbb E[x_{i(r)}x_{i(s)}]
     \end{align}
     expresses arbitrary moments of $(x_1,\ldots,x_r)$ in terms of their second moments ($\hateq$ covariance matrix). Since it is linear in all its arguments, it remains true if we replace $x_1,\ldots, x_r$ by linear combinations of them. This yields then a \enquote{Gaussian family} and Equation~\eqref{eq:wick-formula} is called the \enquote{Wick formula} ({Wick} 1950, {Isserlis} 1918).
   \item\label{remark:GUE-3} We will take as entries for our Gaussian random matrices complex Gaussians
     \begin{align*}
       z=\frac{x+iy}{\sqrt{2}},\quad\quad \parbox{20em}{where $x$ and $y$ are independent real-valued Gaussian random variables of variance 1.}
     \end{align*}
     Then, it follows
$$
\expectation[z^2]=\expectation[\overline{z}^2]={\textstyle\frac{1}{2}}\left(\expectation[x^2]-\expectation[y^2]\right)=0,
\qquad
       \expectation [ z \overline{z}]=\expectation[\overline{ z}z]={\textstyle\frac{1}{2}}\left(\expectation[x^2]+\expectation[y^2]\right)=1.
$$
   \item\label{remark:GUE-4} We will also need a scaling adjusted to the matrix dimension $N\in \naturalnumbers$, to have a non-trivial limit of the averaged eigenvalue distributions for $N\to \infty$. Note that, since our random matrix $A_N=(a_{i,j})_{i,j=1}^N$ will be self-adjoint, we have
     \begin{align*}
       \normalizedtrace(A_N^2)=\frac{1}{N}\sum_{i,j=1}^N\mathbb E[a_{i,j}a_{j,i}]=\frac{1}{N}\sum_{i,j=1}^N\hspace{-3em}\underset{\hspace{4em}
\displaystyle{\text{should be}\sim}\frac{1}{N}}{\underbrace{\mathbb E[a_{i,j}\overline{a_{i,j}}]}}.
       \end{align*}
   \end{enumerate}
   
 \end{remark}

 \begin{definition}
   \label{definition:GUE}
   For $N\in \naturalnumbers$, a \emph{(self-adjoint!) Gaussian random $N\times N$-matrix} is a random $N\times N$-matrix $A=(a_{i,j})_{i,j=1}^N$ such that $A=A^\ast$ and such that the entries $(a_{i,j})_{i,j=1}^N$ form a complex Gaussian family with covariance
   \begin{align}
     \label{eq:gaussian-family-covariance}
     \mathbb E[a_{i,j}a_{k,l}]=\frac{1}{N}\delta_{i,l}\delta_{j,k}
   \end{align}
   for all $i,j,k,l\in [N]$. We say then that $A$ is \emph{\GUEN}. (\GUE\ stands for `` Gaussian unitary ensemble''; the distribution of such matrices is invariant under conjugation with unitary matrices.)
 \end{definition}

 \begin{remark}
   \begin{enumerate}
   \item Let $N\in \naturalnumbers$ and  $A=(a_{i,j})_{i,j=1}^N$ \GUEN. Note that $a_{i,j}=\overline{a_{j,i}}$ for all $i,j\in [N]$ since $A=A^\ast$. (In particular, $a_{i,i}$ is real for every $i\in [N]$). Thus, all $\ast$-moments of second order in the $(a_{i,j})_{i,j=1}^N$ are given by Equation~\eqref{eq:gaussian-family-covariance}. And thus all $\ast$-moments of $A$ are determined via the Wick formula \eqref{eq:wick-formula}.
     \item One can also study versions with real entries (\textsc{goe} $\hateq$ Gaussian orthogonal ensemble) or with quaternionic entries (\textsc{gse} $\hateq$ Gaussian symplectic ensemble).
   \end{enumerate}
 \end{remark}

 \begin{theorem} [Genus expansion for \GUE]
   \label{theorem:gue}
   Let $N\in \naturalnumbers$ and let $A$ be \GUEN. Then, for all $m\in \naturalnumbers$,
   \begin{align*}
     (\normalizedtrace\otimes \expectation)(A^m)=\sum_{\pi\in \setofpartitionsof_2(m)}N^{\#(\gamma\pi)-1-\frac{m}{2}},
   \end{align*}
   where we identify every $\pi\in \setofpartitionsof_2(m)$ with an element of the permutation group $S_m$ by declaring the blocks orbits, i.e.\ by defining $\pi(r)\equalperdefinition s$ and $\pi(s)\equalperdefinition r$ for every $\{r,s\}\in \pi$; and where $\gamma$ is the long cycle
     $\gamma= (1 \ 2\ 3\ \ldots \ m)\in S_m$,
   i.e.\ $\gamma(k)\equalperdefinition k+1$ for all $k\in [m-1]$ and $\gamma(m)\equalperdefinition 1$; and where we denote by $\#\sigma$ the number of orbits of $\sigma\in S_m$.
 \end{theorem}
 \begin{proof}
   For every $m\in \naturalnumbers$, applying the Wick formula \eqref{eq:wick-formula} yields
   \begin{align*}
     (\normalizedtrace\otimes \expectation)(A^m)&=\mathbb E[\normalizedtrace(A^m)]\\
                 &=\frac{1}{N}\sum_{i:[m]\to [N]}\hspace{-5.5em}\underset{\hspace{5.5em}\displaystyle =\sum_{\pi\in \setofpartitionsof_2(m)}\prod_{\substack{\{r,s\}\in \pi\\r<s}}\hspace{-2em}\underset{\hspace{2em}\displaystyle =\frac{1}{N}\cdot \delta_{i(r),\hspace{-0.5em}\underset{\scriptstyle=i(\gamma\pi(r))}{\textstyle\underbrace{\scriptstyle i(\gamma(s))}}}\hspace{-0.5em}\cdot\hspace{-0em} \delta_{i(s),\hspace{-0.5em}\underset{\scriptstyle=i(\gamma\pi(s))}{\textstyle \underbrace{\scriptstyle i(\gamma(r))}}}}{\underbrace{\mathbb E[a_{i(r),i(\gamma(r))}a_{i(s),i(\gamma(s))}]}}}{\underbrace{\mathbb E[a_{i(1),i(2)}a_{i(2),i(3)}\ldots a_{i(m),i(1)}]}}\\
                 &=\frac{1}{N^{1+\frac{m}{2}}}\sum_{\pi\in \setofpartitionsof_2(m)}\sum_{i:[m]\to [N]}\hspace{-7.5em} \underset{\hspace{7.5em}\parbox{14em}{$i$ must be constant on each orbit\\ of $\gamma\pi$ in order to contribute}}{\underbrace{\prod_{s=1}^N\delta_{i(s),i(\gamma\pi(s))}}}\\
                 &=\frac{1}{N^{1+\frac{m}{2}}}\sum_{\pi\in \setofpartitionsof_2(m)}N^{\# (\gamma\pi)}\\
     &=\sum_{\pi\in \setofpartitionsof_2(m)}N^{\# (\gamma\pi)-1-\frac{m}{2}},
   \end{align*}
   which is what we needed to show.
 \end{proof}

 \begin{example}
   For $m=4$ we have $\frac{m}{2}+1=3$, and for the three $\pi\in \setofpartitionsof_2(4)$:
   \begin{align*}
     \begin{array}{c|c|c}
       \pi & \gamma\pi & \# (\gamma\pi)-3\\ \hline
       (1 \ 2)\,(3 \ 4) & (1 \ 3) \, (2) \, (4) & 0\\
       (1 \ 3)\,(2 \ 4) & (1 \ 4 \ 3 \ 2) & -2\\
       (1 \ 4)\,(2 \ 3) & (1)\, (2 \ 4) \, (3) & 0.
     \end{array}
   \end{align*}
  Thus $(\normalizedtrace\otimes \expectation)(A^4)=2+N^{-2}$ for $A$ \GUEN.
 \end{example}

 \begin{proposition}
   \label{proposition:characterization-non-crossing-pair-partitions}
   Let $m\in \naturalnumbers$ and $\pi\in \setofpartitionsof_2(m)$.
   \begin{enumerate}
   \item\label{proposition:characterization-non-crossing-pair-partitions-1} 
We have for all $\pi\in \setofpartitionsof_2(m)$ that
\begin{equation}\label{eq:genus-estimate}
\# (\gamma\pi)-1-\frac{m}{2}\leq 0.
\end{equation}
     \item\label{proposition:characterization-non-crossing-pair-partitions-2} Equality holds in Equation \eqref{eq:genus-estimate} if and only if $\pi\in\setofnoncrossingpartitionsof_2(m)$.
   \end{enumerate}
 \end{proposition}
 \begin{proof}
   Note the following:
   \begin{itemize}
   \item If $\gamma\pi$ has no fixed point, then $|V|\geq 2$ for all $V\in \gamma\pi$, which entails $\#(\gamma\pi)\leq \frac{m}{2}$ and thus $\#(\gamma\pi)-1-\frac{m}{2}<0$.
   \item If $\gamma\pi$ does have a fixed point, say $(\gamma\pi)(i)=i$ for some $i\in [m]$, then this means: $\pi(i)=\gamma^{-1}(i)$, i.e.\ $\pi=\tilde{\pi}\cup \{\gamma^{-1}(i),i\}$ for some partition $\tilde{\pi}$. In this case, we remove $\{\gamma^{-1}(i),i\}$ from $\pi$ and obtain a new pairing $\tilde{\pi}$. Defining a new long cycle $\tilde{\gamma}=(1 \ \ldots \ \gamma^{-2}(i) \ \gamma(i) \ \ldots \ m)$, it then follows $\#(\tilde{\gamma}\tilde{\pi})=\#(\gamma\pi)-1$. (We have lost the  orbit $\{i\}$ of $\gamma\pi$ in $\tilde{\gamma}\tilde{\pi}$, but the orbit encompassing $\{\gamma^{-1}(i),\gamma(i)\}$ of $\gamma\pi$ survives as the orbit containing $\gamma(i)$ in $\tilde{\gamma}\tilde{\pi}$.) Thus, $\#(\tilde{\gamma}\tilde{\pi})-1-\frac{m-2}{2}=\#(\gamma\pi)-1-\frac{m}{2}$.
   \item We iterate this procedure until we find no more fixed points, thus proving Claim~\ref{proposition:characterization-non-crossing-pair-partitions-1}.
   \item Equality can only arise during this iteration if we always find a fixed point to remove until nothing is left.
      \item Since removing a fixed point in $\gamma\pi$ corresponds to removing a pair in $\pi$ which is an interval (i.e.\ consists of neighbors with respect to $\gamma$), the recursive characterization of non-crossing partitions shows that equality in Equation~\eqref{eq:genus-estimate} holds exactly if $\pi\in \setofnoncrossingpartitionsof_2(m)$.\qedhere
   \end{itemize}
 \end{proof}

 \begin{corollary}[{Wigner} 1955]
   \label{corollary:gue}
   For every $N\in \naturalnumbers$, let $A_N$ be a \GUEN. Then, for every $m\in \naturalnumbers$,
   \begin{align*}
     \lim_{N\to\infty}(\normalizedtrace\otimes \expectation)(A^m_N)=\#\setofnoncrossingpartitionsof_2(m).
   \end{align*}
   In other words, $A_N\convergesindistributionto s$ for a standard semicircular variable $s$.
 \end{corollary}

 \begin{remark}
   \begin{enumerate}
   \item We have shown here convergence on average, i.e.\ that, for all $m\in \naturalnumbers$, the sequence $(\mathbb E[\normalizedtrace(A^m_N)])_{N\in \naturalnumbers}$ converges (to the corresponding moment of the semicircle). One can refine this (e.g., by variance estimates) to stronger forms of convergence: The sequence of classical random variables $(\normalizedtrace (A^m_N))_{N\in\naturalnumbers}$ also converges almost surely. We will not consider here such questions, but such stronger versions for the convergence are usually also true for our further results in this section. 

The following figure compares the histogram for the N=3000 eigenvalues of one realization of a \GUEN\  with the semicircular density.

\begin{center}
\includegraphics[width=12cm]{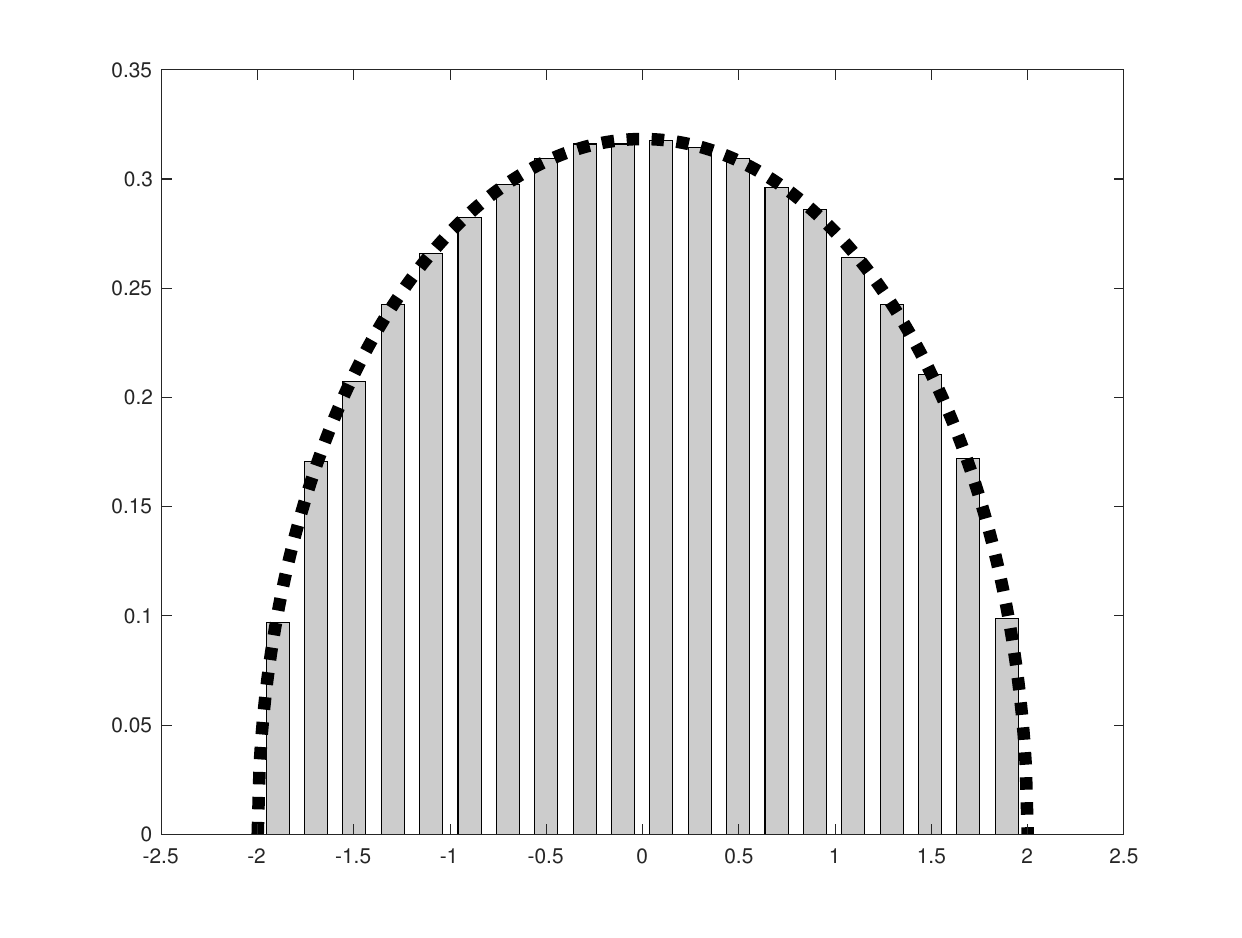}
\end{center}

   \item The occurrence of the semicircle as a basic distribution for limits of sequences of random matrices (long before free probability was invented) hints at a closer relation between free probability and random matrices. Voiculescu made this concrete by showing that also freeness shows up asymptotically in the random matrix world.
   \end{enumerate}
 \end{remark}

 \begin{theorem}[{Voiculescu} 1991]
   \label{theorem:gue-family}
   Let $t\in \naturalnumbers$ be arbitrary and, for every $N\in \naturalnumbers$, let $A_N^{(1)},\ldots, A_N^{(t)}$ be $t$ independent \GUEN, i.e., the upper triangular entries of all the different $(A_N^{(i)})_{i=1}^t$ taken together form a family of independent random variables. Then,
   \begin{align*}
     (A_N^{(1)},\ldots, A_N^{(t)})\convergesindistributionto (s_1,\ldots,s_t),
   \end{align*}
   where $s_1,\ldots,s_t$ are free standard semicircular variables. In particular, the random matrices $A_N^{(1)},\ldots, A_N^{(t)}$ are \emph{asymptotically free}, for $N\to\infty$.
 \end{theorem}
 \begin{proof}
   One can essentially repeat the proof of Theorem~\ref{theorem:gue}, which showed the claim for $t=1$. Let $N\in \naturalnumbers$ be arbitrary and write $A_N^{(r)}=\colon(a_{i,j}^{(r)})_{i,j=1}^N$ for every $r\in [t]$. Then, by our assumption, $\{a_{i,j}^{(r)}\mid r\in [t],\, i,j\in [N]\}$ forms a Gaussian family with covariance
   \begin{align*}
     \mathbb E\left[a_{i,j}^{(r)}a_{k,l}^{(p)}\right]=\frac{1}{N}\delta_{i,l}\delta_{j,k}\delta_{r,p}
   \end{align*}
   for all $i,j,k,l\in [N]$ and $p,r\in [t]$. Consequently, for every $p:[m]\to [t]$, 
   \begin{align*}
     (\normalizedtrace\otimes \expectation)\left(A^{(p(1))}_N\ldots A^{(p(m))}_N\right)
     &=\frac{1}{N}\sum_{i:[m]\to [N]}\mathbb E\left[a^{(p(1))}_{i(1),i(2)}\ldots a^{(p(m))}_{i(m),i(1)}\right]\\
                                             &=\frac{1}{N}\sum_{i:[m]\to [N]}\sum_{\pi\in \setofpartitionsof_2(m)}\prod_{\substack{\{r,s\}\in \pi\\r<s}}\hspace{-3.5em} \underset{\hspace{3.5em}\displaystyle=\frac{1}{N}\delta_{i(r),i(\gamma(s))}\delta_{i(s),i(\gamma(r))}\delta_{p(r),p(s)}}{\underbrace{\mathbb E\left[a_{i(r),i(\gamma(r))}^{(p(r))}a_{i(s),i(\gamma(s))}^{(p(s))}\right]}}\\
     &=\sum_{\substack{\pi\in \setofpartitionsof_2(m)\\\pi\leq \ker(p)}}N^{\#(\gamma\pi)-1-\frac{m}{2}}.
   \end{align*}
   Thus, in the limit, for every $p:[m]\to [t]$,
   \begin{align*}
     \lim_{N\to\infty}(\normalizedtrace\otimes \expectation)\left(A^{(p(1))}_N\ldots A^{(p(m))}_N\right)=\sum_{\substack{\pi\in \setofnoncrossingpartitionsof_2(m)\\\pi\leq \ker(p)}}1%\\
                                                              &=\sum_{\pi\in\setofnoncrossingpartitionsof(m)}\hspace{-6em}\underset{\hspace{6em}\displaystyle =
                                                                \begin{cases}
                                                                  1,&\text{if }\pi\in \setofnoncrossingpartitionsof_2(m), \pi\leq \ker(p),\\
                                                                  0,&\text{otherwise}
                                                                \end{cases}
}{\underbrace{\kappa_\pi(s_{p(1)},\ldots,s_{p(m)})}},
   \end{align*}
and the latter is exactly the formula for the moment
     $\varphi(s_{p(1)}\ldots s_{p(m)})$.
   That concludes the proof.
 \end{proof}
 
 \begin{remark}
   \begin{enumerate}
   \item Theorem~\ref{theorem:gue-family} shows that we can model a free pair of semicirculars asymptotically by independent Gaussian random matrices. Can we also model other distributions in this way? It turns out that we can replace one of the Gaussian random matrix sequences by a sequence of \enquote{deterministic} matrices with arbitrary limit distribution.
   \item For $N\in \naturalnumbers$, \emph{deterministic $N\times N$-matrices} are just elements
     \begin{align*}
       D\in M_N(\complexnumbers)\subseteq M_N(L^{\infty-}(\Omega,\mathbb P)),
     \end{align*}
     i.e.\ ordinary matrices embedded in the algebra of random matrices. Our state $\normalizedtrace\otimes \mathbb E$ on these is then just given by taking the normalized trace: $\normalizedtrace\otimes \mathbb E(D)=\normalizedtrace(D)$.
   \end{enumerate}
 \end{remark}

 \begin{theorem}
   \label{theorem:gue-deterministic}
   Let $t,N\in \naturalnumbers$ be arbitrary, let $A^{(1)},\ldots, A^{(t)}$ be $t$ independent \GUEN and let $D\in M_N(\complexnumbers)$ be a deterministic $N\times N$-matrix. Then, for all $m\in \naturalnumbers$, all $q:[m]\to \naturalnumbers_0$ and all $p:[m]\to [t]$ we have
$$
       (\normalizedtrace\otimes \expectation)\bigl( A^{(p(1))}D^{q(1)}\ldots A^{(p(m))}D^{q(m)} \bigr)
     = \sum_{\substack{\pi\in \setofpartitionsof_2(m)\\\pi\leq \ker(p)}}\normalizedtrace_{\pi\gamma}\bigl[D^{q(1)},\ldots, D^{q(m)}\bigr] \cdot N^{\#(\gamma\pi)-1-\frac{m}{2}},
$$
   where, for all $n\in \naturalnumbers$, $\sigma\in S_n$ and $D_1,\ldots, D_n\in M_N(\complexnumbers)$,
   \begin{align*}
     \normalizedtrace_\sigma(D_1,\ldots, D_n)\equalperdefinition \prod_{\text{$c$ cycle of $\sigma$}}\normalizedtrace \Bigl(\prod_{i\in c}^{\longrightarrow} D_i\Bigr).
   \end{align*}
 \end{theorem}
 \begin{proof}
   We write $A^{(p)}=\colon (a_{i,j}^{(p)})_{i,j=1}^N$ for every $p\in [t]$ and $D^q=\colon (d_{i,j}^{(q)})_{i,j=1}^N$ for every $q\in \naturalnumbers_0$.  Then, with the help of the Wick formula \eqref{eq:wick-formula} from Remark~\hyperref[remark:GUE-2]{\ref*{remark:GUE}~\ref*{remark:GUE-2}}, we compute  for all $m\in \naturalnumbers$, $p:[m]\to [t]$ and $q:[m]\to\naturalnumbers_0$:
   \begin{align*}
     (\normalizedtrace\otimes &\expectation)\Bigl(\prod_{n\in [m]}^{\longrightarrow} \bigl(A^{(p(n))}D^{q(n)}\bigr)\Bigr)\\
\quad\\
     &=\frac{1}{N}\sum_{i,j:[m]\to[N]}\mathbb E\Bigl[\hspace{-1.7em}\underset{\displaystyle =\Bigl(\prod_{n=1}^ma^{(p(n))}_{i(n),j(n)}\Bigr)\Bigl(\prod_{n=1}^md^{(q(n))}_{j(n),i(\gamma(n))}\Bigr)}{\underbrace{\prod_{n=1}^m\left(a_{i(n),j(n)}^{(p(n))}d_{j(n),i(\gamma(n))}^{(q(n))}\right)}}\hspace{-1.7em}\Bigr]\\
\quad\\
     &=\frac{1}{N}\sum_{i,j:[m]\to [N]}\Bigl(\prod_{n=1}^md^{(q(n))}_{j(n),i(\gamma(n))}\Bigr)\sum_{\pi\in \setofpartitionsof_2(m)}\prod_{\substack{\{r,s\}\in \pi\\r<s}}  \hspace{-3.85em}\underset{\hspace{3.85em}\displaystyle =\frac{1}{N}\underset{\displaystyle=\delta_{i,j\circ\pi}}{\underbrace{\delta_{i(r),j(s)}\delta_{j(r),i(s)}}}\delta_{p(r),p(s)}}{\underbrace{\mathbb E\Bigl[a^{(p(r))}_{i(r),j(r)}a^{(p(s))}_{i(s),j(s)}\Bigr]}}\\
     &=\frac{1}{N^{1+\frac{m}{2}}}\sum_{\substack{\pi\in \setofpartitionsof_2(m)\\\pi\leq \ker(p)}}\sum_{j:[m]\to [N]} \Bigl(\prod_{n=1}^md^{(q(n))}_{j(n),j((\pi\gamma)(n))}\Bigr).
   \end{align*}
   Now, note that for every $m\in \naturalnumbers$, $q:[m]\to\naturalnumbers_0$ and  $\sigma\in S_m$:
   \begin{align*}
     \sum_{j:[m]\to [N]}\Bigl(\prod_{n=1}^md^{(q(n))}_{j(n),j(\sigma(n))}\Bigr)=\trace_\sigma\bigl[D^{q(1)},\ldots, D^{q(m)}\bigr]
     =N^{\# \sigma}\cdot \normalizedtrace_\sigma\bigl[D^{q(1)},\ldots, D^{q(m)}\bigr].
   \end{align*}
   For example, if $m=5$ and $\sigma= (1\ 4)\,(2 \ 5)\, (3)$, then
   \begin{align*}
     \sum_{j:[5]\to [N]}d_{j(1),j(4)}^{(q(1))}&d_{j(2),j(5)}^{(q(2))}d_{j(3),j(3)}^{(q(3))}d_{j(4),j(1)}^{(q(4))}d_{j(5),j(2)}^{(q(5))}\\
     &=\trace(D^{q(1)}D^{q(4)})\cdot \trace(D^{q(2)}D^{q(5)}) \cdot \trace(D^{q(3)})\\
     &=N^3\cdot \normalizedtrace(D^{q(1)}D^{q(4)})\cdot \normalizedtrace(D^{q(2)}D^{q(5)}) \cdot \normalizedtrace(D^{q(3)}).
\end{align*}
   Hence, for every $m\in \naturalnumbers$, $p:[m]\to [t]$ and $q:[m]\to \naturalnumbers_0$ we have
   $$
       (\normalizedtrace\otimes \expectation)\Bigl(\prod_{n\in [m]}^{\longrightarrow} \bigl(A^{(p(n))}D^{q(n)}\bigr)\Bigl)
     =\sum_{\substack{\pi\in \setofpartitionsof_2(m)\\\pi\leq \ker(p)}}\sum_{j:[m]\to [N]}\hspace{-0.7em}\normalizedtrace_{\pi\gamma}\bigl[D^{q(1)},\ldots, D^{q(m)}\bigr]\cdot N^{\# (\gamma\pi)-1-\frac{m}{2}}
$$
   because $\#(\gamma\pi)=\#(\pi\gamma)$. That is what we needed to prove.
 \end{proof}

 \begin{remark}
   For an asymptotic statement $N\to\infty$, our deterministic matrices $(D_N)_{N\in \naturalnumbers}$ need to have a limit in distribution. Hence we assume that
     $D_N\convergesindistributiontoo d$
   for some random variable $d$ in a probability space $(\mathcal A,\varphi)$, i.e.\
     $\lim_{N\to\infty} \normalizedtrace(D^m_N)=\varphi(d^m)$ for all $m\in \naturalnumbers$.
   This then, of course, implies that also
   \begin{align*}
     \lim_{N\to\infty}\normalizedtrace_{\pi\gamma}\bigl[D^{q(1)},\ldots, D^{q(m)}\bigr]=\varphi_{\pi\gamma}(d^{q(1)},\ldots,d^{q(m)})
   \end{align*}
   and thus we obtain the following corollary.
 \end{remark}

 \begin{theorem}
   \label{theorem:gue-deterministic-asymptotics}
   Let $t\in \naturalnumbers$ be arbitrary, for each $N\in \naturalnumbers$, let $A_N^{(1)},\ldots, A_N^{(t)}$ be $t$ independent \GUEN\ and $D_N\in M_N(\complexnumbers)$ a deterministic $N\times N$-matrix such that
     $D_N\convergesindistributiontoo d$
   for some random variable $d$ in a probability space $(\mathcal A,\varphi)$. Then, for all $m\in \naturalnumbers$, $q:[m]\to \naturalnumbers_0$ and $p:[m]\to [t]$,
   \begin{align}
     \label{eq:theorem-gue-deterministic-asymptotics}
     \lim_{N\to\infty}\mathbb E\bigl[\normalizedtrace\bigl(A_N^{(p(1))}D_N^{q(1)}\ldots A_N^{(p(m))}D_N^{q(m)}\bigr)\bigr]&=\sum_{\substack{\pi\in \setofnoncrossingpartitionsof_2(m)\\\pi\leq \ker(p)}}\varphi_{\pi\gamma}\bigl(d^{q(1)},\ldots, d^{q(m)}\bigr).
   \end{align}
   In other words,
     $(A_N^{(1)},\ldots, A_N^{(t)},D_N)\convergesindistributionto (s_1,\ldots, s_t,d)$,
   where $s_1,\ldots,s_t$ are semicirculars and where $s_1,\ldots,s_t,d$ are free.
 \end{theorem}
 \begin{proof}
   The limit formula \eqref{eq:theorem-gue-deterministic-asymptotics} follows from Theorem~\ref{theorem:gue-deterministic} since for every $m\in \naturalnumbers$ and $\pi\in \setofpartitionsof_2(m)$ the sequence
     $\normalizedtrace_{\pi\gamma}[D_N^{q(1)},\ldots, D_N^{q(m)}]$
   converges to $\varphi_{\pi\gamma}(d^{q(1)},\ldots,d^{q(m)})$ and since the limit of
     $N^{\#(\gamma\pi)-1-\frac{m}{2}}$
   singles out the non-crossing $\pi\in \setofpartitionsof_2(m)$. By Theorem~\ref{theorem:gue-family} and Assignment~\hyperref[assignment-2]{2}, Exercise~1, it suffices to see asymptotic freeness, for $N\to\infty$, between
     $\{A_N^{(1)},\ldots, A_N^{(t)}\}$ and $D_N$.
   To do so, one has to recall from the proof of Theorem~\ref{theorem:free-convolution-semigroup} that, for every $m\in \naturalnumbers$, $p:[m]\to [t]$ and $q:[m]\to\naturalnumbers_0$,
   \begin{align*}
     \varphi(s_{p(1)}d^{q(1)}\ldots s_{p(m)}d^{q(m)})
     &=\sum_{\sigma\in \setofnoncrossingpartitionsof(2m)}\kappa_\sigma\bigl(s_{p(1)},d^{q(1)},\ldots, s_{p(m)},d^{q(m)}\bigr)\\
     &=\sum_{\substack{\pi\in\setofnoncrossingpartitionsof_2(m)\\\pi\leq \ker(p)}} \underset{\displaystyle =1}{\underbrace{\kappa_\pi\bigl(s_{p(1)},\ldots,s_{p(m)}\bigr)}}\underset{\displaystyle=\varphi_{K(\pi)}\bigl(d^{q(1)},\ldots,d^{q(m)}\bigr)}{\underbrace{\sum_{\substack{\tilde{\pi}\in \setofnoncrossingpartitionsof_2(m)\\\pi\cup \tilde\pi\in \setofnoncrossingpartitionsof(2m)}}\kappa_{\tilde{\pi}}\bigl(d^{q(1)},\ldots, d^{q(m)}\bigr)}},
   \end{align*}
   where we have used the moment-cumulant formula from Remark~\ref{remark:moment-cumulant-formula}.
   This identity implies the formula \eqref{eq:theorem-gue-deterministic-asymptotics} from the claim, provided that for every $\pi\in \setofnoncrossingpartitionsof_2(m)$ the Kreweras complement $K(\pi)$ corresponds, as a  permutation, to $\pi\gamma$. This is indeed the case. We check this just for an example: for the partition
     $\pi=\{\{1,2\},\{3,6\}, \{4,5\},\{7,8\}\}\in \setofnoncrossingpartitionsof_2(8)$ this can be seen from
   \begin{align*}
     \pi\gamma = (1) \, (2\ 6\ 8)\, (3 \ 5)\, (4)\, (7)
   \end{align*}
   \begin{center}
     \begin{tikzpicture}[baseline=-2pt]
      \node[inner sep=1pt] (n1) at (0em,0) {};
      \node[inner sep=1pt] (n2) at (2em,0) {};
      \node[inner sep=1pt] (n3) at (4em,0) {};
      \node[inner sep=1pt] (n4) at (6em,0) {};
      \node[inner sep=1pt] (n5) at (8em,0) {};
      \node[inner sep=1pt] (n6) at (10em,0) {};
      \node[inner sep=1pt] (n7) at (12em,0) {};
      \node[inner sep=1pt] (n8) at (14em,0) {};
      \node[inner sep=1pt] (m1) at (1em,0) {};
      \node[inner sep=1pt] (m2) at (3em,0) {};
      \node[inner sep=1pt] (m3) at (5em,0) {};
      \node[inner sep=1pt] (m4) at (7em,0) {};
      \node[inner sep=1pt] (m5) at (9em,0) {};
      \node[inner sep=1pt] (m6) at (11em,0) {};
      \node[inner sep=1pt] (m7) at (13em,0) {};
      \node[inner sep=1pt] (m8) at (15em,0) {};
      \draw[densely dashed] (n1) --++(0,-1.5em) -| (n2);
      \draw (m1) --++(0,-1em);
      \draw (m4) --++(0,-1em);
      \draw (m7) --++(0,-1em);
      \draw[densely dashed] (n4) --++(0,-1.5em) -| (n5);      
      \draw[densely dashed] (n3) --++(0,-2.5em) -| (n6);
      \draw (m3) --++(0,-2em) -| (m5);      
      \draw[densely dashed] (n7) --++(0,-1.5em) -| (n8);
      \draw (m2) --++(0,-3em) -| (m6);
      \draw ($(m2)+(0,-3em)$)  -| (m8);
      \node [fill=white, right, inner sep=1pt] at ($(n1)+(-0.375em,0em)$) {$1\hspace{0.5em}\overline{1}\hspace{0.5em}2\hspace{0.5em}\overline{2}\hspace{0.5em}3\hspace{0.5em}\overline{3}\hspace{0.475em}4\hspace{0.5em}\overline{4}\hspace{0.5em}5\hspace{0.5em}\overline{5}\hspace{0.5em}6\hspace{0.5em}\overline{6}\hspace{0.5em}7\hspace{0.5em}\overline{7}\hspace{0.5em}8\hspace{0.5em}\overline{8}$};  
    %\draw [brown] (-0.25em,-3.125em) rectangle (15.25em, 0.75em);
    \useasboundingbox (-0.25em,-3.125em) rectangle (15.25em, 0.75em);
  \end{tikzpicture}
\end{center}
and
\begin{align*}
  K(\pi)= \{ \{1\},\{2,6,8\},\{3,5\},\{4\},\{7\} \}.
\end{align*}
That concludes the proof.
\end{proof}

\begin{example}
  \begin{enumerate}
  \item Theorem~\ref{theorem:gue-deterministic-asymptotics} yields now first non-trivial results for asymptotic eigenvalue distributions of random matrices via free convolution. Note that 
      $A_N^{(1)}+A_N^{(2)}$ for two independent
     \GUEN\ is not very interesting because the random matrix $A_N^{(1)}+A_N^{(2)}$ is just another \GUEN, only with different variance. More interesting is a sequence
$A_N+D_N$
    for a \GUE\ $A_N$ and a deterministic sequence $D_N$, with a limit 
$D_N\convergesindistributiontoo d$.
    In that case, 
      $(A_N,D_N)\convergesindistributiontoo (s,d)$, where $s$ is a semicircular variable and $s$ and $d$ are free. Then
      $A_N+D_N \convergesindistributiontoo s+d$, and
        hence the limiting averaged (and also almost sure) eigenvalue distribution of $A_N+D_N$ is given by $\mu_s\boxplus \mu_d$.

Let us give a concrete example for this. In the following figure we have generated a \GUEN\ for $N=3000$ and added to this a diagonal deterministic matrix with eigenvalue distribution
$\mu_d=\frac 14(2\delta_{-2}+\delta_{-1}+\delta_{+1})$. The histogram of the $N$ eigenvalues is compared to $\mu_s\boxplus \mu_d$, which we calculated via our subordination iteration machinery according to Theorem \ref{theorem:subordination-iteration}.

\begin{center}
\includegraphics[width=14cm]{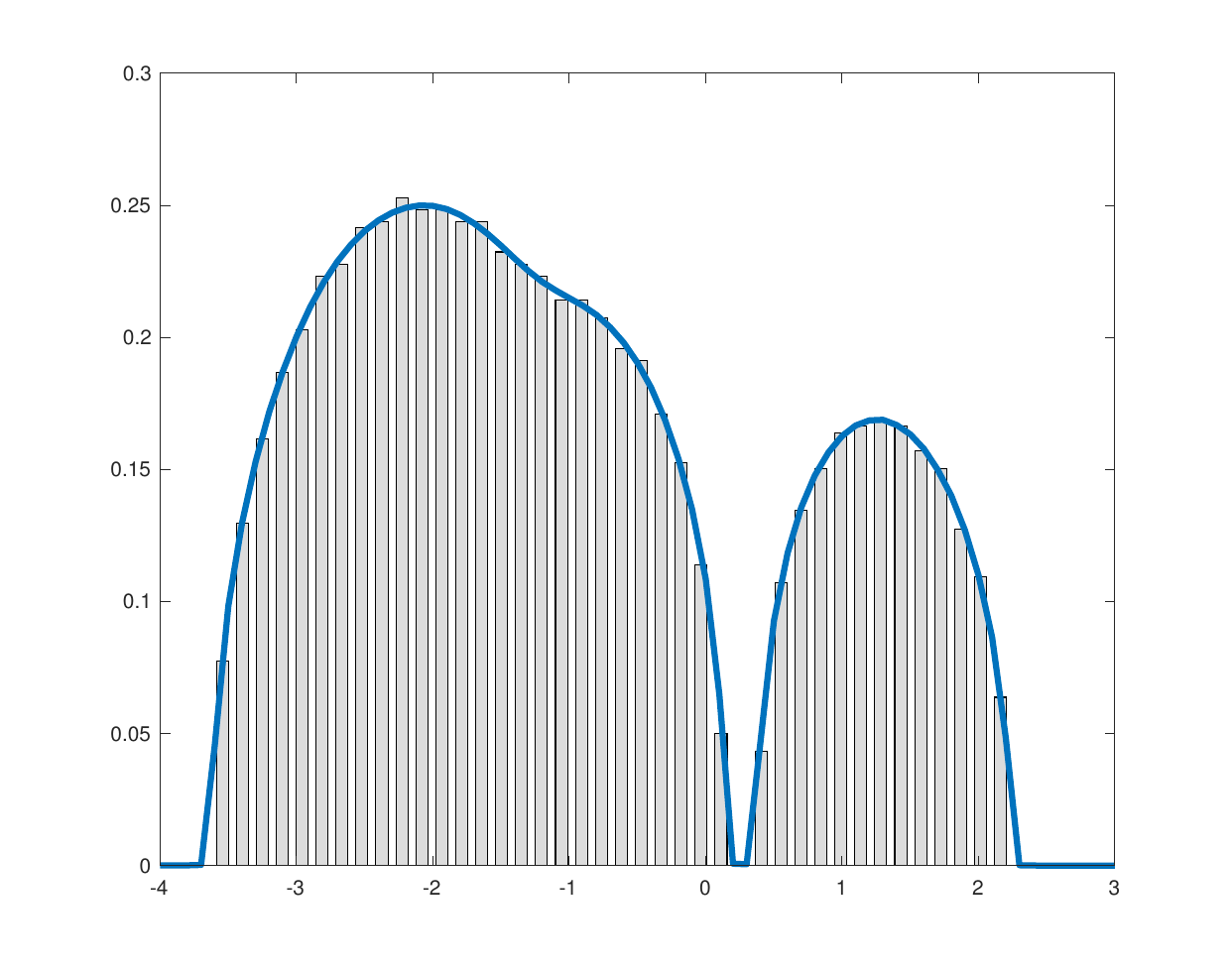}
\end{center}

      \item  Note that according to the proof of Theorem~\ref{theorem:gue-deterministic-asymptotics} the deterministic matrices
         $D^{q(1)}_N,\ldots, D_N^{q(m)}$
        do not need to be powers of one deterministic matrix $D_N$, but can actually be $m$ arbitrary deterministic matrices. For example, we can choose, alternatingly, powers of two deterministic matrices $D_N$ and $E_N$. Then, under the assumption that the sequence
          $(D_N,E_N)$
        has a limit $(d,e)$ in distribution, we obtain
        \begin{align*}
          (A_N,D_N,E_N)\convergesindistributionto (s,d,e), \quad\text{where  $s$ and $\{d,e\}$ are free}.
        \end{align*}
        The freeness between $s$ and $\{d,e\}$ implies (cf. Exercise~3, Assignment~\hyperref[assignment-6]{6}) that $sds$ and $e$ are free as well. Hence, we see that
        \begin{align*}
          (A_ND_NA_N,E_N)\convergesindistributionto (sds,e) \quad\text{where $sds$ and $e$ are free}.
        \end{align*}
        Thus, we can model more free situations asymptotically via random matrices.
        But note: Whereas $d$ and $e$ can be chosen to have arbitrary distributions,  $sds$ will always have a free compound Poisson distribution (see Exercise~3, Assignment~\hyperref[assignment-8]{8}). The limiting eigenvalue distribution of $A_ND_NA_N+E_N$ is thus given by the measure $\mu_{sds}\boxplus \mu_e$. \par
        Can we also model asymptotically via random matrices  distributions $\mu_1\boxplus \mu_2$, where $\mu_1$ and $\mu_2$ are \emph{arbitrary} probability measures on $\realnumbers$?
        For this we have to use unitary random matrices  instead of Gaussian ones. Note: For random variables $u,d\in \mathcal A$ in a $\ast$-probability space $(\mathcal A,\varphi)$ such that $u$ is unitary, the random variable $udu^\ast$ has the same distribution as $d$, provided $\varphi$ is a trace:
        \begin{align*}
          \varphi[(udu^\ast)^m])=\varphi(ud^mu^\ast)=\varphi(d^m).
        \end{align*}
We will elaborate more on this in the next section.
  \end{enumerate}
\end{example}

\newpage

\section{Unitary Random Matrices and Asymptotic Freeness\\
(and a Remark on Wigner Matrices)}
%[Unitary Random Matrices]

\begin{definition}
  \begin{enumerate}
  \item For every $N\in \naturalnumbers$, let
    \begin{align*}
      \mathcal U(N)\equalperdefinition \{ U\in M_N(\complexnumbers) \mid UU^\ast=1=U^\ast U\}
    \end{align*}
    be the unitary $N\times N$-matrices over $\complexnumbers$. Then, $\mathcal{U}(N)$ is (with respect to matrix multiplication and the subspace topology induced by $M_N(\complexnumbers)$) a compact group and thus there exists a Haar measure $\lambda$ on $\mathcal{U}(N)$. (This means that $\lambda$ is invariant under translations by group elements.) The Haar measure $\lambda$ is finite and uniquely determined by the translation invariance condition up to a scalar factor. So, we can normalize $\lambda$ to a probability measure. Random matrices distributed according to this probability measure will be called \emph{Haar unitary random matrices}.
  \item A unitary $u$ in $\ast$-probability space $(\mathcal A,\varphi)$ is called  a \emph{Haar unitary} if
    \begin{align*}
      \varphi(u^k)=\delta_{k,0} \quad\text{for all } k\in \integers,
    \end{align*}
    meaning $\varphi(1)=1$ and $\varphi(u^n)=0=\varphi((u^\ast)^n)$ for all $n\in \naturalnumbers$.
  \end{enumerate}
\end{definition}

\begin{notation}
  For every $N\in \naturalnumbers$, every $m\in \naturalnumbers$ with $m\leq N$ and every $\alpha\in S_m$, define
  \begin{align*}
    \weingarten(N,\alpha)=\mathbb E\left[u_{1,1}\ldots u_{m,m}\overline{u_{1,\alpha(1)}}\ldots \overline{u_{m,\alpha(m)}}\right],
  \end{align*}
  where $(u_{i,j})_{i,j=1}^N$ is a Haar unitary random $N\times N$-matrix. Then, we call $\weingarten$
  the \emph{Weingarten function}.
\end{notation}

\begin{facts}
  \label{facts:haar-unitary-random-matrices}
  \begin{enumerate}
  \item\label{facts:haar-unitary-random-matrices-1} One can give a \enquote{Wick-type} formula for Haar unitary random matrices:
    Let $N\in \naturalnumbers$ be arbitrary and let $U=(u_{i,j})_{i,j=1}^N$ be a Haar unitary $N\times N$-random matrix. Then, for all $m,m'\in \naturalnumbers$ with $m,m'\leq N$ and all $i,i',j,j':[m]\to [N]$,
    \begin{multline*}
        \mathbb E\bigl[ u_{i(1),j(1)}\ldots u_{i(m),j(m)}\overline{u_{i'(1),j'(1)}}\ldots \overline{u_{i'(m'),j'(m')}} \bigr]\\
\quad\\
    =\delta_{m,m'}\sum_{\alpha,\beta\in S_m}\delta_{i(\beta(1))i'(1)}\ldots \delta_{i(\beta(m)),i'(m)}\\
 \cdot\>\delta_{j(\alpha(1)),j'(1)}\ldots \delta_{j(\alpha(m)),j'(m)}\cdot \weingarten(N,\alpha^{-1}\beta),
\end{multline*}
    \begin{center}
    \begin{tikzpicture}[scale=0.9]
  \coordinate (a1) at (0em, 0em);
  \coordinate (a2) at (6em, 0em);
  \coordinate (a3) at (12em, 0em);
  \coordinate (b1) at (18em, 0em);
  \coordinate (b2) at (24em, 0em);
  \coordinate (b3) at (30em, 0em);
  \draw ($(a3)+(1em,0em)$) -- ++ (0,2em) -| ($(b2)+(-1.25em,0em)$);
  \draw ($(a3)+(-1em,0em)$) -- ++ (0,-2em) -| ($(b2)+(-1.25em,0em)$);
  \draw[->, shorten >=1em] ($(a2)+(1em,0em)$) -- ++ (0,3em) -| node [pos=0.1, below] {$\alpha^{-1}$} ($(b2)+(1.25em,0em)$);
  \draw[->, shorten >=1em] ($(a2)+(-1em,0em)$) -- ++ (0,-3em) -| node [pos=0.1, above] {$\beta$} ($(b1)+(-1.25em,0em)$);
  \draw ($(a1)+(-1em,0em)$) -- ++ (0,4em) -| ($(b3)+(1.25em,0em)$);
  \draw ($(a1)+(-1em,0em)$) -- ++ (0,-4em) -| ($(b3)+(-1.25em,0em)$);
  \node[fill=white, inner sep=2pt] at (a1) {$i(1),j(1)$};
  \node[fill=white, inner sep=2pt] at (a2) {$i(2),j(2)$};
  \node[fill=white, inner sep=2pt] at (a3) {$i(3),j(3)$};
  \node[fill=white, inner sep=2pt] at (b1) {$i'(1),j'(1)$};
  \node[fill=white, inner sep=2pt] at (b2) {$i'(2),j'(2)$};
  \node[fill=white, inner sep=2pt] at (b3) {$i'(3),j'(3)$};
  %\draw [brown] (-2.125em,-4.125em) rectangle (32.225em,4.125em);
  \useasboundingbox (-2.125em,-4.125em) rectangle (32.225em,4.125em);  
\end{tikzpicture}
\end{center} 
or, equivalently,
$$
    \mathbb E\Bigl[\bigl(\prod_{n=1}^m u_{i(n),j(n)}\bigr)\bigl(\prod_{n'=1}^{m'}\overline{u_{i'(n'),j'(n')}}\bigr)\Bigr]
  =\delta_{m,m'}\sum_{\alpha,\beta\in S_m}\delta_{i\circ \beta,i'}\delta_{j\circ \alpha,j'}\weingarten(N,\alpha^{-1}\beta).
$$
  \item\label{facts:haar-unitary-random-matrices-2} The asymptotics of the Weingarten function can be summarized as follows. For all $m\in \naturalnumbers$ and $\alpha\in S_m$  there is some $\phi(\alpha)\in \complexnumbers$ (depending only on $\alpha$, not on $N$) such that
    \begin{align}\label{eq:weingarten-asymptotics}
      \weingarten (N,\alpha)=\phi(\alpha)N^{\#\alpha-2m}+O\left(N^{\#\alpha-2m-2}\right) \quad (N\to\infty).
    \end{align}
  \end{enumerate}
\end{facts} 

\begin{remark}
  \begin{enumerate}
  \item The fact that, for $m,m',N\in \naturalnumbers$ with $m,m'\leq N$,  a $\ast$-moment
    \begin{align*}
       \mathbb E\left[u_{i(1),j(1)}\ldots u_{i(m),j(m)}\overline{u_{i'(1),j'(1)}}\ldots \overline{u_{i'(m'),j'(m')}}\right]
    \end{align*}
    in the entries of a Haar unitary $U=(u_{i,j})_{i,j=1}^N$ random matrix is zero if $m\neq m'$ follows from the fact that with $U$ also $\lambda U$ is a Haar unitary random matrix for any $\lambda\in \complexnumbers$ with $|\lambda|=1$. Then, the above moment of $U$ must equal the corresponding moment
    \begin{align*}
      \mathbb E\left[\lambda u_{i(1),j(1)}\ldots \lambda u_{i(m),j(m)}\overline{\lambda u_{i'(1),j'(1)}}\ldots \overline{\lambda u_{i'(m'),j'(m')}}\right],
    \end{align*}
    of $\lambda U$, which, due to $\lambda^{-1}=\overline\lambda$, is identical to
    \begin{align*}
      \lambda^{m-m'}\mathbb E\left[u_{i(1),j(1)}\ldots u_{i(m),j(m)}\overline{ u_{i'(1),j'(1)}}\ldots \overline{ u_{i'(m'),j'(m')}}\right].
    \end{align*}
    The only way to satisfy this equation is thus that the moment vanishes if $m\neq m'$.
  \item This shows that for every Haar unitary random matrix $U$ it holds
    \begin{align*}
      (\normalizedtrace\otimes \expectation)[U^k]=0 \quad\text{for all }k\in \integers\backslash \{0\}.
    \end{align*}
    Thus, for each $N\in\naturalnumbers$, a Haar unitrary random $N\times N$-matrix is a Haar unitary. And thus, for a sequence $(U_N)_{N\in\naturalnumbers}$ of Haar unitray random matrices, the limit exists in distribution and is also a Haar unitary $u$, i.e., $U_N\convergesindistributiontoo u$.
    The interesting question is, of course, if Haar unitary random matrices are also asymptotically free from deterministic matrices. We only consider the most relevant case of \enquote{randomly rotated} matrices.
  \end{enumerate}
\end{remark}

\begin{theorem}
  \label{theorem:haar-unitary-deterministic-asymptotics}
  For every $N\in \naturalnumbers$, let $A_N,B_N\in M_N(\complexnumbers)$, and suppose
    $A_N\convergesindistributiontoo a$ and $B_N\convergesindistributiontoo b$
  for random variables $a$ and $b$ in some $\ast$-probability space. Furthermore, for every $N\in \naturalnumbers$, let $U_N$ be a Haar unitary random $N\times N$-matrix. Then,
  \begin{align*}
    (U_NA_NU_N^\ast,B_N) \convergesindistributionto (a,b), \qquad \text{where $a$ and $b$ are free}.
  \end{align*}
  In particular, the random matrices $U_NA_NU_N^\ast$ and $B_N$ are asymptotically free from each other, for $N\to\infty$.
\end{theorem}

 Note that, for every $N\in \naturalnumbers$, the random matrix $U_NA_NU_N^\ast$ has the same distribution as $A_N$. Thus, also
    $U_NA_NU_N^\ast \convergesindistributiontoo a$.
  The point here is that the random rotation between the eigenspaces of $A_N$ and  $B_N$ makes the two matrices asymptotically free.

Before we start the proof let us give an example for the use of this for the calculation of the asymptotic eigenvalue distribution of the sum of such randomly rotated matrices. We take for $A_N$ and $B_N$ diagonal matrices with eigenvalue distributions
$$
\mu_a=\frac 14(\delta_{-2}+\delta_{-1}+\delta_{1}+\delta_{2})\qquad\text{and}\qquad
\mu_b=\frac 14(2\delta_{-2}+\delta_{-1}+\delta_{1}).
$$
We generate then a Haar unitary random $N\times N$-matrix and consider
$U_NA_NU_N^*+B_N$. According to Theorem \ref{theorem:haar-unitary-deterministic-asymptotics} the eigenvalue distribution of this converges, for $N\to\infty$, to $\mu_a\boxplus \mu_b$. We have proved this only for the averaged version; however, it is also true in the almost sure sense, i.e., for generic realizations of $U_N$. In the following figure we compare the histogram of the eigenvalues of $U_NA_NU_N^*+B_N$, for $N=3000$ and for one realization of $U_N$, with the distribution $\mu_a\boxplus \mu_b$, which we calculated via the subordination iteration machinery from 
Theorem \ref{theorem:subordination-iteration}.
\begin{center}
\includegraphics[width=14cm]{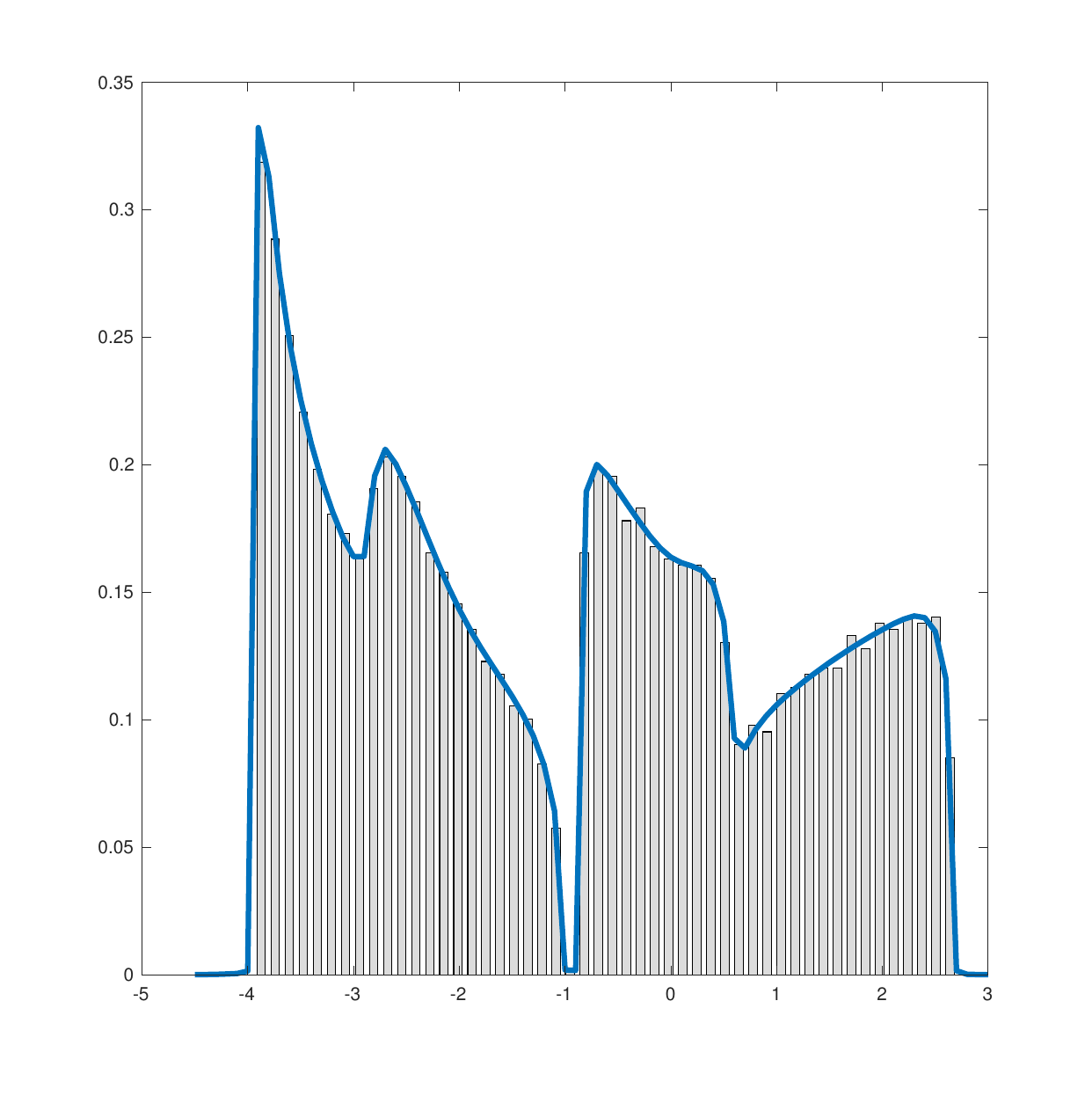}
\end{center}
\begin{proof}
  We have to calculate moments of the form
  \begin{align}
    \label{eq:haar-unitary-deterministic-asymptotics-1}
    (\normalizedtrace\otimes \expectation)\bigl[\prod_{n\in [m]}^{\longrightarrow}\bigl(U_NA_N^{q(n)}U_N^\ast B_N^{p(n)}\bigr)\bigr]
  \end{align}
  for $m,N\in \naturalnumbers$ with $m\leq N$ and for $p,q:[m]\to \naturalnumbers_0$ (where we have used that $(U_NA_NU_N^\ast)^n=U_NA_N^nU_N^\ast$ for all $n\in \naturalnumbers$) and then consider the limit $N\to\infty$. Consider such $m,N,p,q$. We write
  \begin{align*}
    A^{q(k)}_N=\colon (a^{(k)}_{i,j})_{i,j=1}^N\qquad
    \text{and}\qquad B^{p(k)}_N&=\colon (b^{(k)}_{i,j})_{i,j=1}^N
  \end{align*}
  for every $k\in [m]$, as well as
$\gamma \equalperdefinition (1 \ 2 \ \ldots \ m-1 \ m)\in S_m$.
  Then, by Fact~\hyperref[facts:haar-unitary-random-matrices-1]{\ref*{facts:haar-unitary-random-matrices}~\ref*{facts:haar-unitary-random-matrices-1}}, the corresponding moment \eqref{eq:haar-unitary-deterministic-asymptotics-1} is given by
  \begin{align*}
 \eqref{eq:haar-unitary-deterministic-asymptotics-1}    &=\frac{1}{N}\sum_{\substack{i,i',j,j':\\ [m]\to [N]}}\expectation\Bigl[\prod_{n=1}^m\bigl(u_{i(n)j(n)}a^{(n)}_{j(n),j'(n)}\overline{u_{i'(n),j'(n)}}b_{i'(n),i(\gamma(n))}^{(n)}\bigr)\Bigr].\\
     &=\frac{1}{N}\sum_{\substack{i,i',j,j':\\ [m]\to [N]}}\prod_{t=1}^m\left(a^{(t)}_{j(t),j'(t)}b_{i'(t),i(\gamma(t))}^{(t)}\right)\underset{\displaystyle = \sum_{\alpha,\beta\in S_m}\delta_{i\circ \beta,i'}\delta_{j\circ \alpha,j'}\weingarten(N,\alpha^{-1}\beta)}{\underbrace{\expectation\Bigl[\bigl(\prod_{n=1}^mu_{i(n)j(n)}\bigr)\bigl(\prod_{n=1}^{m}\overline{u_{i'(n),j'(n)}}\bigr)\Bigr]}}\\
     &=\frac{1}{N}\sum_{\substack{\alpha,\beta\\\in S_m}} \weingarten(N,\alpha^{-1}\beta)\cdot\underset{\displaystyle =\trace_\alpha[A_N^{q(1)},\ldots,A_N^{q(m)}]}{\underbrace{\Bigl(\sum_{j: [m]\to [N]}\prod_{r=1}^ma^{(r)}_{j(r),j(\alpha(r))}\Bigr)}}\cdot\underset{\displaystyle=\trace_{\beta^{-1}\gamma}[B_N^{p(1)},\ldots,B_N^{p(m)}]} {\underbrace{\Bigl(\sum_{i': [m]\to [N]}\prod_{s=1}^mb_{i'(s),i'(\beta^{-1}\gamma(s))}^{(s)}\Bigr)}}\\
\quad\\
      &=\sum_{\alpha,\beta\in S_m}\weingarten(N,\alpha^{-1}\beta)N^{\#\alpha+\#(\beta^{-1}\gamma)-1}%\\
\cdot  \normalizedtrace_\alpha[A_N^{q(1)},\ldots,A_N^{q(m)}]\cdot\normalizedtrace_{\beta^{-1}\gamma}[B_N^{p(1)},\ldots,B_N^{p(m)}].
  \end{align*}
  We know, by \eqref{eq:weingarten-asymptotics} and with $\phi$ as prescribed there, that for all $\alpha,\beta\in S_m$
  \begin{align*}
    \weingarten(N,\alpha^{-1}\beta) =\phi(\alpha^{-1}\beta)N^{\#(\alpha^{-1}\beta)-2m}+ O\bigl(N^{\#(\alpha^{-1}\beta)-2m-2}\bigr).
  \end{align*}
  Hence, for all $\alpha,\beta\in S_m$, the leading order terms in the above expression for \eqref{eq:haar-unitary-deterministic-asymptotics-1} have the factor
    $N^{\#(\alpha^{-1}\beta)+\#\alpha+\#(\beta^{-1}\gamma)-2m-1}$.
  If we write, for every $\alpha\in S_m$,
    $|\alpha| \equalperdefinition m - \#\alpha$,
  where $\#\alpha$ denotes the number of orbits of $\alpha$, then it is easy to see the following.
  \begin{enumerate}
  \item For every $\alpha\in S_m$, the number $|\alpha|$ is the minimal non-negative integer $k$ such that $\alpha$ can be written as a product of $k$ transpositions. (Transpositions of $S_m$ are of the form $\beta = ( i\ j )$ for $i,j\in [m]$ with $i\neq j$, i.e., $\beta(i)=j$, $\beta(j)=i$ and $\beta(l)=l$ for all $l\in [m]$ with $l\neq i,j$.)
  \item The map $|\cdot|:\, S_m\to\naturalnumbers_0$ satisfies the triangle inequality: For all $\alpha,\beta\in S_m$, we have
      $|\alpha\beta|\leq |\alpha|+|\beta|$.
  \end{enumerate}
  Now, note that for all $\alpha,\beta\in S_m$
  \begin{align*}
    \#(\alpha^{-1}\beta)+\#\alpha+\#(\beta^{-1}\gamma)-2m-1=m-1-|\alpha^{-1}\beta|-|\alpha|-|\beta^{-1}\gamma|\leq 0
  \end{align*}
  since
  \begin{align*}
    |\alpha^{-1}\beta|+|\alpha|+|\beta^{-1}\gamma|\geq |\alpha(\alpha^{-1}\beta)(\beta^{-1}\gamma)|=|\gamma|=m-1.
  \end{align*}
  Thus, in the limit, for every $m\in \naturalnumbers$ and all $p,q:[m]\to \naturalnumbers$,
  \begin{align*}
    \lim_{N\to\infty} (\normalizedtrace\otimes & \expectation)\bigl[\prod_{n\in [m]}^{\longrightarrow}\bigl(U_NA_N^{q(n)}U_N^\ast B_N^{p(n)}\bigr)\bigr]\\
    &=\sum_{\substack{\alpha,\beta\in S_m\\|\alpha^{-1}\beta|+|\alpha|+|\beta^{-1}\gamma|\\=m-1}}\normalizedtrace_\alpha[a^{q(1)},\ldots,a^{q(m)}]\cdot\normalizedtrace_{\beta^{-1}\gamma}[b^{p(1)},\ldots,b^{p(m)}]\cdot\phi(\alpha^{-1}\beta).
  \end{align*}
  By multi-linearity, this result goes over from monomials to arbitrary polynomials $f_1,\ldots,f_m\in \complexnumbers[x]$ and $g_1,\ldots,g_m\in \complexnumbers[x]$:
  \begin{align}
    \label{eq:haar-unitary-deterministic-asymptotics-2}
        \lim_{N\to\infty} (\normalizedtrace\otimes &\expectation)\bigl[\prod_{n\in [m]}^{\longrightarrow}\bigl(U_Nf_n(A_N)U_N^\ast g_n(B_N)\bigr)\bigr]\\
    &=\sum_{\substack{\alpha,\beta\in S_m\\|\alpha^{-1}\beta|+|\alpha|+|\beta^{-1}\gamma|\\=m-1}}\normalizedtrace_\alpha\left[f_1(a),\ldots,f_m(a)\right]\cdot \normalizedtrace_{\beta^{-1}\gamma}\left[g_1(b),\ldots,g_m(b)\right]\cdot\phi(\alpha^{-1}\beta).\notag
  \end{align}
  We have to see that formula \eqref{eq:haar-unitary-deterministic-asymptotics-2} describes variables $a,b$ which are free. So, let $m\in \naturalnumbers$ be arbitrary and let $f_1,\ldots,f_m\in \complexnumbers[x]$ and $g_1,\ldots,g_m\in \complexnumbers[x]$ be such that
    $\varphi(f_i(a))=0$ and $\varphi(g_i(b))=0$  for all $i\in [m]$.
  Then, we have to show 
    $$\varphi[f_1(a)g_1(b)\ldots f_m(a)g_m(b)]=0.$$
  For this, note: If for $m\in \naturalnumbers$ and permutations $\alpha,\beta\in S_m$ it holds $|\alpha^{-1}\beta|+|\alpha|+|\beta^{-1}\gamma|=m-1$, then at least one of the two permutations $\alpha$ and $\beta^{-1}\gamma$ must have a fixed point. (Indeed, the assumption implies that
  \begin{align*}
    |\alpha|\leq \frac{m-1}{2} \quad\text{or}\quad |\beta^{-1}\gamma|\leq \frac{m-1}{2}.
  \end{align*}
  Say this is true for $\alpha$. Then, $\alpha$ can be written as a product of at most $\lfloor \frac{m-1}{2}\rfloor$ transpositions, so can move at most $m-1$ elements. That requires $\alpha$ to have a fixed point.)
  \par
  But then, for all $\alpha,\beta\in S_m$ summed over in \eqref{eq:haar-unitary-deterministic-asymptotics-2}, at least one of the corresponding terms $\normalizedtrace_\alpha\left[f_1(a),\ldots,f_m(a)\right]$ and $\normalizedtrace_{\beta^{-1}\gamma}\left[g_1(b),\ldots,g_m(b)\right]$  is equal to zero. Hence, each term in \eqref{eq:haar-unitary-deterministic-asymptotics-2} vanishes.
\end{proof}

\begin{remark}
  The formula \eqref{eq:haar-unitary-deterministic-asymptotics-2} in the proof of Theorem~\ref{theorem:haar-unitary-deterministic-asymptotics} says that for two free sets $\{a_1,\dots,a_m\}$ and $\{b_1,\dots,b_m\}$ we have the following formula for mixed moments:
  $$
      \varphi(a_1b_1a_2b_2\ldots a_mb_m)
      =\sum_{\substack{\alpha,\beta\in S_m \\ |\alpha^{-1}\beta|+|\alpha|+|\beta^{-1}\gamma|\\=m-1}}\hspace{-1em}\varphi_\alpha(a_1,a_2,\ldots, a_m)\cdot\varphi_{\beta^{-1}\gamma}(b_1,b_2,\ldots, b_m)\cdot\phi(\alpha^{-1}\beta).
$$
  On the other hand, from our combinatorial cumulant machinery we know the formula
  \begin{align*}
    \varphi(a_1b_1a_2b_2\ldots a_mb_m)&=\sum_{\pi\in\setofnoncrossingpartitionsof(m)}\kappa_\pi(a_1,\ldots,a_m)\varphi_{K(\pi)}(b_1,\ldots,b_m)\\
    &=\sum_{\substack{\pi,\sigma\in \setofnoncrossingpartitionsof(m)\\\sigma\leq \pi}}\varphi_\sigma(a_1,\ldots,a_m)\varphi_{K(\pi)}(b_1,\ldots,b_m)\mu(\sigma,\pi).
  \end{align*}
  This suggests a bijective correspondence between the appearing terms in both formulas. This is indeed true. In particular, for every $m\in \naturalnumbers$, the non-crossing partitions $\setofnoncrossingpartitionsof(m)$ can be embedded in $S_m$ as $S_{\setofnoncrossingpartitionsof(\gamma)}\hateq \setofnoncrossingpartitionsof(m)$, where ({Biane} 1997)
  \begin{align*}
    S_{\setofnoncrossingpartitionsof(\gamma)} \equalperdefinition \{\alpha\in S_m\mid |\alpha|+|\alpha^{-1}\gamma|=m-1\}.
  \end{align*}
  If we define, for all $m\in \naturalnumbers$ and $\alpha,\beta\in S_m$,
    $d(\alpha,\beta)\equalperdefinition |\beta \alpha^{-1}|$,
  then $d:\,S_m\times S_m\to\naturalnumbers_0$ is a metric on $S_m$ and the set $S_{\setofnoncrossingpartitionsof(\gamma)}$ consists exactly of all permutations located on a geodesic from the identity permutation to $\gamma$.
\end{remark}

\begin{remark}
  Up to now, our basic random matrix model was the \GUE, corresponding to independent identically Gaussian-distributed entries (up to symmetry). A straightforward generalization of these are \emph{Wigner matrices}, whose entries are also identically distributed (up to symmetry) but the common distribution of which can be arbitrary.
  \par
  From the random matrix literature it is known that Wigner matrices behave in many respects like \GUE. In particular, a sequence of Wigner matrices converges also to a semicircle. But what about asymptotic freeness?
\end{remark}

\begin{definition}
  Let $N\in \naturalnumbers$ and let $\mu$ be a probability measure on $\realnumbers$. A corresponding \emph{Wigner random $N\times N$-matrix}  is of the form
  \begin{align*}
    A=\frac{1}{\sqrt{N}}(a_{i,j})_{i,j=1}^N,
  \end{align*}
  where
  \begin{itemize}
  \item the family $(a_{i,j})_{i,j\in [N], j\leq i}$ is a tuple of independent and identically $\mu$-distributed real random variables,
  \item it holds $A=A^\ast$, i.e.\ $a_{i,j}=a_{j,i}$ for all $i,j\in [N]$.
  \end{itemize}
\end{definition}

\begin{theorem}[Wigner 1955]
  Let $\mu$ be a probability measure on $\realnumbers$ all of whose moments exist and whose first moment is zero. For every $N\in \naturalnumbers$, let $A_N$ be a corresponding $N\times N$-Wigner matrix. Then,
    $A_N\convergesindistributiontoo s$,
  where $s$ is a semicircular element (whose variance is given by the second moment of $\mu$).
\end{theorem}
\begin{proof}[Sketch of proof]
  For all $m,N\in \naturalnumbers$ we calculate
  \begin{align*}
    (\normalizedtrace\otimes\expectation)(A_N^m)&=\frac{1}{N^{1+\frac{m}{2}}}\sum_{i:[m]\to[N]} \expectation\left[ a_{i(1),i(2)}\ldots a_{i(m),i(1)} \right]\\
                  &=\frac{1}{N^{1+\frac{m}{2}}}\sum_{\sigma\in \setofpartitionsof(m)}\sum_{\substack{i:[m]\to [N]\\\ker(i)=\sigma}}\underset{\substack{\displaystyle=\colon \expectation[\sigma]\\\displaystyle\text{depends only on }\sigma}}{\underbrace{\expectation\left[ a_{i(1),i(2)}\ldots a_{i(m),i(1)}\right]}}\\
    &=\frac{1}{N^{1+\frac{m}{2}}}\sum_{\sigma\in \setofpartitionsof(m)}\expectation[\sigma]\cdot \underset{\displaystyle\sim N^{\#\sigma}}{\underbrace{N\cdot (N-1)\cdot \ldots\cdot (N-\#\sigma+1)}}.
  \end{align*}
  For $m\in \naturalnumbers$, identify any partition $\sigma\in \setofpartitionsof(m)$ with a graph $G_\sigma$ as follows:
  \begin{center}
\begin{tikzpicture}
  \coordinate (i1) at (0em,0em);
  \coordinate (i2) at ($(i1)+(300:5em)$);
  \coordinate (i3) at ($(i2)+(240:5em)$);
  \coordinate (i4) at ($(i3)+(180:5em)$);
  \coordinate (i5) at ($(i4)+(120:5em)$);
  \coordinate (i6) at ($(i5)+(60:5em)$);
  \draw (i1) -- node[pos=0.5, right, darkgray] {$a_{i_1,i_2}$} (i2);
  \draw (i2) -- node[pos=0.5, right, darkgray] {$a_{i_2,i_3}$} (i3);
  \draw (i3) -- node[pos=0.5, below, darkgray] {$a_{i_3,i_4}$} (i4);
  \draw (i4) -- node[pos=0.5, left, darkgray] {$a_{i_4,i_5}$} (i5);
  \draw (i5) -- node[pos=0.5, left, darkgray] {$a_{i_5,i_6}$} (i6);
  \draw (i6) -- node[pos=0.5, above, darkgray] {$a_{i_6,i_1}$} (i1);
  \node [circle,scale=0.4, draw=black, fill=lightgray, label=above right:{$i_1$}] at (i1) {};
  \node [circle,scale=0.4, draw=black, fill=lightgray, label=right:{$i_2$}] at (i2) {};
  \node [circle,scale=0.4, draw=black, fill=lightgray, label=below right:{$i_3$}] at (i3) {};
  \node [circle,scale=0.4, draw=black, fill=lightgray, label=below left:{$i_4$}] at (i4) {};
  \node [circle,scale=0.4, draw=black, fill=lightgray, label=left:{$i_5$}] at (i5) {};
  \node [circle,scale=0.4, draw=black, fill=lightgray, label=above left:{$i_6$}] at (i6) {};
  \draw[->] ($(i2)+(4em,0)$) -- node[pos=0.5, above]{identify ver-}node[pos=0.5, below]{tices via $\sigma$} ++(4em,0em);
  %------------------------------------------------------
  \coordinate (j1) at (19em,0em);
  \coordinate (j2) at ($(j1)+(300:10em)$);
  \coordinate (j3) at ($(j2)+(180:10em)$);
    \draw (j1) arc (240:-120:2em) node [pos=0.5, right] {$a_{i_1,i_2}$};  
  \node[circle, draw=black, fill=lightgray] (n1) at (j1) {$1=2=4$};
  \node[circle, draw=black, fill=lightgray] (n2) at (j2) {$5$};
  \node[circle, draw=black, fill=lightgray] (n3) at (j3) {$3=6$};
  \draw (n1) -- node [pos=0.5, right] {$a_{i_4,i_5}$} (n2);
  \draw (n2) -- node [pos=0.5, below] {$a_{i_5,i_6}$} (n3);
  \draw (n3) -- node [pos=0.3, left] {$a_{i_2,i_3}$} node [pos=0.6, left] {$a_{i_3,i_4}$} node [pos=0.9, left] {$a_{i_6,i_1}$} (n1);
\node at (9em,3em) {$\sigma=(1 \ 2 \ 4)\, (3 \ 6)\, (5)$};
\end{tikzpicture}
\end{center}
In order for $\sigma\in \setofpartitionsof(m)$ to contribute a leading order term in the sum, every edge of the corresponding graph  $G_\sigma$ must be labeled  with \emph{exactly two} matrix entries. In that case, $G_\sigma$ is a tree and $\sigma$ is the Kreweras complement of a non-crossing pairing, like, e.g.,
\begin{align*}
  \sigma= \{\{1, 5, 7\}, \{2, 4\}, \{3\}, \{6\}, \{8\}\}
\end{align*}
whose graph is
\begin{center}
  \begin{tikzpicture}
    \coordinate (j1) at (0em,0em);
    \coordinate (j2) at ($(j1)+(220:8em)$);
    \coordinate (j3) at ($(j2)+(315:7em)$);
    \coordinate (j4) at ($(j1)+(310:9em)$);
    \coordinate (j5) at ($(j1)+(350:8em)$);
    \draw (j1) -- node [pos=0.5,shift={(-0.75em,0.75em)}] {$a_{i_1,i_2}$} node [pos=0.5, shift={(0.5em,-1em)}] {$a_{i_4,i_5}$} (j2);
    \draw (j2) -- node [pos=0.5,shift={(-0.75em,-0.75em)}] {$a_{i_2,i_3}$} node [pos=0.5,shift={(1em,0.75em)}] {$a_{i_3,i_4}$} (j3);
    \draw (j1) -- node [pos=0.6, shift={(-1em,-0.5em)}] {$a_{i_5,i_6}$} node [pos=0.6,shift={(1em,0.75em)}] {$a_{i_6,i_7}$} (j4);
    \draw (j1) -- node [pos=0.55, shift={(0em,-0.7em)}] {$a_{i_7,i_8}$} node [pos=0.55, shift={(0em,0.7em)}] {$a_{i_8,i_1}$} (j5);
    \node[draw=black, fill=lightgray, circle] at (j1) {$1=5=7$};
    \node[draw=black, fill=lightgray, circle] at (j2) {$2=4$};
    \node[draw=black, fill=lightgray, circle] at (j3) {$3$};
    \node[draw=black, fill=lightgray, circle] at (j4) {$6$};
    \node[draw=black, fill=lightgray, circle] at (j5) {$8$};
    \node at (-11em,-4.5em) {$G_\sigma=$};    
  \end{tikzpicture}
\end{center}
and whose Kreweras complement is
   \begin{center}
     \begin{tikzpicture}[baseline=-2pt]
      \node[inner sep=1pt] (n1) at (0em,0) {};
      \node[inner sep=1pt] (n2) at (2em,0) {};
      \node[inner sep=1pt] (n3) at (4em,0) {};
      \node[inner sep=1pt] (n4) at (6em,0) {};
      \node[inner sep=1pt] (n5) at (8em,0) {};
      \node[inner sep=1pt] (n6) at (10em,0) {};
      \node[inner sep=1pt] (n7) at (12em,0) {};
      \node[inner sep=1pt] (n8) at (14em,0) {};
      \node[inner sep=1pt] (m1) at (1em,0) {};
      \node[inner sep=1pt] (m2) at (3em,0) {};
      \node[inner sep=1pt] (m3) at (5em,0) {};
      \node[inner sep=1pt] (m4) at (7em,0) {};
      \node[inner sep=1pt] (m5) at (9em,0) {};
      \node[inner sep=1pt] (m6) at (11em,0) {};
      \node[inner sep=1pt] (m7) at (13em,0) {};
      \node[inner sep=1pt] (m8) at (15em,0) {};
      \draw (n1) -- ++ (0,-3em) -| (n5);
      \draw ($(n5)+(0,-3em)$)  -| (n7);
      \draw (n2) -- ++ (0,-2em) -| (n4);
      \draw (n3) -- ++ (0,-1em);
      \draw (n6) -- ++ (0,-1em);
      \draw (n8) -- ++ (0,-1em);
      \draw[densely dotted] (m1) -- ++ (0,-2.5em) -| (m4);
      \draw[densely dotted] (m2) -- ++ (0,-1.5em) -| (m3);
      \draw[densely dotted] (m5) -- ++ (0,-1.5em) -| (m6);
      \draw[densely dotted] (m7) -- ++ (0,-1.5em) -| (m8);                  
      \node [fill=white, right, inner sep=1pt] at ($(n1)+(-0.375em,0em)$) {$1\hspace{0.5em}\overline{1}\hspace{0.5em}2\hspace{0.5em}\overline{2}\hspace{0.5em}3\hspace{0.5em}\overline{3}\hspace{0.475em}4\hspace{0.5em}\overline{4}\hspace{0.5em}5\hspace{0.5em}\overline{5}\hspace{0.5em}6\hspace{0.5em}\overline{6}\hspace{0.5em}7\hspace{0.5em}\overline{7}\hspace{0.5em}8\hspace{0.5em}\overline{8}$};  
    %\draw [brown] (-0.25em,-3.125em) rectangle (15.25em, 0.75em);
    \useasboundingbox (-0.25em,-3.125em) rectangle (15.25em, 0.75em);
  \end{tikzpicture}
  \begin{align*}
    K(\sigma)= \left\{ \left\{\overline 1,\overline 4\right\}, \left\{\overline 2,\overline 3\right\}, \left\{\overline 5,\overline 6\right\}, \left\{\overline 7,\overline 8\right\} \right\}
  \end{align*}
  \end{center}
Thus, for every $m\in \naturalnumbers$, the $m$-th moment is asymptotically given by $\#\setofnoncrossingpartitionsof_2(m)$.
\end{proof}

\begin{remark}
  \begin{enumerate}
  \item Having multiple independent Wigner matrices does not change the combinatorics, but just introduces the additional constraint on partitions that they must connect the same matrices (as in the proof of Theorem~\ref{theorem:gue-family}, compared to that of Theorem~\ref{theorem:gue}). Hence, independent Wigner matrices are also asymptotically free. Inserting deterministic matrices, however, is more tricky.
  \item If the deterministic matrix is diagonal, though, then the combinatorics of the indices are still the same and one obtains asymptotic freeness between Wigner and deterministic matrices. For arbitrary deterministic matrices, the structure gets more complicated and how to deal with it is not obvious. The asymptotic freeness result still holds true, but the proof (\href{http://www.wisdom.weizmann.ac.il/~zeitouni/cupbook.pdf}{[\cite{AGZ}, Theorem 5.4.5]},\href{http://rolandspeicher.com/literature/mingo-speicher/}{[\cite{MSp}, Theorem 4.20]}) is annoyingly complicated.
  \end{enumerate}
\end{remark}

\begin{theorem}
  Let $p\in \naturalnumbers$ and let $\mu_1,\ldots,\mu_p$ be probability measures on $\realnumbers$ all of whose moments exist and whose first moments vanish. For every $N\in \naturalnumbers$, let $A_N^{(1)},\ldots, A_N^{(p)}$ be independent Wigner $N\times N$-matrices, with $A_N^{(i)}$ corresponding to $\mu_i$ for every $i\in [p]$. Furthermore, let $q\in \naturalnumbers$ and, for every $N\in \naturalnumbers$, let $D_N^{(1)},\ldots, D_N^{(q)}$ be deterministic $N\times N$-matrices such that 
  \begin{align*}
    \sup_{N\in \naturalnumbers, r\in [q]}\|D^{(r)}_N\|<\infty,
\qquad\text{and such that}
\qquad
    (D_N^{(1)},\ldots,D_N^{(q)})
\convergesindistributionto (d_1,\ldots,d_q)
\end{align*}
for $N\to\infty$.
  Then, as $N\to\infty$,
  \begin{align*}
    (A_N^{(1)},\ldots, A_N^{(p)}, D_N^{(1)},\ldots,D_N^{(q)}) \convergesindistributionto (s_1,\ldots,s_p,d_1,\ldots,d_q), 
  \end{align*}
  where each $s_i$ is a semicircular and $s_1,\dots,s_p, \{d_1,\dots,d_q\}$
are free.

  In particular, the two  families $\{A_N^{(1)},\ldots, A_N^{(p)}\}$ and  $\{D_N^{(1)},\ldots,D_N^{(q)}\}$ of Wigner and deterministic matrices are asymptotically free from each other.
\end{theorem}

\newpage

%%%%%%%%%%%%%%%%%%%%%%%

\section{Von Neumann Algebras: The Free Group Factors and Relation to Freeness}
%[Von Neumann Algebras and Freeness]

\begin{definition}
  \begin{enumerate}
  \item A \emph{von Neumann algebra} is a $\ast$-subalgebra of $\mathcal B(\mathcal H)$ (for a Hilbert space $\mathcal H$) which contains $1_{\mathcal B(\mathcal H)}$ and is closed with respect to the weak operator topology. (The \emph{weak operator topology (WOT)} is the locally convex topology defined by the seminorms $(p_{\xi,\eta})_{\xi,\eta\in \mathcal H}$ with
      $p_{\xi,\eta}(x)\equalperdefinition \left|\left\langle x\xi,\eta\right\rangle\right|$
    for all $\xi,\eta\in \mathcal H$ and $x\in \mathcal B(\mathcal H)$.)

\item For every Hilbert space $\mathcal H$ and every subset $\mathcal A\subseteq \mathcal B(\mathcal H)$ we define the \emph{commutant} of $\mathcal A$ by
  \begin{align*}
    \mathcal A'\equalperdefinition \{ y\in \mathcal B(\mathcal H)\mid xy=yx \text{ for all }x\in \mathcal A\}
  \end{align*}
  and the \emph{bicommutant} by
    $\mathcal A''\equalperdefinition (\mathcal A')'$.
\item A von Neumann algebra $\mathcal M$ is called a \emph{factor} if
    $\mathcal M\cap \mathcal M'=\complexnumbers\cdot 1_{\mathcal B(\mathcal H)}$.
    \end{enumerate}
\end{definition}

\begin{facts}
  \label{facts:von-neumann-algebras}
  \begin{enumerate}
  \item\label{facts:von-neumann-algebras-1} \emph{Bi\-com\-mu\-tant theorem}: Let $\mathcal H$ be a Hilbert space and let $\mathcal A$ be a unital $\ast$-subalgebra of $\mathcal B(\mathcal H)$. Then, $\mathcal A$ is a von Neumann algebra if and only if $\mathcal A=\mathcal A''$.
  \item\label{facts:von-neumann-algebras-2} Von Neumann algebras are closed under \emph{measurable functional calculus}: Let $\mathcal M$ be a von Neumann algebra, $x=x^\ast\in \mathcal M$ and let $f:\, \sigma(x)\to \complexnumbers$ be a measurable bounded function on the spectrum $\sigma(x)$ of $x$. Then $f(x)\in \mathcal M$.
  \item\label{facts:von-neumann-algebras-3} In particular, von Neumann algebras contain lots of projections. Consider $x=x^\ast\in \mathcal M\subseteq\mathcal B(\mathcal H)$. For every $\lambda\in \realnumbers$, we write 
      $E_x(\lambda)\equalperdefinition \mathbbm 1_{(-\infty,\lambda]}(x)$
    for the corresponding spectral projection; then the spectral theorem shows
    \begin{align*}
      x=\int_{\sigma(x)}\lambda \, dE_x(\lambda).
    \end{align*}
    Since the $\mathbbm 1_{(-\infty,\lambda]}$ are measurable functions, we have that all spectral projections $E_x(\lambda)$ of $x$ are in  $\mathcal M$.
  \item\label{facts:von-neumann-algebras-4} Von Neumann algebras are closed under \emph{polar decomposition}: Let $\mathcal H$ be a Hilbert space and $x\in \mathcal B(\mathcal H)$. The operator $x$ has a unique polar decomposition in $\mathcal B(\mathcal H)$ of the form $x=u|x|$, where $|x|\equalperdefinition \sqrt{x^\ast x}$ and where $u$ is a partial isometry (i.e.\ $u=uu^\ast u$) with $\ker(u)=\ker(x)$. For any von Neumann algebra $\mathcal M\subseteq \mathcal B(\mathcal H)$ with $x\in \mathcal M$ it holds
that $|x|,u\in \mathcal M$.
    \item\label{facts:von-neumann-algebras-5} For von Neumann algebras, topological and algebraic properties determine each other. Hence the relevant notion of isomorphism is purely algebraic: Two von Neumann algebras $\mathcal M\subseteq \mathcal B(\mathcal H)$ and $\mathcal N\subseteq \mathcal B(\mathcal K)$ (for Hilbert spaces $\mathcal H$ and $\mathcal K$) are \emph{isomorphic}, $\mathcal M\cong\mathcal N$, if there exists a $\ast$-isomorphism $\Phi:\,\mathcal M\to \mathcal N$. This isomorphism then has automatically the right continuity properties.
  \end{enumerate}
\end{facts}

\begin{example}
  \label{example:von-neumann-group-algebras}
  \begin{enumerate}
  \item\label{example:von-neumann-group-algebras-1} Our basic example is the von Neumann algebra version of the group algebra from Example~\ref{example:group-algebra}. For a (discrete) group $G$ we defined the group algebra as
    \begin{align*}
      \complexnumbers G\equalperdefinition \bigl\{ \sum_{g\in G}\alpha_gg\;\mid\; \alpha:\, G\to \complexnumbers,\, \alpha_g\neq 0\text{ for only finitely many }g\in G   \bigr\}.
    \end{align*}
    We now let $\complexnumbers G$ act on itself by left multiplication -- this is the left regular representation. For infinite $G$ we have to complete $\complexnumbers G$ to  a Hilbert space
    \begin{align*}
      \ell^2(G)\equalperdefinition \bigl\{\sum_{g\in G}\alpha_g g \; \mid \; \alpha:G\to\complexnumbers,\, \sum_{g\in G}|\alpha_g|^2<\infty\bigr\}
    \end{align*}
    with inner product determined by
      $\langle g,h\rangle\equalperdefinition\delta_{g,h}$ for all $g,h\in G$.
    Define now the \emph{left regular representation}
    \begin{align*}
      \lambda: \, \complexnumbers G\to \mathcal B(\ell^2(G)),\, \sum_{g\in G}\alpha_g g\mapsto \sum_{g\in G}\alpha_g\lambda(g),      
    \end{align*}
    where, for all $g\in G$, the operator $\lambda(g)$ is determined by 
      $\lambda (g)h\equalperdefinition gh$ for all $h\in G$,
    meaning for all $\alpha\in \ell^2(G)$:
    \begin{align*}
      \lambda(g)\underset{\displaystyle \in \ell^2(G)}{\underbrace{\sum_{h\in G}\alpha_hh}}=\sum_{h\in G}\alpha_hgh.
    \end{align*}
    Note that, for all $g\in G$,
      $\lambda(g)^\ast=\lambda(g^{-1})$
    and thus
      $\lambda(g)\lambda(g)^\ast=\lambda(g)^\ast\lambda(g)=1$,
    making $\lambda(g)$ a unitary operator. Thus,
      $\lambda(\complexnumbers G)\subseteq \mathcal B(\ell^2(G))$.
    We define now
    \begin{align*}
      L(G)\equalperdefinition \overline{\lambda (\complexnumbers G)}^{\mathrm{WOT}}=\lambda(\complexnumbers G)''\subseteq \mathcal B(\ell^2(G)).
    \end{align*}
    Thus, $L(G)$ is a von Neumann algebra, called a \emph{group von Neumann algebra}.\par
    If $e$ denotes the neutral element of $G$, then the map $\tau\equalperdefinition \tau_G$ from Example~\ref{example:group-algebra} is of the form
    \begin{align*}
      \tau\bigl(\underset{\displaystyle =\colon x}{\underbrace{\sum_{g\in G}\alpha_g g}}\bigr)=\alpha_e=\bigl\langle \underset{\displaystyle= x}{\underbrace{\bigl(\sum_{g\in G}\alpha_gg\bigr)}}e,e\bigr\rangle,
    \end{align*}
i.e.,
      $\tau(x)=\langle xe,e\rangle$ for all $x\in \complexnumbers G$.
    In this form, $\tau$ extends to a state
      $\tau:\, L(G)\to \complexnumbers$
    on $L(G)$. It is faithful and a trace on $L(G)$. (Note that it is defined on all of $\mathcal B(\ell^2(G))$, but that there it is neither faithful nor a trace in the infinite-dimensional case. Moreover, for any Hilbert space $\mathcal H$ with $\dim(\mathcal H)=\infty$, \emph{no} trace exists on $\mathcal B(\mathcal H)$, showing that $L(G)$ cannot be all of $\mathcal B(\ell^2(G))$ in the case of an infinite group.)
  \item\label{example:von-neumann-group-algebras-2} If all conjugacy classes $\{hgh^{-1}\mid h\in G\}$ for $g\neq e$ are infinite (in which case we say that $G$ is \emph{i.c.c}), then $L(G)$ is a factor. Basic examples for this are:
    \begin{enumerate}
    \item\label{example:von-neumann-group-algebras-2-1}  $G\equalperdefinition S_\infty\equalperdefinition\bigcup_{n=1}^\infty S_n$, the infinite permutation group. $\mathcal R\equalperdefinition L(S_\infty)$ is the so-called \emph{hyperfinite factor} ({Murray and von Neumann} 1936).
      (This is the simplest and nicest of the non-trivial von Neumann algebras.)
    \item\label{example:von-neumann-group-algebras-2-2}  
$G\equalperdefinition \mathbb F_n$, $n\in \naturalnumbers\cup \{\infty\}$, the free group on $n$ generators. This is given by words in $n$ non-commuting generators $g_1,\ldots,g_n$ and their inverses $g_1^{-1},\ldots, g_n^{-1}$ only subject to the group relations, i.e.\ $g_ig_i^{-1}=e=g_i^{-1}g_i$ for every $i\in [n]$.  The free group $\mathbb F_n$ is i.c.c for every $n\geq 2$ (including $n=\infty$). Hence, $L(\mathbb F_n)$ is, for $n\geq 2$, also a factor, the so-called \emph{free group factor}.
    \end{enumerate}
  \item\label{example:von-neumann-group-algebras-3} Murray and von Neumann showed that $\mathcal R\not\cong L(\mathbb F_n)$ for all $n$. But, besides that, not much more was known about the $L(\mathbb F_n)$. In particular, at least since Kadison's Baton-Rouge problem list from 1967 the following famous \emph{free group factor isomorphism problem} is still open:
    \begin{center}
      Is $L(\mathbb F_m)\cong L(\mathbb F_n)$ for $2\leq m,n$ and $m\neq n$?
    \end{center}
Voiculescu introduced and developed free probability theory to attack this and similar questions.
  \item\label{example:von-neumann-group-algebras-4} For $n\in \naturalnumbers$, consider $L(\mathbb F_n)$ with its $n$ generators
      $u_i\equalperdefinition \lambda (g_i)$ for $i\in [n]$,
    where $g_1,\ldots,g_n$ are the generators of $\mathbb F_n$. From Part~\ref{example:von-neumann-group-algebras-1} we know that each $u_i\in L(\mathbb F_n)$ is a unitary. Each of them has moments
    \begin{align*}
      \tau(u_i^k)=\langle \underset{\displaystyle=g_i^k}{\underbrace{u_i^ke},e}\rangle=\delta_{k,0} \quad\text{for all }k\in \integers,
    \end{align*}
    i.e., each $u_i$ is a Haar unitrary in the $\ast$-probability space $(L(\mathbb F_n),\tau)$. \par
    Furthermore, $\mathbb F_n$ can be written as the free product in the group-theo\-re\-ti\-cal sense (see Example~\hyperref[example:group-algebra-4]{\ref*{example:group-algebra}~\ref*{example:group-algebra-4}})
  $\mathbb F_n=G_1\ast \ldots\ast G_n$,
where  $G_i$, for $i\in [n]$, is the subgroup of $\mathbb F_n$ generated by $g_i$ (i.e. $G_i\cong \integers$). And, thus, the $\ast$-subalgebras $(\complexnumbers G_i)_{i=1}^n$ are, by Proposition~\ref{proposition:freeness-group-algebras}, freely independent in $(L(\mathbb F_n),\tau)$.
  \end{enumerate}
\end{example}

\begin{remark}
  \begin{enumerate}
  \item Hence, for every $n\in \naturalnumbers$, the free group factor $L(\mathbb F_n)$ is, as a von Neumann algebra, generated by $n$ free Haar unitaries. Note that our definition of freeness only imposes conditions on the joint moments of polynomials in the generators. If this definition is to be of any use for the von Neumann algebra context, these conditions  better transfer to \enquote{measurable functions} in the generators. That means, freeness better pass to the closures of polynomials in the weak operator topology. This is indeed the case.
  \item For the extension of freeness from algebras to generated von Neumann algebras we need, of course, some kind of continuity for our state $\tau$. (On the $C^\ast$-level, positivity gives continuity in norm, on the von Neumann algebra level, however, we need more than positivity!)
    The relevant notion for a linear functional $\tau:\mathcal M\to \complexnumbers$ on a von Neumann algebra $\mathcal M\subseteq \mathcal B(\mathcal H)$ is \enquote{normality}. Such a functional $\tau$ is called \emph{normal} if
      $\tau(\sup_{\lambda\in\Lambda} x_\lambda)=\sup_{\lambda\in \Lambda}\tau(x_\lambda)$
    for every increasing net $(x_\lambda)_{\lambda\in \Lambda}$ of self-adjoint operators from $\mathcal M$.
    This abstract definition (which is independent of the representation of $\mathcal M$ on $\mathcal H$) is, for the concrete realization $\mathcal M\subseteq \mathcal B(\mathcal H)$, equivalent to the following condition: The restriction of $\tau$ to the unit ball of $\mathcal M$ is continuous with respect to the weak operator topology of $\mathcal B(\mathcal H)$. Note in particular that, typically, our states are of the form
      $\tau(x)=\langle x\xi,\xi\rangle$ for some unit vector $\xi\in \mathcal H$.
    And such vector states are \emph{always} normal!
  \item Recall also that, in the von Neumann algebra context, it is often advantageous to switch between the weak and the strong operator topologies. The strong operator topology (SOT) is the locally convex topology generated by the seminorms $(p_\xi)_{\xi\in \mathcal H}$, where
      $p_\xi(x)\equalperdefinition \|x\xi\|$
    for every $\xi\in \mathcal H$ and all $x\in \mathcal B(\mathcal H)$. For a unital $\ast$-subalgebra $\mathcal A\subseteq \mathcal B(\mathcal H)$ the closures of $\mathcal A$ in WOT and in SOT
    agree; furthermore, a normal state $\tau$ is also continuous with respect to the strong operator topology when restricted to the unit ball.
  \end{enumerate}
\end{remark}

\begin{theorem}
  \label{theorem:freeness-transfers-to-von-neumann-algebras}
  Let $\mathcal M$ be a von Neumann algebra and $\tau:\,\mathcal M\to \complexnumbers$ a normal state. Let $(\mathcal A_i)_{i\in I}$ be unital $\ast$-subalgebras of $\mathcal M$ which are free with respect to $\tau$. For every $i\in I$, let $\mathcal M_i\equalperdefinition\vonneumannalgebrageneratedby(\mathcal A_i)$ be the von Neumann subalgebra of $\mathcal M$ generated by $\mathcal A_i$; i.e., $\mathcal M_i=\mathcal A_i^{''}$ is the smallest von Neumann algebra which contains $\mathcal A_i$. Then, the algebras  $(\mathcal M_i)_{i\in I}$ are also free.
\end{theorem}
\begin{proof}
 Consider elements $a_1,\ldots,a_k\in \mathcal M$, where $k\in \naturalnumbers$, $i: [k]\to I$ with $i(1)\neq i(2)\neq \ldots \neq i(k)$, where $a_j\in \mathcal M_{i(j)}$ for every $j\in [k]$, and where $\tau(a_j)=0$ for all $j\in [k]$.
  Then, we have to show that $\tau(a_1\ldots a_k)=0$. For each $j\in [k]$, we can approximate $a_j$ by a net $(a_j^{(\lambda)})_{\lambda\in \Lambda_j}$ from $\mathcal A_{i(j)}$ in the strong operator topology. In fact, we can assume that all $\Lambda_j$ are equal to some directed set $\Lambda$. By Kaplansky's density theorem, one can also ensure that
    $\| a_j^{(\lambda)}\|\leq \|a_j\|$ for all $j\in [k]$ and $\lambda \in \Lambda$.
  (Note that $\|\cdot\|$ is \emph{not} continuous in the weak operator topology.) By replacing
     $a_j^{(\lambda)}$ by $a_j^{(\lambda)}-\tau(a_j^{(\lambda)})$, we can moreover assume that
    $\tau(a_j^{(\lambda)})=0$ for all $j\in [k]$ and $\lambda\in \Lambda$.
  Then, the freeness assumption for the algebras  $(\mathcal A_i)_{i\in I}$ applies, for each $\lambda$, to the random variables $a_1^{(\lambda)},\ldots, a_k^{(\lambda)}$ and implies that
    $\tau(a_1^{(\lambda)}\ldots a_k^{(\lambda)})=0$ for all $\lambda\in \Lambda$.
  But then, since multiplication on bounded sets is continuous with respect to the strong operator topology, we get
  \begin{align*}
   \tau(a_1 \ldots a_k)= \lim_{\lambda \in \Lambda}\underset{\displaystyle =0\text{ for all }\lambda\in \Lambda}{\underbrace{\tau(a_1^{(\lambda)}\ldots a_k^{(\lambda)})}}=0
  \end{align*}
  That is what we needed to show.
\end{proof}

\begin{remark}
  For $n\in\naturalnumbers$, it now makes sense to say that $L(\mathbb F_n)$ is generated as a von Neumann algebra by the $n$ $\ast$-free Haar unitaries $u_1=\lambda(g_1),\ldots, u_n=\lambda(g_n)$. Next, we want to see that not the concrete form of the $u_1,\ldots,u_n$ is important here, but only their $\ast$-distribution.
\end{remark}

\begin{theorem}
  \label{theorem:von-neumann-algebras-generator-isomorphism}
  Let $n\in \naturalnumbers$ and let $\mathcal M$ and $\mathcal N$ be von Neumann algebras, generated as von Neumann algebras by elements $a_1,\ldots, a_n$ in the case of $\mathcal M$ and $b_1,\ldots,b_n$ in the case of $\mathcal N$. Let $\varphi:\, \mathcal M\to \complexnumbers$ and $\psi:\, \mathcal N\to \complexnumbers$ be faithful normal states. If $(a_1,\ldots, a_n)$ and $(b_1,\ldots,b_n)$ have the same $\ast$-distribution with respect to $\varphi$ and $\psi$, respectively, then the map $a_i\mapsto b_i$ ($i=1,\dots,n$) extends to a $\ast$-isomorphism from $\mathcal M$ to $\mathcal N$. In particular, then $\mathcal M\cong \mathcal N$.
\end{theorem}
\begin{proof}
  This follows from the facts that the GNS construction with respect to a faithful normal state gives a faithful representation of the von Neumann algebra, and that the $\ast$-moments of the generators contain all the information which is needed for this construction. \par
  To give  a more concrete facet of this on the $\ast$-algebra level: If we want to extend $a_i\mapsto b_i$ ($i\in I$) to the generated $\ast$-algebras, then for every polynomial $p$ in $n$ non-commuting variables we need the following implication to hold:
    $p(a_1,\ldots,a_n)=0 \implies p(b_1,\ldots,b_n)=0$.
  That this is indeed the case, can be seen as follows. $p(a_1,\ldots,a_n)=0$ implies
  \begin{align*}
    p(a_1,\ldots,a_n)^\ast p(a_1,\ldots,a_n)=0
\qquad
  \text{and thus}\qquad
\varphi(p(a_1,\ldots,a_n)^\ast p(a_1,\ldots,a_n))=0.
  \end{align*}
  Note that the latter equation involves a sum over $\ast$-moments of $(a_1,\ldots,a_n)$.
  Because those agree, by assumption, with the $*$-moments of $(b_1,\ldots,b_n)$, it follows that
  \begin{align*}
    \psi(p(b_1,\ldots,b_n)^\ast p(b_1,\ldots,b_n))=\varphi(p(a_1,\ldots,a_n)^\ast p(a_1,\ldots,a_n))=0
  \end{align*}
  As $\psi$ is faithful, we can conclude that
    $p(b_1,\ldots,b_n)=0$,
  as claimed.
\end{proof}

\begin{remark}
  \begin{enumerate}
  \item Theorem~\ref{theorem:von-neumann-algebras-generator-isomorphism} yields quite a change of perspective in the study of von Neumann algebras. We can understand von Neumann algebras by looking at the $\ast$-distributions of generators and not at the concrete action on a Hilbert space. (Of course, via the GNS construction this is the same, but still \ldots. Note in this context: The distribution of the real part of the one-sided shift is the semicircle, see Assignment~\hyperref[assignment-3]{3}, Exercise~3. The one-sided shift is one of the most important operators in operator theory. But, apparently, nobody before Voiculescu had been interested in its distribution.) This is very much in the spirit of classical probability theory, where the concrete realization of random variables as functions on some probability space $(\Omega,\mathcal F, \mathbb P)$ is usually not relevant -- the only thing which counts is their distribution.
    \item For every $n\in \naturalnumbers$, Theorem~\ref{theorem:von-neumann-algebras-generator-isomorphism} shows then that $L(\mathbb F_n)$ is generated by $n$ $*$-free Haar unitaries $u_1,\ldots,u_n$, where we do not have to care about how the $u_1,\ldots, u_n$ are conretely realized on a Hilbert space.
    \item But we now also have the freedom of deforming each of the generators $u_1,\ldots, u_n$ of $L(\mathbb F_n)$ by measurable functional calculus for normal elements (see Fact~\hyperref[facts:von-neumann-algebras-2]{\ref*{facts:von-neumann-algebras}~\ref*{facts:von-neumann-algebras-2}}). Instead of $u_1,\ldots,u_n$ we can consider $f_1(u_1),\ldots, f_n(u_n)$, where $f_1,\ldots, f_n$ are measurable bounded functions. If the $f_1,\ldots, f_n$ are invertible (as measurable functions on the corresponding spectra), then we can also \enquote{go backwards} and be assured that $f_1(u_1),\ldots, f_n(u_n)$ generate the same von Neumann algebra as $u_1,\ldots, u_n$, i.e.\ also $L(\mathbb F_n)$.  Since, for each $i\in [n]$, the generator $u_i$ has a diffuse spectrum, i.e.\ no atoms in its distribution, we can change this distribution via invertible measurable functions to any other non-atomic distribution we like. Note that, by Theorem~\ref{theorem:freeness-transfers-to-von-neumann-algebras}, we do not lose the freeness between the deformed variables; and that the $f(u_1),\ldots, f(u_n)$ are still normal.
  \end{enumerate}
\end{remark}

\begin{corollary}
\label{corollary:free-group-factors-isomorphisms}
  Let $\mathcal M$ be a von Neumann algebra and $\tau$ a faithful normal state on $\mathcal M$. Let $n\in \naturalnumbers$ and assume that $x_1,\ldots,x_n\in\mathcal M$ generate $\mathcal M$, $\mathcal M=\text{vN}(x_1,\dots,x_n)$, and that
  \begin{enumerate}
  \item $x_1,\ldots,x_n$ are $\ast$-free with respect to $\tau$;
  \item each $x_i$, for $i\in[n]$, is normal and its distribution $\mu_{x_i}$ with respect to $\tau$ has no atoms.
  \end{enumerate}
  Then, $\mathcal M\cong L(\mathbb F_n)$.
\end{corollary}

\begin{example}
  Let $\mathcal M$ be a von Neumann algebra, $\tau:\,\mathcal M\to\complexnumbers$ a faithful state and $x\in \mathcal M$.
  If $x$ is normal (i.e. $xx^\ast=x^\ast x$), then the distribution $\mu_x$ of $x$ with respect to $\tau$ is the uniquely determined probability measure $\mu$ on the spectrum $\sigma(x)$ of $x$ with
  \begin{align*}
    \tau(x^k(x^\ast)^l)=\int_{\sigma(x)}z^k\overline{z}^l\, d\mu(z)\qquad
\text{for all $k,l\in \naturalnumbers_0$}.
  \end{align*}
  In particular, for a Haar unitary $u\in\mathcal M$ it holds that $\sigma(u)=\unitcircle=\{z\in \complexnumbers \mid |z|=1\}$ (unit circle) and $\mu_u$ is the uniform distribution on $\unitcircle$. Via the inverse $g$ of the map
  \begin{align*}
     [0,2\pi)\to \unitcircle, \, t\mapsto e^{it}
  \end{align*}
  we can transform $\mu_u$ into the uniform distribution $\mu_{g(u)}$ on $[0,2\pi]$.
  The latter can then be deformed via a map
    \begin{gather*}
    h:\, [0,2\pi]\mapsto [-2,2]\\
                \begin{tikzpicture}[baseline=4em]
                        \begin{axis}[xmin=-0.5, ymin=-0.015, xmax=1.1*2*pi, ymax=.25, axis lines=middle, width=16em, height=11em, xtick={0,1.57,3.14,4.7,6.28}, xticklabels={$0$, $\frac{\pi}{2}$, $\pi$, $\frac{3\pi}{2}$, $2\pi$}, ytick={0.2},yticklabels={$\frac{1}{2\pi}$}]
                          \addplot[thick, samples=50, smooth] coordinates {(0,0.2)(2*pi,0.2)};
          \end{axis}
        \end{tikzpicture}
        \quad\quad\longrightarrow\quad\quad
                \begin{tikzpicture}[baseline=4em]
                        \begin{axis}[xmin=-1.1*2, ymin=0, xmax=1.1*2, ymax=1.1*pi^-1, axis lines=middle, xtick={-2,-1,0,1,2}, xticklabels={$-2$, $-1$, $0$, $1$, $2$}, ytick={0,0.318},yticklabels={$0$,$\frac{1}{\pi}$}, width=16em, height=11em]
                          \addplot[thick, samples=150, smooth] {(0.5*pi^-1)*((4-x^2)^0.5)};
          \end{axis}
        \end{tikzpicture}    
  \end{gather*}
  into a semicircle. Hence, with $f\equalperdefinition h\circ g$, the element $f(u)$ is a semicircular variable. Instead of by $u_1,\ldots,u_n$, we can thus generate $L(\mathbb F_n)$ also by $f(u_1),\ldots, f(u_n)$, i.e.\ by $n$ free semicirculars.
\end{example}

\newpage

\section[Circular and $R$-Diagonal Operators]{Circular and $R$-Diagonal Operators}
\begin{remark}
  In the following we want to realize our operators as matrices, motivated by our asymptotic representation by random matrices. Consider a sequence $(A_{2N})_{N\in \naturalnumbers}$ of \GUE(2$N$),
  \begin{align*}
    A_{2N}=\colon{\frac{1}{\sqrt{2N}}}(a_{i,j})_{i,j=1}^{2N}.
  \end{align*}
  We can then consider $A_{2N}$ as a block $2\times 2$-matrix with ${N}\times {N}$-matrices $A_{i,j}^{(N)}$ as entries,
  %\begin{IEEEeqnarray*}{rCCl}
$$
    A_{2N}=\frac{1}{\sqrt{2N}}\left[
    \begin{array}{ccc|ccc}
      a_{1,1}&\hdots& a_{1,N}& a_{1,N+1} &\hdots & a_{1,2N}\\
      \vdots&\ddots& \vdots& \vdots &\ddots & \vdots\\
      a_{N,1}&\hdots& a_{N,N}& a_{N,N+1} &\hdots & a_{N,2N} \\[0.375em]
      \hline
      a_{N+1,1}&\hdots& a_{N+1,N}& a_{N+1,N+1} &\hdots & a_{N+1,2N}\\
      \vdots&\ddots& \vdots& \vdots &\ddots & \vdots\\
      a_{2N,1}&\hdots& a_{2N,N}& a_{2N,N+1} &\hdots & a_{2N,2N}
    \end{array}\right]
                                                                       =\frac{1}{\sqrt{2}}
                                                                       \begin{bmatrix}
                                                                         A_{1,1}^{(N)} & A_{1,2}^{(N)}\\
                                                                         A_{2,1}^{(N)} & A_{2,2}^{(N)}
                                                                       \end{bmatrix}.
  %\end{IEEEeqnarray*}
$$
  The two random matrices on the diagonal,
    $A_{1,1}, A_{2,2}$,
  then converge in distribution to a free pair $s_1,s_2$ of semicircular variables, by Theorem~\ref{theorem:gue-family}. But what about $A_{1,2}$ and $A_{2,1}$? Note that for all $N$ we have
    $A_{1,2}^{(N)}=A_{2,1}^{(N)\ast}$. 
  Each of these matrices has independent Gaussian entries without any symmetry condition. For every $N\in \naturalnumbers$ we can write
  \begin{align*}
    A_{1,2}^{(N)}= {\frac{1}{\sqrt{2}}}(B_1^{(N)}+iB_2^{(N)}),
  \end{align*}
  with
  \begin{align*}
    B_1^{(N)}\equalperdefinition {\frac{1}{\sqrt{2}}}(A_{1,2}^{(N)}+A_{1,2}^{(N)^\ast}) \quad\text{and}\quad B_2^{(N)}\equalperdefinition {\frac{1}{\sqrt{2}i}}(A_{1,2}^{(N)}-A_{1,2}^{(N)\ast}),
  \end{align*}
  implying that $B_1^{(N)}$ and $B_2^{(N)}$ are independent and each of them is \GUEN. Since, again by Theorem~\ref{theorem:gue-family},
$(B_1^{(N)},B_2^{(N)})\convergesindistributionto (s_0,\tilde{s}_0)$, where $s_0$ and $\tilde s_0$ are free semicirculars,
 it follows
  \begin{align*}
    A_{1,2}^{(N)}\convergesindistributionto {\frac{1}{\sqrt 2}}(s_0+i\tilde s_0)\quad\text{and}\quad A_{2,1}^{(N)}=A_{1,2}^{(N)*}\convergesindistributionto {\frac{1}{\sqrt 2}}(s_0-i\tilde s_0).
  \end{align*}
  Such a random variable
  \begin{align*}
    c\equalperdefinition  {\frac{1}{\sqrt 2}}(s_0+i\tilde s_0)
  \end{align*}
  is called a \emph{circular} element.\par
  Since $A_{2N}$ itself, by Theorem~\ref{corollary:gue}, converges in distribution to a semicircular element $s$, this element $s$ can be represented as a matrix
  \begin{align*}
    s =\frac{1}{\sqrt{2}}
    \begin{bmatrix}
      s_1 & c \\
      c^\ast & s_2
    \end{bmatrix},
  \end{align*}
  where $s_1$ and $s_2$ are semicirculars, $c$ circular and $s_1,s_2,c$ are $*$-free. (Note for this that the matrices $B_1$ and $B_2$ are also independent from $A_{1,1}$ and $A_{2,2}$.) This shows the following.
\end{remark}

\begin{theorem}
  \label{theorem:semicircle-as-matrix}
  Let $(\mathcal A,\varphi)$ be a $\ast$-probability space, let therein $s_1$ and $s_2$ be semicircular elements, $c$ a circular element and $(s_1,s_2,c)$ $\ast$-free. Then, the random variable
  \begin{align*}
    \frac{1}{\sqrt{2}}
    \begin{bmatrix}
      s_1 & c \\
      c^\ast & s_2
    \end{bmatrix}
  \end{align*}
  is a semicircular element in the $\ast$-probability space $(M_2(\mathcal A),\normalizedtrace \otimes \varphi)$.
\end{theorem}

\begin{remark}
  \label{remark:circular-elements}
  The $\ast$-distribution of a circular element $c$ in a $\ast$-probability space $(\mathcal A,\varphi)$ is easy to describe in terms of the free cumulants $\kappa_n$ of $(\mathcal A,\varphi)$. Since,
  \begin{align*}
    c= {\frac{1}{\sqrt 2}}(s+i\tilde s) \quad \text{for free standard semicirculars $s$, $\tilde s$,}
  \end{align*}
  the second-order cumulants of $c$ are, according to the vanishing of mixed cumulants, Theorem~\ref{theorem:freeness-vanishing-mixed-cumulants}, of the free variables $s$ and $\tilde s$,
$$
        \kappa_2(c,c)=\frac{1}{2}(\kappa_2(s,s)-\kappa_2(\tilde s,\tilde s))=0,
\qquad\text{and}\qquad
    \kappa_2(c,c^\ast)=\frac{1}{2}(\kappa_2(s,s)+\kappa_2(\tilde s,\tilde s))=1
$$
  and, likewise,
    $\kappa_2(c^\ast,c^\ast)=0$ and $\kappa_2(c^\ast,c)=1$.
  All other $\ast$-cumulants of $c$ vanish; since, by Example~\hyperref[example:free-cumulants-2]{\ref*{example:free-cumulants}~\ref*{example:free-cumulants-2}}, only second-order cumulants of $s$ and $\tilde s$ are non-zero.
  \par
  We also need to understand, if $\mathcal A$ is a von Neumann algebra and $\varphi$ normal, the polar decomposition (see Fact~\hyperref[facts:von-neumann-algebras-4]{\ref*{facts:von-neumann-algebras}~\ref*{facts:von-neumann-algebras-4}}) of a circular element. We will see that this is of the form
    $c=uq$,
  where
  \begin{itemize}
  \item the random variable $q=\sqrt{c^\ast c}$ has a \emph{quartercircular} distribution, (note that $c^\ast c$ has the same distribution as $s^2$) meaning, for all $t\in \realnumbers$,
    \begin{align*}
      d\mu_q(t)=
      \begin{cases}
        \frac{1}{\pi}\sqrt{4-t^2}\, dt, &\text{if }t\in [0;2],\\
        0,&\text{otherwise.}        
      \end{cases}
            \quad\quad
                \begin{tikzpicture}[baseline=4em]
                        \begin{axis}[xmin=-1.1*2, ymin=-0.05, xmax=1.1*2, ymax=.35, axis lines=middle, xtick={-2,-1,0,1,2}, xticklabels={$-2$, $-1$, $0$, $1$, $2$}, ytick={0,0.318},yticklabels={$0$,$\frac{1}{\pi}$}, width=16em, height=11em, domain=0:2]
                          \addplot[thick, samples=150, smooth] {(0.5*pi^-1)*((4-x^2)^0.5)};
          \end{axis}
        \end{tikzpicture}    
    \end{align*}
  \item the random variable $u$ is a Haar unitary and
  \item the random variables $u$ and $q$ are $\ast$-free.
  \end{itemize}
  This can be proved via random matrices (by approximating a circular element by a sequence of non-symmetric Gaussian random matrices and investigating the polar decompositions of those -- that was the original proof of Voiculescu) or by our combinatorial machinery. The latter yields a generalization to the important class of $R$-diagonal operators. (Note also that, by  uniqueness of polar decomposition, the above polar decomposition of $c$ can be proved by showing that $uq$, with $u$ being a Haar unitary, $q$ a quartercircular and $u,q$ $\ast$-free, has a circular distribution.)
  \par
  The main point about the $\ast$-distribution of $c$ is that its non-zero cumulants are alternating in $c$ and $c^\ast$. This pattern is preserved for the polar part.
\end{remark}

\begin{definition}[Nica and Speicher 1997]
  Let $(\mathcal A,\varphi)$ be a tracial $\ast$-probability space. A random variable $a\in \mathcal A$ is called \emph{$R$-diagonal} if, for all $n\in \naturalnumbers$ and $\varepsilon:\,[n]\to \{1,\ast\}$ it holds that
  \begin{align*}
    \kappa_n(a^{\varepsilon(1)},\ldots,a^{\varepsilon(n)})=0
  \end{align*}
  whenever $n$ is odd or $\varepsilon$ not alternating in $1$ and $\ast$. Thus, for every $n\in \naturalnumbers$ the only non-vanishing $n$-th order $\ast$-cumulants of $a$ may be
  \begin{align*}
    \alpha_n\equalperdefinition \kappa_{2n}(a,a^\ast,a,a^\ast,\ldots,a,a^\ast)
= \kappa_{2n}(a^\ast,a,a^\ast,a,\ldots,a^\ast,a).
  \end{align*}
  We call $(\alpha_n)_{n\in \naturalnumbers}$ the \emph{determining sequence} of $a$.
\end{definition}

\begin{example}
  \begin{enumerate}
  \item By Remark~\ref{remark:circular-elements}, every circular element $c$ is $R$-diagonal and has the determining sequence
      $(1,0,0,0,\ldots)$,
    since $\kappa_2(c,c^\ast)=\kappa_2(c^\ast,c)=1$, and all other $*$-cumulants vanish.
  \item In Assignment~\hyperref[assignment-10]{10}, Exercise 2 we have seen that every Haar unitary $u$ is $R$-diagonal  and that its determining sequence is given by
      $\alpha_n=(-1)^{n-1}C_{n-1}$
    for every $n\in \naturalnumbers$, i.e.\
    \begin{align*}
      \kappa_2(u,u^\ast)=1,\qquad
      \kappa_4(u,u^\ast,u,u^\ast)=-1,\qquad
      \kappa_6(u,u^\ast,u,u^\ast,u,u^\ast)=2,\qquad
                                            \dots
    \end{align*}
    Multiples of Haar unitaries are the only \emph{normal} $R$-diagonal elements.
  \end{enumerate}
\end{example}

\begin{theorem}
  \label{theorem:distribution-r-diagonal-operator}
Let $a$ be an $R$-diagonal element with determining sequence  $(\alpha_n)_{n\in \naturalnumbers}$. Then, for all $n\in \naturalnumbers$,
  \begin{align*}
    \kappa_n(aa^\ast,aa^\ast,\ldots, aa^\ast)=\kappa_n(a^\ast a,a^\ast a,\ldots, a^\ast a)=\sum_{\pi\in \setofnoncrossingpartitionsof(n)}\alpha_\pi,%\hspace{-10em}\underset{\hspace{10em}\substack{\displaystyle\hspace{-9.5em} \nwarrow \\ \displaystyle \text{multiplicative functional}\\\displaystyle\text{corresponding to }(\alpha_n)_{n\in \naturalnumbers}\\\displaystyle\alpha_\pi=\prod_{V\in \pi}\alpha_{\# V}}}{\alpha_\pi}
  \end{align*}
  where $(\alpha_\pi)_{\pi\in \setofnoncrossingpartitionsof}$ is the multiplicative function corresponding to $(\alpha_n)_{n\in \naturalnumbers}$ (as in Remark~\ref{remark:moments-cumulants-convolution}), i.e., for all $n\in \naturalnumbers$ and $\pi\in \setofnoncrossingpartitionsof (n)$,
    $\alpha_\pi\equalperdefinition \prod_{V\in \pi}\alpha_{\# V}$.
  Hence, by M\"obius inversion (Corollary~\ref{corollary:moebius-inversion}), the  $\alpha_n$ -- and thus the $\ast$-distribution of $a$ -- are uniquely determined by the distribution of $a^\ast a$.
\end{theorem}

\begin{proof}
  For all $n\in \naturalnumbers$, by Theorem~\ref{theorem:cumulant-product-rule},
  \begin{align*}
    \kappa_n(aa^\ast,aa^\ast,\ldots,aa^\ast)&=\sum_{\substack{\sigma\in \setofnoncrossingpartitionsof(2n)\\\sigma\vee \hat{0}_n=1_{2n}}}\kappa_\sigma(a,a^\ast,a, a^\ast,\ldots, a,a^\ast),
  \end{align*}
  where $\hat{0}_n=\{ \{1,2\} , \{3,4\}, \ldots, \{2n-1,2n\} \}$.
  In order for a partition $\sigma\in \setofnoncrossingpartitionsof(n)$ to yield a non-zero contribution to this sum, by $R$-diagonality of $a$, in every block of $\sigma$, an $a$ must succeed each occurrence of an $a^\ast$. The condition $\sigma\vee\hat{0}_n=1_{2n}$ enforces then that this $a$ must be the one closest to the $a^\ast$. Hence, the contributing $\sigma$ must at least make the following connections:
  \begin{center}
  \begin{tikzpicture}
    \draw (1.5*0em,0em) -- ++ (0,-1.5em) -| (1.5*12em,0em);
    \draw (1.5*1em,0em) -- ++ (0,-1em) -| (1.5*2em,0em);
    \draw (1.5*3em,0em) -- ++ (0,-1em) -| (1.5*4em,0em);
    \draw (1.5*5em,0em) -- ++ (0,-1em) -| (1.5*6em,0em);
    \draw (1.5*8em,0em) -- ++ (0,-1em) -| (1.5*9em,0em);        
    \draw (1.5*10em,0em) -- ++ (0,-1em) -| (1.5*11em,0em);
    \draw[fill=white,draw=none] (-0.5em,-0.35em) rectangle (18.5em,0.65em);
    \node [right, inner sep=2pt] at (0em,0em) {$\hspace{-0.475em}a\hspace{1em} a^\ast\hspace{-0.2em}, $};
    \node [right, inner sep=2pt] at (1.5*2em,0em) {$\hspace{-0.475em}a\hspace{1em} a^\ast\hspace{-0.2em}, $};
    \node [right, inner sep=2pt] at (1.5*4em,0em) {$\hspace{-0.475em}a\hspace{1em} a^\ast\hspace{-0.2em}, \hspace{1.3em} \ldots \hspace{1.5em} ,$};
    \node [right, inner sep=2pt] at (1.5*9em,0em) {$\hspace{-0.475em}a\hspace{1em} a^\ast\hspace{-0.2em}, $};    
    \node [right, inner sep=2pt] at (1.5*11em,0em) {$\hspace{-0.475em}a\hspace{1em} a^\ast\phantom{\hspace{-0.2em},}$};
  \end{tikzpicture}
\end{center}

  Making the remaining connections for $\sigma$ is now the same problem as choosing a partition $\pi\in \setofnoncrossingpartitionsof(a^\ast a,a^\ast a, \ldots, a^\ast a)$ and the cumulant $\kappa_\sigma(a,a^\ast,a, a^\ast,\ldots, a,a^\ast)$ is then exactly given by $\alpha_\pi$.
\end{proof}

\begin{theorem}
  \label{theorem:product-with-r-diagonal-operator}
  Let $(\mathcal A,\varphi)$ be a tracial $\ast$-probability space, $a,b\in \mathcal A$ and $a,b$ $\ast$-free. If $a$ is $R$-diagonal, then so is $ab$.
\end{theorem}
\begin{proof}
  Let $(\kappa_n)_{n\in \naturalnumbers}$ be the free cumulants of $(\mathcal A,\varphi)$. We have to see that non-alternating $\ast$-cumulants in $ab$ vanish. Let $n\in \naturalnumbers$ be arbitrary. It suffices to consider the situation
  \begin{align*}
    \kappa_n(\ldots,ab,ab,\ldots)&=\sum_{\substack{\pi \in \setofnoncrossingpartitionsof(2n)\\\pi\vee \hat{0}_n=1_{2n}}}\kappa_\pi(\ldots,a,b,\overset{\downarrow}{a},b,\ldots).
  \end{align*}
  Since $a,b$ are $\ast$-free, a partition $\pi\in \setofnoncrossingpartitionsof(2n)$ can only yield a non-vanishing contribution if $\pi$ connects among $\{a,a^\ast\}$ and among $\{b,b^\ast\}$. Consider the block of such a $\pi$ containing $\overset{\downarrow}{a}$. Either $\overset{\downarrow}{a}$ is the first element in its block,
  \begin{center}
  \begin{tikzpicture}
      \draw (5.25em,0em) -- ++ (0,-1.5em) -- ++ (2em,0);
    \draw (11.2em,0em) -- ++ (0,-1.5em) -- ++ (-2em,0);
    \draw [dotted, shorten >= 2pt, shorten <= 2pt] ($(5.25em,0em)+(2em,-1.5em)$) --  ($(11.2em,0em)+(-2em,-1.5em)$);
    \node [right, fill=white, inner sep=1pt] at (0em,0.3em) {$\ldots \hspace{0.5em},\hspace{0.25em}a \hspace{0.5em} b\hspace{0.25em},\hspace{0.25em} \overset{\downarrow}{a}\hspace{0.5em} b \hspace{0.25em}, \hspace{0.5em} \ldots \hspace{0.5em} ,\hspace{0.25em}b^\ast \hspace{0.0em}a^\ast\hspace{0.0em} , \hspace{0.5em} \ldots$};
    \node[align=center] at (11.2em,-4em) {last element in this\\block must be $a^\ast$};
    \draw[->] (11.2em,-3em) -- ++ (0,1em);
  \end{tikzpicture}
\end{center}
or there is a preceding element in the block (which again must be an $a^\ast$):
\begin{center}
  \begin{tikzpicture}
          \draw (3.75em,0em) -- ++ (0,-1.5em) -- + (2em,0) -- + (-2em,0);
    \draw (10.5em,0em) -- ++ (0,-1.5em) -- + (-2em,0) -- + (2em,0);
    \draw [dotted, shorten >= 2pt, shorten <= 3pt] ($(3.75em,0em)+(2em,-1.5em)$) --  ($(10.5em,0em)+(-2em,-1.5em)$);
    \draw [dotted, shorten >= 2pt, shorten <= 2pt] ($(3.75em,0em)+(-2em,-1.5em)$) -- ++ (-2em,0em);
    \draw [dotted, shorten >= 2pt, shorten <= 2pt] ($(10.5em,0em)+(2em,-1.5em)$) -- ++ (2em,0em);
    \node [right, fill=white, inner sep=1pt] at (0em,0.3em) {$\ldots \hspace{0.5em} ,\hspace{0.25em}b^\ast \hspace{0.0em}a^\ast\hspace{0.0em} , \hspace{0.5em} \ldots \hspace{0.5em},\hspace{0.25em}a \hspace{0.5em} b\hspace{0.25em},\hspace{0.25em} \overset{\downarrow}{a}\hspace{0.5em} b \hspace{0.25em}, \hspace{0.5em} \ldots $};
    \node[align=center] at (3.75em,-4em) {preceding element\\must be $a^\ast$};
    \draw[->] (3.75em,-3em) -- ++ (0,1em);
  \end{tikzpicture}
\end{center}
  In either situation, the condition $\pi\vee \hat{0}_n=1_{2n}$ is not satisfied. Hence, there is no contributing $\pi$, implying $\kappa_ n(\ldots,ab,ab,\ldots)=0$.
\end{proof}

\begin{theorem}
  \label{theorem:r-diagonal-operator-polar-decomposition-distribution}
  Let $\mathcal M$ be a von Neumann algebra and $\tau:\,\mathcal M\to \complexnumbers$ a faithful normal trace. Let $a\in \mathcal M$ be an $R$-diagonal operator with respect to $\tau$. If $\ker(a)=\{0\}$, then the polar decomposition
    $a=uq$
  of $a$ has the following $\ast$-distribution:
  \begin{enumerate}
  \item[(i)]\label{theorem:r-diagonal-operator-polar-decomposition-distribution-1} The partial isometry $u$ is a Haar unitary.
  \item[(ii)]\label{theorem:r-diagonal-operator-polar-decomposition-distribution-2} The operator $q$ is positive with distribution $\mu_{\sqrt{a^\ast a}}$.
  \item[(iii)]\label{theorem:r-diagonal-operator-polar-decomposition-distribution-3} The operators $u$ and $q$ are $\ast$-free.
  \end{enumerate}
\end{theorem}
\begin{proof}
  Realize elements $\tilde u$ and $\tilde q$ satisfying $(i), (ii), (iii)$ in some von Neumann algebra $\tilde{ \mathcal {M}}$ with respect to a faithful normal trace $\tilde \tau$. Then, by Theorem~\ref{theorem:product-with-r-diagonal-operator}, the element
    $\tilde a  \equalperdefinition \tilde u\tilde q$
  is $R$-diagonal in $(\tilde {\mathcal {M}},\tilde \tau)$. Since the distribution of $\tilde a^\ast\tilde a$ in $(\tilde{ \mathcal{ M}},\tilde \tau)$ is by construction the same as the distribution of $a^\ast a$ in $(\mathcal M,\tau)$, the $\ast$-distributions of $\tilde a$ in $(\tilde{ \mathcal{ M}},\tilde \tau)$ and $a$ in $(\mathcal M,\tau)$ agree by Theorem~\ref{theorem:distribution-r-diagonal-operator}. Hence, the von Neumann algebras generated by $\tilde a$ in $\tilde{\mathcal M}$ and by $a$ in $\mathcal M$ are isomorphic, where this isomorphism also intertwines $\tau$ and $\tilde \tau$. Hence, the $\ast$-distributions of the polar decompositions of $\tilde a$ and of $a$ are the same. But the polar decomposition of $\tilde a$ in $\tilde {\mathcal {M}}$ is  $\tilde a=\tilde u\tilde q$: To see this one has to note
  \begin{align*}
    \ker(\tilde a)=\ker(\tilde q)=\ker(\tilde a^\ast\tilde a)=\ker(a^\ast a)=\ker(a)=\{0\}=\ker(u).
  \end{align*}
  Thus, the $\ast$-distribution of the random variables $(u,q)$ in $(\mathcal M,\tau)$, as appearing in the polar decomposition $a=uq$, is the one of $(\tilde u,\tilde q)$ in $(\tilde{ \mathcal{ M}},\tilde \tau)$, which is given by $(i), (ii), (iii)$.
\end{proof}

\newpage

%%%%%%%%%%%%%%%%%%%%%%%%%%%%

\section{Applications of Freeness to von Neumann Algebras: Compression of Free Group Factors}
%[Free Group Factor Compressions]

\begin{definition}
  A factor $\mathcal M$ is called $\mathrm{II}_1$-factor if $\mathcal M$ is infinite-dimensional and if there exists a norm-continuous trace $\tau:\mathcal M\to \complexnumbers$.
\end{definition}

\begin{remark}
  \begin{enumerate}
  \item The norm-continuous trace on a $\mathrm{II}_1$-factor is unique and automatically normal and faithful.
  \item Restricted to projections of $\mathcal M$, the trace $\tau$ gives a \enquote{dimension function}:
    \begin{enumerate}
    \item For all $t\in [0,1]$ there exists a projection $p\in \mathcal M$ (i.e.\ $p^\ast=p=p^2$) with $\tau(p)=t$.
    \item Two projections $p,q\in \mathcal M$  are equivalent (i.e.\ there exists a partial isometry $u\in \mathcal M$ with $u^\ast u=p$ and $uu^\ast=q$) if and only if $\tau(p)=\tau(q)$. (The number $\tau(p)$ is measuring the relative size of $\range(p)=p\mathcal H$, where $\mathcal H$ is a Hilbert space with $\mathcal M\subseteq \mathcal B(\mathcal H)$. The range has a \enquote{continuous dimension} $\in[0,1]$.)
    \end{enumerate}
  \item Given a $\mathrm{II}_1$-factor $\mathcal M$, we can build new factors by taking matrices over $\mathcal M$. But one can also go the opposite direction by compressing $\mathcal M$ by projections.
  \end{enumerate}
\end{remark}

\begin{definition}
  Let $\mathcal H$ be a Hilbert space and $\mathcal M\subseteq \mathcal B(\mathcal H)$ a $\mathrm{II}_1$-factor with faithful normal trace $\tau$.
  \begin{enumerate}
  \item Given any projection $p\in \mathcal M$, we consider
    \begin{align*}
      p\mathcal Mp\equalperdefinition \{pxp\mid x\in \mathcal M\}.
    \end{align*}
    This is again a von Neumann algebra with unit $p$, even a $\mathrm{II}_1$-factor, realized in $\mathcal B(p\mathcal H)$. And its von Neumann algebra isomorphism class depends only on the equivalence class of $p$, i.e.\ on $\tau(p)=\colon t\in [0,1]$. This isomorphism class is denoted by $\mathcal M_t$ and called a \emph{compression} of $\mathcal M$. By passing to matrices over $\mathcal M$, i.e.\ considering $M_{n\times n}(\mathcal M)$ for $n\in \naturalnumbers$, by then compressing $M_{n\times n}(\mathcal M)$ by projections and by using the unnormalized trace $\trace_n\otimes \tau$, one can define $\mathcal M_t$ for all $t\in \realnumbers^+$.
  \item We define the \emph{fundamental group} of $\mathcal M$ to be
    \begin{align*}
      \Fundgr(\mathcal M)\equalperdefinition \{t\in \realnumbers^+\mid \mathcal M_t\cong\mathcal M\}.
    \end{align*}
  \end{enumerate}
\end{definition}

\begin{remark}
  \label{remark:compressions-two-one-factors} 
 \begin{enumerate}
  \item\label{remark:compressions-two-one-factors-1} {Murray} and {von Neumann} showed that, for all $\mathrm{II}_1$-factors $\mathcal M$ and all $s,t\in \realnumbers^+$,
      $(\mathcal M_s)_t\cong \mathcal M_{st}$, 
    thus indeed, for all $n\in \naturalnumbers$,
      $(\mathcal M_n)_{{1}/{n}}\cong \mathcal M$ and   $\mathcal M_n\cong M_{n\times n}(\mathcal M)$,
    and that $\Fundgr(\mathcal M)$ is a multiplicative subgroup of $\realnumbers^+$.
  \item\label{remark:compressions-two-one-factors-2} They also proved that for the hyperfinite factor $\mathcal R$ one has
      $\Fundgr(\mathcal R)=\realnumbers^+$.
  \item\label{remark:compressions-two-one-factors-3} Nothing about compressions of the free group factors was known before the work of {Voiculescu}. Now, we know, for all $n,k\in \naturalnumbers$,
    \begin{align*}
      L(\mathbb F_n)_{{1}/{k}}\cong L(\mathbb F_m), \quad\text{where } \frac{m-1}{n-1}=k^2,\qquad
    \text{({Voiculescu} 1990)}
\end{align*}
 and
    \begin{align*}
      \Fundgr(L(\mathbb F_\infty))=\realnumbers^+,
\qquad\qquad\qquad
\text{({Radulescu 1992}).}
    \end{align*}

 We will prove {Voiculescu}'s result for the particular case $n=2$, $k=2$ and $m=5$.
\end{enumerate}
\end{remark}

\begin{theorem}[Voiculescu 1990]
  \label{theorem:voiculescu-free-factor-compression}
$ L(\mathbb F_2)_{{1}/{2}}\cong L(\mathbb F_5)$.
\end{theorem}
\begin{proof}
  By Remark~\hyperref[remark:compressions-two-one-factors-1]{\ref*{remark:compressions-two-one-factors}~\ref*{remark:compressions-two-one-factors-1}} we have
    $L(\mathbb F_5)\cong  (L(\mathbb F_5)_2)_{{1}/{2}}$ and $L(\mathbb F_5)_2\cong M_{2\times 2}(L(\mathbb F_5))$.
  Thus, the claim is equivalent to  
  \begin{align*}
    L(\mathbb F_2)\overset{!}{=}M_{2\times 2}(L(\mathbb F_5)).
  \end{align*}
  Hence, we should realize $L(\mathbb F_2)$ as $2\times 2$-matrices: Generalizing Theorem~\ref{theorem:semicircle-as-matrix}, we can find a $\ast$-probability space $(\mathcal A,\varphi)$ such that in $(M_{2\times 2}(\mathcal A),\normalizedtrace \otimes \varphi)$ the elements
  \begin{align*}
    x_1\equalperdefinition
    \begin{bmatrix}
      s_1 & c \\
      c^\ast & s_2
    \end{bmatrix}
               \quad\text{and}\quad
               x_2\equalperdefinition
               \begin{bmatrix}
                 u &0 \\
                 0& 2u
               \end{bmatrix}
  \end{align*}
  are $\ast$-free,  where, in $(\mathcal A,\varphi)$, both $s_1$ and $s_2$ are semicircular elements, $c$ is a circular variable,  $u$ is a Haar unitary and $s_1,s_2,c,u$ are $\ast$-free. (This can be proved via random matrices. We approximate $x_1$ by a sequence of \GUE\  and $x_2$ by a sequence of deterministic normal matrices which has as limiting distribution the uniform distribution on $\unitcircle \cup 2\unitcircle$.
  \begin{center}
  \begin{tikzpicture}
    \draw[->] (-5em,0em) to (5em,0em);
    \draw[->] (0em,-5em) to (0em,5em);
    \draw (2em,0em) arc (0:360:2em);
    \draw (4em,0em) arc (0:360:4em);
    \node[align=center] (lan) at (7em,5em) {uniform distribution\\on $\unitcircle\cup 2\unitcircle$};
    \draw[->]  (lan.south) -- (3em,3em);
    \draw[->]  (lan.south) -- (1.7em,1.5em);
    %\draw (2em, -0.25em) -- node[pos=-1] {$2$} (2em,0.25em);
    %\draw (4em, -0.25em) -- node[pos=-1] {$4$} (4em,0.25em);
    \draw (0em,0em) -- node[pos=0.4, above] {$1$} ++ (135:2em);
    \draw (0em,0em) -- node[pos=0.6, below] {$2$} ++ (225:4em);
    %\draw[brown] (-14em,-6em) rectangle (14em,6.5em);
    \useasboundingbox (-14em,-6em) rectangle (14em,6.5em);
  \end{tikzpicture}
\end{center}
Note, though, that we could have chosen any other non-atomic limit distribution for the deterministic sequence. The above form is a convenient choice to construct a projection of trace $\frac{1}{2}$, which we need for the compression.)
\par
Moreover, we can assume that $\mathcal A$ is a von Neumann algebra and $\varphi$ faithful and normal. Then, by Corollary~\ref{corollary:free-group-factors-isomorphisms}, the von Neumann subalgebra $\vonneumannalgebrageneratedby(s_1,s_2,c,u)$ of $\mathcal A$ generated by $\{s_1,s_2,c,u\}\cong\{s_1,s_2,\realpart(c),\imaginarypart(c),u\}$ is isomorphic to $L(\mathbb F_5)$. By the same corollary, we can infer $\vonneumannalgebrageneratedby (x_1,x_2)\cong L(\mathbb F_2)$ since $x_1$ and $x_2$ in $M_{2\times 2}(\mathcal A)$ are normal, have non-atomic distributions  and are $\ast$-free. Thus, we now know
\begin{align*}
L(\mathbb F_2)\cong  \mathcal N\equalperdefinition \vonneumannalgebrageneratedby (x_1,x_2)\subseteq M_{2\times 2}(L(\mathbb F_5)),
\end{align*}
but we do not know whether $\{x_1,x_2\}$ generates \emph{all} of $M_{2\times 2}(L(\mathbb F_5))$. That is what we still have to show.
\par
Since $u$ is unitary,
\begin{align*}
  x_2x_2^\ast=\begin{bmatrix}
                 u &0 \\
                 0& 2u
               \end{bmatrix}\begin{bmatrix}
                 u^{-1} &0 \\
                 0& 2u^{-1}
               \end{bmatrix}=
  \begin{bmatrix}
    1 & 0\\
    0 & 4
  \end{bmatrix}\in \mathcal N,
\end{align*}
Via measurable functional calculus, we conclude
\begin{align*}
\hspace{-8em}  \underset{\hspace{8em}\displaystyle =\colon p\hspace{2em} \parbox{6em}{projection of\\ trace $\frac{1}{2}$ in $\mathcal N$ }}{\underbrace{
\begin{bmatrix}
  1& 0\\
  0 & 0
\end{bmatrix}}}  \hspace{-7.1em},\quad
      \begin{bmatrix}
        0 &0 \\
        0 & 1
      \end{bmatrix}
\in \mathcal N,
\end{align*}
where $p$ is now a projection in $\mathcal N$ with trace $\normalizedtrace\otimes \varphi(p)=\frac{1}{2}$.
The presence of $p$ in $\mathcal N$ allows us to infer that $\mathcal N$ is generated by
\begin{align*}
  \begin{bmatrix}
    s_1&0\\
    0& 0
  \end{bmatrix},\quad\begin{bmatrix}
    0&c\\
    0& 0
  \end{bmatrix},\quad
         \begin{bmatrix}
    0&0\\
    0& s_2
  \end{bmatrix},\quad
         \begin{bmatrix}
    0&0\\
    0& u
  \end{bmatrix}\quad\text{and}\quad
         \begin{bmatrix}
    u&0\\
    0& 0
  \end{bmatrix}.
\end{align*}
We now replace the generator
\begin{align*}
  \begin{bmatrix}
    0&c\\
    0& 0
  \end{bmatrix} \qquad\text{by}\qquad  \begin{bmatrix}
    0&v\\
    0& 0
  \end{bmatrix}\quad\text{and}\quad  \begin{bmatrix}
    0&0\\
    0& q
  \end{bmatrix},
\end{align*}
where
\begin{align}
  \label{eq:theorem-voiculescu-free-factor-compression-1}
  \begin{bmatrix}
    0&c\\
    0& 0
  \end{bmatrix}
       =
 \begin{bmatrix}
    0&v\\
    0& 0
  \end{bmatrix}
  \begin{bmatrix}
    0&0\\
    0& q
  \end{bmatrix}
\end{align}
and
$c=vq$
is the polar decomposition of $c$ in $L(\mathbb F_5)$, i.e., $v$ is a Haar unitrary, $q$ is a quartercircular and $u$ and $q$ are $\ast$-free (see Remark~\ref{remark:circular-elements}). It then follows, that the two factors on the right hand side of identity \eqref{eq:theorem-voiculescu-free-factor-compression-1} are the polar
decomposition of the left hand side. Because $\mathcal N$ is a von Neumann algebra, Fact~\hyperref[facts:von-neumann-algebras-4]{\ref*{facts:von-neumann-algebras}~\ref*{facts:von-neumann-algebras-4}} yields that 
\begin{align*}
  \begin{bmatrix}
    0&v\\
    0& 0
  \end{bmatrix},\quad
  \begin{bmatrix}
    0&0\\
    0& q
  \end{bmatrix}\quad\in \mathcal N.
\end{align*}
Hence, we can indeed make the above replacement in the generators and conclude that $\mathcal N$ is generated by
\begin{align*}
  \underset{\displaystyle =\colon y_1}{\underbrace{\begin{bmatrix}
    s_1&0\\
    0& 0
  \end{bmatrix}}},\quad
  \underset{\displaystyle =\colon y_2}{\underbrace{\begin{bmatrix}
    0&v\\
    0& 0
  \end{bmatrix}}},\quad
  \underset{\displaystyle =\colon y_3}{\underbrace{\begin{bmatrix}
    0&0\\
    0& q
  \end{bmatrix}}},\quad
  \underset{\displaystyle =\colon y_4}{\underbrace{\begin{bmatrix}
    0&0\\
    0& s_2
  \end{bmatrix}}},\quad
  \underset{\displaystyle =\colon y_5}{\underbrace{\begin{bmatrix}
    u&0\\
    0& 0
  \end{bmatrix}}},\quad
  \underset{\displaystyle =\colon y_6}{\underbrace{\begin{bmatrix}
    0&0\\
    0& u
  \end{bmatrix}}}.       
\end{align*}
\par
A set of generators for the compression von Neumann algebra $p\mathcal N p$ is given by the collection of all elements 
  $py_{i_1} \ldots y_{i_n}p$
for  $n\in \naturalnumbers$ and  $i:[n]\to [6]$. 
Since
\begin{align*}
  1=
  \underset{\displaystyle = p^2}{\underbrace{\begin{bmatrix}
    1 & 0\\
    0 &0
  \end{bmatrix}}}
        +\hspace{-4.5em}
        \underset{\hspace{5em}\displaystyle =
        \underset{\displaystyle = \tilde{v}^\ast}{\underbrace{\begin{bmatrix}
          0 & 0\\
          v^\ast & 0
        \end{bmatrix}}}
        \underset{\displaystyle =\colon \tilde{v}}{\underbrace{
        \begin{bmatrix}
          0 & v\\
          0 & 0
        \end{bmatrix}}}                   
        }{\underbrace{
\begin{bmatrix}
    0 & 0\\
    0 &1
  \end{bmatrix},}}
\end{align*}
we can replace all instances of $1$ in
  $py_{i_1} \ldots y_{i_n}p=p\cdot y_{i_1}\cdot 1\cdot y_{i_2}\cdot \ldots 1\cdot y_{i_n}\cdot p$ by $1=p^2+\tilde v^*\tilde v$ and thus see that $p\mathcal N p$ is  generated by
$\{\, py_ip,\, py_i\tilde v^\ast,\, \tilde vy_ip,\, \tilde vy_i\tilde v^\ast \mid i\in [6]\, \},$
or, explicitly, by:
\begin{align*}
  \begin{bmatrix}
    s_1 & 0  \\
    0& 0
  \end{bmatrix},\quad
  \begin{bmatrix}
    1 & 0  \\
    0& 0
  \end{bmatrix},\quad
  \begin{bmatrix}
    vqv^\ast & 0  \\
    0& 0
  \end{bmatrix},\quad
  \begin{bmatrix}
    vs_2v^\ast & 0  \\
    0& 0
  \end{bmatrix},\quad
  \begin{bmatrix}
    u & 0  \\
    0& 0
  \end{bmatrix},\quad
  \begin{bmatrix}
    vuv^\ast & 0  \\
    0& 0
  \end{bmatrix},
\end{align*}
where once again we have used $vv^\ast=1$.
\par
By considering $p\mathcal Np$ realized on $p\mathcal{H}$ it is hence generated by
  $s_1, vqv^\ast, vs_2v^\ast, u, vuv^\ast$.
Each of those five elements is normal and has a non-atomic distribution; 
furthermore they are $*$-free since
 $ v\{q,s_2,u\}v^\ast$ is free from $\{s_1,u\}$ and the elements $q,s_2,u$ as well as the elements
$s_1,u$ are free.
\par
Thus, $p\mathcal Np$ is generated by five $\ast$-free elements with non-atomic distributions, proving $p\mathcal Np\cong L(\mathbb F_5)$ by Corollary~\ref{corollary:free-group-factors-isomorphisms}.
\end{proof}

\begin{remark}
  The same proof as for Theorem~\ref{theorem:voiculescu-free-factor-compression} works also for the first identity in Part~\ref{remark:compressions-two-one-factors-3} of Remark~\ref*{remark:compressions-two-one-factors} for arbitrary $n,k\in \naturalnumbers$. Actually, it is also true for $n=\infty$, where it says
    $L(\mathbb F_\infty)_{{1}/{k}}=L(\mathbb F_\infty)$ for all $k\in \naturalnumbers$,
  and hence
    $\rationalnumbers^+\subseteq \Fundgr(L(\mathbb F_\infty))$.
  For $n\in \naturalnumbers$, we obtain no direct consequence for $\Fundgr(L(\mathbb F_n))$ since we do not know whether different $L(\mathbb F_n)$ are isomorphic or not. However, the compression results show that some isomorphisms imply others, e.g.,
  \begin{align*}
    L(\mathbb F_2)\cong L(\mathbb F_3) \quad \text{implies}\quad \underset{\displaystyle =L(\mathbb F_5)}{\underbrace{L(\mathbb F_2)_{{1}/{2}}}}\cong \underset{\displaystyle =L(\mathbb F_9)}{\underbrace{L(\mathbb F_3)_{{1}/{2}}}}.
  \end{align*}
  A refinement of this (by defining also \enquote{interpolated free group factors} $L(\mathbb F_r)$ for all real $r>1$) finally yields the following dichotomy.
\end{remark}

\begin{theorem}[Dykema 1994, Radulescu 1994]
  Exactly one of the following statements is true:
  \begin{enumerate}
  \item All (interpolating) free group factors are isomorphic:
    \begin{align*}
      L(\mathbb F_r)\cong L(\mathbb F_s) \quad\text{for all }r,s\in (1,\infty].
    \end{align*}
    In this case,
    \begin{align*}
      \Fundgr(L(\mathbb F_r))=\realnumbers^+\quad \text{for all }r\in (1,\infty].
    \end{align*}
  \item The (interpolating)  free group factors are pairwise non-isomorphic:
    \begin{align*}
      L(\mathbb F_r)\not\cong L(\mathbb F_s) \quad\text{for all }r,s\in (1,\infty]\text{ with }r\neq s.
    \end{align*}
    In this case,
    \begin{align*}
      \Fundgr(L(\mathbb F_r))=\{1\} \quad \text{for all }r\in (1,\infty).
    \end{align*}
  \end{enumerate}
\end{theorem}

\newpage
%%%%%%%%%%%%%%%%%%%%%%%%%

\section{Some More Operator Algebraic Applications of Free Probability}
%[Further Applications]
\begin{remark}
  The asymptotic freeness results for \GUE\  (in their almost-sure convergence versions) tell us that many random matrix sequences converge in distribution to free semicirculars, i.e.\ to generators of the free group factors. This gives us the hope that typical properties of matrices might survive somehow in the free group factors. This is indeed the case. And this idea resulted in some of the most spectacular operator-algebraic successes of free probability. For example, results as the following can be obtained this way.
\end{remark}

\begin{theorem}
  \label{theorem:more-applications}
  Consider $n\in \naturalnumbers$ with $n\geq 2$. Then, the free group factor $L(\mathbb F_n)$
  \begin{enumerate}
  \item\label{theorem:more-applications-1} does not have property $\Gamma$,
  \item\label{theorem:more-applications-2} does not have a Cartan subalgebra and
  \item\label{theorem:more-applications-3} is prime.
  \end{enumerate}
\end{theorem}

\begin{remark}
  \begin{enumerate}
  \item
    \begin{enumerate}
    \item Claim~\ref{theorem:more-applications-1} of Theorem~\ref{theorem:more-applications} is an old result of {Murray} and {von Neumann} and does not require free probability. But it illustrates the method and can be extended to other factors.
    \item Claim~\ref{theorem:more-applications-2} on the other hand is a spectacular result of {Voiculescu} from 1996, as it had been conjectured before  that every $\mathrm{II}_1$-factor has a Cartan subalgebra. (A \emph{Cartan} subalgebra of a von Neumann algebra $\mathcal M$ is a  maximal abelian subalgebra $\mathcal N$ whose normalizer $\{u\in \mathcal M\mid u\mathcal Nu^\ast=\mathcal N\}$ generates $\mathcal M$.)
    \item Lastly, Claim~\ref{theorem:more-applications-3} is a result by {Liming Ge} from 1998 saying that the free group factor $L(\mathbb F_n)$ cannot be written as $\mathcal M_1\otimes \mathcal M_2$ for $\mathrm{II}_1$-factors $\mathcal M_1$ and $\mathcal M_2$.
    \end{enumerate}
  \item Property $\Gamma$ was introduced by {Murray} and {von Neumann} in order to distinguish the hyperfinite and the free group factor(s). The hyperfinite factor $\mathcal R$ has $\Gamma$, the free group factor $L(\mathbb F_n)$ does not.
  \end{enumerate}
\end{remark}

\begin{definition}
  Let $\mathcal M$ be a $\mathrm{II}_1$-factor and $\tau:\mathcal M\to \complexnumbers$ its faithful normal trace.
  \begin{enumerate}
  \item   A bounded sequence $(t_k)_{k\in \naturalnumbers}$ in $\mathcal M$ is \emph{central} if
  \begin{align*}
    \lim_{k\to\infty}\|[x,t_k]\|_2=0 \quad \text{for all }x\in \mathcal M,
  \end{align*}
  where, for all $x,y\in \mathcal M$,
  \begin{align*}
    \|y\|_2\equalperdefinition \sqrt{\tau(y^\ast y)}
\qquad\text{and}\qquad
    [x,y]\equalperdefinition xy-yx.
  \end{align*}
\item A central sequence $(t_k)_{k\in \naturalnumbers}$ in $(\mathcal M,\tau)$ is \emph{trivial} if
  \begin{align*}
    \lim_{k\to\infty} \|t_k-\tau(t_k)\cdot 1\|_2=0.
  \end{align*}
  \item We say that $(\mathcal M,\tau)$ has \emph{property $\Gamma$} if there exists a non-trivial central sequence in  $(\mathcal M,\tau)$.
  \end{enumerate}  
\end{definition}

\begin{proof}[Proof of Part~\ref{theorem:more-applications-1} of Theorem~\ref{theorem:more-applications} (very rough sketch)]
Consider a von Neumann algebra $\mathcal{M}= \vonneumannalgebrageneratedby(x_1,\dots,x_n)$, which is generated by $n$ selfadjoint operators $x_1,\dots,x_n$. Assume that $\mathcal M$ has property $\Gamma$. So, there is a non-trivial central sequence $(t_k)_{k\in \naturalnumbers}$. By taking real or imaginary part, we can suppose that $t_k$ is self-adjoint for every $k\in \naturalnumbers$. By functional calculus we can deform them to projections which have to \enquote{stay away from $0$ and $1$}: There exists a sequence $(p_k)_{k\in \naturalnumbers}$ of projections in $\mathcal M$ such that there exists $\theta\in (0,\frac{1}{2})$ with $\theta<\tau(p_k)<1-\theta$ for all $k\in \naturalnumbers$ and with
    $\lim_{k\to\infty}\|[x,p_k]\|_2=0$ for all $x\in \mathcal M$;
  in particular, we have this for $x=x_i$ for all $i\in [n]$. Now, consider $N\times N$-matrices
$A_1,\dots,A_n$,
  which approximate $x_1,\ldots, x_n$ in distribution. For every $k\in \naturalnumbers$, our projection $p_{k}\in \mathcal M$ can be approximated by polynomials in $x_1,\ldots, x_n$. Hence, applying these polynomials to the matrices $A_1,\ldots, A_n$ and deforming the results again to projections yields a counterpart of  $p_k$ in $M_{N\times N}(\complexnumbers)$. This means that there exist projections $Q_k\in M_{N\times N}(\complexnumbers)$ such that
    $\trace(Q_k)\sim N\cdot \tau (p_k)$ (for $N\to\infty$)
  and such that
    $\|[ A_i,Q_k]\|_2<\omega_k$ for all $i\in [n]$,
for some sequence $(\omega_k)_{k\in \naturalnumbers}$ with $\lim_{k\to\infty}\omega_k=0$. The matrix $Q_k$ can now be diagonalized as
  \begin{align*}
    Q_k=U
    \begin{bmatrix}
      1 & 0\\
      0 &0
    \end{bmatrix}
    U^\ast
  \end{align*}
  for a unitary matrix $U\in M_{N\times N}(\complexnumbers)$.
We now rotate also all the $A_i$ into this basis,
  \begin{align*}
    U^\ast A_iU=
    \begin{bmatrix}
      B_{i}& C_{i}^\ast\\
      C_{i} & D_{i}
    \end{bmatrix}.
  \end{align*}
  Then $\omega_k>\|[ A_i,Q_k]\|_2$
  means that
  \begin{align*}
\omega_k>
    \left\|    \left[
\begin{bmatrix}
      B_{i}& C_{i}^\ast\\
      C_{i} & D_{i}
    \end{bmatrix},
                  \begin{bmatrix}
                    1 &0\\
                    0& 0
                  \end{bmatrix}
                       \right]\right\|_2
    =\left\|
                       \begin{bmatrix}
                        0 & -C_{i}^\ast\\
                         C_{i} &0
                       \end{bmatrix}
                                         \right\|_2
    =\sqrt{\frac{2}{N}\trace(C_{i}C_{i}^\ast)}.
  \end{align*}
  This means that the off-diagonal entries of \emph{all} the $A_1,\ldots, A_n$ (note that $n\geq 2$) in this basis must be small. But this is not what happens for typical independent \GUE s.
\end{proof}

\begin{remark}
  The technical details to make the preceding proof rigorous are quite heavy and rely on the concept of \enquote{free entropy}, which is a measure for how many matrix tuples $(A_1,\ldots,A_n)$ exist approximating $(x_1,\ldots,x_n)$ in distribution. Many of the open questions around the free group factors are related to technical questions about free entropy. This is still a very active area of research.
\end{remark}

\begin{remark}
  Another big result relying on the idea of approximating operators by matrices is the work of Haagerup on hyperinvariant subspaces. The famous \emph{invariant subspace problem} asks whether for every Hilbert space $\mathcal H$ and every operator $a\in \mathcal B(\mathcal H)$ there exists a closed subspace $\mathcal H_0$ of $\mathcal H$ with $\mathcal H\neq \mathcal H_0\neq \{0\}$ and $a\mathcal H_0\subseteq \mathcal H_0$. With respect to the decomposition $\mathcal H=\mathcal H_0\oplus \mathcal H_0^\perp$ the operator $a$ can then be written as
  \begin{align*}
    a=
    \begin{bmatrix}
      a_{11} &a_{12}\\
      0 &a_{22}
    \end{bmatrix}
  \end{align*}
  for operators $a_{11}\in \mathcal B(\mathcal H_0)$, $a_{12}\in \mathcal B(\mathcal H_0^\perp,\mathcal H_0)$ and $a_{22}\in \mathcal B(\mathcal H_0^\perp)$. If $a\in \mathcal M\subseteq \mathcal B(\mathcal H)$ for a $\mathrm{II}_1$-factor $\mathcal M$, the corresponding \emph{hyperinvariant subspace problem} asks additionally that the projection on $\mathcal H_0$ belong to $\vonneumannalgebrageneratedby(a)\subseteq \mathcal M$.
  \par
  For matrices and normal operators such  subspaces always exist. For Banach space versions there are counter examples (by {Enflo} and by {Read}). The big problem is for non-normal operators on Hilbert spaces (or in $\mathrm{II}_1$-factors). The idea of Haagerup goes as follows: Assume that there exists a sequence  $(A_N)_{N\in \naturalnumbers}$ of deterministic matrices such that
   $(A_N,A_N^*)\convergesindistributiontoo (a,a^*)$.
  For every $N\in \naturalnumbers$, the matrix $A_N$ has an invariant subspace and corresponding projections. However, in general, the sequence of those projections need not converge at all. But now, improve the situation by going over from $A_N$ to
    $\tilde A_N\equalperdefinition A_N+\varepsilon_NX_N,$
where $(\varepsilon_N)_{N\in \naturalnumbers}$ is a sequence in $\complexnumbers$ with $\lim_{N\to\infty}\varepsilon_N=0$ and where $(X_N)_{N\in \naturalnumbers}$ is a sequence of non-selfadjoint \GUE. Then, it still holds that
    $(\tilde A_N,\tilde A_N^*)\convergesindistributiontoo (a,a^*)$ {almost surely}.
So we have now many matrices converging in $*$-distribution to $a$. Their generic behavior is much better than that of the original $A_N$ and one can control the asymptotic spectral properties of $\tilde A_N$ for $N\to\infty$. This is, again, laden with tough technical details, but it led, in the end, to the following.
\end{remark}

\begin{theorem}[{Haagerup} and {Schultz} 2009]
  If $a$ is an operator in a $\mathrm{II}_1$-factor $\mathcal M$, then, for every Borel set $B\subseteq \complexnumbers$, there exists a unique closed $a$-invariant subspace affiliated with $\mathcal M$, 
  \begin{align*}
    a=
    \begin{bmatrix}
      a_{11} & a_{12}\\
      0 & a_{22}
    \end{bmatrix},
  \end{align*}
  such that the Brown measures (a generalization of the distribution $\mu_a$ for non-normal $a$) of $a_{11}$ and $a_{22}$ are concentrated on $B$ and $\complexnumbers\backslash B$, respectively. In particular, if the Brown measure of $a$ is not a Dirac measures, then $a$ has a non-trivial hyperinvariant closed subspace.
\end{theorem}

\newpage

\renewcommand{\theta}{\vartheta}
\renewcommand{\phi}{\varphi}
\renewcommand{\epsilon}{\varepsilon}
\newcommand{\qqed}{\qed\medskip}
\newcommand{\N}{\mathbb N}
\newcommand{\Z}{\mathbb Z}
\newcommand{\Q}{\mathbb Q}
\newcommand{\R}{\mathbb R}
\newcommand{\C}{\mathbb C}
\newcommand{\F}{\mathcal F}
\newcommand{\K}{\mathbb K}
\newcommand{\A}{\mathcal A}
\newcommand{\B}{\mathcal B}
\newcommand{\M}{\mathfrak M}
\newcommand{\D}{\mathcal D}
\newcommand{\norm}[1]{\lVert {#1}\rVert}
\newcommand{\surj}{\twoheadrightarrow}
\newcommand{\inj}{\hookrightarrow}
\newcommand{\limn}[1]{\lim_{#1 \to\infty}}
\newcommand{\tn}[1]{\textnormal{#1}}
\newcommand{\platz}{\vspace{0.5cm}}
\newcommand{\Pkte}[1]{#1 points}
\newcommand{\pn}{\Pkte{10}}
\newcommand{\bw}{\vfill
\begin{flushright}
\textbf{\emph{bitte wenden}}
\end{flushright}
\pagebreak}
\newcommand{\veczwei}[2]{\left(\begin{matrix}#1\\#2\end{matrix}\right)}
\newcommand{\vecdrei}[3]{\left(\begin{matrix}#1\\#2\\ #3 \end{matrix}\right)}
\newcommand{\cA}{\mathcal{A}}
\newcommand{\cB}{\mathcal{B}}
\newcommand{\cP}{\mathcal{P}}
\newcommand{\ff}{\varphi}

\newcommand{\CC}{{\mathbb C}}
\newcommand{\NN}{\mathbb N}
\newcommand{\HH}{\mathcal H}
\newcommand{\DD}{{\mathbb D}}

\newcommand{\kk}{\kappa}

\makeatletter
\def\moverlay{\mathpalette\mov@rlay}
\def\mov@rlay#1#2{\leavevmode\vtop{%
\baselineskip\z@skip \lineskiplimit-\maxdimen
\ialign{\hfil$#1##$\hfil\cr#2\crcr}}}
\makeatother

\def\plangle{\moverlay{(\cr<}}
\def\prangle{\moverlay{)\cr>}}

\renewcommand{\Re}{\operatorname{Re}}
\renewcommand{\Im}{\operatorname{Im}}
\newcommand{\Tr}{\operatorname{Tr}}

\setenumerate{label=(\roman*),itemsep=3pt,topsep=3pt}

\section{Exercises}

\subsection*{Assignment 1}\label{assignment-1}
\begin{auf}
We consider, for a group $G$, the non-commutative probability space $(\CC G,\tau_G)$ from Example~\hyperref[{example:group-algebra}]{1.2}.
\begin{enumerate}
\item
Show that $\tau_G$ is tracial.
\item
We define on $\CC G$ the anti-linear mapping $^*$ by
$$(\sum_g \alpha_g g)^*:=\sum_g \bar \alpha_g g^{-1}.$$
Show that $(\CC G,\tau_G)$ becomes a $*$-probability space with respect to this structure, i.e. that $\tau_G$ is positive.
\item
Show that $\tau_G$ is faithful.
\end{enumerate}
\end{auf}

\begin{auf}
\begin{enumerate}
\item
Consider random variables $a_1,a_2,b_1,b_2$ such that $\{a_1,a_2\}$ and $\{b_1,b_2\}$ are free (which means that $a_1,a_2,b_1,b_2\in\cA$ for some non-commutative probability space $(\cA,\ff)$ such that the unital algebra generated by $a_1$ and $a_2$ is free from the unital algebra generated by $b_1$ and $b_2$). Show, from the definition of freeness, that we have then
$$\ff(a_1b_1a_2b_2)=\ff(a_1a_2)\ff(b_1)\ff(b_2)+\ff(a_1)\ff(a_2)\ff(b_1b_2) -\ff(a_1)\ff(b_1)\ff(a_2)\ff(b_2).$$
\item
Try to find a formula for
$\ff(a_1b_1a_2b_2a_3)$,
if $\{a_1,a_2,a_3\}$ and $\{b_1,b_2\}$ are free. Think about how much time it would take you to calculate a formula for $\ff(a_1b_1a_2b_2a_3b_3)$, if $\{a_1,a_2,a_3\}$ and $\{b_1,b_2,b_3\}$ are free.
\end{enumerate}
\end{auf}

\begin{auf}
\begin{enumerate}
\item
Prove that functions of freely independent random variables are freely independent; more precisely: if $a$ and $b$ are freely independent
and $f$ and $g$ are polynomials, then $f(a)$ and $g(b)$ are freely independent,
too.
\item
Assume that $(\cA,\ff)$ is a $*$-probability space and $\ff$ is faithful. Show the following: If two unital $*$-subalgebras $\cA_1,\cA_2\subseteq\cA$ are freely independent, then
$$\cA_1\cap \cA_2 =\CC 1.$$
\end{enumerate}
\end{auf}

\newpage

\subsection*{Assignment 2}\label{assignment-2}

\setcounter{auf}{0}

\begin{auf}
In this exercise we prove that free independence
behaves well under successive decompositions and thus is associative.
Consider a non-commutative probability space $(\cA,\ff)$.
Let $(\cA_i)_{i\in I}$ be unital subalgebras of $\cA$ and, for each $i\in I$, $(\cB_j^i)_{j\in J(i)}$
unital subalgebras of $\cA_i$. Denote the restriction of $\ff$ to $\cA_i$ by $\ff_i$. Note that then $(\cA_i,\ff_i)$ is, for each $i\in I$, a non-commutative probability space on its own.
Then we have:
\begin{enumerate}
\item
If $(\cA_i)_{i\in I}$ are freely independent in $(\cA,\ff)$ and, for each $i\in I$,
$(\cB_j^i)_{j\in J(i)}$ are freely independent in $(\cA_i,\ff_i)$, then all $(\cB_j^i)_{i\in I; j\in J(i)}$ are freely
independent in $(\cA,\ff)$.
\item
If all $(\cB_j^i)_{i\in I; j\in J(i)}$ are freely independent in $(\cA,\ff)$ and if, for each
$i\in I$, $\cA_i$ is as algebra generated by all $\cB_j^i$ for $j\in J(i)$, then $(\cA_i)_{i\in I}$
are freely independent in $(\cA,\ff)$.
\end{enumerate}
Prove one of those two statements!
\end{auf}

\begin{auf}
Let $(\cA,\ff)$ be a $*$-probability space. Consider a unital subalgebra $\cB\subseteq \cA$ and a Haar unitary $u\in\cA$, such that $\{u,u^*\}$ and $\cB$ are free. Show that then also $\cB$ and $u^*\cB u$ are free, where
$$u^*\cB u:=\{ u^*bu\mid b\in\cB\}.$$
\textit{Remark: A Haar unitary is a unitary $u\in\cA$, i.e., $u^\ast u=1=uu^\ast$ that satisfies $\varphi(u^k)=\delta_{0,k}$ for any $k\in\mathbb{Z}$.}
\end{auf}

\begin{auf}
Let $(\cA,\ff)$ be a non-commutative probability space, and let $a_i\in\cA$ ($i\in I$) be free. Consider a product $a_{i(1)}\dots a_{i(k)}$ with $i(j)\in I$ for $j=1,\dots,k$. Put $\pi:=\ker (i(1),\dots,i(k))\in\cP(k)$. Show the following.
\begin{enumerate}
\item
We can write $\ff(a_{i(1)}\dots a_{i(k)})$ as a polynomial in the moments of the $a_i$ where each summand contains at least $\#\pi$ many factors.
\item
If $\pi$ is crossing then $\ff(a_{i(1)}\dots a_{i(k)})$ can be written as a polynomial in moments of the $a_i$ where each summand contains at least $\#\pi+1$ many factors.
\end{enumerate}
\end{auf}

\newpage

\subsection*{Assignment 3}\label{assignment-3}
\setcounter{auf}{0}

\begin{auf}
Let $f(z)=\sum_{m=0}^\infty C_mz^m$ be the generating function (considered as formal power series) for the numbers $\{C_m\}_{m\in\NN_0}$, where the $C_m$ are defined by $C_0=1$ and by the recursion
$$C_m=\sum_{k=1}^m C_{k-1}C_{m-k},\qquad (m\geq 1).$$
\begin{enumerate}
\item
Show that
$$1+zf(z)^2=f(z).$$
\item
Show that $f$ is also the power series for $$\frac{1-\sqrt{1-4z}}{2z}.$$
\item
Show that
$$C_m=\frac 1{m+1}\binom{2m}m.$$
\end{enumerate}
\end{auf}

\begin{auf}
Show that the moments of the semicircular distribution are given by the Catalan numbers; i.e., for $m\in\NN_0$
$$\frac 1{2\pi}\int_{-2}^{+2} t^{2m} \sqrt{4-t^2}dt=\frac 1{m+1}\binom {2m}m.$$
\end{auf}

\begin{auf}
Let $\HH$ be an infinite-dimensional complex Hilbert space with orthonormal basis $(e_n)_{n=0}^\infty$. On $B(\HH)$ we define the state
$$\ff(a)=\langle e_0,ae_0\rangle.$$
We consider now the creation operator $l\in B(\HH)$ which is defined by linear and continuous extension of
$$l e_n=e_{n+1}.$$
\begin{enumerate}
\item
Show that its adjoint (``annihilation operator'') is given by extension of
$$l^* e_n=\begin{cases}
e_{n-1},& n\geq 1,\\
0,& n=0.
\end{cases}
$$
\item
Show that the operator $x=l+l^*$ is in the $*$-probability space $(B(\HH),\ff)$ a standard semicircular element.
\item
Is $\ff$ faithful on $B(\HH)$? What about $\ff$ restricted to the unital algebra generated by $x$?
\end{enumerate}
\end{auf}

\begin{auf}
Let $(\cA_n,\ff_n)$ ($n\in\NN$) and $(\cA,\ff)$ be non-commutative probability spaces. Let $(b_n^{(1)})_{n\in\NN}$ and $(b_n^{(2)})_{n\in\NN}$ be two sequences of random variables $b_n^{(1)},b_n^{(2)}\in\cA_n$ and let $b^{(1)},b^{(2)}\in\cA$. We say that
$(b_n^{(1)},b_n^{(2)})$ converges in distribution to $(b^{(1)},b^{(2)})$, if we have the convergence of all joint moments, i.e., if
$$\lim_{n\to\infty}\ff_n[b_n^{(i_1)}b_n^{(i_2)}\dots b_n^{(i_k)}]=
\ff( b^{(i_1)}b^{(i_2)}\dots b^{(i_k)})$$
for all $k\in\NN$ and all $i_1,\dots,i_k\in\{1,2\}$.

Consider now such a situation where $(b_n^{(1)},b_n^{(2)})$ converges in distribution to $(b^{(1)},b^{(2)})$. Assume in addition that, for each $n\in\NN$,
$b_n^{(1)}$ and $b_n^{(2)}$ are free in $(\cA_n,\ff_n)$. Show that then freeness goes over to the limit, i.e., that also $b^{(1)}$ and $b^{(2)}$ are free in $(\cA,\ff)$.

\end{auf}

\newpage

\subsection*{Assignment 4}\label{assignment-4}

\setcounter{auf}{0}

\begin{auf}
Let $(\cA,\ff)$ be a non-commutative probability space and let the unital subalgebras $\cA_i$ ($i\in I$) be freely independent.
Assume also that $\cA$ is generated as an algebra by the $\cA_i$. Show the following: If, for each $i\in I$, $\ff\vert_{\cA_i}$ is a homomorphism on $\cA_i$ then $\ff$ is a homomorphism on $\cA$.

\emph{Remark: The map $\ff:\cA\to\CC$ is a homomorphism if it is linear and multiplicative in the sense $\ff(ab)=\ff(a)\ff(b)$ for all $a,b\in\cA$.}

\end{auf}

\begin{auf}
Prove that $\#NC(k)=C_k$ by showing that $\#NC(k)$ satisfies the recursion relation for the Catalan numbers.
\end{auf}

\begin{auf}
Define, for $n\in\NN_0$,  the Bell numbers by $B_n:=\#\cP(n)$, where $B_0=1$.
\begin{enumerate}
\item
Show that the Bell numbers satisfy the recursion:
$$B_{n+1}=\sum_{k=0}^n \binom nk B_k.$$
\item
Show that the Bell numbers are also the moments of a Poisson distribution of parameter $\lambda=1$, i.e., that
$$B_n=\frac 1e \sum_{p=0}^\infty p^n \frac 1{p!}.$$
\end{enumerate}
\end{auf}

\newpage

\subsection*{Assignment 5}\label{assignment-5}

\setcounter{auf}{0}

\begin{auf}
Let $f,g\colon\N\to\C$ be functions and define the Dirichlet convolution by
\begin{align*}
f*g(n)=\sum_{d\mid n} f(d)g(\frac{n}{d}).
\end{align*}
We call $f\colon\N\to\C$ multiplicative if $f(mn)=f(m)f(n)$ for all $n,m$ with $\gcd(n,m)=1$ and define functions $\chi\colon\N\to\C$ by $\chi(n)=1$ for all $n\in\N$ and $\delta:\N\to\C$ by
$$\delta(n)=\begin{cases}
1, & n=1,\\
0, & n\not =1.
\end{cases}$$
Furthermore, we define
the (number theoretic) M\"obius function $\mu\colon\N\to\C$ by
\begin{itemize}
\item
$\mu(1)=1$
\item
$\mu(n)=0$ if $n$ is has a squared prime factor.
\item
$\mu(n)=(-1)^k$ if $n=p_1\dots p_k$ and all primes $p_i$ are different.
\end{itemize}
Note that $\chi$, $\delta$, and $\mu$ are multiplicative.
Show the following.

\begin{enumerate}
\item
Let $f$ and $g$ be multiplicative. Show that $f*g$ is multiplicative.
\item

Show that $\mu$ solves the inversion problem
\begin{align*}
g=f*\chi \iff f=g*\mu.
\end{align*}
[Hint: Show that $\mu*\chi=\delta$.]
\item
Show that Euler's phi function $\phi$ satisfies $\phi=f*\mu$, where $f$ is the identity function, i.e., $f(n)=n$ for all $n$. \\
\textit{Remark: The only information you need about Euler's phi function is that $$\sum_{d\mid n} \phi(d)=n$$}
\end{enumerate}
\end{auf}

\begin{auf}
Let $P=B_n$ be the poset of subsets of $[n]=\{1,\dots,n\}$, where $T\leq S$ means that $T\subseteq S$.
\begin{enumerate}
\item Show that the M\"obius function of this poset is given by
\begin{align*}
\mu(T,S)=(-1)^{\#S-\#T}\qquad\qquad (T\subseteq S\subseteq [n]).
\end{align*}
\item
Conclude from M\"obius inversion on this poset $B_n$ the following inclusion-exclusion principle: Let $X$ be a finite set and $X_1,\dots,X_n\subseteq X$. Then we have
\begin{align*}
\#(X_1\cup\cdots\cup X_n)=\sum_{k=1}^n (-1)^{k-1} \sum_{1\leq i_1<\cdots<i_k\leq n} \# (X_{i_1}\cap \cdots \cap X_{i_k}).
\end{align*}
[Hint: Consider the functions
\begin{align*}
f(I)=\# \bigcap_{i\in I}X_i\quad
\text{and} \quad
g(I)=\#\{ x\in X\mid x\in X_i \ \forall i\in I; x\not\in X_j \ \forall j\not\in I\}.
\end{align*}
You might also assume that $X=X_1\cup\cdots \cup X_n$.]
\end{enumerate}
\end{auf}

\begin{auf}
In this exercise we want to introduce another example of a $\ast$-probability space containing a semicircular random variable, namely the \emph{full Fock space}. Let ($V$,$\langle\cdot,\cdot\rangle$) be an inner product space over $\C$. Then we define the \emph{full Fock space} by
\begin{align*}
\mathcal{F}(V)=\C\Omega\oplus\bigoplus_{n=1}^\infty V^{\otimes n},
\end{align*}
where $\Omega$ is a unit vector, called \emph{vacuum vector}. We define on 
$\mathcal{F}(V)$
an inner product by extension of
\begin{align*}
\langle v_1\otimes\dots\otimes v_n,w_1\otimes\dots\otimes w_m \rangle_{\mathcal{F}(V)}=\delta_{nm} \langle v_1,w_1\rangle\dots\langle v_n ,w_n\rangle.
\end{align*}
(Note that, for $n=0$, this says that $\Omega$ is orthogonal to $V^{\otimes m}$ for all $m\geq 1$.)\\
Defining the vacuum expectation
\begin{align*}
\phi\colon\mathcal{A}:={B}(\mathcal{F}((V))\to\mathbb{C}, \quad a\mapsto\langle \Omega ,a\Omega\rangle_{\mathcal{F}(V)}
\end{align*}
makes $(\mathcal{A},\phi)$ into a $\ast$-probability space.
Let $v\in V$ be a non-zero vector, then we define a linear operator on $\mathcal{F}(V)$ by
\begin{align*}
l(v)\Omega = v,\quad l(v)v_1\otimes\dots\otimes v_n=v\otimes v_1\otimes\dots\otimes v_n
\end{align*}
and its adjoint by
\begin{align*}
l^\ast(v)\Omega = 0, \quad l^\ast(v)v_1\otimes\dots\otimes v_n=\langle v, v_1\rangle v_2\otimes\dots\otimes v_n.
\end{align*}
\begin{enumerate}
\item
Show that $x(v):=l(v)+l^\ast(v)$ is a standard semicircular element in $(\mathcal{A},\varphi)$ if $v$ is a unit vector.
\item
Let $v_1\perp v_2$ be orthogonal unit vectors in $V$. Show that $l(v_1)$ and $l(v_2)$ are $\ast$-free, i.e., that $\textrm{alg}(1,l(v_1),l^\ast(v_1))$ and $\textrm{alg}(1,l(v_2),l^\ast(v_2))$ are free. This means in particular that $x(u_1)$ and $x(u_2)$ are free semicirculars.
\end{enumerate}
\end{auf}

\begin{auf}
Let $(\mathcal{A},\phi)$ be a $\ast$-probability space.
\begin{enumerate}
\item
Assume that $\phi$ is a trace. Show that then $\kappa_n$ is invariant under cyclic permutations, i.e.
\begin{align*}
\kappa_n(a_1,a_2,\dots,a_n)=\kappa(a_2,\dots a_n,a_1)
\end{align*}
for all $a_1,\dots ,a_n\in\mathcal{A}$.
\item
Let $\phi$ be invariant under all permutations, i.e. $\phi(a_1\dots a_n)=\phi(a_{\sigma(1)}\dots a_{\sigma(n)})$ for all $n\in\N$ and all $\sigma\in S_n$. Are the free cumulants $\kappa_n$ invariant under all permutations?
\end{enumerate}
\end{auf}

\newpage

\subsection*{Assignment 6}\label{assignment-6}

\setcounter{auf}{0}

\begin{auf}
Let $b$ be a symmetric Bernoulli random variable; i.e., $b^2=1$ and its distribution (corresponding to the probability measure $\frac 12 (\delta_{-1}+\delta_{1})$) is given by the moments
$$\ff(b^n)=\begin{cases}
1,& \text{$n$ even,}\\
0,& \text{$n$ odd}.
\end{cases}$$
Show that the free cumulants of $b$ are of the form
$$\kappa_n(b,b,\dots,b)=
\begin{cases}
(-1)^{k-1}C_{k-1}, &\text{if $n=2k$ even,}\\
0,& \text{if $n$ odd}.
\end{cases}$$
\end{auf}

\begin{auf}
Prove Theorem~\hyperref[theorem:freeness-vanishing-mixed-cumulants-random-variables]{3.24}; i.e., show that for random variables $a_i$ ($i\in I$) in some non-commutative probability space $(\cA,\ff)$ the following statements are equivalent.

(i) The random variables $a_i$ ($i\in I$) are freely independent.

(ii) All mixed cumulants in the random variables vanish, i.e., for any $n\geq 2$ and all $i(1),\dots, i(n)\in I$ such that at least two of the indices $i(1),\dots,i(n)$ are different we have:
$$\kappa_n(a_{i(1)},a_{i(2)},\dots,a_{i(n)})=0.$$
\end{auf}

\begin{auf}
Let $x$ and $y$ be free random variables. Assume that $x$ is even, i.e., that all its odd moments vanish: $\ff(x^{2k+1})=0$ for all $k\in\NN_0$. Show that then also $y$ and $xyx$ are free.
\end{auf}

\begin{auf}
In this exercise we want to adress the question of the existence of non-commutative random variables for a prescribed distribution. Read Lecture 6 (Pages 91-102) of the book \href{https://rolandspeicher.files.wordpress.com/2019/02/nica-speicher.pdf}{``A. Nica, R. Speicher: Lectures on the Combinatorics of Free Probability''} and do its Exercise 6.16:
\\
Let $\mu$ be a probability measure with compact support on $\R$. Show that one can find a $\ast$-probability space $(\mathcal{A},\phi)$ where
$\phi$ is a faithful trace and a sequence $(x_n)_{n\geq 0}$ of freely
independent selfadjoint random variables in $\mathcal{A}$ such that each of the $x_i$'s has distribution $\mu$.
\end{auf}

\newpage

\subsection*{Assignment 7}\label{assignment-7}

\setcounter{auf}{0}

\begin{auf}
1) Show that the Catalan numbers are exponentially bounded by
$$C_n\leq 4^n\qquad \text{for all $n\in\NN$.}$$

2) Show that we have for the M\"obius function on $NC$ that
$$\vert\mu(0_n,\pi)\vert \leq 4^n\qquad\text{for all $n\in\NN$ and all $\pi\in NC(n)$.}$$
[Hint: Use the multiplicativity of the M\"obius function and the known value of the M\"obius function $\mu(0_n,1_n)=(-1)^{n-1}C_{n-1}$.]
\end{auf}

\begin{auf}
The complementation map $K:NC(n) \to NC(n)$ is defined as follows: We consider additional numbers $\bar 1,\bar 2,\dots,\bar n$
and interlace them with $1,\dots, n$ in the following alternating way:
$$1\, \bar 1 \,2\, \bar 2\, 3\, \bar 3\, \dots n\,\bar n.$$
Let $\pi\in NC(n)$ be a non-crossing partition of $\{1,\dots,n\}$. Then its \emph{Kreweras
complement} $K(\pi)\in NC(\bar 1,\bar 2,\dots,\bar n)\cong NC(n)$ is defined to be the biggest element
among those $\sigma\in NC(\bar 1,\bar 2,\dots,\bar n)$ which have the property
 that
$\pi\cup\sigma\in NC(1,\bar 1,2,\bar 2,\dots,n,\bar n)$.
\begin{enumerate}
\item
Show that $K: NC(n)\to NC(n)$ is a bijection and a lattice anti-isomorphism, i.e., that $\pi\leq \sigma$ implies $K(\pi)\geq K(\sigma)$.
\item
As a consequence of (1) we have that
$$\mu(\pi,1_n)=\mu(K(1_n),K(\pi))\qquad\text{for $\pi\in NC(n)$.}$$
Show from this that we have
$$\vert \mu(\pi,1_n)\vert\leq 4^n\qquad\text{for all $n\in \NN$ and all $\pi\in NC(n)$.}$$
\item
Do we also have in general the estimate
$$\vert \mu(\sigma,\pi)\vert \leq 4^n\qquad\text{for all $n\in\NN$ and all
$\sigma,\pi\in NC(n)$ with $\sigma\leq \pi$?}$$

\item
Show that we have for all $\pi\in NC(n)$ that
$$\#\pi +\# K(\pi)=n+1.$$
\end{enumerate}
\end{auf}

\begin{auf}
In the full Fock space setting of Exercise 3 from Assignment~\hyperref[assignment-5]{5} consider the operators $l_1:=l(u_1)$ and $l_2:=l(u_2)$ for two orthogonal unit vectors; according to part (2) of that exercise $l_1$ and $l_2$ are $*$-free.
Now we consider the two operators
$$a_1:=l_1^*+\sum_{n=0}^\infty \alpha_{n+1}l_1^n\qquad
\text{and}\qquad
a_2:=l_2^*+\sum_{n=0}^\infty \beta_{n+1}l_2^n$$
for some sequences $(\alpha_n)_{n\in\NN}$ and $(\beta_n)_{n\in\NN}$ of complex numbers.
(In order to avoid questions what we mean with these infinite sums in general, let us assume that the $\alpha_n$ and the $\beta_n$ are such that the sum converges; for example, we could assume that $\vert \alpha_n\vert,\vert\beta_n\vert\leq r^n$ for some $0<r<1$.)
By the $*$-freeness of $l_1$ and $l_2$ we know that $a_1$ and $a_2$ are free.
Prove now that $a_1+a_2$ has the same moments as
$$a:=l_1^*+\sum_{n=0}^\infty (\alpha_{n+1}+\beta_{n+1})l_1^n.$$
\end{auf}

\begin{auf}
Let $l$ be the creation operator from Exercise 3, Assignment~\hyperref[assignment-3]{3} (or $l(v)$ from Exercise 3, Assignment~\hyperref[assignment-5]{5}) in the $*$-probability space $(B(\HH),\ff)$.
\begin{enumerate}
\item
Show that the only non-vanishing $*$-cumulants of $l$ (i.e., all cumulants where the arguments are any mixture of $l$ and $l^*$) are the second order cumulants, and that for those we have:
$$\kk_2(l,l)=\kk_2(l^*,l^*)=\kk_2(l,l^*)=0,\qquad \kk_2(l^*,l)=1.$$

\item
For a sequence $(\alpha_n)_{n\in\NN}$ of complex numbers we consider as in the previous exercise the operator
$$a:=l^*+\sum_{n=0}^\infty \alpha_{n+1}l^n.$$

Show that the free cumulants of $a$ are given by
$$\kk_n^a=\kk_n(a,a,\dots,a)=\alpha_n.$$

\item
Combine this with the previous exercise to show that the free cumulants are additive for free variables; i.e., that for the $a_1$ and $a_2$ from the previous exercise we have for all $n\in\NN$ that
$$\kk_n^{a_1+a_2}=\kk_n^{a_1}+\kk_n^{a_2}.$$
[This is of course a consequence of the vanishing of mixed cumulants; however, the present proof avoids the use of the general cumulants functionals and was Voiculescu's original approach to the additivity of the free cumulants.]

\end{enumerate}
\end{auf}

\newpage

\subsection*{Assignment 8}\label{assignment-8}

\setcounter{auf}{0}

\begin{auf}
A free Poisson distribution $\mu_\lambda$ with parameter $\lambda>0$ is defined in terms of cumulants by the fact that all free cumulants are equal to $\lambda$. Calculate from this the $R$-transform and the Cauchy transform of $\mu_\lambda$. Use the latter to get via the Stieltjes inversion formula an explicit form for $\mu_\lambda$.

[Note that for $\lambda<1$ the distribution has atomic parts.]
\end{auf}

\begin{auf}
Let $x_1$ and $x_2$ be free symmetric Bernoulli variables; i.e., $x_1$ and $x_2$ are free, $x_1^2=1=x_2^2$, and
$$\mu_{x_1}=\mu_{x_2}=\frac 12 (\delta_{-1}+\delta_{+1}).$$
\begin{enumerate}
\item
Show directly, by using the definition of freeness, that the moments of $x_1+x_2$ are given by
$$\ff((x_1+x_2)^n)=\begin{cases}
0,&\text{$n$ odd,}\\
\binom {2m}m,& \text{$n=2m$ even}.
\end{cases}
$$
\item
Show that the moments from (1) are the moments of the arcsine distribution, i.e., show that
$$\frac 1\pi\int_{-2}^2  t^{2m} \frac 1{\sqrt{4-t^2}}dt=\binom {2m}m.$$
\end{enumerate}
\end{auf}

\begin{auf}
\begin{enumerate}
\item
Let $\lambda>0$ and $\nu$ be a probability measure on $\R$ with compact support. Show that the limit in distribution for $N\to\infty$ of
$$[(1-\frac \lambda N)\delta_0 +\frac \lambda N\nu]^{\boxplus N}$$
has free cumulants
$(\kk_n)_{n\geq 1}$ which are given by
$$\kk_n=\lambda \cdot m_n(\nu),\qquad
\text{where}\qquad
m_n(\nu)=\int t^n d\nu(t)$$
is the $n$-th moment of $\nu$.

[A distribution with those cumulants is called a \emph{free compound Poisson distribution} (with rate $\lambda$ and jump distribution $\nu$).]
\item
Let $s$ and $a$ be two free selfadjoint random variables in some $*$- probability space $(\cA,\ff)$ such that $s$ is a standard semicircular element and $a$ has distribution $\mu_a=\nu$. Show the following.
The free cumulants of $sas$ are given by
$$\kk_n(sas,sas,\dots,sas)=\ff(a^n)\qquad \text{for all $n\geq 1$},$$
i.e., $sas$ is a compound Poisson element with rate $\lambda=1$ and jump distribution $\nu$.
\end{enumerate}

\end{auf}

\newpage

\subsection*{Assignment 9}\label{assignment-9}

\setcounter{auf}{0}

\begin{auf}
In the proof of Theorem~\hyperref[theorem:subordination-iteration]{5.11} we used the following consequence of the Schwarz Lemma (alternatively, one can address this also as the simple part of the Denjoy-Wolff Theorem). Suppose $f:\DD\to\DD$ is a non-constant holomorphic function on the unit disc $$\DD:=\{z\in\CC\mid \vert z\vert <1\}$$ and it is not an automorphism of $\DD$ (i.e., not of the form $\lambda (z-\alpha)/(1-\bar \alpha z)$ for some $\alpha\in \DD$ and $\lambda \in\CC$ with $\vert \lambda\vert =1$). If there is a $z_0\in\DD$ with $f(z_0)=z_0$, then for all $z\in\DD$,
$f^{\circ n}(z)\to z_0$. In particular, the fixed point is unique.

Prove this by an application of the Schwarz Lemma.
\end{auf}

\begin{auf}
Let $\mu$ be the standard semicircular distribution (i.e., $R_\mu(z)=z$) and $\nu$ be the free Poisson distribution of parameter 1 (i.e., $R_\nu(z)=1/(1-z)$).
Calculate (explicitly or numerically) the distribution $\mu\boxplus \nu$, by producing plots for its density, via
\begin{enumerate}
\item
determining its Cauchy transform from its $R$-transform:
$$R(z)=z+\frac 1{1-z}.$$
\item
determining its Cauchy transform $G$ from the subordination equation:
$$G(z)=G_\nu (z-G(z)).$$
\end{enumerate}

\end{auf}

\begin{auf}
A probability measure $\mu$ on $\R$ is called \emph{infinitely divisible (in the free sense)} if, for each $N\in\NN$, there exists a probability measure $\mu_N$ on $\R$ such that
$$\mu=\mu_N^{\boxplus N}.$$
(This is equivalent to requiring that the free convolution semigroup $\{\mu^{\boxplus t}\mid t\geq 1\}$ can be extended to all $t\geq 0$; the $\mu_N$ from above are then $\mu^{\boxplus 1/N}$.)
\begin{enumerate}
\item
Show that a free compound Poisson distribution (which was defined on Assignment~\hyperref[assignment-8]{8}, Exercise~3) is infinitely divisible.
\item
Show the the $R$-transform of a free compound Poisson distribution with rate $\lambda$ and jump distribution $\nu$ is given by
$$R(z)=\lambda \int \frac t{1-tz}d\nu(t),$$
and thus can be extended as an analytic function to all of $\C^-$.
\item
Show that a semicircular distribution is infinitely divisible.
\item

Show that a semicircular distribution can be approximated in distribution by free compound Poisson distributions.

[One has that any infinitely divisible distribution can be approximated by free compound Poisson distributions. Furthermore, infinitely divisible distributions are characterized by the fact that their $R$-transforms have an analytic extension to $\CC^-$.]

\end{enumerate}
\end{auf}

\subsection*{Assignment 10}\label{assignment-10}

\setcounter{auf}{0}

\begin{auf}
\begin{enumerate}
\item
Let $A_N$ be a \GUEN\  and $D_N$ a deterministic $N\times N$-matrix, such that $D_N$ converges to $d$ where
$$\mu_d=\frac 12(\delta_{-1}+\delta_{+1}).$$
We know that then the eigenvalue distribution of $A_N+D_N$ converges to
$\mu_s\boxplus \mu_d$. Check this by comparing the density of $\mu_s\boxplus \mu_d$ with histograms of matrices $A_N+D_N$ for large $N$.
\item
Let $U_N$ be a Haar unitary $N\times N$ random matrix. Let $A_N$ and $B_N$ be deterministic $N\times N$ matrices such that $A_N$ converges to $a$ and $B_N$ converges to $b$, where
$$\mu_a=\frac 12 (\delta_{-1}+\delta_{+1}),\qquad
\mu_b=\frac 14 (\delta_{-1}+2\delta_0+\delta_{+1}).
$$
We know that then the eigenvalue distribution of $U_NA_NU_N^*+B_N$ converges to $\mu_a\boxplus\mu_b$. Check this by comparing the density of
$\mu_a\boxplus\mu_b$ with eigenvalue histograms of matrices $U_NA_NU_N^*+B_N$ for large $N$.

\end{enumerate}

\end{auf}

\begin{auf}
Calculate the $*$-cumulants of a Haar unitary $u$. (For this compare also Exercise~1 in Assignment~\hyperref[assignment-6]{6}.)
\end{auf}

\begin{auf}

For permutations $\alpha,\beta\in S_m$ we define:
$$\vert \alpha \vert:=m-\# \alpha,\qquad
d(\alpha,\beta):=\vert \beta \alpha^{-1}\vert.$$
\begin{enumerate}
\item
Show that $\vert\alpha \vert$ is equal to the minimal non-negative integer $k$ such that $\alpha$ can be written as a product of $k$ transpositions.
\item
Show that $\vert \cdot\vert$ satisfies: $\vert \alpha \beta\vert \leq \vert\alpha\vert +\vert
\beta\vert$ for all $\alpha,\beta\in S_m$.
\item
Show that $d$ is a distance (or metric).
\end{enumerate}
\end{auf}

\begin{auf}
We put for $\gamma= (1,2,\dots,m-1,m)\in S_m$
$$S_{NC}(\gamma):=\{\alpha\in S_m\mid d(e,\alpha)+d(\alpha,\gamma)=m-1=d(e,\gamma)\};$$
i.e., elements in $S_{NC}(\gamma)$ are those permutations in $S_m$ which lie on a geodesic from the identity element $e$ to the long cycle $\gamma$. One can show that $S_{NC}(\gamma)$ is canonically isomorphic to $NC(m)$; hence one has an embedding from $NC(m)$ into the permutation group $S_m$.
\begin{enumerate}
\item
Identify this embedding for $m=1,2,3,4$.
\item
Check by a non-trivial example for $m=4$ that under this embedding a pair $\sigma,\pi$ in $NC(m)$ with $\sigma\leq \pi$ is mapped to a pair $\alpha,\beta$ in $S_m$ with $\vert \alpha^{-1}\beta\vert+\vert \alpha\vert +\vert \beta^{-1}\gamma\vert=m-1$.
\end{enumerate}
\end{auf}

%%%%%%%%%%%%%%%%%%%

\end{document}